\documentclass[11pt]{amsart}
\pdfoutput=1
\usepackage{amsmath, amssymb, amsthm, amsbsy, mathtools, bbm, amsfonts, color, verbatim, array, floatflt, yfonts}

\usepackage[margin=1.2in]{geometry}
\usepackage[all]{xy}
\usepackage[english]{babel}
\usepackage[shortlabels]{enumitem}
\usepackage{ stmaryrd } 

\usepackage{graphicx, color, overpic}
\usepackage[usenames,dvipsnames,svgnames,table]{xcolor}
\usepackage[pdfpagelabels]{hyperref}
\usepackage{cleveref}
\usepackage{multirow}
\usepackage{easybmat}
\usepackage{adjustbox}

\usepackage{tikz}
\newcommand*\circled[1]{\tikz[baseline=(char.base)]{
    \node[shape=circle,draw,inner sep=2pt] (char) {#1};}}

\theoremstyle{plain}
\newtheorem{thm}{Theorem}[section]
\newtheorem*{thm*}{Theorem}
\newtheorem{prop}[thm]{Proposition}
\newtheorem*{prop*}{Proposition}
\newtheorem{lemma}[thm]{Lemma}
\newtheorem*{lemma*}{Lemma}
\newtheorem{corollary}[thm]{Corollary}

\theoremstyle{definition}
\newtheorem{definition}[thm]{Definition}
\newtheorem{example}[thm]{Example}

\newtheorem{rmk}[thm]{Remark}

\newcommand{\cC}{\mathcal{C}}

\newcommand{\cR}{\mathcal{R}}

\newcommand{\R}{\mathbb{R}}
\newcommand{\Z}{\mathbb{Z}}

\newcommand{\Q}{\mathbb{Q}}

\newcommand{\cH}{\mathcal H}

\newcommand{\Id}{\operatorname{I}} 
\newcommand{\sW}{W} 
\newcommand{\aW}{{\overline{W}}} 
\newcommand{\TW}{T} 
\newcommand{\aH}{H} 
\newcommand{\sH}{H_0} 
\newcommand{\On}{\operatorname{O}(n)} 
\newcommand{\aS}{\overline{S}} 
\newcommand{\sS}{S} 

\newcommand{\Mov}{\textsc{Mov}} 
\newcommand{\Mod}{\textsc{Mod}} 
\newcommand{\ModH}{\textsc{Mod}_\aH} 
\newcommand{\ModW}{\textsc{Mod}_\aW} 
\newcommand{\Fix}{\textsc{Fix}} 
 
\newcommand{\Cmin}{\cC_{\min}} 
\newcommand{\Pc}{\operatorname{Pc}} 
\newcommand{\Cent}{\operatorname{C}} 

\newcommand{\Span}{\textsc{Span}}
\newcommand{\Range}{\operatorname{Im}}
\newcommand{\Ker}{\operatorname{Ker}}
\newcommand{\rank}{\operatorname{rk}}

\newcommand{\diag}{\operatorname{diag}}

\newcommand{\bfc}{\mathbf{c}}

\newcommand{\x}{\mathbf{x}}

\newcommand{\wbg}{w_{\beta,\gamma}}

\newcommand{\Bbg}{B_{\beta,\gamma}}

\newcommand{\Mbg}{M_{\beta,\gamma}}
\newcommand{\Sbg}{S_{\beta,\gamma}}
\newcommand{\Tbg}{T_{\beta,\gamma}}
\newcommand{\wbd}{w_{\beta,\delta}}

\newcommand{\Sbd}{S_{\beta,\delta}}

\newcommand{\Pd}{P_\delta}
\newcommand{\Qd}{Q_\delta}
\newcommand{\Pb}{P_\beta}

\newcommand{\Pbd}{P_{\beta,\delta}}
\newcommand{\Qbd}{Q_{\beta,\delta}}

\newcommand{\supp}{\operatorname{Supp}}

\newcommand{\xconj}{[x]}

\newcommand{\wconj}{[w]_\sW} 
\newcommand{\conj}[1]{[#1]} 

\newcommand{\coconj}[2]{\operatorname{C}(#1,#2)} 
\newcommand{\coconjW}[2]{\operatorname{C}(#1,#2)} 
\newcommand{\scoconjW}[2]{\operatorname{C}_\sW(#1,#2)} 
\newcommand{\CompW}{\operatorname{Comp}} 
\newcommand{\tcCoconj}{\operatorname{C}}

\numberwithin{equation}{subsection}

\definecolor{amethyst}{rgb}{0.6, 0.4, 0.8}
\definecolor{kellygreen}{rgb}{0.3, 0.73, 0.09}
\definecolor{americanrose}{rgb}{1.0, 0.01, 0.24}

\newcommand{\liz}[1]{\textcolor{kellygreen}{#1}}

\begin{document}

\hypersetup{pdfauthor={Milicevic, Schwer, Thomas},pdftitle={The geometry of conjugation in affine Coxeter groups}}

\title[Conjugation in Affine Coxeter groups]{The Geometry of Conjugation in \\ affine Coxeter Groups}

\author{Elizabeth Mili\'{c}evi\'{c}}
\address{Elizabeth Mili\'{c}evi\'{c}, Department of Mathematics \& Statistics, Haverford College, 370 Lancaster Avenue, Haverford, PA, USA}
\email{emilicevic@haverford.edu}

\author{Petra Schwer}
\address{Petra Schwer, Institute of Mathematics, Heidelberg University, Im Neuenheimer Feld 255, 69120 Heidelberg, Germany}
\email{schwer@uni-heidelberg.de}

\author{Anne Thomas}
\address{Anne Thomas, School of Mathematics \& Statistics, Carslaw Building F07,  University of Sydney NSW 2006, Australia}
\email{anne.thomas@sydney.edu.au}

\thanks{EM was supported by NSF Grant DMS 2202017. PS was supported by the DFG Project SCHW 1550/4-1. This research was also supported in part by ARC Grant DP180102437.}


\begin{abstract} 
We develop new and precise geometric descriptions of the conjugacy class $\conj{x}$ and coconjugation set $\coconj{x}{x'} = \{ y \in \aW \mid yxy^{-1} = x' \}$ for all elements $x,x'$ of any affine Coxeter group $\aW$. The centralizer of $x$ in $\aW$ is the special case $\coconj{x}{x}$. The key structure in our description of the conjugacy class $\conj{x}$ is the mod-set $\ModW(w) = (w-\Id)R^\vee$, where~$w$ is the finite part of $x$ and $R^\vee$ is the coroot lattice. The coconjugation set $\coconj{x}{x'}$ is then described by $\ModW(w')$ together with the fix-set of $w'$, where $w'$ is the finite part of $x'$. For any element $w$ of the associated finite Weyl group $\sW$, the mod-set of $w$ is contained in the classical move-set $\Mov(w) = \Range(w - \Id)$. We prove that the rank of $\ModW(w)$ equals the dimension of $\Mov(w)$, and then further investigate type-by-type the surprisingly subtle structure of the $\Z$-module $\ModW(w)$. As corollaries, we determine exactly when $\ModW(w) = \Mov(w) \cap R^\vee$, in which case our closed-form descriptions of conjugacy classes and coconjugation sets are as simple as possible.
\end{abstract}

\maketitle




\section{Introduction}\label{sec:Intro}

The study of conjugacy classes and centralizers in Coxeter groups has a long history, going back to classical work of Frobenius, Schur, Specht~\cite{Specht37}, and Young~\cite{Young30}. Carter~\cite{Carter72} gave the first systematic study of conjugacy classes in all finite Weyl groups, and Carter~\cite{Carter72} and Springer~\cite{Springer} studied centralizers of Coxeter elements in these groups. In his thesis from 1994, Krammer showed that there exists a polynomial-time algorithm for the conjugacy problem in all finite rank Coxeter groups (see~\cite{Krammer}). Geck and Pfeiffer gave precise descriptions of conjugacy classes in all finite Coxeter groups and characterized their minimal length elements (see~\cite{GeckPfeiffer} and~\cite[Chapter 3]{GeckPfeifferBook}, as well as an independent proof by He and Nie~\cite{HeNie12}). This characterization of minimal length elements was then established for all twisted Coxeter groups by Geck, Kim, and Pfeiffer~\cite{GKP}, all affine Coxeter groups by He and Nie~\cite{HeNie14}, and all infinite Coxeter groups by Marquis~\cite{Marquis21}. Motivated by the representation theory of Hecke algebras, much of this work dating back to \cite{GeckPfeiffer} has focused on obtaining minimal length representatives through the operation of cyclic shifts; that is, conjugation by one simple reflection at a time.

In this work, we take a distinctive new approach and develop precise geometric descriptions of the entire \emph{conjugacy class}
\[
\conj{x} = \{ yxy^{-1} \mid y \in \aW \} 
\]
and \emph{coconjugation set}
\[
\coconjW{x}{x'} = \{ y \in \aW \mid yxy^{-1} = x' \}
\]	
of arbitrary elements $x,x'$ in any affine Coxeter group~$\aW$. 
Note that $\Cent(x,x)$ is the centralizer of $x$ in $\aW$.

In our companion paper \cite{MST4}, we provide closed geometric descriptions of conjugacy classes and coconjugation sets for all split subgroups of the full isometry group of Euclidean space. Since every affine Coxeter group $\aW$ splits as a semidirect product $\aW = T \rtimes \sW$, where $T = \{ t^\lambda \mid \lambda \in R^\vee \}$ is the group of translations by elements of the associated coroot lattice~$R^\vee$, and $\sW$ is the associated finite Weyl group, all the statements in \cite{MST4} hold true for $\aW$.

As we explain further below, a key player in the results of \cite{MST4} is the \emph{mod-set} of an element $x \in \aW$, which we introduce as:
\[
\ModW(x)= (\Id - x)R^\vee=(x - \Id)R^\vee.
\]
The mod-set is a $\aW$-adapted version of the classical \emph{move-set} $
\Mov(x) = \Range(x - \Id) = (x - \Id)\R^n$.
For $x = t^\lambda w \in \aW$, with $\lambda \in R^\vee$ and $w \in \sW$, we have $\ModW(x) = \lambda + \ModW(w)$ (see~\cite[Lemma 2.2]{MST4}), and hence it is enough to study mod-sets for elements of the associated finite Weyl group $\sW$.
We view $\ModW(w)$ as a (free) $\Z$-submodule of $R^\vee$, and note that $\ModW(w)$ is thus a submodule of $\Mov(w) \cap R^\vee$ (see~\cite[Lemma 2.3]{MST4}).

In this paper, we give explicit algebraic descriptions of $\ModW(w)$ for all finite Weyl groups~$\sW$ and all $w \in \sW$. In particular, we  determine exactly when $\ModW(w)$ equals $\Mov(w) \cap R^\vee$, in which case we say that $w$ \emph{fills its move-set}. Combined with the results of~\cite{MST4}, we thus obtain the precise geometry of all conjugacy classes and coconjugation sets in all affine Coxeter groups.

\subsection{Geometric description of (co)conjugation}

The work in this paper refines and builds upon the results of \cite{MST4}. We now briefly review the main contributions of~\cite{MST4} in the setting of an affine Coxeter group $\aW = T \rtimes \sW$. Throughout, we assume that $\aW$ is irreducible. Recall that any $x \in \aW$ can be expressed uniquely as the product of a \emph{translation part} $t^\lambda $, where  $\lambda\in R^\vee$,  and a \emph{spherical part} $w \in \sW$. For $w \in \sW$, the \emph{fix-set} $\Fix(w) = \Ker (w - \Id)$ is the orthogonal complement of the move-set $\Mov(w) = \Range(w - \Id)$.

\begin{figure}[htb]
	\begin{minipage}{0.4\textwidth}
		\centering
		{\begin{overpic}[width=0.95\textwidth]{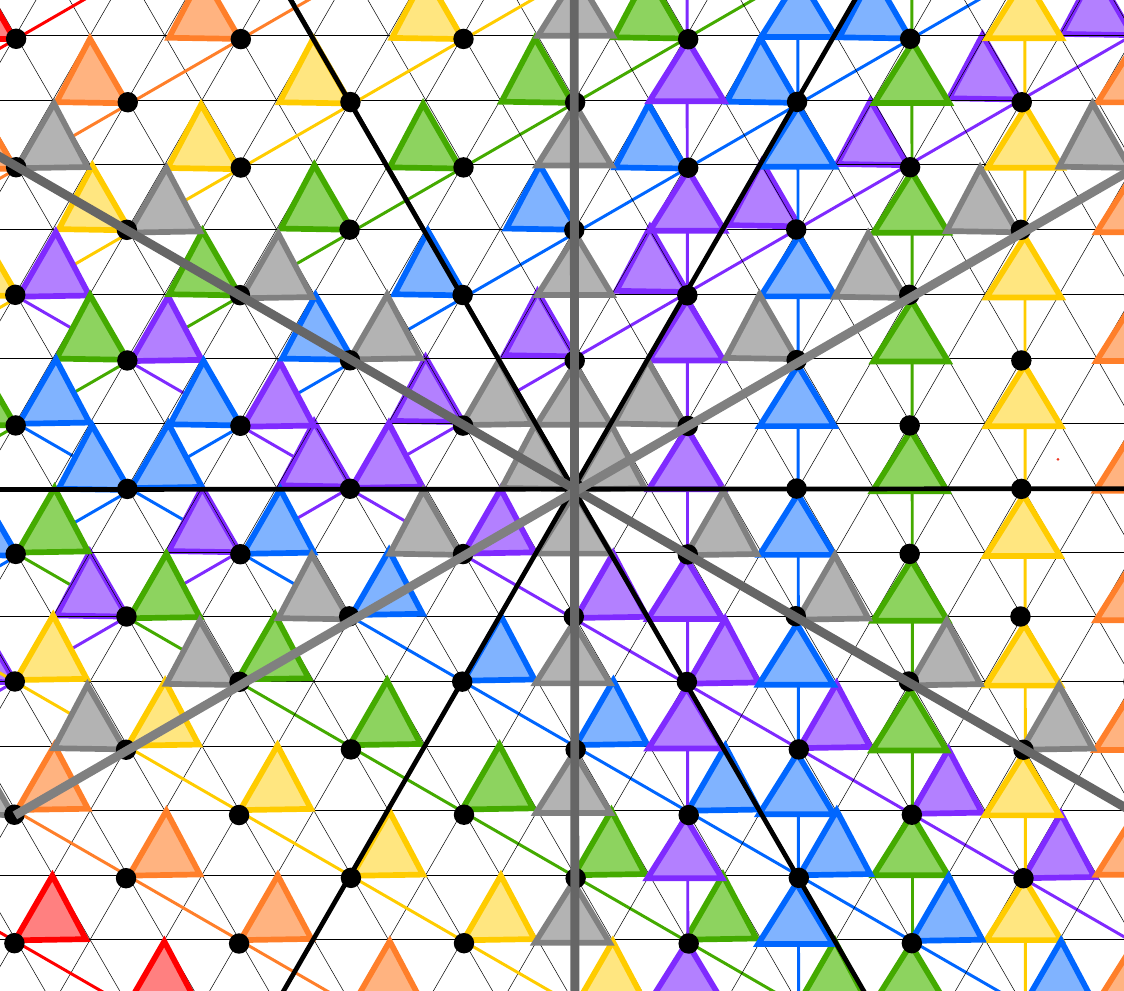}
				\put(50,47){\tiny{$\mathbf{1}$}}
				\put(50,52){\tiny{$\mathbf{s_0}$}}
				\put(46,46){\tiny{$\mathbf{s_1}$}}
				\put(53,46){\tiny{$\mathbf{s_2}$}}
				\put(49,40){\tiny{$\mathbf{w_0}$}}	
		\end{overpic}}
	\end{minipage}
	\begin{minipage}{0.56\textwidth}
		\centering
		{\begin{overpic}[width=0.95\textwidth]{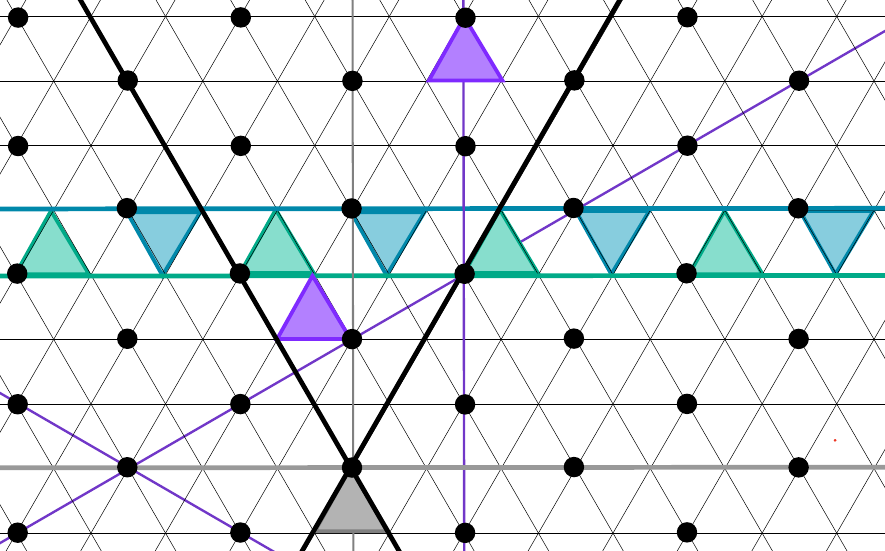}
				\put(34.5,26){\tiny{$\mathbf{x}$}}
				\put(51,54.5){\tiny{$\mathbf{x'}$}}
				\put(41,22){\tiny{$\mathbf{\lambda}$}}
				\put(54,61){\tiny{$\mathbf{\lambda'}$}}
				\put(74,29){\tiny{$\mathbf{\eta_{s_2}}$}}
				\put(83,41){\tiny{$\mathbf{\eta_{s_2 s_1}}$}}
				\put(38.5,13){\tiny{$\mathbf{1}$}}
				\put(34,12){\tiny{$\mathbf{s_1}$}}
				\put(43,12){\tiny{$\mathbf{s_2}$}}
				\put(32,7){\tiny{$\mathbf{s_1s_2}$}}
				\put(42,7){\tiny{$\mathbf{s_2s_1}$}}
				\put(37.5,4){\tiny{$\mathbf{w_0}$}}
				\put(74,10.5){\tiny{$\mathbf{\Fix(w_0)}$}}
				\put(54,17.5){\tiny{$\mathbf{\alpha_1^\vee}$}}
				\put(22,17.5){\tiny{$\mathbf{\alpha_2^\vee}$}}	
		\end{overpic}}
	\end{minipage}
	\caption{\footnotesize{On the left, some conjugacy classes $\conj{t^\lambda w}$ in type $\tilde{A}_2$, where $w$ is a reflection. On the right, the coconjugation set $\coconj{x}{x'}$. See Example \ref{eg:introA2} for details.}}
	\label{fig:introA2}
\end{figure}
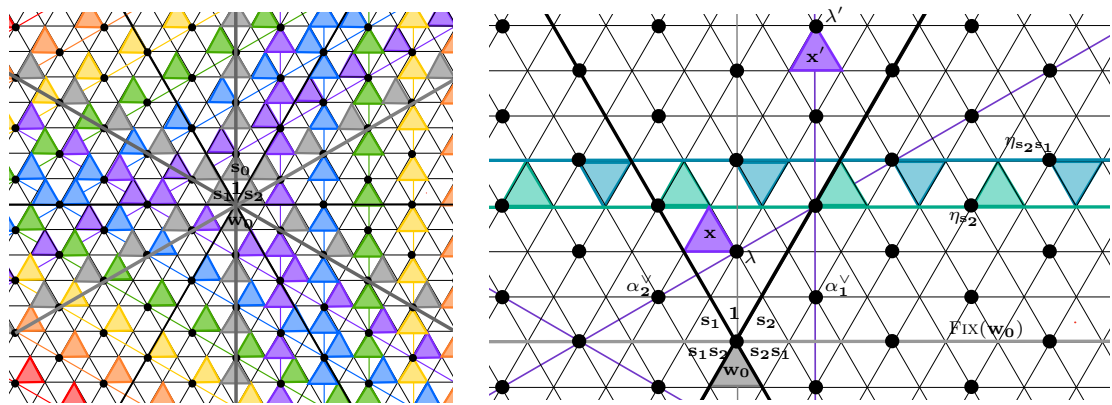

Two of the main findings of~\cite{MST4} can now be stated informally as follows:
\begin{itemize}
	\item For any $x = t^\lambda w \in \aW$, the conjugacy class $\xconj$ is obtained by first translating $x$ by all elements of $\ModW(w) \subseteq \Mov(w)$, then conjugating the so-obtained collection $t^{\ModW(w)}x$ by all elements of~$\sW$. 
	\item For any $x = t^\lambda w$ and $x' = t^{\lambda'}w'$ in $\aW$, the coconjugation set $\coconj{x}{x'}$ has a closed-form description involving $\ModW(w')$, and its shape is described by translates of $\Fix(w')$. 
\end{itemize}
We give formal statements of these and other results from~\cite{MST4} in Section~\ref{sec:summary}, and illustrate them via the following examples.

\begin{example}
	\label{eg:introA2} 
	Let $\aW$ be the affine Coxeter group of type $\tilde{A}_2$. Then $\aW = T \rtimes \sW$ where $\sW \cong S_3$ is the Weyl group of type $A_2$, generated by simple reflections $s_1$ and $s_2$, and the group $\aW$ is generated by $\{s_1, s_2\}$ together with the affine simple reflection $s_0$. The action of $\aW$ induces the tessellation of $\R^2$ by equilateral triangles, as depicted in Figure~\ref{fig:introA2}.  There is a natural bijection between the elements of $\aW$ and the tiles in this tessellation,  and we identify each $x$ in $\aW$ with its corresponding triangle.  The coroot lattice $R^\vee$ is the set of heavy dots in this figure. 
	
	On the left of Figure~\ref{fig:introA2}, each set of triangles shaded in the same color is a conjugacy class $\conj{t^\lambda w}$, where $w \in \sW$ is a reflection; that is, $w \in \{ s_1, s_2, w_0\}$, where $w_0 = s_1 s_2 s_1 = s_2 s_1 s_2$ is the longest element of $\sW$. 
	For each reflection $w \in \sW$, the move-set $\Mov(w)$ is the heavy gray line orthogonal to its fixed hyperplane, and the mod-set $\ModW(w)$ is the set of coroot lattice elements on this gray line. In other words, each reflection $w \in \sW$ fills its move-set. The other, colored lines on the left of Figure~\ref{fig:introA2} are move-sets $\Mov(t^\lambda w)$ for certain $\lambda \not \in \ModW(w)$.  Each conjugacy class $\conj{t^\lambda w}$ is thus a triple of ``lines" of triangles (of the same color). The remaining conjugacy classes $[t^\lambda w]$ for $w \in \sW$ a reflection are obtained by reflecting this picture in the vertical gray line $\ModW(w_0)$.

	The right of Figure~\ref{fig:introA2} depicts the coconjugation set $\coconjW{x}{x'}$, where $x = t^\lambda s_1$  and $x' = t^{\lambda'}w_0$. To describe $\coconjW{x}{x'}$, we first determine the $u \in \sW$ such that $us_1 u^{-1} = w_0$ and $\lambda' - u\lambda \in \ModW(w_0)$; this gives $u \in \{ s_2, s_2 s_1 \}$. Then for each such $u$, we translate the horizontal gray line $\Fix(w_0)$ by a particular solution $\eta_u \in R^\vee$ to the equation $\lambda' - u\lambda = (\Id - w_0)\eta$. The elements of $\coconjW{x}{x'}$ are the triangles $t^\mu u$ along these translates, as depicted in teal and aqua on the right of Figure~\ref{fig:introA2}.
	
	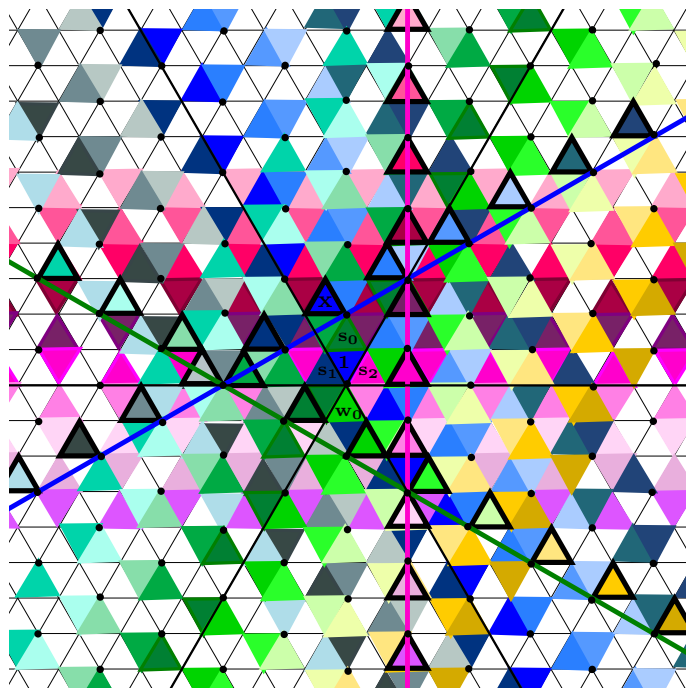
\begin{figure}[htb]
	{\begin{overpic}[width=0.58\textwidth]{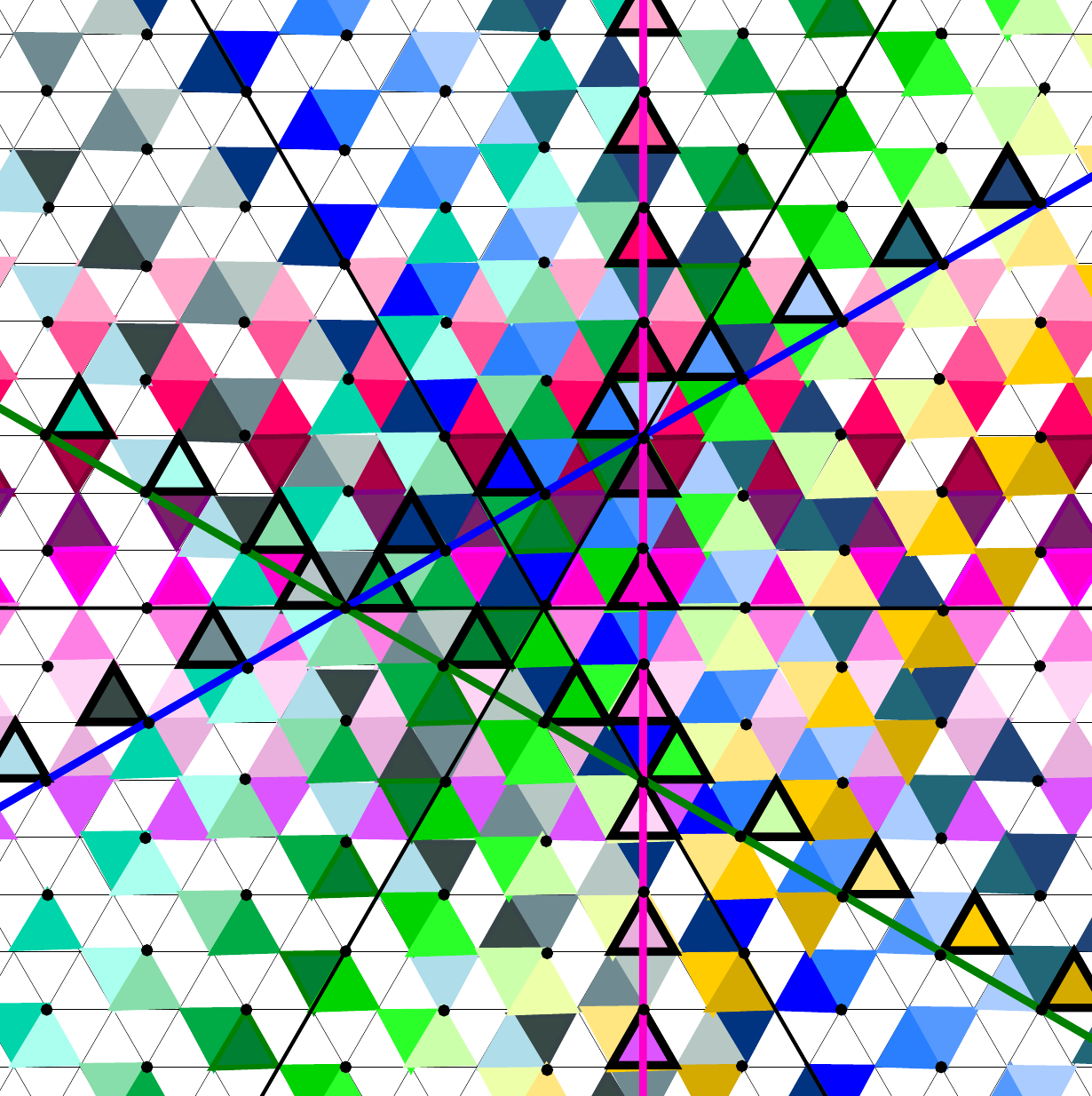}
			\put(48.5,47){\tiny{$\mathbf{1}$}}
			\put(45.5,56){\footnotesize{$\mathbf{x}$}}
			\put(48.5,51){\tiny{$\mathbf{s_0}$}}
			\put(45.5,46){\tiny{$\mathbf{s_1}$}}
			\put(51.5,46){\tiny{$\mathbf{s_2}$}}
			\put(48,40){\tiny{$\mathbf{w_0}$}}	
	\end{overpic}}
	\caption{Coconjugation sets $\coconj{x}{x'}$ in type $\tilde{A}_2$; see Example \ref{eg:introA2} for details.}
	\label{fig:coconjA2chimneys}
\end{figure}
	
	In Figure~\ref{fig:coconjA2chimneys}, which is adapted from~\cite[Figure 1.1]{MST2}, the elements of the conjugacy class $\conj{x}$ are outlined in black, for $x = t^{\lambda}s_1$ as shown in purple on the right of Figure~\ref{fig:introA2}. Then for each of the depicted elements $x' \in \conj{x}$, the coconjugation set $\coconj{x}{x'}$ is the collection of triangles of the same color as $x'$. We see that as $x' = t^{\lambda'}w'$ varies through $\conj{x}$, the elements of $\coconj{x}{x'}$ vary through translates of $\Fix(w')$. In particular,  the centralizer of $x$ is the ``band" of blue triangles running from top left to bottom right which contains $x$ and the identity element~$1$. Note that $x$ can be expressed as $x = s_0 s_1 s_2$, and hence is a Coxeter element of~$\aW$.
\end{example}

\begin{example}\label{eg:introC2} 
Now let $\aW = T \rtimes \sW$ be of type $\tilde{C}_2$, so that $\sW$ is of type $C_2$. Again, $\sW$ is generated by simple reflections $s_1$ and $s_2$, and $\aW$ by $\{ s_0, s_1, s_2\}$. In Figure~\ref{fig:conjC2}, which depicts some of the conjugacy classes $\conj{t^\lambda w}$ for $w \in \sW$ a reflection, we again see ``lines" of conjugates. On the left of this figure, the reflections $s_1$ and $s_2 s_1 s_2$ fill their move-sets, as in type $\tilde{A}_2$. However, on the right, the conjugacy classes leave ``gaps" along the colored lines which are the move-sets. This is due to the fact that the reflections $s_2$ and $s_1 s_2 s_1$ do not fill their move-sets. More precisely, for $w \in \{s_2, s_1 s_2 s_1\}$, the mod-set $\ModW(w)$ is an index 2 submodule of $\Mov(w) \cap R^\vee$.
\end{example}

	\begin{figure}[htb]
		\begin{minipage}{0.4\textwidth}
			\centering
			{\begin{overpic}[width=0.95\textwidth]{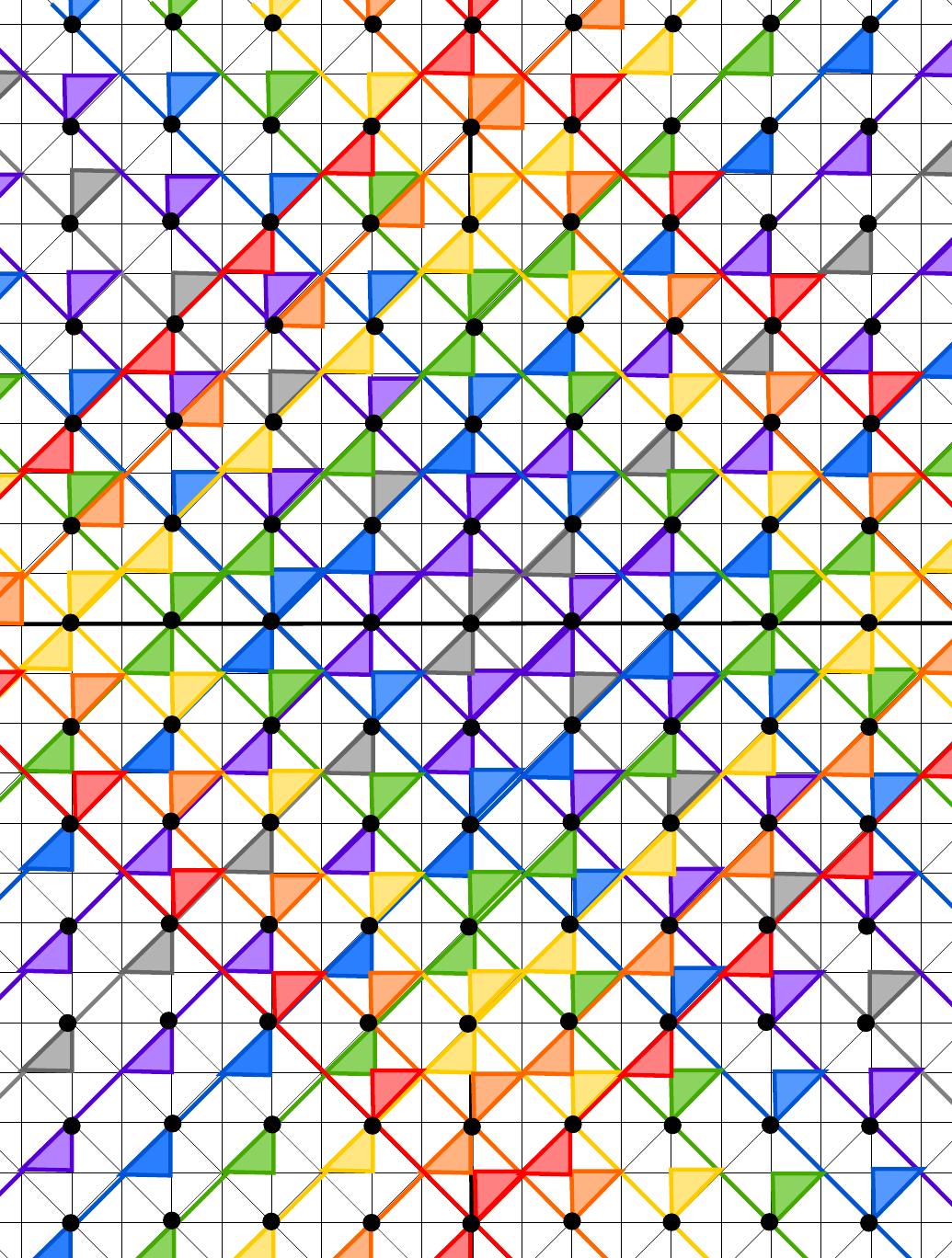}
					\put(39.5,50.8){\tiny{$\mathbf{1}$}}
					\put(42,51.2){\tiny{$\mathbf{s_0}$}}
					\put(36.5,52.7){\tiny{$\mathbf{s_1}$}}
					\put(39.3,48.3){\tiny{$\mathbf{s_2}$}}
					\put(44,40){\tiny{$\mathbf{\alpha_1^\vee}$}}
					\put(32.5,58){\tiny{$\mathbf{\alpha_2^\vee}$}}
			\end{overpic}}
		\end{minipage}
		\begin{minipage}{0.485\textwidth}
			\centering
			{\begin{overpic}[width=0.95\textwidth]{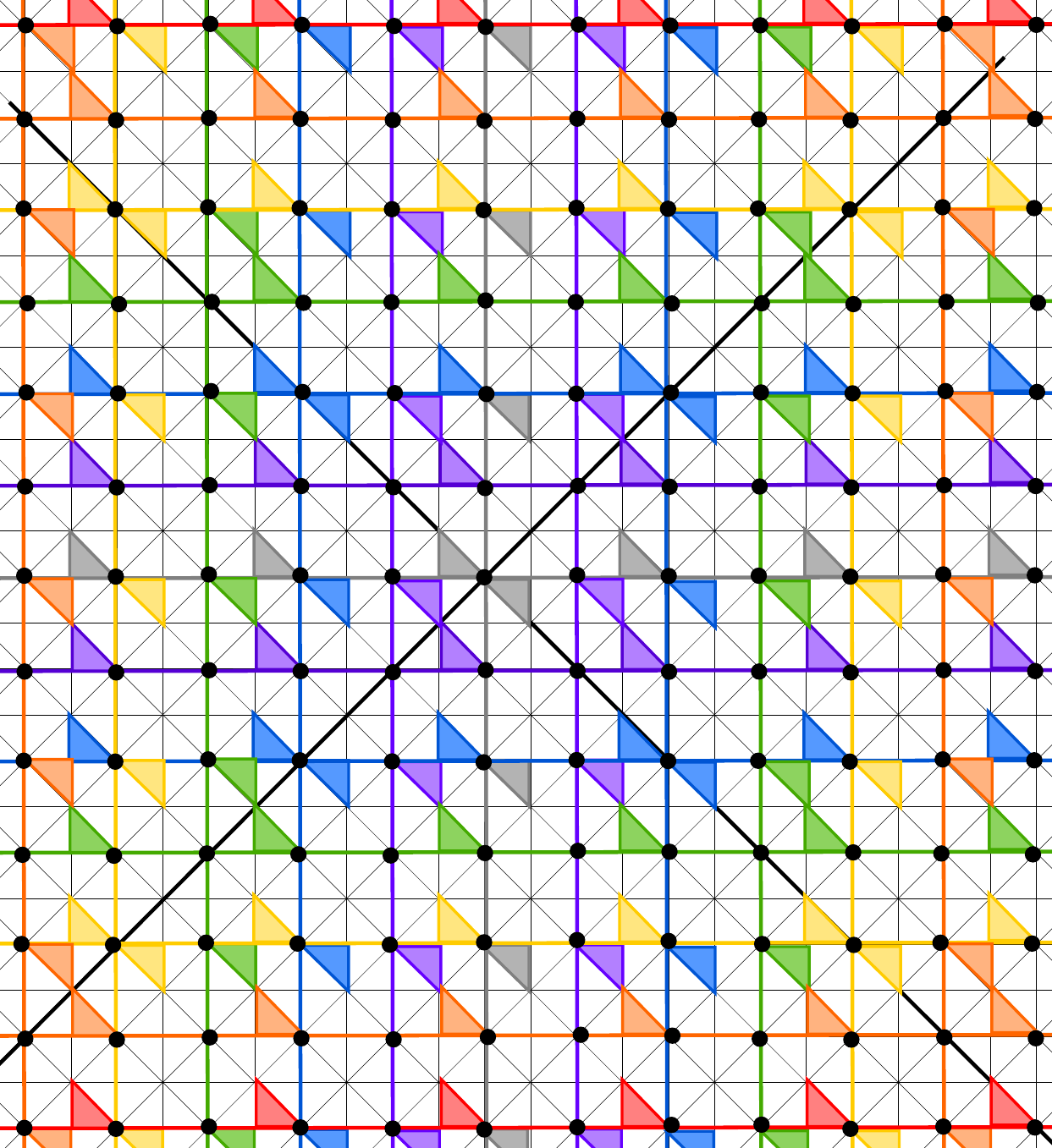}
					\put(44.5,50){\tiny{$\mathbf{1}$}}
					\put(46.5,50.5){\tiny{$\mathbf{s_0}$}}
					\put(41.5,52){\tiny{$\mathbf{s_1}$}}
					\put(44.3,47.5){\tiny{$\mathbf{s_2}$}}
					\put(48,39){\tiny{$\mathbf{\alpha_1^\vee}$}}
					\put(43,58){\tiny{$\mathbf{\alpha_2^\vee}$}}
			\end{overpic}}
		\end{minipage}
		\caption{\footnotesize{The conjugacy classes $\conj{t^\lambda w}$ in type $\tilde{C}_2$ for $w \in \{ s_1, s_2 s_1 s_2\}$ fill their move-sets (on the left), but do not fill their move-sets for $w \in \{ s_2, s_1 s_2 s_1\}$ (on the right); see Example \ref{eg:introC2} for details.}}
		\label{fig:conjC2}
	\end{figure}
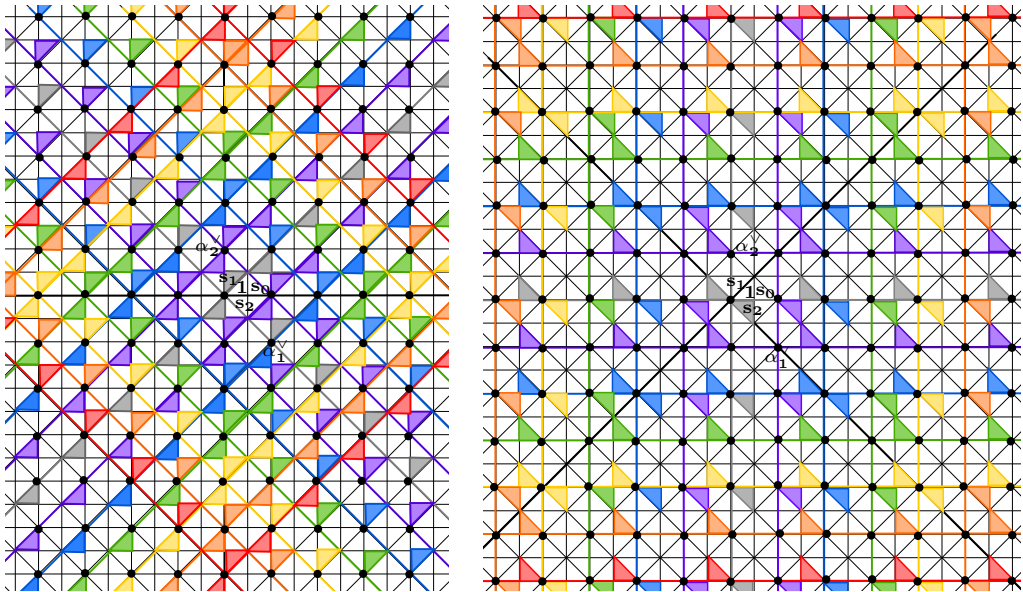

\begin{rmk}\label{rmk:components} We call each ``line" of same-color triangles on the left of Figure~\ref{fig:introA2} and in both parts of Figure~\ref{fig:conjC2} a \emph{component} of the corresponding conjugacy class. The ``global" action of $\aW$ on itself by conjugation can then be described as follows: the conjugation action of $\sW$ permutes the components of any conjugacy class, while that of the translation subgroup $T$ induces a transitive action by translations within each component. 

We also observe that linearization induces a natural surjection from the components of any $\conj{t^\lambda w}$ to the components of $\conj{w}$. Although this surjection is sometimes a bijection, as on the left of Figure~\ref{fig:introA2}, in general it is not injective, as seen in Figure~\ref{fig:conjC2}. There is moreover a natural bijection between the components of $\conj{w}$ (the ``lines" of gray triangles in these figures) and the elements of the conjugacy class of $w$ in $\sW$. We prove the statements sketched in this remark for arbitrary $\aW$ in~\cite{MST4}.
\end{rmk}

\begin{rmk}\label{rmk:cyclic}  The components of a conjugacy class and the conjugation action of $\aW$ on itself, as described in the first paragraph of Remark~\ref{rmk:components}, are closely related to the minimal length conjugacy class representatives found in~\cite{HeNie14} and (the affine case of)~\cite{Marquis21}. To explain this, let $\aW = T \rtimes \sW$ be an irreducible affine Coxeter group, so that $\sW$ is generated by the simple reflections $\{ s_1,\dots, s_n\}$, and $\aW$ has Coxeter generating set  $\aS = \{ s_0, s_1, \dots, s_n \}$, where $s_0$ is the affine simple reflection. Then elements $x, x' \in \aW$ are related by a \emph{cyclic shift} if there is a reduced word $s_{i_1} \cdots s_{i_k}$ for $x$, with $s_{i_j} \in \aS$ for $1 \leq j \leq k$, such that $x' = s_{i_2} \cdots s_{i_k}s_{i_1}$ or $x' = s_{i_k}s_{i_1} \cdots s_{i_{k-1}}$. 
The characterization mentioned in the first paragraph of this introduction is equivalent to the statement that any element of a conjugacy class can be transformed into a minimal length element via a sequence of cyclic shifts (see the introduction to~\cite{Marquis21}).

Now each cyclic shift corresponds to conjugation by an element $s_i \in \aS$. If $1 \leq i \leq n$, then conjugation by $s_i$ permutes the set of components of the conjugacy class $\conj{x}$. If $i = 0$, then $s_0 = t^{\theta^\vee} s_\theta$, where $s_{\theta}$ is the reflection corresponding to the highest root $\theta$ and $\theta^\vee$ is the corresponding coroot. Conjugation by $s_\theta$ again permutes the set of components of $\conj{x}$, but conjugation by $t^{\theta^\vee}$ induces a translation within each component. That is, $t^{\theta^\vee}$ induces a shift along each of the ``lines" of triangles in Figures~\ref{fig:introA2} and~\ref{fig:conjC2} above. The effect of iterating certain cyclic shifts on any given $x \in \aW$ is thus to spiral in gradually towards the minimal length elements of its conjugacy class, sometimes by switching components, and sometimes by translating within the current component.
\end{rmk}

\begin{rmk}\label{rmk:cent} For $x \in \aW$, the coconjugation sets $\coconj{x}{x'}$ include as a special case $\Cent(x,x)$, the centralizer of $x$. If $x$ is a Coxeter element in an arbitrary irreducible Coxeter group (of finite rank), then the centralizer of $x$ is just the cyclic subgroup $\langle x \rangle$; this was established by Carter~\cite{Carter72} for finite Weyl groups, by Blokhina~\cite{Blokhina} for infinite simply-laced or affine Coxeter groups, and recently by Hollenbach and Wegener~\cite{HollenbachWegener} in full generality. For instance, as noted in the last paragraph of Example~\ref{eg:introA2}, the element $x = t^{\alpha_1^\vee + \alpha_2^\vee}s_1$ in Figure~\ref{fig:coconjA2chimneys} is a Coxeter element in $\aW$ of type $\tilde{A}_2$. Hence the elements of $\coconj{x}{x} = \langle x \rangle$ are certain alcoves lying along certain translates of the hyperplane $\Fix(s_1)$. 
\end{rmk}

\begin{rmk}\label{rmk:Marquis} Let $x$ be an arbitrary infinite-order element of an irreducible affine Coxeter group $\aW$. Marquis~\cite[Theorem D]{Marquis23} describes the algebraic structure of the conjugacy class and centralizer of $x$ in terms of an associated point $\eta = \eta_x$ in the visual boundary $\partial \aW$ of~$\aW$; for~$\aW$ of rank $n+1$, this boundary can be identified with the sphere $\mathbb{S}^{n-1}$. The point $\eta = \eta_x \in \partial \aW$ is defined to be the endpoint of some (hence any) \emph{axis} for $x$, where an axis is a line in $\R^n$ on which~$x$ acts by translations. Marquis then associates several groups and a \emph{transversal complex} $\Sigma^\eta$ to the point~$\eta \in \partial \aW$, and uses these to give algebraic descriptions of the conjugacy class~$\conj{x}$ and the centralizer $\Cent(x,x)$.  

For example, for the glide-reflection $x = t^{\alpha_1^\vee + \alpha_2^\vee}s_1$ depicted in Figure~\ref{fig:coconjA2chimneys}, the point $\eta_x \in \partial \aW$ is the top left endpoint of the line which runs halfway between $\Fix(s_1)$ and $(\alpha_1^\vee + \alpha_2^\vee) + \Fix(s_1)$ (compare~\cite[Figure 1]{Marquis23}). The associated group $\aW^\eta$ (in Marquis' notation) is the reflection subgroup of $\aW$ of type $\tilde{A}_1$ generated by $s_1$ and $s_0 s_2 s_0$, and $\Sigma^\eta$ is the corresponding Coxeter complex. We observe that, in this example, the mod-set of the spherical part of $x$, namely $\Z \alpha_1^\vee$, can then be identified with the translation subgroup of $\aW^\eta$; it would be interesting to relate the mod-set to the group $\aW^\eta$ and its action upon $\Sigma^\eta$ in general. The centralizer of $x$ in Figure~\ref{fig:coconjA2chimneys} is then described in terms of the elements of $\aW$ which fix $\eta$ and other data (see part~(6) of~\cite[Theorem D]{Marquis23}).
\end{rmk}

\begin{rmk}\label{rmk:algorithm} Our results in~\cite{MST4} include an algorithm to solve the conjugacy problem and compute coconjugation sets in all split crystallographic groups, and hence in all affine Coxeter groups. As discussed in~\cite[Section 4]{MST4}, we have not yet implemented this algorithm nor investigated its complexity, and we expect both of these tasks to be substantial. 

We note that Krammer~\cite[Section 4.2]{Krammer} gives an algorithm for the conjugacy problem in affine Coxeter groups which uses their semidirect product structure and has linear runtime. It would be desirable to further compare his algorithm to ours. 
\end{rmk}

\subsection{Mod-sets and move-sets}

We now highlight the results of this paper in greater detail. 
 Our first main theorem establishes a close relationship between the $\Z$-module $\ModW(w) = (\Id - w)R^\vee$, the subspace $\Mov(w) = \Range(\Id - w)$, and reflection length, for arbitrary $w \in \sW$. Recall that the \emph{reflection length} of $w \in \sW$, which we denote by $\ell_\cR(w)$, is the minimal integer~$k$ such that $w$ is a product of $k$ reflections in $\sW$. The following result appears as Corollary~\ref{cor:allWeyl}; see Section \ref{sec:modMove} for more details and precise definitions.

\begin{thm}\label{thm:modMoveW} Let $\aW = \TW \rtimes \sW$ be an affine Coxeter group. Then for all $w \in \sW$,
	\[
	\rank_\Z(\ModW(w)) = \dim_\R(\Mov(w)) = \ell_\cR(w).
	\]
\end{thm}

\noindent The first equality in this theorem tells us that the move-set is the ``enveloping subspace" of the mod-set. That is, the geometry of conjugacy classes in~$\aW$ is coarsely described by the linear subspaces comprising the move-sets of the elements of $\sW$ (as seen in Figures~\ref{fig:introA2} and~\ref{fig:conjC2}). 

\begin{corollary}\label{cor:modMoveW} 
Let $\aW = \TW \rtimes \sW$ be an affine Coxeter group. Then for all $w \in \sW$,
	\begin{enumerate}
		\item\label{item:indexW} $\ModW(w)$ is a finite-index submodule of $\Mov(w) \cap R^\vee$; and
		\item\label{item:equalW} $\ModW(w) = \Mov(w) \cap R^\vee$ if and only if $R^\vee / \ModW(w)$ is torsion-free.
	\end{enumerate}
\end{corollary}

\noindent Part~\eqref{item:indexW} here says that any ``gaps" between elements in the same component of a conjugacy class are bounded and constant within conjugacy classes (see the right of Figure~\ref{fig:conjC2}). Part~\eqref{item:equalW} then gives a criterion for an element to fill its move-set. See Section \ref{sec:modsetintro} for additional discussion of Corollary \ref{cor:modMoveW} and its consequences.

In fact, we prove the first equality in Theorem~\ref{thm:modMoveW} as well as Corollary~\ref{cor:modMoveW} for all split crystallographic groups which are contained in affine Coxeter groups. See Theorem~\ref{thm:modMove} for the precise statement, and Remark~\ref{rmk:cryst} below regarding the relationship between crystallographic groups and affine Coxeter groups.

\begin{figure}[htb]
	\begin{minipage}{0.5\textwidth}
		\begin{example}\label{eg:conjA2elliptic} 
			Suppose $\aW = T \rtimes \sW$ is of type~$\tilde{A}_2$.  The rotations in $\sW$ are the elements $w \in \{ s_1 s_2, s_2 s_1\}$, and for these $w$, we have $\Mov(w) = \R^2$, therefore $\Mov(w) \cap R^\vee = R^\vee$. On the other hand, we have $\ModW(w) = \{ c_1 \alpha_1^\vee + c_2 \alpha_2^\vee \mid c_1 + c_2 \equiv 0 \mod 3 \}$, where $c_1, c_2 \in \Z$ and $\{ \alpha_1^\vee, \alpha_2^\vee \}$ are the simple coroots. Hence $\ModW(w)$ has index $3$ in $\Mov(w) \cap R^\vee$. In particular, $w$ does not fill its move-set.
			
			\hspace{3mm} In Figure~\ref{fig:conjA2elliptic}, the set $\ModW(w)$ for $w \in \sW$ a rotation is shown by the gray dots, and the~$2$ other cosets of $\ModW(w)$ in $R^\vee$ are depicted by light and dark purple dots, respectively. Note that $\ModW(w)$ is $\sW$-invariant, while for any $\lambda \not \in \ModW(w)$, the $\sW$-orbit of $\lambda + \ModW(w)$ has $2$ elements, corresponding to the $2$ shades of purple dots in this figure. The conjugacy class $\conj{t^\lambda w}$ for any $\lambda \in \ModW(w)$ is the set of gray triangles, while the $2$ conjugacy classes $\conj{t^\lambda w}$ for $\lambda \not \in \ModW(w)$ are shown as light and dark pink triangles, respectively.
		\end{example}
	\end{minipage}
	\begin{minipage}{0.48\textwidth}
		\begin{center}
			\vspace{-6mm}
			{\begin{overpic}[width=0.95\textwidth]{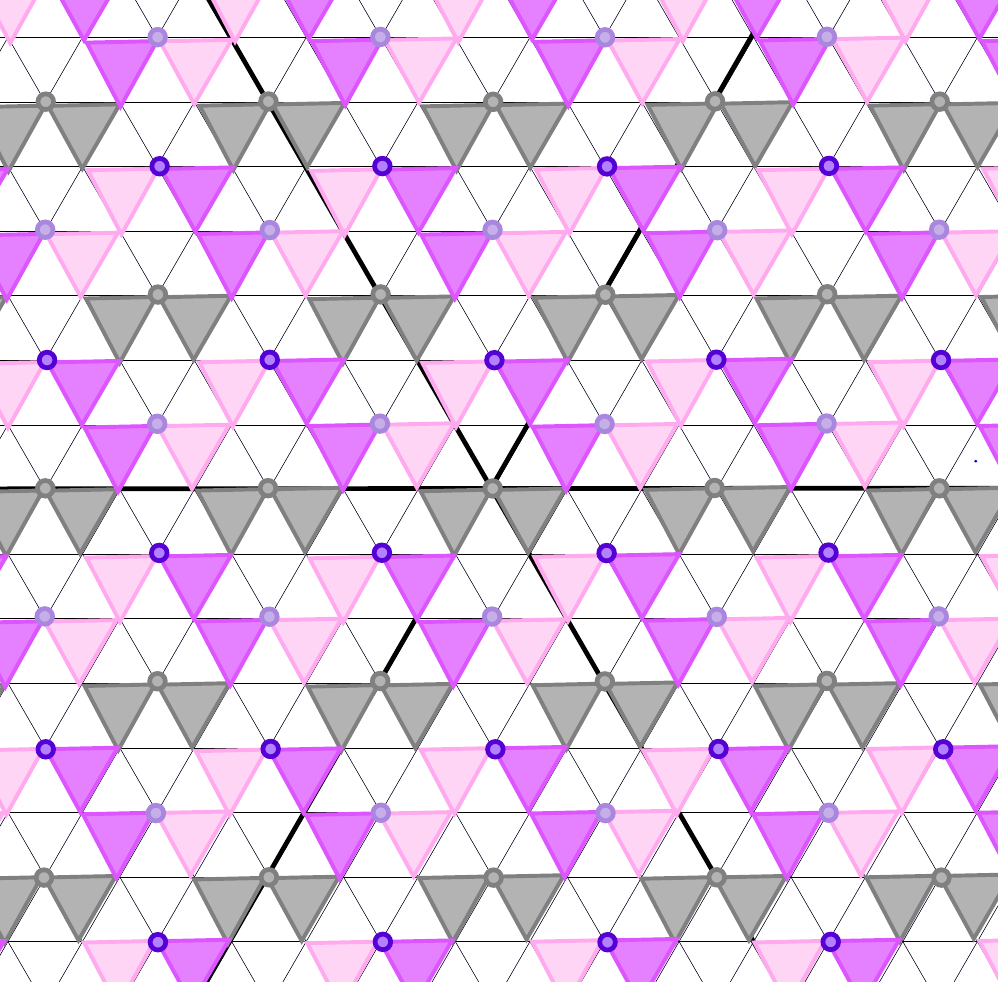}
					\put(48,52){\tiny{$\mathbf{1}$}}
					\put(48,57.5){\tiny{$\mathbf{s_0}$}}
					\put(44,51.5){\tiny{$\mathbf{s_1}$}}
					\put(51.4,51.5){\tiny{$\mathbf{s_2}$}}
					\put(41,46){\tiny{$\mathbf{s_1s_2}$}}
					\put(50,46){\tiny{$\mathbf{s_2s_1}$}}
					\put(59,57){\tiny{$\mathbf{\alpha_1^\vee}$}}
					\put(36,57){\tiny{$\mathbf{\alpha_2^\vee}$}}	
			\end{overpic}}
		\end{center}
		\caption{\footnotesize{Conjugacy classes $\conj{t^\lambda w}$ in type $\tilde{A}_2$, where $w$ is a rotation; see Example \ref{eg:conjA2elliptic} for details.}}
		\label{fig:conjA2elliptic}
	\end{minipage}
	
\end{figure}
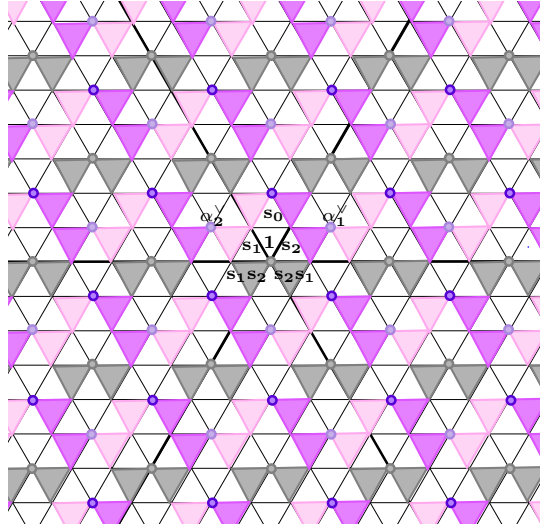 

\vspace{-2mm}

\begin{rmk} The tesselation of Euclidean space induced by the action of $\aW$ is sometimes known as the Coxeter complex of $\aW$. Now any diagram automorphism $\tau$ of the affine Coxeter group $\aW$ induces a (non-type-preserving) automorphism of the Coxeter complex. Hence if $x, y \in \aW$ and $y = \tau(x)$, we can obtain the conjugacy class of $y$ from that of $x$ by applying the corresponding automorphism of the Coxeter complex to $\conj{x}$. 

For example, in type $\tilde{A}_2$, any permutation of the affine Coxeter generating set $\aS = \{ s_0, s_1, s_2 \}$ induces a diagram automorphism of $\aW$. Swapping $s_1$ and $s_2$ induces a reflection in the vertical gray line $\Mov(w_0)$ on the left of  Figure~\ref{fig:introA2}; as already mentioned in Example~\ref{eg:introA2}, this reflection yields the remaining conjugacy classes of the form $\conj{t^\lambda w}$ where $w$ is a reflection in $\sW$. A cyclic permutation of the elements of $\aS$ induces a diagram automorphism which cyclically permutes the $3$ conjugacy classes depicted in Figure~\ref{fig:conjA2elliptic}.

In type $\tilde{C}_2$, the only nontrivial diagram automorphism swaps $s_0$ and $s_2$, and this induces a reflection of the Coxeter complex (in a line from top left to bottom right which is not depicted in Figure~\ref{fig:conjC2}). On the left of Figure~\ref{fig:conjC2}, this reflection preserves all conjugacy classes, while on the right, it yields the remaining conjugacy classes of the form $\conj{t^\lambda w}$ where $w \in \{ s_2, s_1 s_2 s_1\}$.
\end{rmk}

Our proof of Theorem~\ref{thm:modMoveW} relies upon a result of Carter~\cite[Lemma 3]{Carter72} which characterizes minimal length reflection presentations in the finite Weyl group $\sW$. We recall this statement as Theorem~\ref{thm:Carter} below. Carter's proof of this result yields that the dimension of $\Mov(w)$ equals the reflection length of $w \in \sW$ (see Corollary~\ref{cor:dimMov}). With some care, we adapt Carter's argument to work over~$\Z$, and hence determine that $\ModW(w)$ has rank equal to $\ell_R(w)$. We then extend the first equality in Theorem~\ref{thm:modMoveW} to all split crystallographic subgroups of $\aW$, and prove Corollary~\ref{cor:modMoveW} in this level of generality. We note that our proof of Theorem~\ref{thm:modMoveW} is type- and rank-free, essentially because Carter's proof of~Theorem~\ref{thm:Carter} is.

\begin{rmk}
After publishing the first version of this work on the arXiv, Dermenjian and Evetts directed us to their recent preprint \cite{DermenjianEvetts}, in which they study conjugacy class growth for finitely generated virtually abelian groups. Lemma 4.5 in \cite{DermenjianEvetts} provides a characterization of conjugacy classes similar to \Cref{thm:ConjClass} below using an algebraic approach involving commutator subgroups in place of Mod-sets.  There is also a connection between the proof of Theorem 4.11 in \cite{DermenjianEvetts} and our proof of Theorem \ref{thm:modMoveW}. 
\end{rmk}

\begin{rmk}\label{rmk:cryst} It seems to be folklore that every (split) crystallographic group is contained in an affine Coxeter group in dimension $n = 2,3$, but that this no longer holds for~$n \geq 4$. We sketch how to obtain these facts from the literature in Appendix~\ref{app:cryst}. We do not know if the conclusions of Theorem~\ref{thm:modMove} are true for split crystallographic groups which are not contained in affine Coxeter groups.
\end{rmk}


\subsection{Structure of mod-sets}\label{sec:modsetintro}

To complete the description of conjugacy classes in all affine Coxeter groups, we provide explicit descriptions of the mod-sets $\ModW(w) = (\Id - w)R^\vee$ for all $w \in \sW$. The examples given above already show the delicate behavior of these $\Z$-modules, and our detailed results are thus necessarily type-by-type. The principal work of this paper is in establishing the theorems below.

Throughout, we follow the conventions of Bourbaki~\cite{Bourbaki4-6} (see Table~\ref{table:dynkin}). Suppose $\aW = T \rtimes \sW$  is of rank $n$. Then in all types, $\sW$ is generated (as a group) by the simple reflections $\{ s_1, \dots, s_n\}$, and the corresponding coroot lattice $R^\vee$ is generated (as a $\Z$-module) by the simple coroots $\{\alpha_1^\vee, \dots, \alpha_n^\vee \}$. Define $[n] = \{1,\dots,n\}$. To simplify notation, in this section we write $\Mod(w)$ for $\ModW(w)$.

For all $u,w \in \sW$, we have $ \Mod(uwu^{-1}) = u \Mod(w)$, by~\cite[Lemma 2.1]{MST4}.  Hence it suffices to consider the mod-sets for a single element in each conjugacy class of the finite group $\sW$. We do so using the minimal length conjugacy class representatives determined by Geck and Pfeiffer in~\cite[Chapter 3]{GeckPfeifferBook}. In the classical types, we recall and give a ``visual" rephrasing of the construction of these representatives. In types $B$ and $C$, this involves the normal form forests of du Cloux (see~\cite[Section 3.4]{BjoernerBrenti}).

We first consider $\sW$ the finite Weyl group of type $A_n$ for $n \geq 1$.  The conjugacy classes of $\sW$ are parameterized by weakly decreasing compositions (that is, partitions) $\beta = (\beta_1,\dots,\beta_p)$ of $n+1$; see~\cite[Proposition 3.4.1]{GeckPfeifferBook}, for instance. Each such $\beta$ determines a subset $J_\beta \subseteq [n]$, such that $n \in J_\beta$ if and only if $\beta_p \geq 2$. The corresponding representative $w_\beta$ is the product of the simple reflections $\{ s_j \mid j \in J_\beta \}$ in increasing order. The set $I_\beta$ records the partial sums of the parts of $\beta$, and $I_\beta - 1$ subtracts 1 from each partial sum. The following result appears as Theorem~\ref{thm:An}; see Section \ref{sec:TypeA} for more details and precise definitions.

\begin{thm}
	\label{thm:introAn}
	Suppose $\sW$ is of type $A_n$ with $n \geq 1$.  Let $\beta = (\beta_1, \dots, \beta_p)$ be a partition of $n+1$ with corresponding conjugacy class representative $w_\beta \in \sW$.  
	\begin{enumerate}
		\item The module $\Mod(w_\beta) = (\Id-w_\beta)R^\vee$ equals
		\[  \left\{\sum_{i=1}^n c_i \alpha_i^\vee \ \middle| \ c_i \in \Z, \quad c_i=0\ \text{for}\ i \in [n] \setminus J_\beta, \quad  \sum_{i=1}^n c_i \equiv 0  \; \operatorname{mod}\left( \gcd(\beta_k) \right) \right\}. \]
		\item If $\beta_p = 1$ then the module $\Mod(w_\beta) = (\Id-w_\beta)R^\vee$ has a $\Z$-basis given by 
\[ \{ \alpha_j^\vee \mid j \in J_\beta \}, \]
 and if $\beta_p \geq 2$ then the module $\Mod(w_\beta) = (\Id-w_\beta)R^\vee$ has a $\Z$-basis given by 
\[ \left\{ \alpha_i^\vee - \alpha_{i+1}^\vee \ \middle|\ i \in J_\beta \backslash (I_\beta - 1) \right\} \cup \left\{ \alpha_i^\vee - \alpha_{i+2}^\vee \ \middle| \ i \in (I_\beta - 1) \backslash \{ n\} \right\} \cup \left\{ \gcd(\beta_k)\alpha_n^\vee \right\}.\]
		\item  If $\beta = (1,\dots, 1)$, so that $w_\beta$ is trivial, then $(\Id - w_\beta)$ has Smith normal form $\diag(0^n)$. For $w_\beta$ nontrivial and any $w \in [w_\beta]$, the Smith normal form of $(\Id-w)$ equals \[ S_\beta = \diag(1^{n-p}, \gcd(\beta_k), 0^{p-1}).\]
	\item\label{item:An_quotient} For any partition $\beta$ and any $w \in [w_\beta]$, we have \[ R^\vee / \Mod(w) \cong  \left( \Z/ \gcd(\beta_k) \Z \right) \oplus \Z^{p-1}.\]
	\end{enumerate}
\end{thm}

From part \eqref{item:equalW} of Corollary~\ref{cor:modMoveW} and part~\eqref{item:An_quotient} of Theorem~\ref{thm:introAn}, we characterize those elements of $\sW$ of type $A_n$ which fill their move-sets, as follows.

\begin{corollary}\label{cor:fillAn} Suppose $\sW$ is of type $A_n$ with $n \geq 1$. Let $w \in \sW$, and let the conjugacy class of $w$ in $\sW$ be indexed by  the partition $\beta = (\beta_1,\dots,\beta_p)$. The following are equivalent:
	\begin{enumerate}
		\item $w$ fills its move-set; that is, $\Mod(w) = \Mov(w) \cap R^\vee$; and
		\item $\gcd(\beta_k) = 1$. 
	\end{enumerate}
	In particular, if $\beta_p = 1$, equivalently the conjugacy class of $w$ in $\sW$ is represented by $w_\beta$ contained in the type $A_{n-1}$ subsystem generated by $\{s_1,\dots,s_{n-1}\}$, then $w$ fills its move-set.
\end{corollary}

\begin{example} For $\sW$ of type $A_2$, the conjugacy class $\{ s_1, s_2, s_1s_2s_1 \}$ of reflections corresponds to the partition $\beta = (2,1)$, and is represented by $w_\beta = s_1$. As seen in Example~\ref{eg:introA2}, we have $\Mod(s_1) = \Z \alpha_1^\vee$. The conjugacy class $\{ s_1 s_2, s_2 s_1 \}$ of rotations corresponds to the partition $\beta = (3)$, and is represented by $w_\beta = s_1 s_2$. As seen in Example~\ref{eg:conjA2elliptic}, we have $\Mod(s_1s_2) = \{ c_1 \alpha_1^\vee + c_2 \alpha_2^\vee \mid c_1, c_2 \in \Z, \, c_1 + c_2 \equiv 0 \mod 3 \}$. Our $\Z$-basis for $\Mod(s_1 s_2)$ is given by $\{ \alpha_1^\vee - \alpha_2^\vee , 3 \alpha_2^\vee \}$.
\end{example}

We next consider $\sW$ of type $C_n$, for $n \geq 2$. As will be seen later in this introduction, the statements in type $B_n$ are much more delicate than in type $C_n$, which is why we present type~$C_n$ first. Following \cite[Proposition 3.4.7]{GeckPfeifferBook}, the conjugacy classes of $\sW$ are parameterized by ordered pairs of compositions~$(\beta, \gamma)$ such that $\beta$ is weakly decreasing, $\gamma$ is weakly increasing, and $|\beta| + |\gamma| = n$. Write $|\beta| = m$. Then each such pair determines subsets $J_\beta \subseteq [m-1]$ and $ I_\gamma, J_\gamma \subseteq [n]$, such that $I_\gamma = \emptyset$ if and only if $|\gamma| = 0$. The corresponding representative is given by $w_{\beta,\gamma} = w_\beta w_\gamma$, where the element $w_\beta$ is contained in the type $A_{m-1}$ subsystem of $\sW$ generated by $\{ s_1, \dots, s_{m-1} \}$, according to our conventions (see Table \ref{table:dynkin}), and is constructed exactly as sketched in the paragraph above Theorem~\ref{thm:introAn}. The element $w_\gamma$ is contained in the type $C_{n-m}$ subsystem generated by $\{ s_{m+1}, \dots, s_n \}$ and is constructed from~$\gamma$ by a more complicated process, so that $w_\gamma$ is nontrivial if and only if $|\gamma| \geq 1$; see Proposition \ref{prop:wbgC}. The following result appears as Theorem~\ref{thm:Cn}; see Section \ref{sec:TypeC} for more details and precise definitions. 

\begin{thm}\label{thm:introCn}
	Suppose $\sW$ is of type $C_n$ with $n \geq 2$.  Let $(\beta,\gamma)$ be a pair of compositions such that $\beta = (\beta_1,\dots,\beta_p)$ is weakly decreasing, $\gamma = (\gamma_1,\dots,\gamma_q)$ is weakly increasing, and $|\beta| + |\gamma| = n$, with corresponding conjugacy class representative $\wbg \in \sW$.  Write $m = |\beta|$, so that $0 \leq m \leq n$ and $|\gamma| = n-m$. 
	\begin{enumerate}
		\item The module $\Mod(\wbg) = (\Id-\wbg)R^\vee$ equals 
		\[  \left\{\sum_{i=1}^n c_i \alpha_i^\vee \ \middle| \ c_i \in \Z, \ c_i = 0 \mbox{ for }i \in [m] \setminus J_\beta, \ c_i \equiv 0 \; \operatorname{mod}\left(2 \right) \mbox{ for } i \in I_\gamma  \right\}. \]
		\item The module $\Mod(\wbg)$ has $\Z$-basis given by \[\{ \alpha_i^\vee \mid i \in J_\beta \} \cup \{ \alpha_i^\vee \mid i \in J_\gamma \setminus I_\gamma \} \cup \{ 2\alpha_i^\vee \mid i \in I_\gamma \}.\]
		\item \label{item:CnSNF} For any $w \in [\wbg]$, the Smith normal form of $(\Id-w)$ equals 
		\[ \Sbg =  \diag(1^{n - p - q}, 2^q, 0^p).\]
		\item \label{item:Cn_quotient} For any $w \in [\wbg]$, the quotient of $R^\vee$ by the mod-set is
		\[R^\vee / \Mod(w) \cong (\Z/ 2\Z)^q \oplus  \Z^p.\]
	\end{enumerate}
\end{thm}

\noindent We note that, in contrast to our results in type $A_n$ (see Theorem~\ref{thm:introAn}), neither the value of $\gcd(\beta_k)$ nor the indexing set~$I_\beta$ appears in our statements in type $C_n$. 

From part \eqref{item:equalW} of Corollary~\ref{cor:modMoveW} and part~\eqref{item:Cn_quotient} of Theorem~\ref{thm:introCn}, we characterize those elements of $\sW$ of type $C_n$ which fill their move-sets, as follows.

\begin{corollary}\label{cor:fillCn} Suppose $\sW$ is of type $C_n$ with $n \geq 2$. Let $w \in \sW$, and let the conjugacy class of $w$ in $\sW$ be indexed by  the pair of compositions $(\beta, \gamma)$. The following are equivalent:
	\begin{enumerate}
		\item $w$ fills its move-set; that is, $\Mod(w) = \Mov(w) \cap R^\vee$; 
		\item the conjugacy class of $w$ in $\sW$ is represented by the element $w_\beta$ of the type $A_{n-1}$ subsystem on $\{ s_1,\dots,s_{n-1}\}$; and
		\item $|\gamma| = 0$.
	\end{enumerate}
\end{corollary}

\begin{example}\label{eg:C2intro} In type $C_2$, the conjugacy class $\{ s_1, s_2 s_1 s_2 \}$ corresponds to the pair of compositions $(\beta, \gamma)$ with $\beta = (2)$ and $|\gamma| = 0$, and is represented by $\wbg = w_\beta = s_1$. Here, $J_\beta = \{ 1 \}$ and $I_\gamma = J_\gamma = \emptyset$. As seen on the left of Figure~\ref{fig:conjC2}, we have $\Mod(s_1) = \Z\alpha_1^\vee$, and $\wbg$ fills its move-set. The conjugacy class $\{ s_2, s_1 s_2 s_1 \}$ corresponds to $\beta = (1)$ and $\gamma = (1)$, and is represented by $\wbg = w_\gamma = s_2$. Here, $J_\beta =  \emptyset$ and $J_\gamma = I_\gamma = \{ 2 \}$, so that $J_\gamma \backslash I_\gamma = \emptyset$. As seen on the right of Figure~\ref{fig:conjC2}, we have $\Mod(s_2) = 2\Z\alpha_2^\vee$, and $\wbg$ does not fill its move-set.
\end{example}

Next, suppose $\sW$ is of type $B_n$ for $n \geq 2$. The minimal length representatives of conjugacy classes $\wbg$ are identical to those in type $C_n$, and the Cartan matrices in types $B_n$ and $C_n$ differ only by exchanging the $(n-1,n)$ and $(n,n-1)$ entries. However, this small change alters the results on mod-sets considerably. The following result appears as Theorem~\ref{thm:SNFB}; see Section~\ref{sec:TypeB} for more details and precise definitions. We write $\gcd(\beta_k,2)=\gcd(\beta_1,\dots,\beta_p,2)$, and if $\beta_p \geq 2$, we write $\gcd(\beta_k, \beta_p - 2)=\gcd(\beta_1, \dots, \beta_p, \beta_p - 2)$.

\begin{thm}\label{thm:introBnSNF}
	Suppose $\sW$ is of type $B_n$ with $n \geq 2$.  Let $(\beta,\gamma)$ be a pair of compositions such that $\beta = (\beta_1,\dots,\beta_p)$ is weakly decreasing, $\gamma = (\gamma_1,\dots,\gamma_q)$ is weakly increasing, and $|\beta| + |\gamma| = n$, with corresponding conjugacy class representative $\wbg \in \sW$.   For any $w \in [\wbg]$, the Smith normal form of $(\Id-w)$ is as follows:
 \begin{enumerate}
  \item\label{introSNFB_gcd2_parity} If $\gcd(\beta_k,2) = 2$ (including $|\beta| = 0$), $|\gamma| \geq 1$, and all parts of $\gamma$ have the same parity, then 
  \[
 \Sbg = \diag(1^{n - q-p}, 2^{q},0^p).
 \]
   \item\label{introSNFB_gcd2_nonparity} If $\gcd(\beta_k,2) = 2$ (including $|\beta| = 0$), $|\gamma| \geq 1$, and $\gamma$ has a change in parity, then 
  \[
 \Sbg = \diag(1^{n - q-p+1}, 2^{q-2},4,0^p).
 \]
  \item\label{introSNFB_gcd1} If $\gcd(\beta_k,2) = 1$ and $|\gamma| \geq 1$, then 
  \[
 \Sbg = \diag(1^{n-q-p+1}, 2^{q-1}, 0^p).
 \]
\item\label{introSNFB_gamma0_2} If $|\gamma| = 0$, $\beta_p \geq 2$, and $\gcd(\beta_k,\beta_p - 2) \geq 2$, then 
 \[
 \Sbg = \diag(1^{n-p-1}, \gcd(\beta_k,\beta_p -2), 0^p).
 \]
  \item\label{introSNFB_gamma0_1} If $|\gamma| = 0$, and either $\beta_p = 1$, or $\beta_p \geq 2$ and $\gcd(\beta_k, \beta_p - 2) = 1$, then 
 \[
 \Sbg = \diag(1^{n-p}, 0^p).
 \]
 \end{enumerate}
\end{thm}

The Smith normal form is canonical and fully characterizes the isomorphism type of $R^\vee/\Mod(w)$ for any $w \in [\wbg]$. Hence, the cases required to state Theorem~\ref{thm:introBnSNF} indicate the delicate nature of the results in type $B$; compare the uniform statement for the Smith normal form in type $C$ given by part \ref{item:CnSNF} of Theorem \ref{thm:introCn}. Many further cases and lengthy descriptions appear when finding a basis for $\Mod(\wbg)$, so we refer the reader directly to Theorem~\ref{thm:BasisB} for our results on $\Z$-bases for $\Mod(\wbg)$ when either $|\beta| = 0$ or $|\gamma| = 0$.
We provide many illustrative examples in type $B_n$, including sketches of how to find a basis for all pairs $(\beta,\gamma)$, but leave the general proof of our results in type $B_n$ to the reader, since much of the argument is similar to that in types $A_n$ and $C_n$.

From part \eqref{item:equalW} of Corollary~\ref{cor:modMoveW} and Theorem~\ref{thm:introBnSNF}, we characterize those elements of $\sW$ of type $B_n$ which fill their move-sets, as follows.

\begin{corollary}\label{cor:fillBn} Suppose $\sW$ is of type $B_n$ with $n \geq 2$. Let $w \in \sW$, and let the conjugacy class of $w$ in $\sW$ be indexed by  the pair of compositions $(\beta, \gamma)$. The following are equivalent:
	\begin{enumerate}
		\item $w$ fills its move-set; that is, $\Mod(w) = \Mov(w) \cap R^\vee$; and
		\item 
		\begin{enumerate}
		\item  $\beta$ has at least one odd part, and $\gamma$ has exactly one part; or
		\item $\beta_p \geq 2$, $\gcd(\beta_k,\beta_p - 2) = 1$, and $|\gamma| = 0$; or
		\item $\beta_p = 1$ and $|\gamma| = 0$.
		\end{enumerate}
	\end{enumerate}
\end{corollary}

\noindent We remark that for $|\beta| = m$ with $0 \leq m < n$, the composition $\gamma$ having exactly one part is equivalent to $w_\gamma$ being the Coxeter element $w_\gamma = s_n s_{n-1} \dots s_{m+1}$ in the type $B_{n-m}$ subsystem generated by $\{ s_{m+1}, \dots, s_n \}$.

\begin{example} In type $B_2$, as in type $C_2$ (see Example~\ref{eg:C2intro}), the conjugacy class $\{ s_1, s_2 s_1 s_2 \}$ corresponds to $\beta = (2)$ and $|\gamma| = 0$, and is represented by $\wbg = w_\beta = s_1$. By case \eqref{introSNFB_gamma0_2} of Theorem~\ref{thm:introBnSNF}, the Smith normal form for $\Id - s_1$ is $\diag(2,0)$. Thus $R^\vee/\ModW(s_1) \cong (\Z/2\Z) \oplus \Z$, and $s_1$ does not fill its move-set in type $B_2$. The conjugacy class $\{ s_2, s_1 s_2 s_1 \}$ corresponds to $\beta = (1)$ and $\gamma = (1)$, and is represented by $\wbg = w_\gamma = s_2$. We have $\Sbg = \diag(1,0)$ by case \eqref{introSNFB_gcd1} of Theorem~\ref{thm:introBnSNF}, and so $R^\vee/\ModW(s_2) \cong \Z$, and $s_2$  fills its move-set in type $B_2$. Compare Example \ref{eg:C2intro} in type $C_2$, where these situations are reversed.
\end{example}

Now suppose $\sW$ is of type $D_n$ for $n \geq 4$. Following~\cite[Section 3.4]{GeckPfeifferBook}, the conjugacy classes of $\sW$ are parameterized by ordered pairs of compositions~$(\beta, \delta)$ such that $\beta$ is weakly decreasing, $\delta$ is weakly increasing and has an even number of parts, and $|\beta| + |\gamma| = n$.  In the special case $(\beta,0)$ where all parts of $\beta$ are even, there are two distinct conjugacy classes, parameterized by $\beta^+$ and $\beta^-$. Let $|\beta|=m$. Then the conjugacy class representatives are given by $\wbd = w_\beta w_\delta$, where $w_\beta$ is contained in the type $A_{m-1}$ subsystem generated by $\{ s_1, \dots, s_{m-1} \}$ and $w_\delta$ is contained in the subsystem generated by $\{ s_{m+1},\dots,s_n\}$ (which could have various types). The following result appears as Theorem~\ref{thm:DnSNF}; see Section~\ref{sec:TypeD} for more details and precise definitions. As in type $B$, we write $\gcd(\beta_k,2)$ for $\gcd(\beta_1,\dots,\beta_p,2)$.

\begin{thm}\label{thm:introDnSNF}
Suppose $\sW$ is of type $D_n$ with $n \geq 4$.  Let $(\beta,\delta)$ be a pair of compositions such that $\beta = (\beta_1,\dots,\beta_p)$ is weakly decreasing, $\delta = (\delta_1,\dots,\delta_{2r})$ is weakly increasing with an even number of parts, and $|\beta| + |\delta| = n$.  Let $w_{\beta^\pm, \delta} \in \sW$ be the corresponding conjugacy class representative in $\sW$.  

For any $w \in [w_{\beta^\pm, \delta}]$, the Smith normal form of $(\Id-w)$ is as follows:
\begin{enumerate}
 \item\label{introSNFD_gcd2_parity} If $\gcd(\beta_k,2) = 2$ (including the case that $|\beta| = 0$), $|\delta| \geq 2$, and all parts of $\delta$ have the same parity, then 
  \[
 \Sbd = \diag(1^{n - 2r-p}, 2^{2r},0^p).
 \]
   \item\label{introSNFD_gcd2_nonparity} If $\gcd(\beta_k,2) = 2$ (including the case that $|\beta| = 0$), $|\delta| \geq 2$, and $\delta$ has a change in parity, then 
  \[
 \Sbd = \diag(1^{n - 2r-p+1}, 2^{2r-2},4,0^p).
 \]
\item\label{introD_gcd1} If $\gcd(\beta_k,2) = 1$ and $|\delta| \geq 2$, then
\[ \Sbd = \diag(1^{n-2r-p+1},2^{2r-1},0^p).\] 
\item\label{introD_beta_delta0} If $|\delta| = 0$ and $\beta = (1,\dots, 1)$, so that $\wbd$ is trivial, then $\Sbd = \diag(0^n)$. If $|\delta| = 0$ and $\beta \neq (1,\dots, 1)$, then \[ S_{\beta^\pm, \delta} = \diag(1^{n-p-1}, \gcd(\beta_k,2), 0^{p}).\] 
\end{enumerate}
\end{thm}

For bases in type $D_n$, there are even more cases than in type $B_n$, and so we refer the reader directly to Theorems \ref{thm:DnCuspidal} and \ref{thm:Dnnoncuspidal}. Roughly speaking, this proliferation of cases is due to the structural differences in type $D_n$ between $n$ odd and $n$ even, the various possibilities for the interaction of a type $A_{m-1}$ subsystem with the node of valence $3$ in the Dynkin diagram, and the various types possible for the subsystem generated by $\{ s_{m+1},\dots,s_n\}$. From part \eqref{item:equalW} of Corollary~\ref{cor:modMoveW} and Theorem~\ref{thm:introDnSNF}, we obtain the following result.

\begin{corollary}\label{cor:fillDn} Suppose $\sW$ is of type $D_n$ with $n \geq 4$. Let $w \in \sW$, and let the conjugacy class of $w$ in $\sW$ be indexed by the pair of compositions $(\beta, \delta)$. The following are equivalent:
	\begin{enumerate}
		\item $w$ fills its move-set; that is, $\Mod(w) = \Mov(w) \cap R^\vee$; and
		\item  $\beta$ has at least one odd part, and $|\delta| = 0$. 
	\end{enumerate}
	In particular, if the conjugacy class of $w$ in $\sW$ is \emph{not} represented by the element $w_\beta$ of the type $A_{n-1}$ subsystem on $\{ s_1,\dots,s_{n-1}\}$, then $w$ does \emph{not} fill its move-set.
\end{corollary}

\noindent Despite the vast differences in the characterizations of mod-sets in types $A$, $B$, $C$, and $D$, note that the condition for $w \in \sW$ to fill its move-set is almost identical across types.

In all classical types, the results outlined in this section are obtained by the following process. After presenting the conjugacy class representatives from~\cite[Chapter 3]{GeckPfeifferBook}, we determine the matrix for each such representative, with respect to the $\Z$-basis $\{ \alpha_1^\vee, \dots, \alpha_n^\vee \}$ for $R^\vee$. We then formulate algorithms involving column and/or row operations over $\Z$ to obtain an explicit description of the corresponding mod-set, to find a basis for the mod-set, and to determine the Smith normal form. From the Smith normal form, one can write down the isomorphism type of the quotient of $R^\vee$ by the mod-set, and so complete the proofs. 

In order to aid readability, for the classical types we have included many examples of our algorithms in the body of this work. We caution that, in the classical types, the conventions of~\cite{Bourbaki4-6}, which we follow in this paper, differ from those used in~\cite[Chapter 3]{GeckPfeifferBook}. In Table~\ref{table:dynkin} we have set out, side-by-side for each type, the conventions from~\cite{Bourbaki4-6} and~\cite[Chapter 3]{GeckPfeifferBook}.

In Section \ref{sec:exceptional}, we provide a complete system of minimal length representatives for all conjugacy classes of $\sW$ of exceptional type. We use the tables in Appendix B of \cite{GeckPfeifferBook} for the cuspidal classes, and then iterate over all proper parabolic subgroups of $\sW$, choosing representatives for the cuspidal classes in each proper parabolic using the algorithms from~\cite[Chapter 3]{GeckPfeifferBook}. For each representative $w \in \sW$, we present the Smith normal form $S_w$ for $(\Id - w)$, which is canonical and fully characterizes the isomorphism type of the quotient $R^\vee/\Mod(w)$, and answers the question of whether or not $w$ fills its move-set.  In the exceptional types, the Smith normal form $S_w$ was calculated using the \verb!smith_form()!  command in Sage \cite{sagemath}.

\subsection{Discussion of related work}

Our geometric approach to describing conjugacy classes and coconjugation sets is quite distinct from what has appeared in the literature. We review here some previous results on conjugacy classes and centralizers in Coxeter groups, in addition to those works already discussed in the first paragraph of the introduction and Remark~\ref{rmk:cyclic}. We do not know of any previous work on coconjugation sets in Coxeter groups, per se.

If $G$ is a simple algebraic group, then any element $g \in G$ admits a generalized Jordan decomposition $g = su$, where $s$ is semisimple and $u$ is unipotent. The conjugacy classes of $G$ can thus be understood by studying the semisimple and unipotent conjugacy classes within $G$; see the surveys \cite{Spalt,SprStein,LiebSeit}. Work of Lusztig cleverly exchanges the study of the unipotent conjugacy classes for the conjugacy classes of the finite Weyl group \cite{LuszRT11}.  Lusztig's map from conjugacy classes in $W$ to unipotent classes in $G$ is injective when restricted to the elliptic conjugacy classes in $W$, and Adams, He, and Nie generalize Lusztig's result, proving that the straight conjugacy classes in $\aW$ play the analogous role in the affine setting \cite{AHN}. These conjugacy classes in affine Weyl groups also enjoy a wide range of applications, playing a central role in the character theory of affine Hecke algebras, homological properties of affine Deligne-Lusztig varieties, and geometric representations of $p$-adic groups via an action of the cocenter of the Hecke algebra; see \cite{HeNie12,HeNie14,HeNie15,CiuHe}.

Much of the literature on centralizers, a special case of coconjugation sets, is in the setting of either finite Coxeter groups, or arbitrary Coxeter groups, the latter often motivated by the still-open isomorphism problem for Coxeter groups (discussed in~\cite{SantosRegoSchwer}, for example). Much of this work also focuses on obtaining an algebraic description of centralizers of certain ``special" elements. For example, as discussed in Remark~\ref{rmk:cent}, the centralizer of a Coxeter element $w$ is equal to the cyclic group $\langle w \rangle$. Richardson~\cite{Richardson} investigated centralizers of involutions in arbitrary Coxeter groups, and centralizers of reflections in arbitrary Coxeter groups have been studied by Howlett~\cite{Howlett}, Brink~\cite{Brink}, and Allcock~\cite{AllcockCentralizers}. For $W$ an arbitrary finite Coxeter group, the final statement in this direction seems to be the work of Konvalinka, Pfeiffer, and R\"over~\cite{KonvalinkaPfeifferRoever}, which gives an algebraic description of the centralizer of any $w \in W$; see Remark~\cite{rmk:Marquis} for a brief discussion of Marquis' work in the setting of arbitrary Coxeter groups.

\subsection{Structure of the paper}

In Section~\ref{sec:preliminaries} we fix notation and recall some background, and then restate the main definitions and results  of~\cite{MST4} in the setting of affine Coxeter groups. Section~\ref{sec:modMove} establishes Theorem~\ref{thm:modMoveW} and Corollary~\ref{cor:modMoveW}, concerning the relationship between mod-sets and move-sets. Our main results in types $A_n$, $C_n$, $B_n$, and $D_n$ are presented in Sections~\ref{sec:TypeA},~\ref{sec:TypeC},~\ref{sec:TypeB}, and~\ref{sec:TypeD}, respectively, and we discuss the exceptional types in Section~\ref{sec:exceptional}. Appendix~\ref{app:cryst} briefly considers the relationship between split crystallographic groups and affine Coxeter groups, and Appendix~\ref{app:dynkin} describes several relevant conventions for labeling the nodes of Dynkin diagrams.

\subsection{Acknowledgements}

We thank Bob Howlett, Gregory Maloney, Timoth\'ee Marquis, Jon McCammond, and Travis Scrimshaw for helpful conversations and/or correspondence. We thank the Sydney Mathematical Research Institute for hosting a visit by PS in August 2019. We are grateful to MPIM Bonn for hosting EM's sabbatical in 2020, and for supporting a collaborative visit by PS in March 2020. We thank the Mathematisches Forschungsinstitut Oberwolfach for hosting PS for Research in Pairs in 2020.  We thank Haverford College for supporting a visit by AT in June 2022, as well as one by PS in April 2023. Finally, many experiments which led to these results were conducted in Sage, and we thank the Sage developers for implementing related procedures \cite{sagemath}.



\section{Preliminaries}
\label{sec:preliminaries}

We fix notation and collect some useful results on affine Coxeter systems in Section~\ref{sec:affine}, then discuss move- and fix-sets in \Cref{sec:moveFix}. We discuss cuspidal and Coxeter elements of the finite Coxeter group $\sW$ in \Cref{sec:cuspidalPrelim}.  In \Cref{sec:summary}, we state the main definitions and results of~\cite{MST4} in the setting of affine Coxeter groups.


\subsection{Affine Coxeter systems, roots, and coroots}\label{sec:affine}
 
This section reviews definitions and notation on affine Coxeter systems and their associated root systems and coroot lattices. We assume the reader is familiar with this material at the level of the reference \cite{Humphreys}. 

Let $(\aW,\aS)$ be an irreducible affine Coxeter system of rank $n$. Let $V$ be the associated $n$-dimensional real vector space on which $\aW$ acts by affine transformations, which we can identify with $n$-dimensional Euclidean space $\R^n$.  Denote the origin of $V$ by $0$. Then $(\aW,\aS)$ has associated spherical Coxeter system $(\sW,\sS)$ such that $\sW$ is the stabilizer in $\aW$ of $0$ and $\sS = \{ s_1, \dots, s_n \}$ is the set of elements of $\aS = \{ s_0,s_1, \dots, s_n\}$ which fix $0$.  We call $\{ s_1, \dots, s_n\}$ the \emph{simple reflections}. Denote the index set for $\sS$ by $[n]:= \{1,2,\dots, n\}$. We typically use the letters $x, y$ for elements of $\aW$, and $u, w$ for elements of $\sW$, and we write $w_0$ for the longest element of the spherical Coxeter group $\sW$.

The vector space $V$ admits an ordered basis $\Delta = (\alpha_i)_{i \in [n]}$, and a symmetric bilinear form $B(\alpha_i, \alpha_j) = -\cos \frac{\pi}{m(i,j)}$, where $m(i,j)$ is the $(i,j)$-entry of the associated Coxeter matrix.  Given any $i \in [n]$, define a linear functional $\alpha_i^\vee \in V^*$ by $\langle \alpha_i^\vee, v \rangle := 2B(\alpha_i, v)$ for any $v \in V$, where $\langle \cdot, \cdot \rangle: V^* \times V \to \Z$ denotes the evaluation pairing. The ordered set $\Delta^\vee = (\alpha_i^\vee)_{i \in [n]}$ is then a basis for the dual space $V^*$. The elements of $\Delta$ are called the \emph{simple roots}, and the elements of $\Delta^\vee$ are the \emph{simple coroots}.

The action of $\aW$ on $V$ induces an action of $\aW$ on $V^*$, as follows. Recall that, using the conventions in \cite{Bourbaki4-6}, every generator $s_i \in \sS$ acts as a linear reflection:
\begin{equation}\label{eq:siaction}
	s_i(\alpha_j^\vee) = \alpha_j^\vee - \langle \alpha_j^\vee, \alpha_i \rangle \alpha_i^\vee = \alpha_j^\vee - c_{ij} \alpha_i^\vee,
\end{equation} 
where $c_{ij}$ equals the $(i,j)$-entry of the Cartan matrix. The fix-set of the reflection $s_i$ is the linear hyperplane \[ \cH_{\alpha_i} = \{ z \in V^* \mid \langle z, \alpha_i \rangle = 0\}.\] We sometimes write $s_i = s_{\alpha_i}$ for $i \in [n]$. 

The vectors $\Phi = \{w\alpha_i \mid w \in \sW, \ i \in [n]\} \subset V$ are called \emph{roots}, and $\Phi^\vee = \{w\alpha_i^\vee  \mid w \in \sW, \ i \in [n]\} \subset V^*$ are the corresponding \emph{coroots}.
For any root $\beta = w\alpha_i$, the corresponding coroot $\beta^\vee$ is given by $\langle \beta^\vee, v \rangle := 2B(\beta, v)$ for $v \in V$, and we write $s_\beta$ for the linear reflection $ws_{\alpha_i} w^{-1} = s_{w\alpha_i}$. For $\beta = w\alpha_i$, the reflection $s_\beta$ has fix-set the linear hyperplane
\[ \cH_{\beta} = \{ z \in V^* \mid \langle z, \beta \rangle = 0\} = w\cH_{\alpha_i}. \] 

There is a distinguished lattice in $V^*$, which is preserved by the actions of $\aW$ and $\sW$, called the \emph{coroot lattice} $R^\vee = \bigoplus \Z \alpha_i^\vee$. Denote by $t^\lambda$ the translation in $V^*$ by the coroot $\lambda \in R^\vee$. The set of all such translations $T = \{ t^\lambda \mid \lambda \in R^\vee\}$ is the translation subgroup of the affine Coxeter group $\aW$.  For any $\lambda, \mu \in R^\vee$, we have $t^\lambda t^\mu = t^{\lambda + \mu} = t^{\mu + \lambda} = t^\mu t^\lambda$, and given any $w \in \sW$, we have $w t^\lambda w^{-1} = t^{w\lambda}$.  

We may also canonically identify $V$ and its dual $V^*$, and we use this identification freely throughout. Then for any root $\beta\in \Phi$, the corresponding coroot $\beta^\vee$ is identified with the vector $2\frac{\beta}{(\beta, \beta)}$, where $( \cdot, \cdot)$ is the standard inner product on $\R^n$.

 
\subsection{Move-sets and fix-sets}\label{sec:moveFix}

The \emph{move-set} and \emph{fix-set} of elements of~$\aW$ play an important role in our study of conjugation. They are defined as follows.    
\begin{definition}
	\label{def:mov-fix-elliptic}
	For any $x \in \aW$: 
	\begin{enumerate}
		\item the \emph{move-set} of $x$ is $\Mov(x) = \{ p \in V^* \mid q + p = x(q), \; q \in V^* \}$; and
		\item the \emph{fix-set} of $x$ is $\Fix(x) = \{ p \in V^* \mid x(p) = p\}$. 
	\end{enumerate}
\end{definition}

\noindent For example, the reflection $s_\beta \in \sW$ has $\Mov(s_\beta) = \R \beta^\vee$ and $\Fix(s_\beta) = \cH_\beta$. Note that for $i \in [n]$ we have $\Mov(s_i) =  \R\alpha_i^\vee$ and $\Fix(s_i) = \cH_{\alpha_i}$.

For all  elements $w \in \sW$, the sets 
\[
\Mov(w)= \Range(w-\Id) \quad \text{and} \quad \Fix(w)=\Ker(w-\Id)
\] 
are sub-vector spaces of $V^*$ and in fact orthogonal complements, i.e. $V^*=\Mov(w) \oplus \Fix(w)$.
In addition, by straightforward computation, one has 
\[
\Mov(uwu^{-1})=u\Mov(w) \text{ for all } u, w \in\sW.
\] 
For move-sets of elements of $\aW$, we have, by Proposition~1.21 of~\cite{LMPS}, that 
\[
\Mov(x) = \lambda + \Mov(w), 
\]
where $x = t^\lambda w \in \aW$ with $\lambda \in R^\vee$ and $w \in \sW$. 

We write $\cR$ for the set of all reflections in $\sW$. That is, \[ \cR = \{ s_\beta \mid \beta \in \Phi \} = \{ w s_i w^{-1} \mid w \in\sW, i \in [n] \}.\] Hence for each $r = w s_i w^{-1} \in \cR$, we have $\Mov(r) = \R (w \alpha_i^\vee) = w \Mov(s_i)$. 

The \emph{reflection length} of an element $w \in\sW$, denoted $\ell_\cR(w)$, is the minimal integer~$k$ such that $w = r_1 \cdots r_k$ where each $r_i$ is a reflection in $\sW$. We will use the following fundamental result of Carter~\cite{Carter72}. 

\begin{thm}[Lemma 3 of~\cite{Carter72}]\label{thm:Carter} Let $s_{\beta_1},\dots,s_{\beta_k} \in \cR$ and let $w = s_{\beta_1} \cdots s_{\beta_k} \in \sW$. Then $\ell_\cR(w) = k$ if and only if the set of roots $\{ \beta_1, \dots, \beta_k \} \subset \Phi$ is linearly independent over~$\R$.
\end{thm}

The next result recalls part of the statement of~\cite[Lemma 1.26]{LMPS}. 

\begin{corollary}\label{cor:dimMov} Let $w \in \sW$. Suppose $\ell_\cR(w) = k$, with $w = s_{\beta_1} \cdots s_{\beta_k}$ where each $s_{\beta_i} \in \cR$. Then $\Mov(w)$ has an $\R$-basis given by $\{ \beta_1^\vee, \dots, \beta_k^\vee \}$, and hence $\dim_\R(\Mov(w)) = \ell_R(w)$.
\end{corollary}
\begin{proof} As explained at the start of the proof of \cite[Proposition 5.1]{BradyMcCammond}, since each $s_{\beta_i}$ moves points just in the $\beta_i^\vee$-direction, we have $\Mov(w) \subseteq \Span_\R\{ \beta_1^\vee, \dots, \beta_k^\vee \}$. Now by one direction of Theorem~\ref{thm:Carter}, since $\ell_\cR(w) = k$, then the set $\{ \beta_1, \dots, \beta_k \}$ is linearly independent over $\R$. Moreover, Carter's proof of the other direction of Theorem~\ref{thm:Carter} shows that when the set $\{ \beta_1, \dots, \beta_k \}$, and thus $\{ \beta_1^\vee, \dots, \beta_k^\vee \}$, is linearly independent over $\R$, the subspace $(w - \Id)\R^n$ contains each of $\beta_1, \dots, \beta_k$, and thus also $\beta_1^\vee, \dots, \beta_k^\vee$, in turn. Since $\Mov(w)=(w - \Id)\R^n$, this completes the proof.
\end{proof}

\begin{rmk} In~\cite[Lemma 1.26]{LMPS}, the claims in Corollary~\ref{cor:dimMov} are stated to be proved via~\cite[Lemma 6.4]{BradyMcCammond}, which considers move-sets for elements of the full isometry group of Euclidean space. Tracing~\cite[Lemma 6.4]{BradyMcCammond} back, we arrive at~\cite[Proposition~5.1]{BradyMcCammond}. However, we do not understand the proof of the following statement within \cite[Proposition~5.1]{BradyMcCammond}: if the roots corresponding to the reflections $r_i$ are linearly independent, then $\dim_\R(\Mov(w)) = k$. Since Corollary~\ref{cor:dimMov} is important for our results, we have included here a proof which relies upon Theorem~\ref{thm:Carter} instead, for the sake of completeness.
\end{rmk}

\subsection{Cuspidal and Coxeter elements}
\label{sec:cuspidalPrelim}

In this section we recall some background on cuspidal and Coxeter elements of $\sW$, mostly following Chapter 3 of \cite{GeckPfeifferBook}.

Given any subset $J \subseteq [n]$, we denote by $W_J$ the \emph{standard parabolic subgroup} of $\sW$ generated by $S_J = \{ s_j \mid j \in J \}$. Recall that $(W_J, S_J)$ is a spherical Coxeter system. A \emph{parabolic subgroup} of $\sW$ is then any subgroup which is conjugate to a standard parabolic subgroup. 
For $w \in \sW$, the \emph{parabolic closure} of $w$, denoted $\Pc(w)$, is the smallest parabolic subgroup of~$\sW$ which contains $w$.
The relationship between conjugation in $\sW$ and parabolic closures is given by the next lemma, whose proof is an easy exercise.

\begin{lemma}
	\label{lem:ConjPc} 
	For all $u,w \in \sW$, we have $u \Pc(w) u^{-1} = \Pc(uwu^{-1})$.
\end{lemma} 

For any $w \in \sW$, we write $\wconj$ for its conjugacy class in $\sW$.

\begin{definition}[Cuspidal]
	\label{def:cuspidal}
	A conjugacy class $\cC$ of $\sW$ is called \emph{cuspidal} if $\cC \cap W_J = \emptyset$ for all proper subsets $J \subsetneq [n]$.  
	An element $w \in \sW$ is called \emph{cuspidal} if its conjugacy class $\wconj$ is cuspidal.
\end{definition}

These definitions can also be phrased in terms of parabolic closures, as follows: a conjugacy class $\cC$ of $\sW$ is \emph{cuspidal} if $\Pc(w) = \sW$ for some, hence by \Cref{lem:ConjPc} any, $w \in \cC$, and $w \in \sW$ is \emph{cuspidal} if $\Pc(w) = \sW$.  If $w \in \sW$ is an element of a parabolic subgroup $\sW_J$, we will sometimes describe $w$ as being \emph{cuspidal in $\sW_J$} if $\Pc(w) = \sW_J$.

Let $w \in \sW$ and let $\cC$ be a $\sW$-conjugacy class.
We write $\supp(w)$ for the \emph{support} of $w$ in~$S$; that is, the set of simple reflections appearing in some, and hence any, reduced expression for~$w$ (see Corollary~1.2.3 of~\cite{GeckPfeifferBook}).
We denote by $\Cmin$  the set of minimal length elements of~$\cC$. 	

\begin{prop}
	\label{prop:cuspidal}
	Let $w \in \sW$ have $\sW$-conjugacy class $\cC$.  The following are equivalent:
	\begin{enumerate}
		\item\label{item1} $\cC$ is cuspidal.
		\item\label{item2} $\Fix(w) = \{ 0 \}$.
		\item\label{item3} $\Mov(w) = V^*$.
		\item\label{item7} $\supp(w) = S$ for all $w \in \Cmin$.
		\item\label{item8} There exists an element $w \in \Cmin$ such that $\supp(w) = S$.
	\end{enumerate}
\end{prop}

\begin{proof}  
The equivalence of items \eqref{item1}, \eqref{item7}, and \eqref{item8} is Proposition~3.1.12 of~\cite{GeckPfeifferBook}.  Now suppose $\cC$ is cuspidal.  Then as remarked on p.~77 of \cite{GeckPfeifferBook}, the fix-set of $w$ in $V$, and thus in $V^*$, is trivial, so \eqref{item2} holds.  Therefore as $\Fix(w)$ and $\Mov(w)$ are orthogonal complements in~$V^*$, we obtain \eqref{item3}.   
\end{proof}

\begin{definition}[Coxeter elements]
	\label{defn:Coxeter}
	A \emph{Coxeter element} of $\sW$ is a product of all elements of~$S$, in any given order.
\end{definition}

Theorem~3.1.4 and Proposition~3.1.6 of~\cite{GeckPfeifferBook} imply the following statement.

\begin{prop}	
	\label{prop:CoxeterCuspidal} 
	The Coxeter elements of $\sW$ are contained in a single conjugacy class, which is cuspidal.
\end{prop}

In type $A_n$, by Example~3.1.16 of~\cite{GeckPfeifferBook}, the unique cuspidal conjugacy class of $\sW$ is the one containing all Coxeter elements.
However in general there exist cuspidal elements which are neither Coxeter nor in the conjugacy class containing the Coxeter elements, such as $w_0 = s_1 s_2 s_1 s_2$ in type $C_2$, equivalently type $B_2$.

\subsection{Summary of previous results on (co)conjugation}
\label{sec:summary}

In this section we give formal statements of the main definitions and results of \cite{MST4} on Euclidean isometry groups, in the setting of affine Coxeter groups. We also recall some relevant results from~\cite{MST2}. 

We first restate a key definition from the introduction above.

\begin{definition}[Mod-set, see Definition 1.1 of \cite{MST4}]  For any $x \in \aW$, the \emph{mod-set of $x$ (with respect to $\aW$)}  is given by:
	\[ 	\label{eq:mod} 
	\ModW(x) = (x - \Id) R^\vee = (\Id - x)R^\vee. 
	\]
\end{definition}

The next result gathers some first properties of mod-sets. Parts \eqref{mod-conj}, \eqref{mod-affine-translate}, and \eqref{mod-mov} in the next statement are special cases of Lemmas 2.1, 2.2, and 2.3 of~\cite{MST4}, respectively.

\begin{lemma}[Properties of mod-sets]\label{lem:mod} Let $\lambda \in R^\vee$ and $w \in \sW$. Then:
\begin{enumerate}
\item \label{mod-conj} For all $u \in \sW$, $u \ModW(w) = \ModW(uwu^{-1})$;
\item \label{mod-affine-translate}  $ \ModW(t^\lambda w)=\lambda+\ModW(w)$; and
\item \label{mod-mov} $\ModW(w) \subseteq \Mov(w) \cap R^\vee$.
\end{enumerate}
\end{lemma}

The following theorem describes conjugacy classes in $\aW$ in terms of mod-sets. For any $x \in \aW$, we write $\xconj = \{ yxy^{-1} \mid y \in \aW \}$ for its conjugacy class in $\aW$.

\begin{thm}[Closed form of conjugacy classes, see Theorem~1.2 of \cite{MST4}]
	\label{thm:ConjClass}  Let $x = t^\lambda w \in \aW$, where $\lambda \in R^\vee$ and $w \in \sW$.  Then the conjugacy class of $x$ in $\aW$ satisfies
	\begin{equation}\label{eq:conj1}
		\xconj    
		= \bigcup_{u \in \sW} u \left( t^{\ModW(w) } x \right) u^{-1} 
	\end{equation}
	and also
	\begin{equation}\label{eq:conj2}
		\xconj    
		= \bigcup_{u \in \sW} t^{u( \lambda + \ModW(w) )} u w  u^{-1} = \bigcup_{u \in \sW} t^{u\ModW(x) } u w  u^{-1}.
	\end{equation}
\end{thm}

For any $x\in\aW$ the class $\xconj$ is thus obtained by first translating $x$ by all elements of $\ModW(w)$, and then conjugating the so-obtained collection $t^{\ModW(w)}x$ by all elements of $\sW$. Alternatively, for each $u \in \sW$, we translate the $u$-conjugate of the spherical part $w$ of $x$ by the set $t^{u( \lambda + \ModW(w) )}=t^{u \lambda} t^{u\ModW(w)}$.

Let us emphasize here that the conjugacy class of every element $x=t^\lambda w\in\aW$ is determined by finite data: the conjugacy class of the spherical part $w$ in the finite Coxeter group $\sW$, together with the collection of shifted mod-sets $t^{u( \lambda + \ModW(w))}$, where $u \in \sW$ varies.

\begin{corollary}[Conjugacy classes and move-sets, see Corollary 1.3 of~\cite{MST4}] 
Let $x = t^\lambda w \in \aW$, where $\lambda \in R^\vee$ and $w \in \sW$. Then,
\begin{equation}\label{eq:conjMov1}
		\xconj    
		\subseteq \bigcup_{u \in \sW} u \left( t^{\Mov(w) \cap R^\vee} x \right) u^{-1}
		\end{equation}
and also
\begin{equation}\label{eq:conjMov2}
		\xconj  
		\subseteq \bigcup_{u \in \sW} t^{u(\Mov(x) \cap R^\vee)} u w  u^{-1}.
		\end{equation}
\end{corollary}

\noindent The second main definition of~\cite{MST4} is motivated by these containments.

\begin{definition}[Filling, see Definition 1.4 of~\cite{MST4}]\label{defn:filling} Let $x \in \aW$. We say that $x$ \emph{fills its move-set}, or that \emph{filling occurs for $x$}, if
\[
\ModW(x) = \Mov(x) \cap R^\vee.
\]
\end{definition}

In~\cite{MST2} we studied the filling condition for simple reflections, using direct computations.  
We will use the next result to study the filling property for arbitrary elements of $\aW$. Its proof (in~\cite{MST4}) is similar to that of \cite[Lemma~5.8]{MST2}. 

\begin{prop}[Conjugacy classes and filling, see Proposition 2.4 of~\cite{MST4}]\label{prop:conjFilling} For all $x = t^\lambda w \in \aW$, where $\lambda \in R^\vee$ and $w \in \sW$, the following are equivalent:
\begin{enumerate}
\item\label{fill:x} $x$ fills its move-set;
\item\label{fill:w} $w$ fills its move-set;
\item\label{fill:xconjw} $\xconj    =  \bigcup_{u \in \sW} u \left( t^{\Mov(w) \cap R^\vee} x \right) u^{-1}$; and
\item\label{fill:xconjx} $\xconj =   \bigcup_{u \in \sW} t^{u(\Mov(x) \cap R^\vee)} u w  u^{-1}$.
\end{enumerate}
\end{prop}

A \emph{component} of a conjugacy class $\xconj$ is defined in~\cite[Definition 2.5]{MST4} to be a subset of $\xconj$ of the form $u(t^{\ModW(w)} x)u^{-1}$ with $u\in \sW$; that is, a set appearing in the union given by \eqref{eq:conj1} above.
We write $\CompW(x)$ for the set of components of $\xconj$. By definition, the group $\sW$ acts transitively by conjugation on $\CompW(x)$. The following result is discussed in the introduction in Remark~\ref{rmk:components}.

\begin{thm}[Components, see Theorem~1.11 of~\cite{MST4}] 
	\label{thm:Components}   Let $x = t^\lambda w \in \aW$, where $\lambda \in R^\vee$ and $w \in \sW$. Then:
	\begin{enumerate}
		\item The conjugation action of $\TW$ induces a transitive action by translations on the elements of each component of $\xconj$.
		\item Linearization induces a natural surjection from $\CompW(x)$ to $\CompW(w)$. 
		\item There is a natural bijection between  $\CompW(w)$ and  $[w]_{\sW}$.   
	\end{enumerate}
\end{thm}

\medskip

For $x,x' \in \aW$ we define the \emph{coconjugation set} $$\coconjW{x}{x'} = \{ y \in \aW \mid yxy^{-1} = x' \},$$ and for $w,w' \in \sW$ we define the \emph{spherical coconjugation set} $$\scoconjW{w}{w'} = \{ u \in \sW \mid uwu^{-1} = w' \}.$$ Our geometric description of $\coconjW{x}{x'}$ involves the following subset of the corresponding spherical coconjugation set.

\begin{definition}[see Definition 1.12 of~\cite{MST4}] Let $x = t^\lambda w$ and $x' = t^{\lambda'} w'$ be elements of~$\aW$, where $\lambda, \lambda' \in R^\vee$ and $w,w' \in \sW$. The \emph{translation-compatible part} of the coconjugation set $\scoconjW{w}{w'}$ is the set
\[
\tcCoconj_\sW^{\lambda,\lambda'}(w,w') = \{ u \in \scoconjW{w}{w'} \mid \lambda' - u\lambda \in \ModW(w') \}.
\]
\end{definition}

The final main result of~\cite{MST4} describes coconjugation sets as follows. 

\begin{thm}[Coconjugation, see Theorem 1.13 of~\cite{MST4}]
	\label{thm:coconj} Let $x=t^{\lambda}w$ and $x' = t^{\lambda'}w'$ be elements of $\aW$, where $\lambda, \lambda' \in R^\vee$ and $w, w' \in \sW$.  Then 
	\begin{equation}\label{eq:coconjNonempty}
		\coconjW{x}{x'} \neq \emptyset \;\; \Longleftrightarrow\;  \tcCoconj_\sW^{\lambda,\lambda'}(w,w') \neq \emptyset.
	\end{equation}
	Moreover, if these sets are nonempty, then
	\begin{equation}\label{eq:coconjFix}
	\coconjW{x}{x'} = \bigsqcup_{u \in \tcCoconj_\sW^{\lambda,\lambda'}(w,w')} t^{\eta_{u}+ (\Fix(w') \cap R^\vee)} u 
	\end{equation}
	where for each $u$, the element $\eta_{u} \in R^\vee$ is a solution to the equation 
	\begin{equation}\label{eq:coconj}
		\lambda'- u\lambda=(\Id - w')\eta.
	\end{equation}
	In the special case that $\Fix(w)=\{0\}$, we have that 
	\[ 
	\eta_u=(\Id-w')^{-1}(\lambda'-u\lambda)
	\] 
	is the unique solution to~\eqref{eq:coconj}, and $\coconjW{x}{x'}$ is in bijection with $\tcCoconj_\sW^{\lambda,\lambda'}(w,w')$.	
\end{thm}


\section{Mod-sets and move-sets}\label{sec:modMove}


In this section we prove Theorem~\ref{thm:modMoveW} and Corollary~\ref{cor:modMoveW} of the introduction. We also prove the generalization Theorem~\ref{thm:modMove} below, which applies to both affine Coxeter groups and their split crystallographic subgroups.

The bulk of this section is devoted to the proof of Theorem~\ref{thm:modMoveW}. This says that for $w \in \sW$, 
\[
\rank_\Z(\ModW(w)) = \dim_\R(\Mov(w)) = \ell_\cR(w),
\]
where $\ell_R(w)$ is the reflection length of $w$ (see Section~\ref{sec:affine}). To prove these equalities, we establish some results on intersections of hyperplanes in Section~\ref{sec:hyperplanes}, then use these to complete the proof of  Theorem~\ref{thm:modMoveW} in Section~\ref{sec:proofsModMoveW}. We then state and prove Theorem~\ref{thm:modMove} in Section~\ref{sec:proofsModMoveH}.

\subsection{Hyperplane intersections}\label{sec:hyperplanes}

In this section we establish two technical results on hyperplane intersections that we will use in our proof of Theorem~\ref{thm:modMoveW}. These two results are obvious if working over $\R$, but require some care over $\Z$.

Recall from Section~\ref{sec:affine} that for any root $\alpha \in \Phi$, we denote by $\cH_{\alpha}$ the (linear) hyperplane which is the fix-set of the reflection $s_\alpha$.

\begin{lemma}\label{lem:intersection} Let $\{ \beta_1, \dots, \beta_k \} \subset \Phi$ be a set of roots such that the set of corresponding coroots $\{ \beta_1^\vee, \dots, \beta_k^\vee \}$ is linearly independent over~$\Z$. Then the submodule of $R^\vee$ given by
\begin{equation}\label{eq:intersection}
\left( \bigcap_{i=1}^k \cH_{\beta_i} \right) \cap R^\vee
\end{equation}
has rank $(n-k)$ as a $\Z$-module.
\end{lemma}

\begin{proof} 
To simplify notation, for $1 \leq i \leq k$ write $\cH_i$ for $\cH_{\beta_i}$, and write $M$ for the intersection given in~\eqref{eq:intersection}. Then $M = 
 \bigcap_{i=1}^k \left( \cH_i \cap R^\vee \right)$, and for $1 \leq i \leq k$, we have 
 \[ 
 \cH_i \cap R^\vee = \{ \mu \in R^\vee \mid \langle \mu,  \beta_i \rangle = 0 \}.
 \] 
Write $\mu = \sum_{i \in [n]} m_i \alpha_i^\vee$ for unknown $m_i \in \Z$, and expand $\beta_i = \sum_{j \in [n]} b_{ij} \alpha_j$ where $b_{ij} \in \Z$.
Since the evaluation pairing on $V^* \times V$ is bilinear,  the condition $\langle \mu, \beta_i \rangle = 0$ then becomes a single equation in the $n$ variables $m_i$.  Moreover, since $\langle \alpha_i^\vee, \alpha_j \rangle \in \Z$ for all $1 \leq i,j \leq n$, and all  $b_{ij} \in \Z$, the equation  $\langle \mu, \beta_i \rangle = 0$ has $\Z$-coefficients.
We may thus identify the points of $M$ with the set of solutions over $\Z$ to the corresponding system of $k$ homogeneous linear equations in $n$ variables, with coefficients in $\Z$. 

Now the set $\{ \beta_1^\vee, \dots, \beta_k^\vee \}$ is linearly independent over $\Q$, since by assumption it is linearly independent over $\Z$. Hence over the field $\Q$, the set of solutions to this system of homogeneous equations is a subspace of dimension $(n - k)$. Therefore the $\Z$-rank of the submodule $M$ of~$R^\vee$ is at most $(n-k)$.

If $\rank_\Z(M) = m  < n-k$, let $\{ \eta_1^\vee, \dots, \eta_m^\vee \}$ be a $\Z$-basis for $M$. Then the $\Q$-span of 
$\{ \eta_1^\vee, \dots, \eta_m^\vee \}$ is a proper subspace of the set of solutions over $\Q$, so there is a $\lambda \in R^\vee$ which is a solution over $\Q$ with $\lambda \not \in \Span_\Q \{ \eta_1^\vee, \dots, \eta_m^\vee \}$. Now $\lambda \in \Span_\Q\{\alpha_1^\vee, \dots, \alpha_n^\vee \}$,  so for large enough $N$ we have $N\lambda \in \Span_\Z\{\alpha_1^\vee, \dots, \alpha_n^\vee \}$. That is, for large enough $N$ we have $N\lambda \in R^\vee$. And since the system of equations defining $M$ is homogeneous, then $N\lambda$ is also a solution to this system. Hence $N\lambda \in \Span_\Z\{ \eta_1^\vee, \dots, \eta_m^\vee \}$. But then $\lambda \in \Span_\Q \{ \eta_1^\vee, \dots, \eta_m^\vee \}$, a contradiction. We conclude that $\rank_\Z(M) = n-k$.
\end{proof}

\begin{corollary}\label{cor:intersection}  Let $\{ \beta_1, \dots, \beta_k \} \subset \Phi$ be a set of roots such that the set of corresponding coroots $\{ \beta_1^\vee, \dots, \beta_k^\vee \}$ is linearly independent over~$\Z$.  Then for all $1 \leq j \leq k$ there exists $\mu_j \in R^\vee$ such that
\[
\mu_j \in \left(\bigcap_{i=j+1}^k  \cH_{\beta_i} \right) \quad \mbox{but} \quad \mu_j \not \in \cH_{\beta_j}.
\]
\end{corollary}
\begin{proof} By Lemma~\ref{lem:intersection}, the intersection $ \left(\bigcap_{i=j+1}^k  \cH_{\beta_i} \right) \cap R^\vee$ has $\Z$-rank $n - (k-j) = n-k +j$, while its submodule $\left(\bigcap_{i=j}^k  \cH_{\beta_i} \right) \cap R^\vee$ has $\Z$-rank $n-(k-(j-1)) = n-k + j -1$. The result follows.
\end{proof}

\subsection{Rank vs dimension for affine Coxeter groups}\label{sec:proofsModMoveW}

In this section we establish Theorem~\ref{thm:modMoveW} of the introduction. We consider reflections in $\sW$, then arbitrary elements of $\sW$.

\begin{lemma}\label{lem:ModWr} For all reflections $r = s_\alpha \in \sW$, the mod-set $\ModW(r)$ is generated as a $\Z$-module by some nonzero integer multiple of $\alpha^\vee$.
\end{lemma}

\begin{proof} 
 Since $\ModW(s_\alpha) \subset \Mov(s_\alpha) =  \R\alpha^\vee$, we have $\rank_\Z(\ModW(s_\alpha)) \leq 1$. Now by the case $j = k = 1$ of Corollary~\ref{cor:intersection}, there is a $\mu \in R^\vee$ such that $\mu \not \in \cH_{\alpha} = \Fix(s_\alpha)$. Then $(\Id - s_\alpha)\mu \neq 0$, and \[ (\Id - s_\alpha)\mu \in \ModW(s_\alpha) \subseteq \left(\Mov(s_\alpha) \cap R^\vee\right) = \left(\R\alpha^\vee \cap R^\vee\right) = \Z \alpha^\vee.\] Thus $\ModW(s_\alpha)$ contains some nonzero integer multiple of $\alpha^\vee$. The result follows.
\end{proof}

\begin{corollary}\label{cor:reflection} For all reflections $r \in \sW$, we have \[ \rank_\Z(\ModW(r)) = \dim_R(\Mov(r)) = \ell_\cR(r) = 1.\]
\end{corollary}

We now consider arbitrary elements of $\sW$. We make essential use of one direction of Theorem~\ref{thm:Carter}, and generalize Carter's proof of the other direction of that result. 

\begin{prop}\label{prop:rankMod} Let $w \in \sW$. Then $\rank_\Z(\ModW(w)) = \ell_\cR(w)$.
\end{prop}
\begin{proof} Let $w = r_1 \cdots r_k$ be a minimal length reflection factorization of $w$, so that $\ell_\cR(w) = k$, and for $1 \leq i \leq k$ let $\beta_i \in \Phi$ be a root such that $r_i = s_{\beta_i}$. To simplify notation, write $\cH_i$ for $\cH_{\beta_i}$. Then by Theorem~\ref{thm:Carter}, the set $\{ \beta_1, \dots, \beta_k \}$ is linearly independent over $\R$. Hence the set $\{ \beta_1^\vee, \dots, \beta_k^\vee \}$ is linearly independent over $\R$, and so $\{ \beta_1^\vee, \dots, \beta_k^\vee \} \subset \Phi^\vee$ is linearly independent over $\Z$. 

By Lemma~\ref{lem:ModWr}, each $r_i$ only moves elements of $R^\vee$ by integer multiples of $\beta_i^\vee$, and so the motion of any $\mu \in R^\vee$ under $w$ is a $\Z$-linear combination of the $\beta_i^\vee$. Thus $\ModW(w) \subseteq \Span_\Z\{ \beta_1^\vee, \dots, \beta_k^\vee \}$ and so $\rank_\Z(\ModW(w)) \leq k$. 

We now use Corollary~\ref{cor:intersection} to prove by induction that $\ModW(w)$ contains some nonzero integer multiple of each of $\beta_1^\vee, \dots, \beta_k^\vee$. Since the $\beta_i^\vee$ are linearly independent over $\Z$, this establishes $\rank_\Z(\ModW(w)) = k$.

First, by the case $j = 1$ of Corollary~\ref{cor:intersection}, there is a $\mu_1 \in R^\vee$ such that $\mu_1 \in \left(\cap_{i=2}^k  \cH_i \right)$ but $\mu_1 \not \in \cH_1$. Then since $r_2, \dots, r_k$ fix $\mu_1$ while $r_1$ does not, by Lemma~\ref{lem:ModWr} we have $w \mu_1 = r_1 \mu_1 = \mu_1 - c_1 \beta_1^\vee$ for some $c_1 \in \Z$ with $c_1 \neq 0$. It follows that $\ModW(w)$ contains a nonzero integer multiple of $\beta_1^\vee$, namely $c_1 \beta_1^\vee$.

Assume inductively that $\ModW(w)$ contains $c_1 \beta_1^\vee, \dots, c_{j-1} \beta_{j-1}^\vee$ where $c_1,\dots,c_{j-1}$ are nonzero integers. Define $c = \prod_{i=1}^{j-1} c_i$, so that $c \in \Z$ and $c \neq 0$. By Corollary~\ref{cor:intersection} there is a $\mu_j \in R^\vee$ such that $\mu_j \in \left(\cap_{i=j+1}^k  \cH_i \right)$ but $\mu_j \not \in \cH_j$. Hence $c\mu_j \in R^\vee$, $c \mu_j \in \left(\cap_{i=j+1}^k  \cH_i \right)$, but $c\mu_j \not \in \cH_j$. Now 
 \[w (c\mu_j) = cr_1 \cdots  r_j \mu_j = cr_1 \cdots r_{j-1}(\mu_j - c_j \beta_j^\vee) = c\mu_j - (cc_j) \beta_j^\vee - \sum_{i = 1}^{j-1} c_{ij}\beta_i^\vee,\] for some $c_j \in \Z$ with $c_j \neq 0$ and some $c_{ij} \in \Z$ divisible by $c_i$, for $1 \leq i \leq j-1$. Then by subtracting suitable multiples of  $c_1 \beta_1^\vee, \dots, c_{j-1} \beta_{j-1}^\vee$, we obtain that $\ModW(w)$ contains a nonzero scalar multiple of $\beta_j^\vee$, namely $cc_j\beta_j^\vee$. This completes the proof.
\end{proof}

\begin{corollary}\label{cor:allWeyl} For all $w \in\sW$, we have \[ \rank_\Z(\ModW(w)) = \dim_\R(\Mov(w)) = \ell_\cR(w).\]
\end{corollary}
\begin{proof} This follows from Corollary~\ref{cor:dimMov} and Proposition~\ref{prop:rankMod}.
\end{proof}

\subsection{Mod-sets and move-sets for split crystallographic groups}
\label{sec:proofsModMoveH}

We now state and prove Theorem~\ref{thm:modMove}. This generalizes the first equality in Theorem~\ref{thm:modMoveW}, and both parts of Corollary~\ref{cor:modMoveW}.

In order to state Theorem~\ref{thm:modMove}, recall that an \emph{$n$-dimensional crystallographic group} is a discrete, cocompact group of isometries of $n$-dimensional Euclidean space $\R^n$. We define two $n$-dimensional crystallographic groups to be \emph{equivalent}, denoted~$\sim$, if they are conjugate under some affine transformation of $\R^n$ (the classification of crystallographic groups is usually up to this relation). Then by abuse of terminology, we say that an $n$-dimensional crystallographic group~$\aH$ is \emph{contained in an affine Coxeter group} if there is an $n$-dimensional crystallographic group $\aH'$, and an affine Coxeter group $\aW$ (not necessarily irreducible) which is also an $n$-dimensional crystallographic group, such that $\aH \sim \aH'$ and $\aH' \leq \aW$. 

It is classical that the full isometry group $G$ of $n$-dimensional Euclidean space splits as $G = \R^n \rtimes \On$, where $\R^n$ is the additive group of all translations and $\On$ is the orthogonal group. An $n$-dimensional crystallographic group $H$ is \emph{split} if it respects this splitting, that is, $\aH = T_\aH \rtimes \sH$ where $T_\aH$ is the translation subgroup of $\aH$ and $\sH = \aH \cap \On$.  The set $L_\aH = \{ \lambda : t^\lambda \in T_\aH \}$ is a cocompact lattice in $\R^n$, and so $L_\aH$ can be viewed as a free $\Z$-module of rank $n$.  

Theorem~\ref{thm:modMove} describes the relationships between the mod-set $\ModH(h_0)$ for $h_0 \in \sH$, the move-set $\Mov(h_0)$ for $h_0 \in \sH$, and the lattice $L_\aH$, in the case that $\aH$ is contained in an affine Coxeter group.

\begin{thm}[Mod-sets and move-sets]\label{thm:modMove} Let $\aH = T_\aH \rtimes \sH$ be a split crystallographic group. Assume that $\aH$ is contained in an affine Coxeter group $\sW = T \rtimes \sW$. Then for all $h_0 \in \sH$:
\begin{enumerate}
\item\label{item:rank} $\rank_\Z(\ModH(h_0)) = \dim_\R(\Mov(h_0))$; 
\item\label{item:index} $\ModH(h_0)$ is a finite-index submodule of $\Mov(h_0) \cap L_\aH$; and
\item\label{item:equal} $\ModH(h_0) = \Mov(h_0) \cap L_\aH$ if and only if $L_\aH / \ModH(h_0)$ is torsion-free.
\end{enumerate}
\end{thm}

The next result establishes part \eqref{item:rank} of this theorem.

\begin{corollary}\label{cor:containedAffine} Let $\aH = T_{\aH} \rtimes\sH$ be a split crystallographic group which is contained in an affine Coxeter group $\aW = \TW \rtimes \sW$. Then for every $h_0 \in \sH$, we have $$\rank_\Z(\Mod_{\aH}(h_0))= \dim_\R(\Mov(h_0)).$$
\end{corollary}
\begin{proof} We may assume without loss of generality that $\aH$ is a subgroup of $\aW$, and so in particular $\sH \leq \sW$ and $L_\aH \subseteq R^\vee$.

Let $h_0 \in \sH$. Then $h_0 \in\sW$, so we may write $h_0$ as a minimal length product of reflections in $\sW$, say $h_0 = r_1 \cdots r_k$ with $r_i = s_{\beta_i}$ for $1 \leq i \leq k$.
Since $L_\aH \subseteq R^\vee$, the same argument as in the proof of Proposition~\ref{prop:rankMod} then tells us that $\Mod_{\aH}(h_0) \subseteq \Span_\Z\{ \beta_1^\vee, \dots, \beta_k^\vee \}$, and hence $\rank_\Z(\ModH(h_0)) \leq k$.

Choose $\mu_1, \dots, \mu_k \in R^\vee$ as in the statement of Corollary~\ref{cor:intersection}. Now the lattices $L_{\aH}$ and $R^\vee$ are both free $\Z$-modules of rank $n$, so $L_{\aH}$ must be a finite-index submodule of $R^\vee$.  Thus there is an integer, say $d$, such that $d\mu_i \in L_{\aH}$ for all $1 \leq i \leq k$. Then since each $\cH_{\beta_i}$ is a subspace, we have
\[
d\mu_j \in \left(\bigcap_{i=j+1}^k  \cH_{\beta_i} \right) \quad \mbox{but} \quad d\mu_j \not \in \cH_{\beta_j}.
\]
 Now by a similar argument to that given in the proof of Proposition~\ref{prop:rankMod}, but replacing each $\mu_j$ by $d\mu_j$, we see that $\Mod_{\aH}(h_0)$ contains some nonzero integer multiple of each of $\beta_1^\vee, \dots, \beta_k^\vee$, and so it has $\Z$-rank equal to  $k$.
\end{proof}

We now prove the remainder of Theorem~\ref{thm:modMove}. In particular, the next result establishes Corollary~\ref{cor:modMoveW}.

\begin{corollary} Let $\aH = T_{\aH} \rtimes\sH$ be a split crystallographic group which is contained in an affine Coxeter group. Then for all $h_0 \in \aH$:
\begin{enumerate}
 \item $\ModH(h_0)$ is a finite-index submodule of $\Mov(h_0) \cap L_{\aH}$; and
 \item $\ModH(h_0) = \Mov(h_0) \cap L_{\aH}$ if and only if $L_{\aH}/\ModH(h_0)$ is torsion-free.
 \end{enumerate}
\end{corollary}
\begin{proof}  By Corollary~\ref{cor:containedAffine}, we have $\rank_\Z(\ModH(h_0)) = \dim_\R(\Mov(h_0)) = k$, say, and so $\rank_\Z(\Mov(h_0) \cap L_\aH) \leq \dim_\R(\Mov(h_0)) = k$. Hence $\Mov(h_0) \cap L_{\aH}$  and its submodule $\ModH(h_0)$ are both free $\Z$-modules of rank $k$, and (1) follows. 

For (2), given (1) it suffices to show that the quotient $L_\aH / (\Mov(h_0) \cap L_\aH)$ is torsion-free. If $\Mov(h_0) = \R^n$ then $\Mov(h_0) \cap L_\aH = L_\aH$ so this quotient is trivial and we are done. Assume now that $\Mov(h_0)$ is a proper subspace of $\R^n$, and suppose $\lambda \in L_\aH$ represents a torsion element of $L_\aH / (\Mov(h_0) \cap L_\aH)$. Then $\lambda \not \in \Mov(h_0)$ but there is some $c \in \Z$ with $c \neq 0$ so that $c\lambda \in \Mov(h_0)$. This is a contradiction, as $\Mov(h_0)$ is a subspace of $\R^n$.
\end{proof}

In light of the results in this section and the main results obtained in \cite{MST4} (as restated in \Cref{sec:summary}), it has become clear that the mod-sets are the crucial objects ruling the shape of conjugacy classes in affine Coxeter groups, and are critical to the structure of their coconjugation sets as well. The remainder of the paper, starting with \Cref{sec:TypeA}, is hence devoted to the study of precisely these $\Z$-modules.

\section{Type $A$ mod-sets} \label{sec:TypeA}

In this section, we give an explicit description of all mod-sets in type $A$, stated as Theorem \ref{thm:An}.  These results will also be used in Sections~\ref{sec:TypeC},~\ref{sec:TypeB}, and~\ref{sec:TypeD}, since types $B, C,$ and $D$ all have type $A$ subsystems.  Let $\sW$ be the finite Weyl group of type $A_n$, with $n \geq 1$.  In \Cref{sec:repsA}, we recall the complete system of minimal length representatives for the conjugacy classes of~$\sW$, which we rephrase in Proposition \ref{prop:wbetaA}.  We then give several key examples in \Cref{sec:examplesA}, before proving our results in type $A_n$ in \Cref{sec:proofsA}.  

Throughout this paper, we order the nodes of the Dynkin diagram for $\sW$ of type $A_n$ increasing from left to right, as in both \cite{Bourbaki4-6} and \cite{GeckPfeifferBook}.  See Table \ref{table:dynkin} in~\Cref{app:dynkin} for a direct comparison of choices of labeling.

\subsection{Conjugacy class representatives and mod-sets in type $A$}
\label{sec:repsA}

Recall from \cite[Prop.~3.4.1]{GeckPfeifferBook}, for example, that the conjugacy classes of $\sW$ of type $A_n$ are parameterized by partitions of $n+1$; that is, weakly decreasing sequences $\beta = (\beta_1, \dots, \beta_p)$ of nonnegative integers such that $| \beta | = \sum \beta_i = n+1$.   For each such $\beta$, we will define a standard parabolic subgroup $\sW_\beta$ of $\sW$, and a Coxeter (equivalently cuspidal, in type $A$) element $w_\beta$ of $\sW_\beta$, so that the set of all such $w_\beta$ forms a complete system of minimal length representatives for the conjugacy classes of~$\sW$.

We first provide a summary of our algorithm for obtaining a minimal length conjugacy class representative $w_\beta$ associated to a partition $\beta$. The following result is classical, though we emphasize a visualization, as this perspective will be helpful in the other classical types.

\begin{prop}\label{prop:wbetaA}
Let $\beta = (\beta_1, \dots, \beta_p)$ be a partition of $n+1$. Define an element $w_\beta$ of the finite Weyl group $\sW$ of type $A_n$ as follows:
\begin{enumerate}
\item subtract $1$ from each part of $\beta$, 
\item take the connected subdiagrams of the Dynkin diagram with  $\beta_1 - 1$, $\beta_2 - 1$, \dots, $\beta_p - 1$ nodes, respectively, going from left to right, omitting only a single node between any two such nonempty subdiagrams, and 
\item multiply together the simple reflections indexed by the nodes of these subdiagrams in increasing order to obtain $w_\beta$.
\end{enumerate}
Then the set of all such $w_\beta$ forms a complete system of minimal length representatives for the conjugacy classes of $\sW$.
\end{prop}

The following figure illustrates the construction in Proposition \ref{prop:wbetaA}; see Section \ref{sec:examplesA} for additional examples, including Example \ref{eg:431A7}  for more details on Figure \ref{fig:wbetaA}.

\begin{figure}[h]
\begin{center}
 \resizebox{3in}{!}
 {
\begin{overpic}{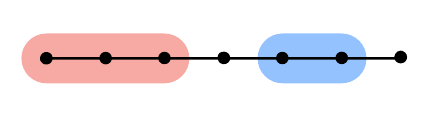}
\put(8,5){$s_1$}
\put(22,5){$s_2$}
\put(36,5){$s_3$}
\put(50,5){$s_4$}
\put(63.5,5){$s_5$}
\put(77.5,5){$s_6$}
\put(91,5){$s_7$}
\put(14,25){$4-1 = \textcolor{red}{3}$}
\put(61,25){$3-1 = \textcolor{blue}{2}$}
\end{overpic}
}
\caption{Constructing $w_\beta = \textcolor{red}{(}s_1s_2s_3\textcolor{red}{)}\textcolor{blue}{(}s_5s_6\textcolor{blue}{)}$ for $\beta = (4,3,1)$ in type $A_7$.}
\label{fig:wbetaA}
\end{center}
\end{figure}

We now establish the notation needed to formally state Theorem \ref{thm:An} characterizing all mod-sets in type $A$.  Fix a partition $\beta = (\beta_1,\dots,\beta_p)$ of $n+1$.  We define a strictly increasing subsequence of $0,1,\dots,n$ as follows.  Let $j^\beta_1 = 0$, and for $2 \leq k \leq p$, define \[j^\beta_k = \sum_{i=1}^{k-1} \beta_i = \beta_1 + \dots + \beta_{k-1}. \]  
In other words, for $p \geq 2$ and any $1 \leq k < p$, we have $j^\beta_{k+1} = j^\beta_k + \beta_k$.  Hence in particular, $0 = j^\beta_1 < j^\beta_2 < \dots < j^\beta_p \leq n$.

We next define subsets $J^\beta_k$ of $[n]$, so that $[n]$ will be the disjoint union of the $J^\beta_k$ with the set of integers $\{ j^\beta_k \mid 2 \leq k \leq p \}$.  For $k \in [p]$, if $\beta_k = 1$ put $J^\beta_k = \emptyset$, and if $\beta_k \geq 2$ define $J^\beta_k$ to be the subinterval of $[n]$ given by
\begin{equation}\label{eq:Jbetak}
J^\beta_{k} = \{ j^\beta_k+1, j^\beta_k+2, \dots, j^\beta_k+(\beta_k-1)\}.
\end{equation}
Thus in particular, if $\beta_1 \geq 2$ then $J^{\beta}_1$ is the initial subinterval $J^\beta_1 = \{1,2,\dots,\beta_1 - 1\}$, and if $p \geq 2$ and $\beta_p \geq 2$ then $J^\beta_p$ is the terminal subinterval $J^\beta_p = \{ \beta_1 + \dots + \beta_{p-1} + 1, \dots, n\}$.  Notice that the $J^\beta_k$ are pairwise disjoint, with each $J^\beta_k$ containing $\beta_k -1$ elements, and that if $J^\beta_k$ and $J^\beta_{k+1}$ are both nonempty, then the unique element of $[n]$ which lies strictly between these subintervals is the integer $j_{k+1}^\beta = j_k^\beta + \beta_k$.  When $\beta_k \geq 2$, we may identify the elements of the nonempty interval $J^\beta_k$ with the nodes of the corresponding (connected) subdiagram of type $A_{\beta_k - 1}$ of the Dynkin diagram, as illustrated in Figure \ref{fig:wbetaA}.   We also define the union $J_\beta = \sqcup_{k=1}^p J_k^\beta \subseteq [n]$.  Then if $p =1$ we have $J_\beta = J_1^\beta = [n]$, while for all $p \geq 2$,  we have \[ [n] \setminus J_\beta = \{ j^\beta_k \mid 2 \leq k \leq p \} = \{ \beta_1, \beta_1 + \beta_2, \dots, \beta_1 + \dots + \beta_{p-1} \}.  \] 

Next, for each $k \in [p]$, write $W_{J^{\beta}_k}$ for the (possibly trivial) standard parabolic subgroup of $\sW$ generated by the simple reflections $\{ s_j \mid j \in J^\beta_k\}$.  Then each nontrivial $W_{J^\beta_k}$ is of type $A_{\beta_k - 1}$, and the $W_{J^\beta_k}$ pairwise commute.  Hence the (possibly trivial) standard parabolic subgroup of $\sW$ generated by the simple reflections $\{ s_j \mid j \in J_\beta\}$ has the form
\[
W_\beta = W_{J_\beta} = W_{J^{\beta}_1} \times \dots \times W_{J^{\beta}_p}.
\]
Now if $\beta_k = 1$ let $w^\beta_k$ be the trivial element of $\sW$, and if $\beta_k \geq 2$ let $w^{\beta}_k$ be the Coxeter element of $\sW_{J^{\beta}_k}$ given by:
\begin{equation}\label{eq:w^beta_k}
w^{\beta}_k = s_{j^\beta_k+1}s_{j^\beta_k+2} \cdots s_{j^\beta_k+(\beta_k-1)} \in \sW_{J^{\beta}_k}.
\end{equation}
That is, $w^\beta_k$ is the product of the simple reflections in $\sW_{J^\beta_k}$, in increasing order. (In the language of \cite[Sec.~3.4.2]{GeckPfeifferBook}, the element $w^{\beta}_k$ is the positive block $b_{j^\beta_k, \beta_k}^+$ of length $\beta_k$ starting at $j^\beta_k$.)   

Finally, the partition $\beta = (\beta_1, \dots, \beta_p)$ corresponds to the product \[w_\beta = w^\beta_1 \cdots w^\beta_p \in \sW_\beta.\] That is, $w_\beta$ is the Coxeter element of $W_\beta$ which is the product of its simple reflections, in increasing order.  Note that $\supp(w_\beta) = \{ s_j \mid j \in J_\beta \}$, so that $S \setminus \supp(w_\beta)$ is indexed by $[n] \setminus J_\beta = \{ j_k^\beta \mid 2 \leq k \leq p \}$.  

To simplify notation in the theorem below, given any partition $\beta = (\beta_1, \dots, \beta_p)$ of $n+1$, we define the following $p$-element subsets of $[n+1]$:
\[ I_\beta = \{ j_k^\beta \mid 2 \leq k \leq p \} \cup \{n+1\} \quad \text{and} \quad I_\beta - 1 = \{ i-1 \mid i \in I_\beta\}.\]

With this notation established, we can now state our results describing mod-sets in type~$A$. 

\begin{thm}
	\label{thm:An}
	Suppose $\sW$ is of type $A_n$ with $n \geq 1$.  Let $\beta = (\beta_1, \dots, \beta_p)$ be a partition of $n+1$ with corresponding conjugacy class representative $w_\beta \in \sW$.  
\begin{enumerate}
\item\label{ModA} The module $\Mod(w_\beta) = (\Id-w_\beta)R^\vee$ equals
\[  \left\{\sum_{i=1}^n c_i \alpha_i^\vee \ \middle| \ c_i \in \Z, \quad c_i=0\ \text{for}\ i \in [n] \setminus J_\beta, \quad  \sum_{i=1}^n c_i \equiv 0  \; \operatorname{mod}\left( \gcd(\beta_k) \right) \right\}. \]
\item\label{BasisA} If $\beta_p = 1$ then the module $\Mod(w_\beta) = (\Id-w_\beta)R^\vee$ has a $\Z$-basis given by 
\[ \{ \alpha_j^\vee \mid j \in J_\beta \}, \]
 and if $\beta_p \geq 2$ then the module $\Mod(w_\beta) = (\Id-w_\beta)R^\vee$ has a $\Z$-basis given by 
\[ \left\{ \alpha_i^\vee - \alpha_{i+1}^\vee \ \middle|\ i \in J_\beta \backslash (I_\beta - 1) \right\} \cup \left\{ \alpha_i^\vee - \alpha_{i+2}^\vee \ \middle| \ i \in (I_\beta - 1) \backslash \{ n\} \right\} \cup \left\{ \gcd(\beta_k)\alpha_n^\vee \right\}.\]
\item\label{SNFA}  If $\beta = (1,\dots, 1)$ so that $w_\beta$ is trivial, then $(\Id - w_\beta)$ has Smith normal form $\diag(0^n)$. For $w_\beta$ nontrivial and any $w \in [w_\beta]$, the Smith normal form of $(\Id-w)$ equals \[ S_\beta = \diag(1^{n-p}, \gcd(\beta_k), 0^{p-1}).\]
\item\label{quotientA} For any partition $\beta$ and any $w \in [w_\beta]$, we have \[ R^\vee / \Mod(w) \cong  \left( \Z/ \gcd(\beta_k) \Z \right) \oplus \Z^{p-1}.\]
\end{enumerate}
\end{thm}

\noindent We illustrate this theorem with several examples in Section \ref{sec:examplesA}, and provide the proof in Section \ref{sec:proofsA}.

\subsection{Examples in type $A$}
\label{sec:examplesA}

In this section, we describe four examples which illustrate our general results and proof techniques in type $A$.  We consider the partitions:
\begin{itemize}
	\item $\beta=(4)$  in Example~\ref{eg:CoxeterA3};
	\item $\beta=(6,4)$  in Example~\ref{eg:64A9};
	\item  $\beta=(4,1)$  in Example~\ref{eg:41A4}; and 
	\item $\beta = (4,3,1)$ in Example~\ref{eg:431A7}.
\end{itemize}
For each of these examples, we give an explicit description of the $\Z$-module $\Mod(w_\beta) = (\Id - w_\beta)R^\vee$, construct a $\Z$-basis for $\Mod(w_\beta)$, find the Smith normal form for $(\Id - w_\beta)$, and determine the isomorphism class of the quotient $R^\vee / \Mod(w_\beta)$.
We continue all notation from \Cref{sec:repsA}.

\begin{example}
	\label{eg:CoxeterA3}  
	Let $\sW$ be of type $A_3$, and consider the partition $\beta = (\beta_1) = (4)$ of $n+1 = 4$.  Then $J_\beta = J^\beta_1 = \{1,2,3\}$, so $w_\beta$ is the Coxeter element $w_\beta = s_1 s_2 s_3$ of $\sW$.   We will show that $(\Id - w_\beta)$ has Smith normal form $(1,1,4)$, and hence $R^\vee / \Mod(w_\beta) \cong \Z/4\Z = \Z/\beta_1 \Z$.

	We work with the $\Z$-basis $\Delta^\vee = \{ \alpha_1^\vee, \alpha_2^\vee, \alpha_3^\vee \}$ for $R^\vee$ and compute using \eqref{eq:siaction} that 
	\begin{eqnarray*}
		w_\beta(\alpha^\vee_1) & = & s_1 s_2(\alpha^\vee_1) = s_1(\alpha^\vee_1 + \alpha^\vee_2) = -\alpha^\vee_1 + (\alpha^\vee_1 + \alpha^\vee_2) = \alpha_2^\vee \\
		w_\beta(\alpha^\vee_2) & = & s_1 s_2 s_3 (\alpha^\vee_2) = s_1 s_2 (\alpha^\vee_2 + \alpha^\vee_3) = s_1(\alpha^\vee_3) = \alpha^\vee_3 \\
		w_\beta(\alpha^\vee_3) & = & s_1 s_2 s_3(\alpha^\vee_3) = s_1 s_2(-\alpha^\vee_3) = s_1 (-\alpha^\vee_2 - \alpha^\vee_3) =  -\alpha^\vee_1 - \alpha^\vee_2 -\alpha^\vee_3.
	\end{eqnarray*}
	Hence, with respect to the basis $\Delta^\vee$, the matrix for $w_\beta$ is as on the left, and the matrix $M_\beta$ for $(\Id - w_\beta)$ is as on the right:
	\[
	w_\beta = \begin{pmatrix}
		0 & 0 &  -1 \\ 
		1 & 0 & -1 \\ 
		0 & 1 & -1 
	\end{pmatrix} \quad \quad \quad
	M_\beta = \begin{pmatrix}
		1 &  0 & 1  \\ 
		-1 & 1 & 1 \\ 
		0 & -1 & 2
	\end{pmatrix}.
	\]

	We now solve the $\Z$-linear system $M_\beta \x = \bfc$ by performing certain row operations on $M_\beta$.  We replace row 1 of $M_\beta$ by the sum of its rows 1, 2, and 3, to obtain the matrix
	\[
	M'_\beta = \begin{pmatrix}
		0 &  0 & 4  \\ 
		-1 & 1 & 1 \\ 
		0 & -1 & 2
	\end{pmatrix}.
	\]
	Tracking the effect of these row operations on $\bfc$, and noticing that for $i = 2,3$, the leading entry in row $i$ of $M'_\beta$ is a $-1$ at position $(i,i-1)$,  we see that the corresponding system of linear equations has solution over $\Z$ if and only if
	\[ c_1 + c_2 + c_3 \equiv 0 \; \operatorname{mod}\left( 4 \right). \]
	Hence for the Coxeter element $w_\beta = s_1 s_2 s_3$ in $\sW$ of type $A_3$,
	\[
	\Mod(w_\beta) = M_\beta R^\vee =  \left\{ \sum_{i=1}^3 c_i \alpha_i^\vee \ \middle| \ c_i \in \Z \mbox{ and } c_1 + c_2 + c_3 \equiv 0 \; \operatorname{mod}\left(4 \right) \right\}.
	\]

	Next, we obtain the Smith normal form $S_\beta$ for $M_\beta$.  To do this, we first, for $i = 2,3$, replace row $i$ of $M_\beta$ by the sum of rows $1,\dots,i$ of $M_\beta$, to obtain the upper-triangular matrix 
	\[
	T_\beta = \begin{pmatrix}
	\boxed{1} & 0 & \circled{1}  \\ 
		0 & \boxed{1} & \circled{2} \\ 
		0 & 0 & 4
	\end{pmatrix}.
	\]
	Here, we have boxed the pivot entries $1$ in rows 1 and 2, and we have circled the non-zero entries (in column 3) which we will clear at the next step.  Using the boxed pivots, for $i = 1, 2$ we subtract $i$ times column $i$ from column $3$, to obtain 
	\[
	S_\beta = \begin{pmatrix}
		1 & 0 & 0  \\ 
		0 & 1 & 0 \\ 
		0 & 0 & 4
	\end{pmatrix}
	\]
	which is in Smith normal form.  That is, for $w_\beta = s_1 s_2 s_3$ in $\sW$ of type $A_3$, the matrix $M_\beta$ has Smith normal form $\diag(1,1,4)$.  Hence
	\[
	R^\vee / \Mod(w_\beta) = R^\vee / M_\beta R^\vee \cong \Z/4\Z.
	\]

	Finally, to obtain a $\Z$-basis for $\Mod(w_\beta)$, we go back to $M_\beta$ and perform column operations.  We replace column 3 of $M_\beta$ by the sum of column 3 and $-i$ times column $i$ for $i = 1,2$, to obtain the matrix
	\[
	B_\beta = \begin{pmatrix}
		1 & 0 & 0  \\ 
		-1 & 1 & 0  \\ 
		0 & -1 & 4
	\end{pmatrix}.
	\]
	From here, we see that a $\Z$-basis for $\Mod(w_\beta)$ is given by the set
	\[
 	\left\{ \alpha_1^\vee-\alpha_2^\vee, \alpha_2^\vee-\alpha_3^\vee, 4\alpha_3^\vee \right\}.
	 \]
 	In this example, $J_\beta = \{1,2,3\}$ and $I_\beta = \{ 4\}$, in which case $I_\beta-1 = \{3\}$ and $(I_\beta - 1) \backslash \{ 3\} = \emptyset$. Therefore, this basis matches the one given in Theorem \ref{thm:An}. 
\end{example}

\begin{example}
	\label{eg:64A9}
 	Let $\sW$ be of type $A_9$, and consider the partition $\beta = (\beta_1,\beta_2) = (6,4)$ of $n+1 = 10$.  Then $J^\beta_1 = \{1,2,3,4,5\}$ and $J^\beta_2 = \{7,8,9\}$, so $J_\beta = \{ 1, \dots, 9 \} \setminus \{ 6 \}$ and $w_\beta = w^\beta_1 w^\beta_2 =  (s_1 s_2 s_3 s_4 s_5)(s_7 s_8 s_9)$ is Coxeter in $\sW_\beta = \sW_{J^\beta_1} \times \sW_{J^\beta_2}$.   We will  show that $(\Id - w_\beta)$ has Smith normal form $(1^7,2,0)$, and hence $R^\vee / \Mod(w_\beta) \cong  \Z/2\Z \oplus \Z$.  That is, the torsion in this quotient is determined by $2=\gcd(6,4)=\gcd(\beta_1, \beta_2)$, while its free rank equals $1= p-1$, where $p = 2$ is the number of parts of $\beta$. 
 
	We work with the $\Z$-basis $\Delta^\vee = \{ \alpha_1^\vee, \dots, \alpha_9^\vee \}$ for $R^\vee$.  Now $w^\beta_1 = s_1 s_2 s_3 s_4 s_5$ fixes $\alpha_i^\vee$ for $7 \leq i \leq 9$, and~$w^\beta_2 = s_7 s_8 s_9$ fixes $\alpha_i^\vee$ for $1 \leq i \leq 5$, so we can compute similarly to Example~\ref{eg:CoxeterA3} that $w_\beta(\alpha_i^\vee) = \alpha^\vee_{i+1}$ for $i \in \{1,2,3,4\} \cup \{7,8\}$, while
 	\[
	w_\beta(\alpha^\vee_5) =   - \sum_{i=1}^5 \alpha_i^\vee \quad \mbox{and} \quad 
	 w_\beta(\alpha^\vee_9) = - \sum_{i=7}^9 \alpha_i^\vee.
	\]
	By inserting $s_6 s_6$ at the end of $w^\beta_1$, we also compute using \eqref{eq:siaction} that
	\begin{eqnarray*}
		w_\beta(\alpha^\vee_6)  
		& = & (s_1 s_2 s_3 s_4 s_5) (s_7 s_8 s_9) (\alpha^\vee_6) \\
		& = & (s_1 s_2 s_3 s_4 s_5) s_7 (\alpha^\vee_6) \\
		& = & s_1 s_2 s_3 s_4 s_5 (\alpha_6^\vee + \alpha_7^\vee) \\
		& = & s_1 s_2 s_3 s_4 s_5 s_6 (s_6 \alpha_6^\vee) + \alpha_7^\vee \\ 
		& = & -  s_1 s_2 s_3 s_4 s_5 s_6 ( \alpha_6^\vee) + \alpha_7^\vee \\
		& = & - \left( - \sum_{i=1}^6 \alpha_i^\vee \right) + \alpha_7^\vee \\
		& = & \sum_{i=1}^7 \alpha_i^\vee.
	\end{eqnarray*}
	Hence the matrices for  $w_\beta$ and $M_\beta = \Id - w_\beta$ with respect to the basis $\Delta^\vee$ are:
	\[
	w_\beta = \begin{pmatrix}
		0 & 0 & 0 & 0 & -1 & 1 &  &  &  \\ 
		1 & 0 & 0 & 0 & -1 & 1 &  &  & \\ 
		0 & 1 & 0 & 0 & -1 & 1 &  &  &  \\
		0 & 0 & 1 & 0 & -1 & 1 &  &  &  \\
		0 & 0 & 0 & 1 & -1 & 1 & &  &  \\
		0 &  0&  0&  0& 0 & 1 & 0 &  0&  0\\
 		  &  &  &  &  & 1 & 0 & 0 & -1 \\
 		  &  &  &  &  & 0 & 1 & 0 & -1 \\
 	 	  &  &  &  &  & 0 & 0 & 1 & -1 
	\end{pmatrix}
	\quad
	M_\beta = \begin{pmatrix}
		1 & 0 & 0 & 0 & 1 & -1 & &  &  &  \\ 
		-1 & 1 & 0 & 0 & 1 & -1 & &  &  &  \\ 
		0 & -1 & 1 & 0 & 1 & -1 & &  &  &   \\
		0 & 0 & -1 & 1 & 1 & -1 & &  &  &   \\
		0 & 0 & 0 & -1 & 2 & -1 & &  &  &   \\
		0 & 0 &0  &0  &  0 & 0 & 0&  0& 0 &   \\
 		&  &  &  &   & -1  & 1 & 0 & 1 \\
 		&  &  &  &   & 0  & -1 & 1 & 1 \\
 		&  &  &  &   & 0  & 0 & -1 & 2 
	\end{pmatrix}.
	\]  
	Here, omitted entries are all $0$s.

	We now solve the $\Z$-linear system $M_\beta \x = \bfc$ using row operations.  We replace row 1 of $M_\beta$ by the sum of its rows $1,\dots,9$ to obtain the matrix
	\[
	M'_\beta = \begin{pmatrix}
		0 &  0 & 0  & 0 & 6 & -6 & 0 & 0 & 4\\ 
		-1 & 1 & 0 & 0 & 1 & -1 &  &  & \\ 
		0 & -1 & 1 & 0 & 1 & -1 &  &  &  \\
		0 & 0 & -1 & 1 & 1 & -1 &  &  &  \\
		0 & 0 & 0 & -1 & 2 & -1 &  &  &  \\
		0 & 0 & 0 & 0 & 0 & 0 & 0 & 0 & 0 \\
 		&  &  &  &  & -1 & 1 & 0 & 1 \\
 		&  &  &  &  & 0 & -1 & 1 & 1 \\
 		&  &  &  &  & 0 & 0 & -1 & 2 
	\end{pmatrix}.
	\]
	Tracking the effect of these row operations on $\bfc$, and noticing that for $i = 2,3,4,5,7,8,9$, the leading entry in row $i$ of $M'_\beta$ is its $(i,i-1)$ entry, which equals $-1$, while row 6 of $M'_\beta$ is all~$0$s, we see that the corresponding system of linear equations has solution over $\Z$ if and only if $c_6 = 0$ and there are $x_5, x_6, x_9 \in \Z$ such that
	\[ 6x_5 - 6x_6 + 4x_9 = \sum_{i=1}^9 c_i.\]
	Since $\gcd(6,-6,4) = 2$, by Bezout's Theorem this equation has integer solution if and only if $
	\sum_{i=1}^9 c_i \equiv 0 \; \operatorname{mod}\left( 2\right)$.
	Hence 
	\[
	\Mod(w_\beta) = M_\beta R^\vee =  \left\{ \sum_{i=1}^9 c_i \alpha_i^\vee \ \middle| \ c_i \in \Z, \ c_6 = 0, \ \sum_{i=1}^n c_i \equiv 0 \; \operatorname{mod}\left( 2\right)  \right\}.
	\]

	We next obtain the Smith normal form $S_\beta$ for $M_\beta$, using a combination of column and row operations.  Starting with the matrix~$M_\beta$, we add columns 5, 7, and 8 to column $6$, to obtain 
	\[
	\hat{M}_\beta = \begin{pmatrix}
		1 & 0 & 0 & 0 & 1 & 0 &  &  &  \\ 
		-1 & 1 & 0 & 0 & 1 & 0 &  &  & \\ 
		0 & -1 & 1 & 0 & 1 & 0 &  &  &  \\	
		0 & 0 & -1 & 1 & 1 & 0 &  &  &  \\
		0 & 0 & 0 & -1 & 2 & 1 &  &  &  \\
		0 & 0 & 0 & 0 & 0 & 0 & 0 & 0 & 0 \\
 		&  &  &  &  & 0 & 1 & 0 & 1 \\
		 &  &  &  &  & 0 & -1 & 1 & 1 \\
 		&  &  &  &  & -1 & 0 & -1 & 2 
	\end{pmatrix}.
	\]  
	Next, for $i_1 = 2,3,4,5$ we replace row $i_1$ of $\hat{M}_\beta$ by the sum of its rows $1,\dots,i_1$, and for $i_2 = 8,9$ we replace row $i_2$ of $\hat{M}_\beta$ by the sum of its rows $7,\dots,i_2$.  This yields the matrix
	\[
	T_\beta = \begin{pmatrix}
		\boxed{1} & 0 & 0 & 0 & \circled{1} & 0 &  &  &  \\ 
		0 & \boxed{1} & 0 & 0 & \circled{2} & 0 &  &  & \\ 
		0 & 0 & \boxed{1} & 0 & \circled{3} & 0 &  &  &  \\
		0 & 0 & 0 & \boxed{1} & \circled{4} & 0 &  &  &  \\
		0 & 0 & 0 & 0 & 6 & 1 &  &  &  \\
		0 & 0 & 0 & 0 & 0 & 0 & 0 & 0 & 0 \\
 		&  &  &  &  & 0 & \boxed{1} & 0 & \circled{1} \\
 		&  &  &  &  & 0 & 0 & \boxed{1} & \circled{2} \\
 		&  &  &  &  & -1 & 0 & 0 & 4 
	\end{pmatrix}.
	\]  
	Here, we have boxed the pivot entries 1 in columns $i_1 = 1,2,3,4$ and $i_2 = 7,8$, and circled the non-zero entries in columns $5$ and $9$ which we will clear, using these pivots and column operations, at the next step.  This results in the matrix
	\[
	T_\beta' = \begin{pmatrix}
		\boxed{1} & 0 & 0 & 0 & 0 & 0 &  &  &  \\ 
		0 & \boxed{1} & 0 & 0 & 0 & 0 &  &  & \\ 
		0 & 0 & \boxed{1} & 0 & 0 & 0 &  &  &  \\
		0 & 0 & 0 & \boxed{1} & 0 & 0 &  &  &  \\
		0 & 0 & 0 & 0 & \circled{6} & 1 & 0 & 0 & 0 \\
		0 & 0 & 0 & 0 & 0 & 0 & 0 & 0 & 0 \\
 		&  &  &  &  & 0 & \boxed{1} & 0 & 0 \\
		 &  &  &  &  & 0 & 0 & \boxed{1} & 0 \\
 		&  &  &  &  & \circled{$-1$} & 0 & 0 & 4 \\
	\end{pmatrix}.
	\] 
	We now add row 5 to row 9, which clears the circled $-1$ and puts a 6 in position $(9,5)$, and then subtract 6 times column 6 from column 5, which clears the circled 6, to obtain
	\[
	T_\beta'' = \begin{pmatrix}
		\boxed{1} & 0 & 0 & 0 & 0 & 0 &  &  &  \\ 
		0 & \boxed{1} & 0 & 0 & 0 & 0 &  &  & \\ 
		0 & 0 & \boxed{1} & 0 & 0 & 0 &  &  &  \\
		0 & 0 & 0 & \boxed{1} & 0 & 0 &  &  &  \\
		0 & 0 & 0 & 0 & 0 & \boxed{1} &  &  &  \\
		0 & 0 & 0 & 0 & 0 & 0 & 0 & 0 & 0 \\
 		&  &  &  & 0 & 0 & \boxed{1} & 0 & 0 \\
 		&  &  &  & 0 & 0 & 0 & \boxed{1} & 0 \\
 		0& 0 & 0 & 0 & 6 & 0 & 0 & 0 & 4 \\
	\end{pmatrix},
	\] 
	where we have boxed the pivot entries $1$ in columns $1,2,3,4,6,7,8$.  
	By applying Bezout's Theorem to the last row, whose nonzero entries are $\beta_1 = 6$ and $\beta_2 = 4$, we obtain that the Smith normal form $S_\beta$ for $M_\beta$ is $ \diag(1^7,2,0)$, where $2 = \gcd(6,4)$.  Therefore, the quotient $R^\vee / \Mod(w_\beta) \cong \Z/2\Z \oplus \Z$.

	Finally, to obtain a $\Z$-basis for $\Mod(w_\beta)$, we go back to $M_\beta$ above and perform the following column operations.  First replace column 6 by the sum of columns 5 and 6.  Then, apply the same column operations as used to obtain $B_\beta$ in Example \ref{eg:CoxeterA3}, here on the upper $5\times 5$ block and the lower $3 \times 3$ block, respectively.  Then, replace column 5 by $-6$ times the sum of columns 6, 7, 8 to obtain 
\[
	B'_\beta = \begin{pmatrix}
		1 & 0 & 0 & 0 & 0 & 0 &  &  &  \\ 
		-1 & 1 & 0 & 0 & 0 & 0 &  &  & \\ 
		0 & -1 & 1 & 0 & 0 & 0 &  &  &  \\
		0 & 0 & -1 & 1 & 0 & 0 &  &  &  \\
		0 & 0 & 0 & -1 & 0 & 1 &  &  &  \\
		0 & 0 & 0 & 0 & 0 & 0 & 0 & 0 & 0 \\
 		&  &  &  & 0 & -1 & 1 & 0 & 0 \\
 		&  &  &  & 0 & 0 & -1 & 1 & 0 \\
 		&  &  &  & 6 & 0 & 0 & -1 & 4 
	\end{pmatrix}.
	\]
Since $\gcd(6,4)=2$, by Bezout's Theorem, there exist column operations on columns 5 and 9 which first replaces column 9 with $\gcd(6,4)=2$ and then clears the 6 from column 5, resulting in
\[
	B_\beta = \begin{pmatrix}
		1 & 0 & 0 & 0 &  &  &  &  &  \\ 
		-1 & 1 & 0 & 0 &  &  &  &  & \\ 
		0 & -1 & 1 & 0 &  &  &  &  &  \\
		0 & 0 & -1 & 1 &  &  &  &  &  \\
		0 & 0 & 0 & -1 & 0 & 1 &  &  &  \\
		&  &  &  &  & 0 &  &  &  \\
 		&  &  &  &  & -1 & 1 & 0 & 0 \\
 		&  &  &  &  & 0 & -1 & 1 & 0 \\
 		&  &  &  &  & 0 & 0 & -1 & 2 
	\end{pmatrix}.
	\]
	
	Therefore, one $\Z$-basis for $\Mod(w_\beta)$ is given by
	\[ \{ \alpha_{i}^\vee -\alpha_{i+1}^\vee \mid i =1,2,3,4, 7, 8 \}  \ \cup\  \{\alpha_5^\vee - \alpha_7^\vee\} \ \cup\  \{ 2\alpha_9^\vee \}. \]
In this example, $J_\beta = \{1,2,3,4,5,7,8,9\}$ and $I_\beta = \{ 6,10\}$, in which case $I_\beta-1 = \{5,9\}$ and $(I_\beta - 1) \backslash \{ 9\} = \{5\}$. Therefore, this basis matches the one given in Theorem \ref{thm:An}. 
\end{example}

\begin{example}\label{eg:41A4} 
	Let $\sW$ be of type $A_4$, and consider the partition $\beta = (\beta_1,\beta_2) = (4,1)$ of $n+1 = 5$.  We now have $J^\beta_1 = \{1,2,3\}$ and $J^\beta_2 = \emptyset$, so $J_\beta = J^\beta_1 = \{1,2,3\}$ and $w_\beta = s_1 s_2 s_3$ is Coxeter in $\sW_\beta = \sW_{J^\beta_1}$.   We will show that $(\Id - w_\beta)$ has Smith normal form $(1,1,1,0) = (1^{n-p + 1},0^{p-1})$, and hence $R^\vee / \Mod(w_\beta) \cong \Z = \Z^{p-1}$ is free of rank $p - 1$, where $p = 2$ is the number of parts of $\beta$.

	We work with the $\Z$-basis $\Delta^\vee = \{ \alpha_1^\vee, \alpha_2^\vee, \alpha_3^\vee, \alpha_4^\vee \}$ for $R^\vee$, and similarly to Example~\ref{eg:CoxeterA3} compute that
	\[
		w_\beta(\alpha_1^\vee) = \alpha^\vee_2, \quad 
		w_\beta(\alpha_2^\vee)  =  \alpha^\vee_3, \quad \mbox{ and } \quad w_\beta(\alpha^\vee_3) =  -\alpha^\vee_1 - \alpha^\vee_2 -\alpha^\vee_3.
	\]
	This time we also need
	\[w_\beta(\alpha^\vee_4) = s_1 s_2 s_3 (\alpha^\vee_4) = s_1 s_2 s_3 s_4(s_4\alpha^\vee_4) = - s_1 s_2 s_3 s_4(\alpha_4^\vee) = 
 	\alpha_1^\vee + \alpha_2^\vee + \alpha_3^\vee + \alpha_4^\vee.\]
	Hence the matrices  for $w_\beta$ and $M_\beta = \Id - w_\beta$ with respect to the basis $\Delta^\vee$ are given by
\[
w_\beta = \begin{pmatrix}
0 & 0 &  -1 & 1\\ 
1 & 0 & -1 & 1\\ 
0 & 1 & -1 & 1 \\
0 & 0 & 0 & 1
\end{pmatrix}
\quad \mbox{and} \quad
M_\beta = \begin{pmatrix}
1 &  0 & 1  & -1\\ 
-1 & 1 & 1 & -1\\ 
0 & -1 & 2 & -1 \\
0 & 0 & 0 & 0
\end{pmatrix}.
\]

We now solve the $\Z$-linear system $M_\beta \x = \bfc$.  We replace row 1 of $M_\beta$ by the sum of its rows 1, 2, and 3 to obtain the matrix
\[
M'_\beta = \begin{pmatrix}
0 &  0 & 4  & -3\\ 
-1 & 1 & 1 & -1\\ 
0 & -1 & 2 & -1 \\
0 & 0 & 0 & 0
\end{pmatrix}.
\]
We see that the corresponding system of linear equations has solution over $\Z$ if and only if $c_4 = 0$ and there are  $x_3, x_4 \in \Z$ such that
\[ 4x_3 - 3x_4 = c_1 + c_2 + c_3.\]
However since $\gcd(4,-3) = 1$, by Bezout's Theorem this equation has integer solution for all $c_1, c_2, c_3 \in \Z$.  
Hence for $w_\beta = s_1 s_2 s_3$ in $\sW$ of type $A_4$, we have
\[
\Mod(w_\beta) = M_\beta R^\vee =  \left\{ \sum_{i=1}^3 c_i \alpha_i^\vee \ \middle| \  c_i \in \Z  \right\} = \left\{ \sum_{i=1}^4 c_i \alpha_i^\vee  \ \middle| \  c_i \in \Z  \mbox{ and } c_4 = 0 \right\}.
\]
It is immediate that a $\Z$-basis for $\Mod(w_\beta)$ is given by
$\{ \alpha_1^\vee, \alpha_2^\vee, \alpha_3^\vee \} = \{ \alpha_j^\vee : j \in J_\beta \}$.

Finally, in this case we can use just column operations to obtain the Smith normal form.  Starting with $M_\beta$, for $i = 1,2,3$ we replace column $i$ of $M_\beta$ by the sum of its columns $i,\dots,4$, to obtain the upper-triangular matrix 
\[
T_\beta = \begin{pmatrix}
\boxed{1} & 0 & 0 & \circled{$-1$}  \\ 
0 & \boxed{1} & 0 & \circled{$-1$}  \\ 
0 & 0 & \boxed{1} & \circled{$-1$}  \\
0 & 0 & 0 & 0
\end{pmatrix}.
\]
Here, we have boxed the pivot entries $1$, and we have circled the non-zero entries which we will clear, using column operations, at the next step.  This yields Smith normal form
\[
S_\beta = \begin{pmatrix}
1 & 0 & 0 & 0 \\ 
0 & 1 & 0 & 0\\ 
0 & 0 & 1 & 0 \\
0 & 0 & 0 & 0
\end{pmatrix}.
\]
Hence $
R^\vee / \Mod(w_\beta) = R^\vee / M_\beta R^\vee \cong \Z.
$
\end{example}

Our final example in type $A$ is as follows. 

\begin{example}
	\label{eg:431A7}  
	Let $\sW$ be of type $A_7$, and consider the partition $\beta = (\beta_1,\beta_2,\beta_3) = (4,3,1)$ of $n+1 = 8$.  We now have $J^\beta_1 = \{1,2,3\}$, $J^\beta_2 = \{5,6\}$, and $J^\beta_3 = \emptyset$, so  $w_\beta = (s_1 s_2 s_3) (s_5 s_6)$.   We will show that $(\Id - w_\beta)$ has Smith normal form $(1^5,0^2) = (1^{n-p+1},0^{p-1})$, hence $R^\vee / \Mod(w_\beta) \cong \Z^2 = \Z^{p-1}$, where $p = 3$ is the number of parts of $\beta$.

We work with the $\Z$-basis $\Delta^\vee = \{ \alpha_1^\vee, \dots,\alpha_7^\vee \}$ for $R^\vee$, and similarly to previous examples compute that the matrices for $w_\beta$ and $M_\beta=\Id - w_\beta$ with respect to the basis $\Delta^\vee$ are
\[
w_\beta = \begin{pmatrix}
0 & 0 & -1 &  1 & &&\\ 
1 & 0 & -1 & 1 &  && \\ 
0 & 1 & -1 & 1 & & & \\
0 & 0 & 0 & 1 & 0 & 0 & 0 \\
 & &  & 1 & 0 & -1 & 1\\
& &  & 0 & 1 & -1 & 1 \\
&  &  & 0 &0 &0 & 1
\end{pmatrix}
\quad \mbox{and} \quad
M_\beta = \begin{pmatrix}
1 & 0 & 1 &  -1 & &&\\ 
-1 & 1 & 1 & -1 &  && \\ 
0 & -1 & 2 & -1 & & & \\
0 & 0 & 0 & 0 & 0 & 0 & 0 \\
 & &  & -1 & 1 & 1 & -1\\
& &  & 0 & -1 & 2 & -1 \\
&  &  & 0 &0 &0 & 0
\end{pmatrix}.
\]

We now solve the $\Z$-linear system $M_\beta \x = \bfc$, using row operations.  We replace row 1 of~$M_\beta$ by the sum of its rows $1$ through $7$ to obtain the matrix
\[
M'_\beta = \begin{pmatrix}
0 & 0 & 4 &  -4 & 0 & 3 & -2\\ 
-1 & 1 & 1 & -1 &  && \\ 
0 & -1 & 2 & -1 & & & \\
0 & 0 & 0 & 0 & 0 & 0 & 0 \\
 & &  & -1 & 1 & 1 & -1\\
& &  & 0 & -1 & 2 & -1 \\
&  &  & 0 &0 &0 & 0
\end{pmatrix}.
\]
Since $\gcd(4,-4,3,-2) = \gcd(3,-2) = 1$, while rows $4$ and $7$ are all $0$s, we see that 
\[
\Mod(w_\beta) = M_\beta R^\vee =  \left\{ \sum_{i=1}^7 c_i \alpha_i^\vee \ \middle| \ c_i \in \Z  \mbox{ and } c_4 = c_7 = 0 \right\}.
\]
It is immediate that a $\Z$-basis for $\Mod(w_\beta)$ is given by $
\{ \alpha_1^\vee, \alpha_2^\vee, \alpha_3^\vee, \alpha_5^\vee,  \alpha_6^\vee \} = \{ \alpha_j^\vee \mid j \in J_\beta \}$.

Finally, we obtain the Smith normal form $S_\beta$ for $M_\beta$, using just column operations.  To do this, starting with $M_\beta$, for $i = 4,5,6$ we replace column $i$ by the sum of columns $i,\dots,7$,  to obtain
\[
T_\beta = \begin{pmatrix}
1 & 0 & 1 &  -1 & &&\\ 
-1 & 1 & 1 & -1 &  && \\ 
0 & -1 & 2 & -1 & & & \\
0 & 0 & 0 & 0 & 0 & 0 & 0 \\
 & &  & 0 & \boxed{1} & 0 & \circled{$-1$}\\
& &  & 0 & 0 & \boxed{1} & \circled{$-1$} \\
&  &  & 0 &0 &0 & 0
\end{pmatrix}.
\]
Now for $i = 1,2,3$ we replace column $i$ of $T_\beta$ by the sum of columns $i,\dots,4$.  This yields
\[
T'_\beta = \begin{pmatrix}
\boxed{1} & 0 & 0 &  \circled{$-1$} & &&\\ 
0 & \boxed{1} &  & \circled{$-1$} &  && \\ 
0 & 0 & \boxed{1} & \circled{$-1$} & & & \\
0 & 0 & 0 & 0 & 0 & 0 & 0 \\
 & &  & 0 & \boxed{1} & 0 & \circled{$-1$}\\
& &  & 0 & 0 & \boxed{1} & \circled{$-1$} \\
&  &  & 0 &0 &0 & 0
\end{pmatrix}.
\]
We then use column operations with the boxed pivots to clear the circled entries, and so obtain that the Smith normal form is $S_\beta = (1^5,0^2)$, and hence $R^\vee / \Mod(w_\beta) \cong \Z^2$.
\end{example}

\subsection{Proofs in type $A$}
\label{sec:proofsA}

In order to prove \Cref{thm:An}, in Section~\ref{sec:matricesA} we define a collection of auxiliary matrices and record some related $\Z$-linear-algebraic results.  We then determine the matrix for $w_\beta$, and hence for $M_\beta = \Id - w_\beta$, with respect to $\Delta^\vee$ in Section~\ref{sec:wMA}, and complete the proof of \Cref{thm:An} in Section~\ref{sec:propsAn}.

\subsubsection{Matrix definitions and results in type $A$}\label{sec:matricesA}

In this section, we define a collection of useful matrices, and record some results related to these matrices.  

For $r \geq 1$, we write $\Id_r$ or sometimes just $\Id$ for the $r \times r$ identity matrix, and $0_r$ for the $r \times r$ zero matrix.  For any matrix $M$, we write $\cC_i(M)$ for the $i$th column of $M$ and $\cR_i(M)$ for the $i$th row of~$M$.

We next define $1 \times 1$ matrices $W_1 = [-1]$ and $M_1 = \Id_1 - W_1 = [2]$, and for $r \geq 2$, define 
\begin{equation}
W_r = \begin{pmatrix}\label{eq:WrMr}
0 & 0 & \dots & \dots   & 0 & -1 \\ 
1 & 0 & \ddots & \ddots & 0 & -1 \\ 
0 & 1 & \ddots & \ddots & \vdots & \vdots \\
\vdots & 0 & \ddots & 0 & 0 & -1 \\
\vdots &  & \ddots & 1 & 0 & -1 \\
0 & \dots & \dots & 0 & 1 & -1 
\end{pmatrix}, \quad
M_r = \begin{pmatrix}
1 & 0 & \cdots & \cdots & 0 & 1  \\ 
-1 & 1 & 0  & \ddots & 0 & 1 \\ 
0 & -1 & \ddots & \ddots & \vdots  & \vdots \\
\vdots & 0 & \ddots &  1 & 0 & 1 \\
\vdots &  & \ddots &  -1 & 1 & 1 \\
0 & \cdots & \cdots & 0 & -1 & 2
\end{pmatrix}.
\end{equation}
That is, the $(r,r)$-entry of $M_r= \Id - W_r$ is $2$,  its $(i,i)$- and $(i,r)$-entries are $1$ for $1 \leq i \leq r-1$, its $(i+1,i)$-entry is $-1$ for $1 \leq i \leq r-1$,  and all other entries of $M_r$ are $0$.  We also define $(r+1) \times (r+1)$ upper block-triangular matrices $W_{r,1}$ and $M_{r,1}$ by
\begin{equation}\label{eq:Wr1}
W_{r,1} = \begin{pmatrix} &&& 1 \\ & W_r & & \vdots \\ &&& 1 \\ 0 & \cdots & 0 & 1 \end{pmatrix} = 
\begin{pmatrix}
0 & 0 & \dots & \dots   & 0 & -1 & 1 \\ 
1 & 0 & \ddots & \ddots & 0 & -1 & 1\\ 
0 & 1 & \ddots & \ddots & \vdots & \vdots & \vdots \\
\vdots & 0 & \ddots & 0 & 0 & -1 & 1\\
\vdots &  & \ddots & 1 & 0 & -1 & 1\\
0 & \dots & \dots & 0 & 1 & -1 & 1 \\
0 & \dots & \dots & 0 & 0 & 0 & 1 
\end{pmatrix} \quad \text{and}
\end{equation}

\begin{equation}\label{eq:Mr1}
M_{r,1} = \Id - W_{r,1} = \begin{pmatrix} &&& -1 \\ & M_r & & \vdots \\ &&& -1 \\ 0 & \cdots & 0 & 0 \end{pmatrix} =  \begin{pmatrix}
1 & 0 & \cdots & \cdots & 0 & 1 & -1 \\ 
-1 & 1 & 0  & \ddots & 0 & 1 & -1\\ 
0 & -1 & \ddots & \ddots & \vdots  & \vdots & \vdots \\
\vdots & 0 & \ddots &  1 & 0 & 1 & -1\\
\vdots &  & \ddots &  -1 & 1 & 1 & -1\\
0 & \cdots & \cdots & 0 & -1 & 2 & -1 \\
0 & \cdots & \cdots & 0 & 0 & 0 & 0
\end{pmatrix}.
\end{equation}

Next, define $M_r'$  to be the matrix obtained from $M_r$ by replacing $\cR_1(M_r)$ by the sum $\sum_{i=1}^r \cR_i(M_r)$, so that $M_1' = M_1 = [2]$ and for $r \geq 2$, we have
\[
M_r' = \begin{pmatrix}
0 & 0 & \cdots & \cdots & 0 & r+1  \\ 
-1 & 1 & 0  & \ddots & 0 & 1 \\ 
0 & -1 & \ddots & \ddots & \vdots  & \vdots \\
\vdots & 0 & \ddots &  1 & 0 & 1 \\
\vdots &  & \ddots &  -1 & 1 & 1 \\
0 & \cdots & \cdots & 0 & -1 & 2
\end{pmatrix}.
\]
That is, the matrix $M_r'$ is the same as $M_r$ in rows $2$ through $r$, and the only nonzero entry in the first row of $M_r'$ is the $(r+1)$ in position $(1,r)$.  Similarly, we define $M'_{r,1}$ to be the matrix obtained from $M_{r,1}$ by replacing $\cR_1(M_{r,1})$ by the sum $\sum_{i=1}^r \cR_i(M_{r,1})$, so that 
\[
M'_{r,1} = \begin{pmatrix}
0 & 0 & \cdots & \cdots & 0 & r+1 & -r \\ 
-1 & 1 & 0  & \ddots & 0 & 1 & -1\\ 
0 & -1 & \ddots & \ddots & \vdots  & \vdots & \vdots \\
\vdots & 0 & \ddots &  1 & 0 & 1 & -1\\
\vdots &  & \ddots &  -1 & 1 & 1 & -1\\
0 & \cdots & \cdots & 0 & -1 & 2 & -1 \\
0 & \cdots & \cdots & 0 & 0 & 0 & 0
\end{pmatrix}.
\]
That is, the matrix $M'_{r,1}$ is the same as $M_{r,1}$ in rows $2$ through $r+1$, and the only nonzero entries in the first row of $M'_{r,1}$ are  the $(r+1)$ in position $(1,r)$, and the $-r$ in position $(1,r+1)$.  We use $M_r'$ and $M'_{r,1}$ to establish the following result.

\begin{lemma}\label{lem:MrModSets}
For all $r \geq 1$:
\begin{enumerate}
\item The $\Z$-linear system $M_r \x = \bfc$ has solution $\x \in \Z^r$ if and only if \[\sum_{i=1}^r c_i \equiv 0 \; \operatorname{mod}\left(r+1\right).\] 
\item The $\Z$-linear system $M_{r,1} \x = \bfc$ has solution $\x \in \Z^{r+1}$ if and only if $c_{r+1} = 0$.
\end{enumerate}
\end{lemma}
\begin{proof}  Since $M_1 = [2]$, the result (1) is obvious for $r = 1$.  For $r \geq 2$, write $\bfc = (c_1, \dots, c_r) \in \Z^r$.  Then, tracking the effect of the row operations used to obtain $M_r'$ from $M_r$ on $\bfc$, we see that the system $M_r \x = \bfc$ has solution $\x \in \Z^r$ if and only if the system $M_r' \x = \bfc'$ has solution $\x \in \Z^r$, where $\bfc' = \left( \sum_{i=1}^r c_i, c_2, \dots, c_r\right)$.  Now for $2 \leq i \leq r$, row $i$ of $M_r'$ has leading entry given by the $-1$ in position $(i,i-1)$.  Hence $M_r' \x = \bfc'$ has solution $\x \in \Z^r$ if and only if the equation $(r+1)x_r = \sum_{i=1}^r c_i$, coming from the first row of $M_r'$ and first entry of $\bfc'$, has solution $x_r \in \Z$.  This is equivalent to $\sum_{i=1}^r c_i \equiv 0 \; \operatorname{mod}\left(r+1\right)$.

For (2), notice that for $2 \leq i \leq r$, the leading entry in row $i$ of $M'_{r,1}$ is the $-1$ in position $(i,i-1)$, while row $r+1$ is all $0$s.  Therefore $M_{r,1} \x = \bfc$ has solution $\x \in \Z^{r+1}$ if and only if there are $x_r, x_{r+1} \in \Z$ such that $(r+1)x_r - rx_r = \sum_{i=1}^r c_i$ (from the first row) and $c_{r+1} = 0$.  But $\gcd(r+1, -r) = 1$, so by Bezout's Theorem, the equation $(r+1)x_r - rx_r = c$ in fact has a solution $x_r, x_{r+1} \in \Z$ for any $c \in \Z$.  Thus the only condition for $M_{r,1}\x = \bfc$ to have solution over $\Z$ is that $c_{r+1} = 0$.
\end{proof}

We now consider the coroot lattice $R^\vee$ as a free $\Z$-module with $\Z$-basis $\Delta^\vee = \{ \alpha_i^\vee \}$. When taking the subsystem of type $A_r$ for $1 \leq r < n$ determined by the first $r$ nodes of the Dynkin diagram, we denote the coroot lattice by $R_r^\vee$ and the corresponding basis by $\Delta^\vee_r$. 

The following is then an immediate corollary of Lemma \ref{lem:MrModSets}.

\begin{corollary}\label{cor:ModA}  For all $r \geq 1$, we have
  \[ M_r R_r^\vee = (\Id - W_r)R_r^\vee = \left\{ \sum_{i=1}^r c_i \alpha_i^\vee  \ \middle| \  c_i \in \Z, \ \sum_{i=1}^r c_i \equiv 0 \; \operatorname{mod}\left( r+1\right) \right\}\] and 
 \[ M_{r,1} R^\vee_{r+1} = (\Id - W_{r,1})R^\vee_{r+1} = \left\{ \sum_{i=1}^{r+1} c_i \alpha_i^\vee  \ \middle| \  c_i \in \Z, \ c_{r+1} = 0\right\}.\]
\end{corollary}

We next define $T_r$ to be the matrix obtained from $M_r$ by replacing $\cR_i(M_r)$ by the sum $\sum_{j=1}^i \cR_j(M_r)$ for $2 \leq i \leq r$, so that $T_1 = M_1 = [2]$ and for $r \geq 2$, we have
\begin{equation}\label{eq:Tr}
T_r =  \begin{pmatrix}
1 & 0 & \cdots & \cdots & 0 & 1  \\ 
0 & 1 & 0  & \cdots & 0 & 2 \\ 
0 & 0 & \ddots & \ddots & \vdots  & \vdots \\
\vdots & 0 & \ddots &  1 & 0 & r-2 \\
\vdots & \vdots & \ddots &  0 & 1 & r-1 \\
0 & 0 & \cdots & 0 & 0 & r+1
\end{pmatrix}.
\end{equation}
The definition of $T_{r,1}$ is qualitatively different to that of $T_r$: for $1 \leq i \leq r$, we replace $\cC_i(M_{r,1})$ by the sum $\sum_{i=1}^{r+1} \cC_i(M_{r,1})$, to obtain
\begin{equation}\label{eq:Tr1}
T_{r,1} =  \begin{pmatrix}
1 & 0 & \cdots & \cdots & 0 & -1  \\ 
0 & 1 & 0  & \cdots & 0 & -1 \\ 
0 & 0 & \ddots & \ddots & \vdots  & \vdots \\
\vdots & 0 & \ddots &  1 & 0 & -1 \\
\vdots & \vdots & \ddots &  0 & 1 & -1 \\
0 & 0 & \cdots & 0 & 0 & 0
\end{pmatrix}.
\end{equation}

We then define $S_r$ to be the matrix obtained from $T_r$ by replacing $\cC_r(T_r)$ by the difference $\cC_r(T_r) - \sum_{i=1}^{r-1} i\cC_i(T_r)$.  Then $S_r$ is the diagonal matrix $S_r = \diag(1^{r-1},r+1)$.  We also define $S_{r,1}$ to be the matrix obtained from $T_{r,1}$ by replacing $\cC_{r+1}(T_{r,1})$ by the difference $\cC_{r+1}(T_{r,1}) + \sum_{i=1}^{r} \cC_i(T_{r,1})$.  Then $S_{r,1}$ is the diagonal matrix $S_{r,1} = \diag(1^r,0)$.  Since $S_r$ (respectively, $S_{r,1}$) is in Smith normal form, and has been obtained from $M_r$ (respectively, $M_{r,1}$) by $\Z$-linear row and/or column operations, we obtain:

\begin{lemma}  For all $r \geq 1$, we have
\begin{enumerate}
\item $M_r = \Id - W_r$ has Smith normal form $S_r = \diag(1^{r-1},r+1)$; and
\item $M_{r,1} = \Id - W_{r,1}$ has Smith normal form $S_{r,1} = \diag(1^r,0)$.
\end{enumerate}
\end{lemma}

\begin{corollary}  For all $r \geq 1$, we have
\begin{enumerate}
\item $R_r^\vee / (\Id - W_r)R_r^\vee \cong \Z/(r+1)\Z$; and
\item  $R_{r+1}^\vee / (\Id - W_{r,1})R_{r+1}^\vee \cong \Z$.
\end{enumerate}
\end{corollary}

Finally, we define $B_r$ to be the matrix obtained from $M_r$ by replacing $\cC_r(M_r)$ by the sum $- \sum_{j=i}^{r-1} j\cC_j(M_r)+ \cC_r(M_r)$.  Then 
\begin{equation}\label{eq:Br}
B_r = \begin{pmatrix}
1 & 0 & \cdots & \cdots & 0 & 0  \\ 
-1 & 1 & 0  & \ddots & 0 & 0 \\ 
0 & -1 & \ddots & \ddots & \vdots  & \vdots \\
\vdots & 0 & \ddots &  1 & 0 & 0 \\
\vdots &  & \ddots &  -1 & 1 & 0 \\
0 & \cdots & \cdots & 0 & -1 & r+1
\end{pmatrix}.
\end{equation}
In other words, $B_r$ and $M_r$ differ only in column $r$, which has a single nonzero entry $r+1$ in the last entry. Since $M_r$ and $B_r$ have the same column-space, while $M_{r,1}$ and $S_{r,1}$ have the same column-space (since we used only column operations to obtain $S_{r,1}$ from $M_{r,1}$), it follows that:

\begin{lemma}  For all $r \geq 1$, we have
\begin{enumerate}
\item $M_r R_r^\vee = (\Id - W_r)R_r^\vee$ has a $\Z$-basis given by $
\left\{ \alpha_i^\vee-\alpha_{i+1}^\vee  \mid 1 \leq i \leq r-1 \right\} \cup \{ (r+1)\alpha_r^\vee \}.$
\item $M_{r,1} R_{r+1}^\vee = (\Id - W_{r,1})R_{r+1}^\vee$ has a $\Z$-basis given by $\{ \alpha_i^\vee \mid 1 \leq i \leq r \}  = \Delta_{r+1}^\vee \setminus \{ \alpha_{r+1}^\vee \}$.
\end{enumerate}
\end{lemma}

\subsubsection{The matrices for $w_\beta$ and $M_\beta$}\label{sec:wMA}

In this section, we determine the matrices for $w_\beta$ and $M_\beta = \Id - w_\beta$ with respect to the basis~$\Delta^\vee$ for $R^\vee$ in type $A_n$ for $n \geq 1$, in Corollaries~\ref{cor:matrixwbetaA} and~\ref{cor:matrixMbetaA}, respectively. This result will involve the auxiliary matrices $W_r$, $W_{r,1}$, $M_r$, and $M_{r,1}$ defined in \Cref{sec:matricesA}.

We begin with the following observation.

\begin{lemma}\label{lem:CoxeterA}  Let $\sW$ be of type $A_n$, for $n\geq 1$.  If $\beta = (n+1)$, then $w_\beta$ is the Coxeter element $s_1 s_2 \dots s_n$ of $\sW$, and the matrix for~$w_\beta$ with respect to $\Delta^\vee$ is $W_n$ from \eqref{eq:WrMr}.
\end{lemma}
\begin{proof}  This is obtained by straightforward generalization of the corresponding computations in Example~\ref{eg:CoxeterA3}.
\end{proof}

We next consider the individual parts $\beta_k \geq 2$ of the partition $\beta$.  These determine blocks of size $(\beta_k -1) \times (\beta_k - 1)$ in the matrix for $w_\beta$, as described by the next result.

\begin{lemma}\label{lem:w_beta-blocks}
Let $\sW$ be of type $A_n$, for $n \geq 1$, and let $\beta = (\beta_1,\dots,\beta_p)$ be a partition of $n+1$.
Then for all $k \in [p]$ such that $\beta_k \geq 2$: 
\begin{enumerate}
\item
the $(\beta_k -1) \times (\beta_k-1)$ square submatrix of $w_\beta$ with entries $(w_\beta)_{ij}$ for $i,j \in J^\beta_k$ is equal to the matrix $W_{r}$ from \eqref{eq:WrMr} with $r = \beta_k - 1$; and
\item for all $j \in J^\beta_k$ and $i \in [n] \setminus J^\beta_k$, the $(i,j)$-entry of $w_\beta$ equals $0$.
\end{enumerate}
\end{lemma}

\begin{proof}   
Recall that the element $w^\beta_k$ is the product, in increasing order, of the simple reflections in the type $A_{\beta_k - 1}$ subsystem indexed  by $J^\beta_k = \{ j^\beta_k + 1,\dots,j^\beta_k + (\beta_k - 1)\}$.  Hence $w^\beta_k$ fixes $\alpha_i^\vee$ for all $1 \leq i \leq j^\beta_k - 1$ and all $j^\beta_k + \beta_k + 1 \leq i \leq n$.  The result is then obtained by shifting the indexing and applying \Cref{lem:CoxeterA}.
\end{proof}

For partitions $\beta$ which have at least two parts equal to $1$, we have the following.

\begin{lemma}\label{lem:w_beta-Idblock}
Let $\sW$ be of type $A_n$ for $n \geq 1$, and let $\beta = (\beta_1,\dots,\beta_p)$ be a partition of $n+1$.
Suppose that the last $m \geq 2$ parts of $\beta$ are equal to $1$.  Then:
\begin{enumerate}
\item the $(m -1) \times (m-1)$ square submatrix of $w_\beta$ with entries $(w_\beta)_{ij}$ for $n-m+2 \leq i,j\leq n$ is equal to $\Id_{m-1}$; and
\item for all $1 \leq i \leq n-m+1$ and $n-m+2 \leq j\leq n$, the $(i,j)$-entry of $w_\beta$ equals $0$.
\end{enumerate}
\end{lemma}
\begin{proof}  We have that $\sum_{k=1}^{p-m} \beta_k = (n+1) - m = (n-m) + 1$ is the sum of the parts of $\beta$ which are $\geq 2$.  Hence  $w_\beta$ is a product of a subset of the simple reflections $\{ s_1, \dots, s_{n-m} \}$, and so $w_\beta$ fixes $\alpha_i^\vee$ for $n-m+2 \leq i \leq n$.  The result follows.
\end{proof}

To describe the remaining columns of $w_\beta$; that is, the columns not in the blocks given by the previous two lemmas, we use the following.

\begin{lemma}\label{lem:w_beta-action}
Let $\sW$ be of type $A_n$ for $n \geq 1$, and let $\beta = (\beta_1,\dots,\beta_p)$ be a partition of $n+1$.  
\begin{enumerate}
\item If $2 \leq k \leq p$ is such that $\beta_{k-1} \geq 2$ and $\beta_k \geq 2$, then \[w_\beta\left(\alpha^\vee_{j^\beta_k}\right) = \sum_{i=j^\beta_{k-1} + 1}^{j^\beta_k + 1} \alpha_i^\vee.\]
\item If $2 \leq k \leq p$ is such that $\beta_{k-1} \geq 2$ and $\beta_k = 1$,  then \[w_\beta\left(\alpha^\vee_{j^\beta_k}\right) = \sum_{i=j^\beta_{k-1} + 1}^{j^\beta_k} \alpha_i^\vee.\]
\end{enumerate}
\end{lemma}
\begin{proof}  To simplify notation, we write $j_k$ for $j^\beta_k$.  We will prove (2) first. In the notation of \eqref{eq:w^beta_k}, we compute using \eqref{eq:siaction} that
\begin{eqnarray*}
w^\beta_{k-1}\left(\alpha_{j_k}^\vee\right)  & = & s_{j_{k-1} + 1} \dots s_{j_k - 1} \left( \alpha_{j_k}^\vee \right) \\
& = & s_{j_{k-1} + 1} \dots s_{j_k - 1} s_{j_k} \left( s_{j_k}\alpha_{j_k}^\vee \right) \\
& = & - \left( - \sum_{i=j_{k-1} + 1}^{j_k } \alpha_{i}^\vee \right) \\
& = & \sum_{i=j_{k-1}+1}^{j_k} \alpha_{i}^\vee.
\end{eqnarray*} 
Since $w_\beta = w^\beta_1 \dots w^\beta_{k-1}$ in this case, and the product $w^\beta_1 \dots w^\beta_{k-2}$ fixes $\alpha_i^\vee$ for all $j^\beta_{k-1} + 1 \leq i \leq n$, we obtain that $w_\beta\left(\alpha_{j_k}^\vee\right) = w^\beta_{k-1}\left(\alpha_{j_k}^\vee\right)$ has the required formula.

Now for (1), since $w^\beta_\ell$ fixes $\alpha^\vee_{j_k}$ for all $\ell \notin \{k-1,k\}$ and $w^\beta_{k-1}$ and $w^\beta_k$ commute, we obtain using the proof of (2) that
\begin{eqnarray*}
w_\beta\left(\alpha_{j_k}^\vee\right) 
 & = & w^\beta_{k}w^\beta_{k-1}\left(\alpha_{j_k}^\vee\right) \\
 & = & s_{j_{k} + 1} \dots s_{j_k + \beta_k - 1}\left( \sum_{i=j_{k-1}+1}^{j_k} \alpha_{i}^\vee \right) \\
 & = & \left( \sum_{i=j_{k-1}+1}^{j_k - 1} \alpha_{i}^\vee \right) + s_{j_{k} + 1}\left( \alpha_{j_k}^\vee \right) \\
 & = & \left(\sum_{i=j_{k-1}+1}^{j_k - 1} \alpha_{i}^\vee \right) +  \alpha_{j_k}^\vee  +  \alpha_{j_k + 1}^\vee \\
\end{eqnarray*} 
as required.
\end{proof}

Applying Lemmas \ref{lem:w_beta-blocks}, \ref{lem:w_beta-Idblock}, and \ref{lem:w_beta-action}, we can now describe the matrix for $w_\beta$ in type $A_n$.  In the next two results, (1) generalizes Examples~\ref{eg:CoxeterA3} and~\ref{eg:64A9}, and (2) generalizes Examples~\ref{eg:41A4} and~\ref{eg:431A7}.  The matrix for $w_\beta$ is as follows.

\begin{corollary}\label{cor:matrixwbetaA}   Let $\sW$ be of type $A_n$ for $n \geq 1$, and let $\beta = (\beta_1,\dots,\beta_p)$ be a partition of $n+1$.  
\begin{enumerate}
\item If $\beta_p \geq 2$, then the matrix for $w_\beta$ with respect to the basis $\Delta^\vee$ for $R^\vee$ satisfies:
\begin{enumerate}
\item for all $k \in [p]$, the $(\beta_k - 1) \times (\beta_k -1)$ square submatrix of $w_\beta$ with entries $(w_\beta)_{ij}$ for $i,j \in J^\beta_k$ is equal to $W_{\beta_k - 1}$.
\item for all $2 \leq k \leq p$ and all $j^\beta_{k-1} + 1 \leq i \leq j^\beta_k + 1$, the $(i,j^\beta_k)$-entry of $w_\beta$ is equal to $1$.
\item all other entries of $w_\beta$ are zero.  
\end{enumerate}
\item If the last $m \geq 1$ parts of the partition $\beta$ are equal to $1$, then the matrix for $w_\beta$ with respect to the basis $\Delta^\vee$ for $R^\vee$ satisfies:
\begin{enumerate}
\item for all $k \in [p]$ such that $\beta_k \geq 2$, the $\beta_k \times \beta_k$ square submatrix of $w_\beta$ with entries $(w_\beta)_{ij}$ for $i,j \in J^\beta_k \cup \{ j^\beta_k \}$ is equal to $W_{\beta_k - 1,1}$.
\item for all $2 \leq k \leq p$ such that $\beta_k \geq 2$, the $(j^\beta_k+1,j^\beta_k)$-entry of $w_\beta$ is equal to $1$.
\item the $(m -1) \times (m-1)$ square submatrix of $w_\beta$ with entries $(w_\beta)_{ij}$ for $n-m+2 \leq i,j\leq n$ is equal to $\Id_{m-1}$.
\item all other entries of $w_\beta$ are zero.  
\end{enumerate}
\end{enumerate}
\end{corollary}

Therefore, we immediately obtain a formula for the matrix $M_\beta = \Id - w_\beta$ as follows.

\begin{corollary}\label{cor:matrixMbetaA}   Let $\sW$ be of type $A_n$ for $n \geq 1$, and let $\beta = (\beta_1,\dots,\beta_p)$ be a  partition of $n+1$.  Let $M_\beta$ be the matrix for $\Id - w_\beta$ with respect to the basis $\Delta^\vee$ for $R^\vee$.
\begin{enumerate}
\item\label{Abeta2} If $\beta_p \geq 2$, then $M_\beta$ satisfies:
\begin{enumerate}
\item for all $k \in [p]$, the $(\beta_k - 1) \times (\beta_k -1)$ square submatrix of $w_\beta$ with entries $(w_\beta)_{ij}$ for $i,j \in J^\beta_k$ is equal to $M_{\beta_k - 1}$.
\item for all $2 \leq k \leq p$ and all $i \in J^\beta_{k-1} \cup \{ j^\beta_k + 1\}$, the $(i,j^\beta_k)$-entry of $w_\beta$ is equal to $-1$.
\item all other entries of $M_\beta$ are zero.  
\end{enumerate}
\item\label{Abeta1} If the last $m \geq 1$ parts of the partition $\beta$ are equal to $1$, then $M_\beta$ satisfies:
\begin{enumerate}
\item for all $k \in [p]$ such that $\beta_k \geq 2$, the $\beta_k \times \beta_k$ square submatrix of $w_\beta$ with entries $(w_\beta)_{ij}$ for $i,j \in J^\beta_k \cup \{ j^\beta_k \}$ is equal to $M_{\beta_k - 1,1}$.
\item for all $2 \leq k \leq p$ such that $\beta_k \geq 2$, the $(j^\beta_k+1,j^\beta_k)$-entry of $w_\beta$ is equal to $-1$.
\item all other entries of $M_\beta$ are zero.  
\end{enumerate}
\end{enumerate}
\end{corollary}

Equipped with explicit formulas for the matrix $M_\beta$ from Corollary \ref{cor:matrixMbetaA}, we are now prepared to prove our main theorem in type $A$.

\subsubsection{Proof of \Cref{thm:An}}\label{sec:propsAn}

We now complete the proof of \Cref{thm:An}.  We will use the description of the matrix $M_\beta = \Id - w_\beta$ from \Cref{cor:matrixMbetaA}, as well as all of the auxiliary matrices and results concerning them from \Cref{sec:matricesA}.

Let $\sW$ be of type~$A_n$ with $n \geq 1$ and let $\beta = (\beta_1,\dots,\beta_p)$ be a partition of $n+1$.  Note that the partition $\beta = (1,\dots,1)$ of $n+1$ corresponds to $w_\beta$ being the trivial element of $\sW$.  The statement of \Cref{thm:An}\eqref{SNFA} in this case is obvious, so we assume from now on that $\beta_1 \geq 2$. 

We next consider the case that $\beta_p \geq 2$.  If $p = 1$ then $\beta = (n+1)$, so $\gcd(\beta_k) = \beta_1 = n+1$, and \Cref{thm:An} in this case is immediate from \Cref{lem:CoxeterA} and the results for the matrix~$W_n$ in \Cref{sec:matricesA}.  

We now assume that $p \geq 2$. Note that $M_\beta = \Id - w_\beta$ is as described in \Cref{cor:matrixMbetaA}\eqref{Abeta2}. We proceed to generalize Examples~\ref{eg:CoxeterA3} and~\ref{eg:64A9}, and we use the same notation as developed in those examples here in the general case.

To determine $\Mod(w_\beta) = M_\beta R^\vee$, we replace row 1 of $M_\beta$ by the sum of its rows $1,\dots,n$, to obtain a matrix $M'_\beta$.  Then for $2 \leq k \leq p$, the $(1,j^\beta_k -1)$-entry of $M'_\beta$ is equal to $\beta_k$ and the $(1,j^\beta_k)$-entry of $M'_\beta$ is equal to $-\beta_k$.  Also the $(1,n)$-entry of $M'_\beta$ is equal to $\beta_p$, and all other entries in row~ 1 of $M'_\beta$ are 0.  Rows $2,\dots,n$ of $M'_\beta$ are the same as in $M_\beta$, and hence $M'_\beta$ has row~$i$ all $0$s for $i \in [n] \setminus J_\beta$, while for $i \in J_\beta$ with $i \neq 1$, the leading entry in row $i$ of $M'_\beta$ is its $(i,i-1)$ entry, which equals $-1$.  Since $\gcd(\beta_1,-\beta_1,\dots,\beta_{p-1},-\beta_{p-1}, \beta_p) = \gcd(\beta_k)$, it follows that 
\[
\Mod(w_\beta) 
 = \left\{ \sum_{i=1}^n c_i \alpha_i^\vee \ \middle| \ c_i \in \Z, \ c_i = 0 \mbox{ for all } i \in [n] \setminus J_\beta, \ \sum_{i=1}^n c_i \equiv 0 \; \operatorname{mod}\left( \gcd(\beta_k)\right) \right\}.
\]

We next obtain the Smith normal form $S_\beta$ for $M_\beta$.  To do this, starting with the matrix $M_\beta$, for all $2 \leq k \leq p$, we add to column $j^\beta_k$ of $M_\beta$ its columns with indexes $i \in \{ j^\beta_k - 1\} \cup \{ j^\beta_k + 1, j^\beta_k +2, \dots, j^\beta_k + \beta_k \}$.  That is, we add to column $j^\beta_k$ the column with the highest index in $J^\beta_{k-1}$,  and the columns with all indexes in $J^\beta_k$ except for the highest.  Let $\hat{M}_\beta$ be the so-obtained matrix.  Then $\hat{M}_\beta$ is the same as $M_\beta$ except that for $2 \leq k \leq p$,  column $j^\beta_k$ is now all $0$s except for an entry $1$ in row $j^\beta_k - 1$ (the highest index in $J^\beta_{k-1}$), and an entry $-1$ in row $j^\beta_k + (\beta_{k} - 1)$ (the highest index in $J^\beta_k$).  

Next, for all $k \in [p]$, we carry out row operations on the rows of $\hat{M}_\beta$ which contain the submatrix $M_{\beta_k - 1}$, so as to replace $M_{\beta_k - 1}$ by the matrix $T_{\beta_k - 1}$ defined in \eqref{eq:Tr}.  Call the resulting matrix $T_\beta$.  For each $k \in [p]$, and each $i_k \in J^\beta_k$ with $i_k$ not the highest index in $J^\beta_k$, the matrix $T_\beta$ has a pivot entry $1$ in its $(i_k,i_k)$-position.  At the next step, for $k \in [p]$ we use these pivots to clear all of the nonzero entries in column $i_k$,  where $i_k$ is the highest index in $J^\beta_k$, other than the $(i_k,i_k)$-entry (which equals $\beta_k$).  Call the resulting matrix $T'_\beta$.
We now, for $2 \leq k \leq p$, replace the row of $T'_\beta$ with index the highest in $J^\beta_k$ by the sum of its rows with indices the highest in $J^\beta_1,\dots,J^\beta_{k-1},J^\beta_k$.  Finally, we use column operations to clear all entries in rows $1,\dots,n-1$ which are not already pivot $1$s.  The resulting matrix $T''_\beta$ satisfies the following:
\begin{enumerate}
\item for all $1 \leq k < p$, the $(\beta_k - 1) \times (\beta_k -1)$ square submatrix of $T''_\beta$ with entries $(T''_\beta)_{ij}$ for $i,j \in J^\beta_k$ is $\Id_{\beta_k - 1}$;
\item the $(\beta_p - 1) \times (\beta_p -1)$ square submatrix of $T''_\beta$ with entries $(T''_\beta)_{ij}$ for $i,j \in J^\beta_p$ is $\diag(1^{\beta_p -1}, \beta_p)$;
\item for all $2 \leq k \leq p$, the $(j^\beta_k - 1,j^\beta_k)$-entry is $1$;
\item for all $2 \leq k \leq p$, the $(n,j^\beta_{k} - 1)$-entry is $\beta_{k-1}$;
\item the $(n,n)$-entry is $\beta_p$; and
\item all other entries are $0$.
\end{enumerate}  
By applying Bezout's Theorem to the last row, where the nonzero entries are $\beta_1,\dots,\beta_p$, we obtain that the Smith normal form $S_\beta$ for $M_\beta$ is $\diag(1^{n-p},\gcd(\beta_k),0^{p-1})$.

Finally, to obtain a $\Z$-basis for $\Mod(w_\beta)$, we perform the following column operations on $M_\beta$.  For $2 \leq k \leq p$, first add column $(j_k^\beta - 1)$ to column $j_k^\beta$, which clears all but two entries in column $j_k^\beta$; namely the 1 in row $(j_k^\beta - 1)$ and $-1$ in row $(j_k^\beta + 1)$.  Then use the column operations  which transform each submatrix $M_{\beta_k - 1}$ into the matrix $B_{\beta_k - 1}$ defined in \eqref{eq:Br}, which places a unique nonzero entry $\beta_k$ on the diagonal in column $(j_k^\beta - 1)$, for $2 \leq k \leq p$.  Now for $2 \leq k \leq p$, add to column $(j_k^\beta-1)$ the sum of $-\beta_k$ times all columns strictly to its right, excluding the other columns numbered $(j_\ell^\beta - 1)$ for $k < \ell \leq p$.  The result moves the nonzero entries $\beta_k$ in columns $(j_k^\beta - 1)$ all down to row $n$, as seen in $B'_\beta$ from Example \ref{eg:64A9}.  To produce the matrix $B_\beta$ from which the claimed $\Z$-basis for $\Mod(w_\beta)$ can easily be read, we apply Bezout's Theorem on successive pairs of columns with a nonzero entry in row $n$, clearing out all but $\gcd(\beta_k)$ in position $(n,n)$.

To complete the proof of \Cref{thm:An}, we suppose that the last $m \geq 1$ parts of $\beta = (\beta_1,\dots,\beta_p)$ are equal to $1$.  Thus we are generalizing Examples~\ref{eg:41A4} and~\ref{eg:431A7}.   Notice that in this case, $\gcd(\beta_k) = 1$.  Let $M_\beta = \Id - w_\beta$ be as described in \Cref{cor:matrixMbetaA}\eqref{Abeta1}.

To determine $\Mod(w_\beta)$, we replace row $1$ of $M_\beta$ by the sum of its rows $1$ through $n-m = \sum_{k=1}^{p-m} \beta_k$, to obtain a matrix $M'_\beta$ in which row $i$ is zero for all $i \in [n] \setminus J_\beta$, while for $2 \leq i \leq n-m$ with $i \in J_\beta$, the first entry in row $i$ of $M'_\beta$ is its $(i,i-1)$-entry, which equals $-1$.  Now the first row of $M'_\beta$ has nonzero entries as follows: for $2 \leq k \leq p-m+1$, the $(1,j^\beta_k - 1)$-entry is $\beta_{k-1}$; for $2 \leq k \leq p-m$, the $(1,j^\beta_k)$-entry is $-\beta_{k-1}$; and the $(1,n-m+1)$-entry is $-(\beta_{p-m} - 1)$.  Since $\gcd(\beta_1,-\beta_1,\beta_2,\dots,\beta_{p-m},-(\beta_{p-m} - 1)) = \gcd(\beta_{p-m},-(\beta_{p-m} - 1)) = 1$, it follows that 
\[
\Mod(w_\beta) = \left\{ \sum_{i=1}^n c_i \alpha_i^\vee \ \middle| \ c_i \in \Z \mbox{ and } c_i = 0 \mbox{ for all $i \in [n] \setminus J_\beta$} \right\}.
\]
with a $\Z$-basis given by $\{ \alpha_j^\vee \mid j \in J_\beta \}$.

For the Smith normal form, we start with $M_\beta$ and for $i = j^\beta_{p-m-1},\dots, n-m$, we replace column $i$ of $M_\beta$ by the sum of its columns $i, \dots, n-m+1$.  The resulting matrix has the matrix $T_{\beta_{p-m}-1,1}$ from \eqref{eq:Tr1} replacing $M_{\beta_{p-m}-1,1}$ from \eqref{eq:Mr1}, while column $j^\beta_{p-m-1}$ of the resulting matrix now has all $0$s except for the $-1$s in rows $i \in J^\beta_{p-m-1}$.  We then inductively repeat this process, moving from right to left, to obtain an upper-triangular matrix $T_\beta$ which has blocks $T_{\beta_1-1,1}, T_{\beta_2-1,1}, \dots, T_{\beta_{p-m}-1,1}$ going down the diagonal, up to row $n-m +1$, and then all-$0$ blocks (if $m \geq 2$).  

Each of the $T_{\beta_k-1,1}$ can then be replaced by its Smith normal form $S_{\beta_k - 1,1}$.  Hence for $1 \leq k \leq p-m$, the $k$th block  contributes $\beta_k - 1$ entries equal to $1$ and one entry equal to $0$ to the Smith normal form $S_\beta$.  Now \[ \sum_{k=1}^{p-m} (\beta_k - 1) = \left( \sum_{k=1}^{p-m} \beta_k \right) - (p-m) = \left( \sum_{k=1}^{p} \beta_k \right) - m - (p-m) = n+1 - p.\]
Therefore $S_\beta$ has $n-p+1$ diagonal entries equal to $1$, and all remaining entries $0$.  Hence $S_\beta = (1^{n-p+1},0^{p-1})= (1^{n-p},\gcd(\beta_k),0^{p-1})$, and thus $R^\vee / \Mod(w_\beta) \cong \Z^{p-1}$.

This completes the proof of \Cref{thm:An}. \qed

\section{Type $C$ mod-sets}\label{sec:TypeC}

In this section we give an explicit description of all mod-sets in type $C$, as stated in Theorem \ref{thm:Cn}. Although the finite Weyl groups in types $B$ and $C$ are identical, the mod-sets in these two types differ substantially.  We treat type $C$ first, since the results are considerably more straightforward in type $C$;  compare Theorems \ref{thm:SNFB} and \ref{thm:BasisB} in type $B$.

Let $\sW$ be the finite Weyl group of type $C_n$, with $n \geq 2$.  In \Cref{sec:repsC}, we review the complete system of minimal length representatives for the conjugacy classes of $\sW$ provided in \cite[Ch.~3]{GeckPfeifferBook}, which we rephrase in Proposition \ref{prop:wbgC}.  We then give several key examples in \Cref{sec:examplesC}, and prove our results in type $C_n$ in \Cref{sec:proofsC}.

Throughout this section, we order the nodes of the Dynkin diagram increasing from left to right, as in both Bourbaki \cite{Bourbaki4-6} and Sage \cite{sagemath}, so that the first $n-1$ nodes form a type $A_{n-1}$ subsystem, and the special node is indexed by $n$ on the right.  We note that this is the reverse of the ordering of nodes used in~\cite{GeckPfeifferBook}. We also note that in Sage \cite{sagemath} the Cartan matrices in types $B_n$ and $C_n$ are reversed with respect to the Dynkin diagrams in these types. See~\Cref{app:dynkin} for a direct comparison of these conventions.

\subsection{Conjugacy class representatives and mod-sets in type $C$}\label{sec:repsC}

Following \cite[Proposition 3.4.7]{GeckPfeifferBook}, we first explain how the conjugacy classes of $\sW$ of type $C_n$ are parameterized by ordered pairs of compositions~$(\beta, \gamma)$ such that $\beta$ is weakly decreasing, $\gamma$ is weakly increasing, and $|\beta| + |\gamma| = n \geq 2$.  (In particular, note that at most one of $|\beta| = 0$ and $|\gamma| = 0$ is permitted.)  For each such pair $(\beta,\gamma)$, we will define standard parabolic subgroups $\sW_{\beta}$ and $\sW_\gamma$ of $\sW$, and an  element $\wbg = w_\beta \cdot w_\gamma$ with $w_\beta$ cuspidal in $\sW_\beta$ and $w_\gamma$ cuspidal in $ \sW_\gamma$, so that the set of all such $\wbg$ forms a complete system of minimal length representatives for the conjugacy classes of $\sW$.

We first provide a summary of our algorithm for obtaining a minimal length conjugacy class representative $\wbg$ associated to the pair of compositions $(\beta, \gamma)$.  The following result is contained in \cite[Proposition 3.4.7]{GeckPfeifferBook}, though we present a visual algorithm for quickly constructing these elements; see Figure \ref{fig:wgammaC} for an illustration of Proposition \ref{prop:wbgC}\eqref{C_wgamma}.

\begin{figure}[h]
\begin{center}
 \resizebox{3.5in}{!}
 {
\begin{overpic}{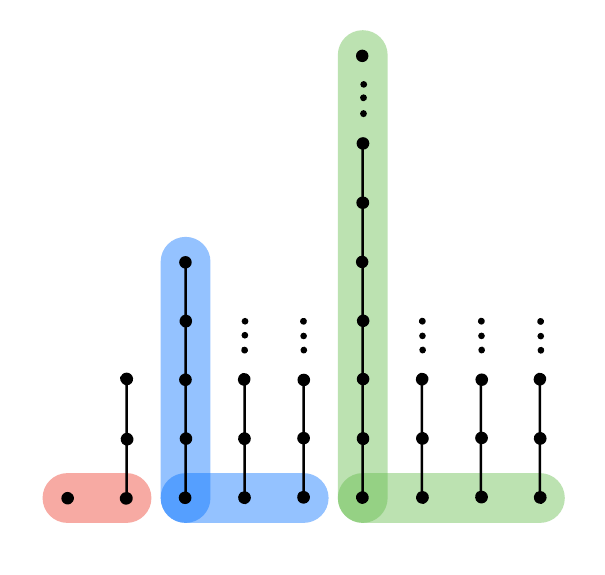}
\put(9,5){$s_9$}
\put(19,5){$s_8$}
\put(29,5){$s_7$}
\put(39,5){$s_6$}
\put(49,5){$s_5$}
\put(59,5){$s_4$}
\put(69,5){$s_3$}
\put(78.5,5){$s_2$}
\put(88.5,5){$s_1$}
\put(15,21.5){$s_9$}
\put(15,31){$s_8$}
\put(24.5,21.5){$s_8$}
\put(24.5,31){$s_9$}
\put(24.5,41){$s_8$}
\put(24.5,51){$s_7$}
\put(35,21.5){$s_7$}
\put(35,31){$s_8$}
\put(45,21.5){$s_6$}
\put(45,31){$s_7$}
\put(64.5,21.5){$s_4$}
\put(64.5,31){$s_5$}
\put(74.5,21.5){$s_3$}
\put(74.5,31){$s_4$}
\put(84.5,21.5){$s_2$}
\put(84.5,31){$s_3$}
\put(54.5,21.5){$s_5$}
\put(54.4,31){$s_6$}
\put(54.5,41){$s_7$}
\put(54.5,51){$s_8$}
\put(54.4,61){$s_9$}
\put(54.5,71){$s_8$}
\put(54.5,86){$s_4$}
\end{overpic}
}
\caption{$w_\gamma = \textcolor{red}{(}s_9s_8\textcolor{red}{)}\textcolor{blue}{(}s_7s_8s_9s_8s_7\cdot s_6s_5\textcolor{blue}{)} \liz{(}s_4s_5s_6s_7s_8s_9s_8s_7s_6s_5s_4\cdot s_3s_2s_1\liz{)}$ for the composition $\gamma = (2,3,4)$ in type $C_9$.}
\label{fig:wgammaC}
\end{center}
\end{figure}

\begin{prop}\label{prop:wbgC}
Let $(\beta, \gamma)$ be a pair of compositions such that $\beta$ is weakly decreasing, $\gamma$ is weakly increasing, and $|\beta| + |\gamma| = n \geq 2$. Define an element $\wbg$ of the finite Weyl group $\sW$ of type $C_n$ as follows:
\begin{enumerate}
\item\label{C_wbeta} Let $|\beta| = m$, and define $w_\beta$ as in Proposition \ref{prop:wbetaA}, using the type $A_{m-1}$ subsystem formed by the first nodes of the Dynkin diagram; note that $w_\beta$ is trivial if $m \in \{0,1\}$.
\item\label{C_wgamma}  Let $|\gamma| = n-m$, and write $\gamma = (\gamma_1, \dots, \gamma_q)$.  If $|\gamma| = 0$, then $w_\gamma$ is trivial.  Otherwise, define $w_\gamma$ as follows:
\begin{enumerate}
\item\label{trees} Draw $|\gamma |$ trees, rooted at $s_{n-i}$ for $i \in \{0, 1, \dots, n-m-1\}$ from left to right.
\item\label{label} Label the vertices of the tree rooted at $s_{n-i}$ vertically by the palindrome \[s_{n-i}, s_{n-i+1}, \dots, s_{n-1}, s_n, s_{n-1}, \dots, s_{n-i+1}, s_{n-i}.\]
\item\label{shade} For each part $\gamma_i$ of the composition $\gamma$ from $i = 1, \dots, q$, shade the entire tree to the right of the previous shaded area, together with the next consecutive $\gamma_i$ roots.
\item\label{multiply} Multiply together the simple reflections indexed by the shaded subgraphs, reading down trees and then left to right within shaded areas, to obtain $w_\gamma$.
\end{enumerate}
\end{enumerate}
Then the set of all $\wbg = w_\beta \cdot w_\gamma$ forms a complete system of minimal length representatives for the conjugacy classes of $\sW$. 
\end{prop}

Figure \ref{fig:wgammaC} illustrates the construction of $w_\gamma$ in Proposition \ref{prop:wbgC}; see  Section \ref{sec:examplesC} for additional examples, including Example \ref{eg:beta0gamma234C9} for more details on Figure \ref{fig:wgammaC}. 
We remark that the underlying collection of trees in Figure \ref{fig:wgammaC} is a normal form forest for the Coxeter group $W$ of type $C_9$, in the sense of du Cloux (see \cite[Section~3.4]{BjoernerBrenti}), using the ordering on generators which goes from right to left on the Dynkin diagram.

We now establish the notation needed to formally state Theorem \ref{thm:Cn} characterizing all mod-sets in type $C_n$.  If either $|\beta| = 0$ or $|\beta| = 1$, we define both $w_\beta$ and $\sW_\beta$ to be trivial.   If $\beta$ is a  partition of $m$ with $2 \leq m \leq n$, then we define both $w_\beta$ and $\sW_\beta$ as in \Cref{sec:repsA}, so that the $w_\beta$ are the minimal length conjugacy class representatives in the type $A_{m-1}$ subsystem of $\sW$ indexed by the first $m-1$ nodes of the Dynkin diagram, as stated in Proposition \ref{prop:wbgC}.

Now let $\gamma = (\gamma_1, \dots, \gamma_q)$ be any weakly increasing composition of $n-m$. If $|\gamma| = 0$, we define both $w_\gamma$ and $\sW_\gamma$ to be trivial.  It thus remains to define $w_\gamma$ and $\sW_\gamma$ for $0 \leq m < n$.  For this, following \cite[Section~3.4]{GeckPfeifferBook} but using the opposite ordering, we define $t_0 = s_n$, and for $1 \leq i \leq n-1$, define $t_i$ to be the conjugate
\begin{equation}\label{eq:ti}
t_i = s_{n-i} t_{i-1} s_{n-i}.
\end{equation}
That is, 
\[
t_1  =  s_{n-1} s_n s_{n-1}, \quad t_2 =  s_{n-2} s_{n-1} s_n s_{n-1} s_{n-2}, \quad \dots,  \quad
t_{n-1}  =  s_1 \dots s_{n-1} s_n s_{n-1} \dots s_1.
\]
Note that the elements $t_i$ provide the labeling of the rooted trees in \ref{C_wgamma}\eqref{label} of Proposition \ref{prop:wbgC}.
We also define a strictly increasing subsequence of $\{0,1,\dots,n-m\}$  in the same way as we did for the partition $\beta$ in \Cref{sec:repsA}.  That is, let $j^\gamma_1 = 0$, and for all $2 \leq k \leq q$, define
\[
j^\gamma_k = \sum_{i=1}^{k-1} \gamma_i = \gamma_1 + \dots + \gamma_{k-1}.
\]
It will be convenient to additionally define $j^\gamma_{q+1} = \sum_{i=1}^q \gamma_i = n-m$.

Now we define subintervals $J^\gamma_k$ of the interval $[n]$ by putting
\[
J^\gamma_k = \{ n-j^\gamma_{k+1}+1, n-j^\gamma_{k+1}+2, \dots, n \} = \left\{ n - \left(\sum_{i=1}^k \gamma_i\right) + 1,  \dots, n \right\}
\]
 for all $k \in [q]$; for example, $J^\gamma_1 = \{ n - \gamma_1 + 1, \dots, n \}$. In the language of Proposition \ref{prop:wbgC}, the set $J^\gamma_k$ indexes the support of the shaded diagram corresponding to part $\gamma_k$ in step \ref{C_wgamma}\eqref{shade}. Note that the $J^\gamma_k$ are \emph{not} defined the same way as for $\beta$.  In particular, these sets are all nonempty, with $n \in J^\gamma_k$ for all $k \in [q]$, and they satisfy the following strict inclusions
\[
J^\gamma_1 \subsetneq J^\gamma_2 \subsetneq \dots \subsetneq J^\gamma_q.
\]
Observe also that for all $k \in [q]$, the set $J^\gamma_k$ has $\sum_{i=1}^k \gamma_i = j^\gamma_{k+1}$ elements. Therefore, $J^\gamma_1 = \{ n \}$ if and only if $\gamma_1 = 1$, and $J^\gamma_q = \{n - (n-m) + 1, \dots, n\} = \{ m+1,\dots,n\}$, and hence $|J^\gamma_q| = |\gamma| = n-m$ for all $\gamma$.  We may thus identify the elements of each~$J^\gamma_k$ with the last $j^\gamma_{k+1} \geq 1$ nodes of the Dynkin diagram for $\sW$; that is, with the nodes of its (connected and nonempty) subdiagram of type $C_{j^\gamma_{k+1}} $.  To simplify notation, we also define \[J_\gamma = J^\gamma_q = \cup_{k=1}^q J^\gamma_k.\]  

Next, for each $k \in [q]$, we write $\sW_{J^\gamma_k}$ for the nontrivial standard parabolic subgroup of $\sW$ generated by the simple reflections $\{ s_j \mid j \in J^\gamma_k \}$.  Then each $\sW_{J^\gamma_k}$ is of type $C_{j^\gamma_{k+1}} $, and we have the strict inclusions
\[
\sW_{J^\gamma_1} \lneq \sW_{J^\gamma_2} \lneq \dots \lneq \sW_{J^\gamma_q}.
\]
We write $W_\gamma$ for the nontrivial standard parabolic subgroup of $\sW$ generated by the simple reflections $\{ s_j \mid j \in J_\gamma \}$, so that $W_\gamma = W_{J^\gamma_q} = \cup_{k=1}^q W_{J^\gamma_k}$.  

Now for all $k \in [q]$, we define an element $w^\gamma_k \in \sW_{J^\gamma_k}$ by
 \begin{equation}\label{eq:wgammak-t}
 w^\gamma_k = t_{j^\gamma_k}s_{n-(j^\gamma_k+1)}s_{n-(j^\gamma_k+2)} \cdots s_{n-j^\gamma_{k+1}+1}.
 \end{equation}
Thus as $j^\gamma_1 = 0$ and $t_0 = s_n$, we have
 \begin{equation}\label{eq:wgamma1}
 w^\gamma_1 = s_n s_{n-1} \cdots s_{n-\gamma_1 + 1}
 \end{equation}
 while for all $2 \leq k \leq q$, by expanding out the expression \eqref{eq:ti} for $i = j^\gamma_k$, we get
 \begin{equation}\label{eq:wgammak}
 w^\gamma_k = s_{n-j^\gamma_k} s_{n-j^\gamma_k+1} \cdots s_{n-1} \cdot s_n s_{n-1} \cdots  s_{n-j^\gamma_{k+1}+1}.
 \end{equation}
 That is, the support of $w^\gamma_k$ is the set $J^\gamma_k$, and $w^\gamma_k$ is the product of the simple reflections $s_j$ for $j \in \{ n- j^\gamma_k, \dots, n-1 \} \subsetneq J^\gamma_k$ in increasing order, followed by the product of the simple reflections $s_j$ for all $j \in J^\gamma_k$ in decreasing order. For the example in Figure \ref{fig:wgammaC}, we have $w_1^\gamma = s_9s_8$, whereas $w^\gamma_2 = s_7s_8\cdot s_9s_8s_7 s_6s_5$ and $w^\gamma_3 = s_4s_5s_6s_7s_8\cdot s_9s_8s_7s_6s_5s_4 s_3s_2s_1$. In particular, each $w^\gamma_k$ has support $J^\gamma_k$, and is a nontrivial element of $\sW_{J^\gamma_k}$.   (In the language of \cite[Section~3.4.2]{GeckPfeifferBook}, the element~$w^\gamma_k$ is the negative block $b_{j^\gamma_k, \gamma_k}^-$ of length $\gamma_k$ starting at $j^\gamma_k$.) 

Finally, the composition $\gamma = (\gamma_1,\dots,\gamma_q)$ corresponds to the product
\[ w_\gamma = w^\gamma_1 \cdots w^\gamma_q \in W_\gamma. \] 
By Proposition 3.4.6 of \cite{GeckPfeifferBook}, the element $w_\gamma$ is cuspidal in the parabolic subgroup $W_\gamma$ of type $C_{n-m}$, generated by the simple reflections indexed by the last $|\gamma| = n-m$ nodes of the Dynkin diagram.  By Proposition 3.4.7 of \cite{GeckPfeifferBook}, the set of $w_{\beta, \gamma} = w_\beta \cdot w_\gamma$ for all distinct pairs of compositions $(\beta, \gamma)$ such that $\beta$ is weakly decreasing, $\gamma$ is weakly increasing, and $|\beta| + |\gamma| = n$, forms a complete system of minimal length representatives of the conjugacy classes of $\sW$ in type $C_n$, as recorded in Proposition \ref{prop:wbgC}.

To simplify notation in the theorem below, for any weakly increasing composition $\gamma = (\gamma_1, \dots, \gamma_q)$ with $|\gamma| \leq n$, we define the following $q$-element subset of $[n]$:
\[
I_\gamma = \left\{ n - j^\gamma_k \mid 1 \leq k \leq q \right\} = \left\{ n, n - \gamma_1, \dots, n - \sum_{k=1}^{q-1} \gamma_k \right\} = \left\{ n - (|\gamma| - \gamma_q), \dots, n - \gamma_1, n \right\}.
\]

With this notation established, we now state our results describing mod-sets in type~$C$.  

\begin{thm}\label{thm:Cn}
Suppose $\sW$ is of type $C_n$ with $n \geq 2$.  Let $(\beta,\gamma)$ be a pair of compositions such that $\beta = (\beta_1,\dots,\beta_p)$ is weakly decreasing, $\gamma = (\gamma_1,\dots,\gamma_q)$ is weakly increasing, and $|\beta| + |\gamma| = n$, with corresponding conjugacy class representative $\wbg \in \sW$.  Write $m = |\beta|$, so that $0 \leq m \leq n$ and $|\gamma| = n-m$. 
\begin{enumerate}
\item\label{ModC} The module $\Mod(\wbg) = (\Id-\wbg)R^\vee$ equals 
\[  \left\{\sum_{i=1}^n c_i \alpha_i^\vee \ \middle| \ c_i \in \Z, \ c_i = 0 \mbox{ for }i \in [m] \setminus J_\beta, \ c_i \equiv 0 \; \operatorname{mod}\left(2 \right) \mbox{ for } i \in I_\gamma  \right\}. \]
\item\label{BasisC} The module $\Mod(\wbg)$ has $\Z$-basis given by \[\{ \alpha_i^\vee \mid i \in J_\beta \} \cup \{ \alpha_i^\vee \mid i \in J_\gamma \setminus I_\gamma \} \cup \{ 2\alpha_i^\vee \mid i \in I_\gamma \}.\]
\item\label{SNFC} For any $w \in [\wbg]$, the Smith normal form of $(\Id-w)$ equals 
\[ \Sbg =  \diag(1^{n - p - q}, 2^q, 0^p).\]
\item\label{quotientC} For any $w \in [\wbg]$, the quotient of $R^\vee$ by the mod-set is
 \[R^\vee / \Mod(w) \cong (\Z/ 2\Z)^q \oplus  \Z^p.\]
\end{enumerate}
\end{thm}

\noindent We illustrate this theorem with several examples in Section \ref{sec:examplesC}, and provide a detailed proof in Section \ref{sec:proofsC}.

\subsection{Examples in type $C$}\label{sec:examplesC}

In this section, we present a sequence of examples to illustrate our general results and proof techniques in type $C$.  We mostly restrict to the case $|\beta| = 0$, since the element $w_\beta$ is in a type $A$ subsystem, and the contribution of $w_\beta$ in type $C$ can thus be explained quickly via our results from \Cref{sec:TypeA}. 
The examples we consider are:
\begin{itemize}
\item $|\beta| = 0$ and $\gamma = (4)$  in Example~\ref{eg:beta0gamma4C4};
\item $|\beta| = 0$ and $\gamma = (3,4)$  in Example~\ref{eg:beta0gamma34C7};
\item $|\beta| = 0$ and $\gamma = (2,3,4)$  in Example~\ref{eg:beta0gamma234C9};
\item $|\beta| = 0$ and $\gamma = (1,3)$  in Example~\ref{eg:beta0gamma13C4}; 
\item $\beta = (1)$ and $\gamma = (3)$  in Example~\ref{eg:beta1gamma3C4};
\item $\beta = (4)$ and $\gamma = (3)$  in Example~\ref{eg:beta4gamma3C7}; and
\item $\beta = (3,3)$ and $\gamma = (3)$  in Example~\ref{eg:beta33gamma3C9}.
\end{itemize}
For each of these examples, we give an explicit description of the $\Z$-module $\Mod(\wbg) = (\Id - \wbg)R^\vee$, construct a $\Z$-basis for $\Mod(\wbg)$, and find the Smith normal form for $(\Id - \wbg)$ and the isomorphism class of the quotient $R^\vee / \Mod(\wbg)$.
We continue all notation from \Cref{sec:repsC}.

\begin{example}\label{eg:beta0gamma4C4}   Let $\sW$ be of type $C_4$ and suppose $|\beta| = 0$ and $\gamma = (4)$, so that $\wbg = w_\gamma$.  We will see that the Smith normal form for $\Id - w_\gamma$ is $\diag(1,1,1,2)$, so that $R^\vee/\Mod(w_\gamma) \cong \Z/2\Z$ is torsion of rank equal to the number of parts of $\gamma$.

We have that $w_\gamma = s_4 s_3 s_2 s_1$ is Coxeter in $\sW$.  Observing that the subexpression $s_3 s_2 s_1$ is Coxeter in the type $A_3$ subsystem indexed by the first $3$ nodes, we may, after reversing the labelings, use our computations for $w_\beta = s_1 s_2 s_3$ in Example~\ref{eg:CoxeterA3} above (which considered the partition $(4)$ in type $A_3$).  Since $s_4(\alpha_3^\vee) = \alpha_3^\vee + 2\alpha_4^\vee$ in type $C_4$ by \eqref{eq:siaction}, we thus obtain 
\begin{eqnarray*}
w_\gamma(\alpha_1^\vee) & = & s_4 (-\alpha_1^\vee - \alpha_2^\vee - \alpha_3^\vee) = -\alpha_1^\vee - \alpha_2^\vee -\alpha_3^\vee - 2\alpha_4^\vee \\
w_\gamma(\alpha_2^\vee) & = & s_4 (\alpha_1^\vee) = \alpha_1^\vee \\
w_\gamma(\alpha_3^\vee) & = & s_4 (\alpha_2^\vee) = \alpha_2^\vee.
\end{eqnarray*}
We also compute
\[
w_\gamma(\alpha_4^\vee) = s_4 s_3(\alpha_4^\vee) = s_4(\alpha_3^\vee + \alpha_4^\vee) = \alpha_3^\vee + 2\alpha_4^\vee - \alpha_4^\vee = \alpha_3^\vee + \alpha_4^\vee.
\]
Hence the matrices for $w_\gamma$ and $N_\gamma = \Id - w_\gamma$ with respect to the basis $\Delta^\vee$ are given by
\[
w_\gamma = \begin{pmatrix}
-1 & 1 &  0 & 0\\ 
-1 & 0 & 1 & 0\\ 
-1 & 0 & 0 & 1 \\
-2 & 0 & 0 & 1
\end{pmatrix}
\quad \mbox{and} \quad
N_\gamma = \begin{pmatrix}
2 & \boxed{-1} &  0 & 0\\ 
1 & 1 & \boxed{-1} & 0\\ 
1 & 0 & 1 & \boxed{-1} \\
\circled{2} & 0 & 0 & 0
\end{pmatrix}.
\]

To determine $\Mod(w_\gamma)$, notice that in rows $i = 1,2,3$ of $N_\gamma$, the last nonzero entry is the boxed $-1$ in the $(i,i+1)$-position.  Since the only nonzero entry in row $4$ is the $2$ in column $1$, it follows that the system $N_\gamma \x = \bfc$ has solution over $\Z$ if and only if $c_4 \equiv 0 \; \operatorname{mod}\left( 2\right)$.  Thus
\[
\Mod(w_\gamma) = N_\gamma R^\vee = \left\{ \sum_{i=1}^4 c_i \alpha_i^\vee \ \middle| \  c_i \in \Z  \mbox{ and } c_4 \equiv 0\; \operatorname{mod}\left( 2\right) \right\}.
\]
It is immediate that a $\Z$-basis for $\Mod(\wbg)$ is given by $\{ \alpha_1^\vee, \alpha_2^\vee, \alpha_3^\vee, 2\alpha_4^\vee \}$.

To determine the Smith normal form $S_\gamma$ for $N_\gamma$, we carry out the following column operations.  Replace $\cC_1(N_\gamma)$ by $\sum_{i=1}^4 i \cC_i(N_\gamma)$, and for $i = 2,3$ replace $\cC_i(N_\gamma)$ by $\sum_{j=i}^4 \cC_j(N_\gamma)$.  This gives the matrix 
\[
T_\gamma = \begin{pmatrix}
0 & -1 &  0 & 0\\ 
0 & 0 & -1 & 0\\ 
0 & 0 & 0 & -1 \\
2 & 0 & 0 & 0
\end{pmatrix}.
\]
Hence $S_\gamma = \diag(1,1,1,2) $, and $R^\vee /\Mod(w_\gamma) \cong \Z/2\Z$. 
 
\end{example}

\begin{example}\label{eg:beta0gamma34C7}  Let $\sW$ be of type $C_7$ and suppose $|\beta| = 0$ and $\gamma = (3,4)$, so that $\wbg = w_\gamma$.  We will see that the Smith normal form for $\Id - w_\gamma$ is $\diag(1^5, 2^2)$, so that $R^\vee/\Mod(w_\gamma) \cong (\Z/2\Z)^2$ is torsion of rank equal to the number of parts of $\gamma$.

We have $w_\gamma = w^\gamma_1 w^\gamma_2 = (s_7 s_6 s_5)(s_4 s_5 s_6 s_7 s_6 s_5 s_4 s_3 s_2 s_1)$.  We first determine the matrices for $w^\gamma_1$ and $w^\gamma_2$, with respect to the basis $\Delta^\vee$, then multiply these together to obtain the matrix for $w_\gamma$.  Now $w^\gamma_1 = s_7 s_6 s_5$ fixes $\alpha_i^\vee$ for $1 \leq i \leq 3$.  Since $w^\gamma_1$ is Coxeter in the type~$C_3$ subsystem indexed by the last $3$ nodes, we obtain using similar arguments to those in Example~\ref{eg:beta0gamma4C4} that 
\[
w_\gamma^1(\alpha_5^\vee) = -\alpha_5^\vee -\alpha_6^\vee -2\alpha_7^\vee, \quad w_\gamma(\alpha_6^\vee) = \alpha_5^\vee, \quad \mbox{and} \quad w_\gamma(\alpha_7^\vee) = \alpha_6^\vee +\alpha_7^\vee.
\]
We also have, by inserting $s_4 s_4$ and working in the type $C_4$ subsystem on nodes $4,5,6,7$, that 
\[
w_\gamma^1(\alpha_4^\vee) = s_7 s_6 s_5 s_4 (s_4 \alpha_4^\vee) = s_7 s_6 s_5 s_4 (-\alpha_4^\vee) = \alpha_4^\vee + \alpha_5^\vee + \alpha_6^\vee + 2\alpha_7^\vee.
\]
Hence the matrix for $w^\gamma_1$ is block lower-triangular of the form
\[
w^\gamma_1 = 
\begin{pmatrix}
1 &0&0&0&&& \\
0 &1&0&0&&& \\
0 &0&1&0&&& \\
0&0&0& 1 &  &  &  \\
0&0&0  & 1 & -1 & 1 & 0 \\
0&0&0 & 1 & -1 & 0 & 1 \\
0&0& 0 & 2 & -2 & 0 & 1 \\
\end{pmatrix}.
\]

We now consider $w^\gamma_2$. Since this ends with the Coxeter element $s_7 s_6 s_5 s_4 s_3 s_2 s_1$, by \eqref{eq:siaction} we have 
\begin{eqnarray*}
w^\gamma_2(\alpha_1^\vee) & = & s_4 s_5 s_6 (-\alpha_1^\vee - \dots - \alpha_6^\vee - 2\alpha_7^\vee) \\
w^\gamma_2(\alpha_i^\vee) & = & s_4 s_5 s_6(\alpha_{i-1}^\vee) \quad \mbox{ for } 2 \leq i \leq 6 \\
w^\gamma_2(\alpha_7^\vee) & = & s_4 s_5 s_6(\alpha^\vee_6 + \alpha^\vee_7).
\end{eqnarray*}
Next we compute
\begin{eqnarray*}
s_4 s_5 s_6(\alpha_i^\vee) & = & \alpha_i^\vee \quad \mbox{ for }i = 1,2 \\
s_4 s_5 s_6(\alpha_3^\vee) & = & \alpha_3^\vee + \alpha_4^\vee \\
s_4 s_5 s_6(\alpha_i^\vee) & = & \alpha_{i+1}^\vee  \quad \mbox{ for }i = 4,5 \\
s_4 s_5 s_6(\alpha_6^\vee) & = & -\alpha_4^\vee - \alpha_5^\vee - \alpha_6^\vee \\
s_4 s_5 s_6(\alpha_7^\vee) & = & \alpha_4^\vee + \alpha_5^\vee + \alpha_6^\vee + \alpha_7^\vee.
\end{eqnarray*}
Putting these two collections of equations together, we compute that $w^\gamma_2(\alpha^\vee_1)$ equals 
\[
(-\alpha_1^\vee - \alpha_2^\vee) - (\alpha_3^\vee + \alpha_4^\vee) - (\alpha_5^\vee + \alpha_6^\vee) + (\alpha_4^\vee + \alpha_5^\vee + \alpha_6^\vee) - 2(\alpha_4^\vee + \alpha_5^\vee + \alpha_6^\vee + \alpha_7^\vee) 
\]
which simplifies to 
\[w^\gamma_2(\alpha^\vee_1) = -\alpha_1^\vee - \alpha_2^\vee - \alpha_3^\vee - 2(\alpha_4^\vee + \alpha_5^\vee + \alpha_6^\vee + \alpha_7^\vee),\]
while 
\begin{eqnarray*}
w^\gamma_2(\alpha_i^\vee) & = & \alpha_{i-1}^\vee \mbox{ for } i = 2,3 \\
w^\gamma_2(\alpha_4^\vee) & = & \alpha_3^\vee + \alpha_4^\vee \\
w^\gamma_2(\alpha_i^\vee) & = & \alpha_{(i-1) + 1}^\vee = \alpha_i^\vee \mbox{ for } i = 5,6 \\
w^\gamma_2(\alpha_7^\vee) & = & -(\alpha_4^\vee + \alpha_5^\vee + \alpha_6^\vee) + (\alpha_4^\vee + \alpha_5^\vee + \alpha_6^\vee + \alpha_7^\vee) = \alpha_7^\vee.
\end{eqnarray*}
Hence the matrix for $w^\gamma_2$ is block lower-triangular of the form
\[
w^\gamma_2 = 
\begin{pmatrix}
-1 & 1 & 0 & 0 &  &  &  \\
-1 & 0 & 1 & 0 &  &  &  \\
-1 & 0 & 0 & 1 &  &  &  \\
-2 & 0 & 0 & 1 &  &  &  \\
-2 & 0 & 0 & 0 & 1 & 0 & 0 \\
-2 & 0 & 0 & 0 & 0 & 1 & 0 \\
-2 & 0 & 0 & 0 & 0 & 0 & 1 \\
\end{pmatrix}.
\]

To compute the matrix product $w_\gamma = w^\gamma_1 w^\gamma_2$, we observe that $w^\gamma_1$ and $w^\gamma_2$ are both block lower-triangular, with $\Id_4$ in the top left-hand corner of $w^\gamma_1$ (respectively, the bottom right-hand corner of $w^\gamma_2$).  Hence $w_\gamma$ is block lower-triangular, with its first $4$ rows the same as in $w^\gamma_2$, and its last $4$ columns the same as in $w^\gamma_1$.
For the remaining $3 \times 3$ block in the lower left-hand corner, we observe that for $i = 5,6,7$, the entries in row $i$ of $w^\gamma_1$ sum to $1$, and these rows are all $0$ in columns $1,2,3$.  Also, column $1$ of $w^\gamma_2$ has all entries in rows $4,5,6,7$ equal to $-2$, and columns $2$ and $3$ of $w^\gamma_2$ are all $0$s in rows $4,5,6,7$.  Hence for $i = 5,6,7$, the $(i,1)$-entry of $w_\gamma$ is $-2$, and the $(i,2)$- and $(i,3)$-entries are $0$.  That is,
\[
w_\gamma = 
\begin{pmatrix}
-1 & 1 & 0 & 0 &  &  &  \\
-1 & 0 & 1 & 0 &  &  &  \\
-1 & 0 & 0 & 1 &  &  &  \\
-2 & 0 & 0 & 1 &  &  &  \\
-2 & 0 & 0 & 1 & -1 & 1 & 0 \\
-2 & 0 & 0 & 1 & -1 & 0 & 1 \\
-2 & 0 & 0 & 2 & -2 & 0 & 1 \\
\end{pmatrix}.
\]
Therefore $N_\gamma = \Id - w_\gamma$ is given by
\[
N_\gamma = 
\begin{pmatrix}
2 & \boxed{-1} & 0 & 0 &  &  &  \\
1 & 1 & \boxed{-1} & 0 &  &  &  \\
1 & 0 & 1 & \boxed{-1} &  &  &  \\
2 & 0 & 0 & 0 &  &  &  \\
2 & 0 & 0 & -1 & 2 & \boxed{-1} & 0 \\
2 & 0 & 0 & -1 & 1 & 1 & \boxed{-1} \\
2 & 0 & 0 & -2 & 2 & 0 & 0 \\
\end{pmatrix}.
\]

For $\Mod(w_\gamma)$, we observe just from $N_\gamma$ that for $i = 1,2,3,5,6$, the last nonzero entry in row $i$ is the boxed $-1$ in position $(i,i+1)$, while for $i = 4,7$, all nonzero entries in row $i$ are $\pm 2$.  Therefore 
\[
\Mod(w_\gamma) = \left\{ \sum_{i=1}^7 c_i \alpha_i^\vee \ \middle| \  c_i \in \Z, \ c_i \equiv 0 \; \operatorname{mod}\left(2\right) \mbox{ for }i \in \{4,7\} \right\},
\]
and it is clear that a $\Z$-basis for $\Mod(w_\gamma)$ is given by $\{ \alpha_1^\vee, \alpha_2^\vee, \alpha_3^\vee, 2\alpha_4^\vee, \alpha_5^\vee, \alpha_6^\vee, 2\alpha_7^\vee \}$.

For the Smith normal form, we carry out column operations on $N_\gamma$.  We replace $\cC_5(N_\gamma)$ by $\cC_5(N_\gamma) + 2\cC_6(N_\gamma) + 3\cC_7(N_\gamma)$, and replace $\cC_6(N_\gamma)$ by $\cC_6(N_\gamma) + C_7(N_\gamma)$, to obtain
\[
T_\gamma = 
\begin{pmatrix}
2 & -1 & 0 & 0 &  &  &  \\
1 & 1 & -1 & 0 &  &  &  \\
1 & 0 & 1 & -1 &  &  &  \\
2 & 0 & 0 & 0 &  &  &  \\
\circled{2} & 0 & 0 & \circled{$-1$} & 0 & \boxed{-1} & 0 \\
\circled{2} & 0 & 0 & \circled{$-1$} & 0 & 0 & \boxed{-1} \\
\circled{2} & 0 & 0 & \circled{$-2$} & \boxed{2} & 0 & 0 \\
\end{pmatrix}.
\]
We can now use the boxed pivot entries in columns $5$, $6$, and $7$ of $T_\gamma$ to clear all of the circled entries, so that the resulting matrix is block-diagonal.  We then replace $\cC_1(T_\gamma)$ by $\cC_1(T_\gamma) + 2\cC_2(T_\gamma) + 3\cC_3(T_\gamma) + 4\cC_4(T_\gamma)$, and for $i = 2,3,4$ replace $\cC_i(T_\gamma)$ by $\sum_{j=i}^4 \cC_j(T_\gamma)$, to obtain
\[
T'_\gamma = 
\begin{pmatrix}
0 & -1 & 0 & 0 &  &  &  \\
0 & 0 & -1 & 0 &  &  &  \\
0 & 0 & 0 & -1 &  &  &  \\
2 & 0 & 0 & 0 &  &  &  \\
 &  & & & 0 & -1 & 0 \\
&  &  &  & 0 & 0 & -1 \\
 & &  &  & 2 & 0 & 0 \\
\end{pmatrix}.
\]
Thus $S_\gamma = \diag(1^5, 2^2)$, and so $R^\vee/\Mod(w_\gamma) \cong (\Z/2\Z)^2$.
\end{example}

\begin{example}\label{eg:beta0gamma234C9}  
Let $\sW$ be of type $C_9$ and suppose $|\beta| = 0$ and $\gamma = (2,3,4)$, so that $\wbg = w_\gamma$.  We will see that the Smith normal form for $\Id - w_\gamma$ is $\diag(1^6, 2^3)$, so that $R^\vee/\Mod(w_\gamma) \cong (\Z/2\Z)^3$ is torsion of rank equal to the number of parts of $\gamma$.

We have $w_\gamma = w^\gamma_1 w^\gamma_2 w^\gamma_3 = (s_9 s_8)(s_7 s_8 s_9 s_8s_7s_6 s_5)(s_4 \cdots s_8 s_9 s_8 \cdots s_4 s_3s_2  s_1)$; see Figure \ref{fig:wgammaC}.  Let $\gamma'$ be the composition $\gamma' = (2,3)$ of $5$.  Then since the product $w^\gamma_1 w^\gamma_2$ has support given by the type $C_5$ subsystem  indexed by $J^\gamma_2 = \{ s_5, \dots, s_9 \}$, we see that the matrix for $w^\gamma_1 w^\gamma_2$ with respect to $\Delta^\vee = \{ \alpha^\vee_1, \dots, \alpha^\vee_9\}$ has a $5 \times 5$ block in its bottom right-hand corner which is equal to the matrix for $w^{\gamma'}_1 w^{\gamma'}_2$ with respect to the subsystem $\Delta_5^\vee = \{ \alpha^\vee_1, \dots, \alpha^\vee_5\}$, and is all $0$s above this block.  Moreover, $w^\gamma_1 w^\gamma_2$ fixes $\alpha_i^\vee$ for $i = 1,2,3$.  Now as 
\[
w^\gamma_1 w^\gamma_2(\alpha_4^\vee) = (s_9 s_8)(s_7 s_8 s_9 s_8s_7s_6 s_5)s_4s_4(\alpha_4^\vee) = -(s_9 s_8)(s_7 s_8 s_9 s_8s_7s_6 s_5s_4)(\alpha_4^\vee),
\]
by considering the composition $\gamma'' = (2,4)$ of $6$ and then shifting indexes, we see that column~$4$ of the matrix for $w^\gamma_1 w^\gamma_2$ has entries $[0,0,0,1,1,1,2,2,2]$.  Therefore the matrix for $w^\gamma_1 w^\gamma_2$ is block lower-triangular of the form
\[
w^\gamma_1 w^\gamma_2 = 
\begin{pmatrix}
\Id_3 &&&& \\
 & 1 &&&& \\
&1 & -1 & 1 & 0 &      \\
&1 & -1 & 0 & 1 &      \\
&2 & -2 & 0 & 1 &      \\
&2 & -2 & 0 & 1 &  -1 & 1    \\
&2 & -2 & 0 & 2 & -2 & 1
\end{pmatrix}.
\]
Now $w^\gamma_3$ is block lower-triangular of the form
\[
w^\gamma_3 = 
\begin{pmatrix}
-1 & 1 & 0 & 0 &      \\
-1 & 0 & 1 & 0 &      \\
-1 & 0 & 0 & 1 &      \\
-2 & 0 & 0 & 1 &      \\
-2 & 0 & 0 &  0& \Id_{5}
\end{pmatrix}
\]
where for $5 \leq i \leq 9$ and $2 \leq j \leq 4$, the $(i,1)$-entry of $w^\gamma_3$ is $-2$ and the $(i,j)$-entry is $0$.  Therefore multiplying these together, we obtain that the matrix for $w_\gamma$ is as on the left, and hence that the matrix for $N_\gamma = \Id - w_\gamma$ is as on the right:
\[
w_\gamma = 
\begin{pmatrix}
-1 & 1 & 0 & 0 & &&&&     \\
-1 & 0 & 1 & 0 &  &&&&    \\
-1 & 0 & 0 & 1 &   &&&&   \\
-2 & 0 & 0 & 1 &    &&&&  \\
-2&0&0&1 & -1 & 1 & 0 &      \\
-2&0&0&1 & -1 & 0 & 1 &      \\
-2&0&0&2 & -2 & 0 & 1 &      \\
-2&0&0&2 & -2 & 0 & 1 &  -1 & 1    \\
-2&0&0&2 & -2 & 0 & 2 & -2 & 1
\end{pmatrix}
\quad
N_\gamma = 
\begin{pmatrix}
2 & -1 & 0 & 0 & &&&&     \\
1 & 1 & -1 & 0 &  &&&&    \\
1 & 0 & 1 & -1 &   &&&&   \\
2 & 0 & 0 & 0 &    &&&&  \\
2&0&0&-1 & 2 & -1 & 0 &      \\
2&0&0&-1 & 1 & 1 & -1 &      \\
2&0&0&-2 & 2 & 0 & 0 &      \\
2&0&0&-2 & 2 & 0 & -1 &  2 & -1    \\
2&0&0&-2 & 2 & 0 & -2 & 2 & 0
\end{pmatrix}.
\]

By similar arguments to the previous examples, we see that
\[
\Mod(w_\gamma) = \left\{ \sum_{i=1}^9 c_i \alpha_i^\vee \ \middle| \ c_i \in \Z, \ c_i \equiv 0 \; \operatorname{mod}\left( 2\right) \mbox{ for }i \in \{4,7,9\} \right\}
\]
and that $\Mod(w_\gamma)$ has a $\Z$-basis given by $\{ \alpha_1^\vee, \alpha_2^\vee, \alpha_3^\vee, 2\alpha_4^\vee, \alpha_5^\vee, \alpha_6^\vee, 2\alpha_7^\vee, \alpha_8^\vee, 2\alpha_9^\vee \}$.  

For the Smith normal form, we replace $\cC_8(N_\gamma)$ by $\cC_8(N_\gamma) + 2\cC_9(N_\gamma)$, to obtain
\[
T_\gamma = 
\begin{pmatrix}
2 & -1 & 0 & 0 & &&&&     \\
1 & 1 & -1 & 0 &  &&&&    \\
1 & 0 & 1 & -1 &   &&&&   \\
2 & 0 & 0 & 0 &    &&&&  \\
2&0&0&-1 & 2 & -1 & 0 &      \\
2&0&0&-1 & 1 & 1 & -1 &      \\
2&0&0&-2 & 2 & 0 & 0 &      \\
\circled{2}&0&0&\circled{$-2$} & \circled{2} & 0 & \circled{$-1$} &  0 & \boxed{-1}    \\
\circled{2}&0&0&\circled{$-2$} & \circled{2} & 0 & \circled{$-2$} & \boxed{2} & 0
\end{pmatrix}
\]
We then use the boxed pivots in columns $8$ and $9$ to clear all of the nonzero entries in rows $8$ and $9$ (these entries have been circled).  Now we carry out similar column operations to those in Example~\ref{eg:beta0gamma34C7} on the $7 \times 7$ block in the top left of $T_\gamma$, to eventually obtain
\[
T'_\gamma = 
\begin{pmatrix}
0 & -1 & 0 & 0 &  &  & && \\
0 & 0 & -1 & 0 &  &  &  &&\\
0 & 0 & 0 & -1 &  &  &  &&\\
2 & 0 & 0 & 0 &  &  &  &&\\
 &  & & & 0 & -1 & 0 &&\\
&  &  &  & 0 & 0 & -1 &&\\
 & &  &  & 2 & 0 & 0 &&\\
  & &  &  &  &  &  &0 &-1\\
    & &  &  &  &  &  &2 &0\\
\end{pmatrix}.
\]
Thus $S_\gamma = \diag(1^6, 2^3)$, and so $R^\vee/\Mod(w_\gamma) \cong (\Z/2\Z)^3$.
\end{example}

\begin{example}\label{eg:beta0gamma13C4}  
Let $\sW$ be of type $C_4$ and suppose $|\beta| = 0$ and $\gamma = (1,3)$, so that $\wbg = w_\gamma$.  We will see that the Smith normal form for $\Id - w_\gamma$ is $\diag(1,1, 2,2)$, so that $R^\vee/\Mod(w_\gamma) \cong (\Z/2\Z)^2$ is torsion of rank equal to the number of parts of $\gamma$.  In particular, $\gamma$ having a part equal to $1$ does \emph{not} change the final result, unlike the situation for $\beta$ in type~$A$; compare Example \ref{eg:41A4}.

Here we have $w_\gamma = w_1^\gamma w_2^\gamma = (s_4)(s_3 s_4 s_3 s_2 s_1)$.  In this case, since $w^\gamma_1 = s_4$ we compute immediately that the matrix for $w^\gamma_1$ is as on the left, while using similar arguments to those in Example~\ref{eg:beta0gamma34C7}, we obtain that $w^\gamma_2$ is as on the right:
\[
w^\gamma_1 = 
\begin{pmatrix} 
1 &0&0&0 \\
0 & 1 &0&0 \\
0&0&1& 0\\
0&0&2&-1
\end{pmatrix}
\quad
w^\gamma_2 = 
\begin{pmatrix} 
-1 &1&0&0 \\
-1 & 0 &1& 0\\
-2&0&1&0 \\
-2&0&2&1
\end{pmatrix}.
\]
Therefore the matrices for $w_\gamma$ and $N_\gamma = \Id - w_\gamma$ are:
\[
w_\gamma = 
\begin{pmatrix} 
-1 &1&0& 0\\
-1 & 0 &1&0 \\
-2&0&1&0 \\
-2&0&2&-1
\end{pmatrix}
\quad
N_\gamma = 
\begin{pmatrix} 
2 &\boxed{-1}&0& 0\\
1 & 1 &\boxed{-1}&0 \\
\boxed{2}&0&0& 0\\
\circled{2}&0&\circled{$-2$}&\boxed{2}
\end{pmatrix}.
\]
Hence
\[
\Mod(w_\gamma) = \left\{ \sum_{i=1}^4 c_i \alpha_i^\vee \  \middle| \  c_i \in \Z, \ c_i \equiv 0 \; \operatorname{mod}\left(2\right) \mbox{ for }i \in \{3,4\} \right\}
\]
has a $\Z$-basis given by $\{ \alpha_1^\vee, \alpha_2^\vee, 2\alpha_3^\vee, 2\alpha_4^\vee\}$.  

To obtain the Smith normal form, we used the boxed pivot in column 4 of $N_\gamma$ to clear the circled entries in row $4$, and then carry out column operations as in previous examples on the $3 \times 3$ block in the top left, to obtain the block-diagonal matrix
\[
T_\gamma = 
\begin{pmatrix} 
0 &-1&0& \\
0 & 0 &-1& \\
2&0&0& \\
&&&2
\end{pmatrix}.
\]
Thus $S_\gamma = \diag (1,1,2,2)$, and so $R^\vee / \Mod(w_\gamma) \cong (\Z/2\Z)^2$.
\end{example}

Our final three examples in type $C$ illustrate the case when the compositions $\beta$ and $\gamma$ are both nonempty.

\begin{example}\label{eg:beta1gamma3C4}   
Let $\sW$ be of type $C_4$ and suppose $\beta = (1)$ and $\gamma = (3)$, so that $w_\beta$ is trivial and hence $\wbg = w_\gamma$.  We will see that the Smith normal form for $\Id - w_\gamma$ is $\diag(1,1,2,0)$, so that $R^\vee/\Mod(w_\gamma) \cong \Z/2\Z \oplus \Z$ has torsion part of rank equal to the number of parts of $\gamma$, and free part of rank equal to the number of parts of $\beta$.

We have that $w_\gamma = s_4 s_3 s_2$.  This is Coxeter in the type $C_3$ subsystem indexed by the last~$3$ nodes, and so by considering $w_\gamma(\alpha_1^\vee) = -w_\gamma s_1 (\alpha_1^\vee)$, similarly to Example~\ref{eg:beta0gamma34C7} above, we obtain that the matrices for $w_\gamma$ and $N_\gamma = \Id - w_\gamma$ with respect to the basis $\Delta^\vee$ are 
\[
w_\gamma = 
\begin{pmatrix}
 1 &  0& 0 & 0 \\
 1 & -1 & 1 & 0 \\
  1 & -1 & 0 & 1 \\
 2 & -2 & 0 & 1 
\end{pmatrix}
\quad \mbox{and} \quad
N_\gamma = \begin{pmatrix}
0 & 0 &  0 & 0\\ 
-1 & 2 & \boxed{-1} & 0\\ 
-1 & 1 & 1 & \boxed{-1} \\
-2 & 2 & 0 & 0
\end{pmatrix}.
\]

To determine $\Mod(w_\gamma)$, notice that row $1$ of $N_\gamma$ is all $0$s, and that in rows $i = 2,3$ of $N_\gamma$, the last nonzero entry is the boxed $-1$ in the $(i,i+1)$-position.  Since the only nonzero entries in row $4$ are the $\pm 2$ in columns $1$ and $2$, it follows that the system $N_\gamma \x = \bfc$ has solution over $\Z$ if and only if $c_1 = 0$ and $c_4 \equiv 0 \; \operatorname{mod}\left( 2\right)$.  Thus
\[
\Mod(w_\gamma) = N_\gamma R^\vee = \left\{ \sum_{i=1}^4 c_i \alpha_i^\vee \ \middle| \  c_i \in \Z, \ c_1 = 0, \mbox{ and } c_4 \equiv 0\; \operatorname{mod}\left( 2\right) \right\}.
\]
It is immediate that a $\Z$-basis for $\Mod(\wbg)$ is given by $\{ \alpha_2^\vee, \alpha_3^\vee, 2\alpha_4^\vee \}$.

To determine the Smith normal form $S_\gamma$ for $N_\gamma$, we carry out the following column operations. Replace $\cC_2(N_\gamma)$ by $\sum_{i=2}^4 (i-1) \cC_i(N_\gamma)$, and for $i = 3,4$ replace $\cC_i(N_\gamma)$ by $\sum_{j=i}^4 \cC_j(N_\gamma)$.  This gives the matrix 
\[
T_\gamma = \begin{pmatrix}
0 & 0 &  0 & 0\\ 
-1 & 0 & \boxed{-1} & 0\\ 
-1 & 0 & 0 & \boxed{-1} \\
-2 & \boxed{2} & 0 & 0
\end{pmatrix}.
\]
Column operations using the boxed pivots can now be used to clear column $1$. Hence $S_\gamma = \diag(1,1,2,0) $, and $R^\vee /\Mod(w_\gamma) \cong \Z/2\Z \oplus \Z$.
 
\end{example}

\begin{example}\label{eg:beta4gamma3C7}  

Let $\sW$ be of type $C_7$ and suppose $\beta = (4)$ and $\gamma = (3)$.  We will see that the Smith normal form for $\Id - w_\gamma$ is $\diag(1^5, 2,0)$, so that $R^\vee/\Mod(w_\gamma) \cong (\Z/2\Z) \oplus \Z$ has torsion part of rank equal to the number of parts of $\gamma$, and free part of rank equal to the number of parts of $\beta$.  In particular, $\gcd(\beta_k)=4$ plays no role in this description, in contrast to type $A$; compare Example \ref{eg:64A9}. Note also that, since $\beta = (4)$ has only one part, the indexing set $(I_\beta - 1) \setminus \{ 3 \}$ is empty here (as in Example \ref{eg:CoxeterA3}).

We have $\wbg = w_\beta w_\gamma = (s_1 s_2 s_3)(s_7 s_6 s_5)$.  So $w_\beta$ fixes $\alpha_i^\vee$ for $i = 5,6,7$.  Now let $\beta'$ be the partition of $4 + 1 = 5$ given by $\beta' = (4,1)$, and observe that, in type $A_4$, we have $w_{\beta'} = s_1 s_2 s_3$.  Thus we can apply results from \Cref{sec:TypeA} to obtain that the matrix for $w_\beta$, with respect to the type $C_7$ basis of simple coroots $\Delta^\vee$, is block-diagonal of the form shown on the left:
\[
w_\beta = \begin{pmatrix} w_{\beta'} & \\ & \Id_3 \end{pmatrix} = 
\begin{pmatrix}
0&0&-1&1& \\
1&0&-1&1& \\
0&1&-1&1& \\
0&0&0&1& \\
&&&& \Id_3
\end{pmatrix}, 
\quad \mbox{while} \quad
w_\gamma = 
\begin{pmatrix}
\Id_3&&&& \\
& 1 &  0& 0 & 0 \\
& 1 & -1 & 1 & 0 \\
 & 1 & -1 & 0 & 1 \\
& 2 & -2 & 0 & 1 \\
\end{pmatrix}
\]
by considering $w_\gamma(\alpha_4^\vee) = -w_\gamma s_4(\alpha_4^\vee)$, as in Example~\ref{eg:beta0gamma34C7} above.  Therefore $\wbg = w_\beta w_\gamma$ and $\Mbg = \Id - w_\beta w_\gamma $ have matrices given by
\[
\wbg = 
\begin{pmatrix}
0&0&-1&1& \\
1&0&-1&1& \\
0&1&-1&1& \\
0&0&0&1&0&0&0\\
&&& 1 & -1 & 1 & 0 \\
&& & 1 & -1 & 0 & 1 \\
&&& 2 & -2 & 0 & 1 \\
\end{pmatrix} \quad
\Mbg = 
\begin{pmatrix}
1&0&1&-1& \\
-1&1&1&-1& \\
0&-1&2&-1& \\
0&0&0&0&0&0&0\\
&&& -1 & 2 & -1 & 0 \\
&& & -1 & 1 & 1 & -1 \\
&&& -2 & 2 & 0 & 0 \\
\end{pmatrix}. 
\]

Now from the shape of $\Mbg$, we see that $\Mbg \x = \bfc$ has solution over $\Z$ if and only if $c_7 \equiv 0 \; \operatorname{mod}\left( 2\right)$ and the equations represented by the $4 \times 4$ matrix $M_\beta' = \Id - w_{\beta'}$ in the top left of $\Mbg$ have solution over $\Z$.  Hence as the partition $\beta'$ has last part equal to $1$, we deduce from results of \Cref{sec:TypeA} that
\[
\Mod(\wbg) = \left\{ \sum_{i=1}^7 c_i \alpha_i^\vee \ \middle| \ c_i \in \Z, \ c_4 = 0, \ c_7 \equiv 0 \; \operatorname{mod}\left( 2\right) \right\},
\]
which has a $\Z$-basis given by $\{ \alpha_1^\vee, \alpha_2^\vee, \alpha_3^\vee\} \cup \{\alpha_5^\vee, \alpha_6^\vee, 2\alpha_7^\vee\}$.  

To find the Smith normal form $\Sbg$, we start with $\Mbg$, and first use column operations on columns $5,6,7$ to obtain
\[
\Tbg = 
\begin{pmatrix}
1&0&1&-1& \\
-1&1&1&-1& \\
0&-1&2&-1& \\
0&0&0&0&0&0&0\\
&&& \circled{$-1$} & 0 & \boxed{-1} & 0 \\
&& & \circled{$-1$} & 0 & 0 & \boxed{-1} \\
&&& \circled{$-2$} & \boxed{2} & 0 & 0 \\
\end{pmatrix}. 
\]
We then use the boxed pivots to clear the circled entries in rows $5,6,7$.  Now we can apply the same column operations as in \Cref{sec:TypeA} to the submatrix $M_{\beta'}$, to obtain
\[
\Tbg' = 
\begin{pmatrix}
\boxed{1}&0&0&0& \\
0&\boxed{1}&0&0& \\
0&0&\boxed{1}&0& \\
0&0&0&0&0&0&0\\
&&& 0& 0 & \boxed{-1} & 0 \\
&& & 0 & 0 & 0 & \boxed{-1} \\
&&& 0 & \boxed{2} & 0 & 0 \\
\end{pmatrix}. 
\]
Therefore $\Sbg = \diag(1^5, 2,0)$, and so $R^\vee/\Mod(\wbg) \cong (\Z/2\Z) \oplus \Z$.

\end{example}

\begin{example}\label{eg:beta33gamma3C9}

Let $\sW$ be of type $C_9$ and suppose $\beta = (3,3)$ and $\gamma = (3)$.  We will see that the Smith normal form for $\Id - w_\gamma$ is $\diag(1^6, 2,0^2)$, so that $R^\vee/\Mod(w_\gamma) \cong (\Z/2\Z) \oplus \Z^2$ has torsion part of rank equal to the number of parts of $\gamma$, and free part of rank equal to the number of parts of $\beta$.  In particular, $\gcd(\beta_k)=3$ and the indexing set $(I_\beta - 1)\setminus \{ 5 \} = \{ 2 \}$ play no role in this description, in contrast to type $A$; compare Example \ref{eg:64A9}.

We have $\wbg = w_\beta w_\gamma = (s_1 s_2)(s_4 s_5)(s_9 s_8 s_7)$, and so $w_\beta$ fixes $\alpha_i^\vee$ for $i = 7,8,9$.  Now let $\beta'$ be the partition of $6 + 1 = 7$ given by $\beta' = (3,3,1)$, and observe that, in type $A_6$, we also have $w_{\beta'} = (s_1 s_2)(s_4 s_5)$.  Thus we can apply results from \Cref{sec:TypeA} to obtain that the matrix for $w_\beta$, with respect to the type $C_9$ basis of simple coroots $\Delta^\vee$, is block-diagonal of the form shown on the left:
\[
w_\beta = \begin{pmatrix} w_{\beta'} & \\ & \Id_3 \end{pmatrix} = 
\begin{pmatrix}
0 & -1 & 1 & & & & \\
1 & -1 & 1 & & & & \\
0 & 0 & 1 & 0 & 0 & 0 \\
& & 1&0&-1&1& \\
 &  & &1&-1&1& \\
 & & &0&0&1& \\
&&&&&& \Id_3
\end{pmatrix}, 
\quad \mbox{and} \quad
w_\gamma = 
\begin{pmatrix}
\Id_5&&&& \\
& 1 &  0& 0 & 0 \\
& 1 & -1 & 1 & 0 \\
 & 1 & -1 & 0 & 1 \\
& 2 & -2 & 0 & 1 \\
\end{pmatrix}
\]
by considering $w_\gamma(\alpha_6^\vee) = -w_\gamma s_6(\alpha_6^\vee)$, similarly to Example~\ref{eg:beta0gamma34C7} above.  Therefore $\Mbg = \Id - w_\beta w_\gamma $ has matrix given by
\[
\Mbg = 
\begin{pmatrix}
1 & 1 & -1 & & & & && \\
-1 & 2 & -1 & & & & &&\\
0 & 0 & 0 & 0 & 0 & 0 && \\
& & -1&1&1&-1&&& \\
 &  & &-1&2&-1& &&\\
 & & &0&0&0& 0 &0&0\\
&&&&& -1 & 2 & -1 & 0 \\
&&&& & -1 & 1 & 1 & -1 \\
&&&&& -2 & 2 & 0 & 0 \\
\end{pmatrix}. 
\]

Now from the shape of $\Mbg$, we see that $\Mbg \x = \bfc$ has solution over $\Z$ if and only if $c_9 \equiv 0 \; \operatorname{mod}\left( 2\right)$ and the equations represented by the $6 \times 6$ matrix $M_\beta' = \Id - w_{\beta'}$ in the top left of $\Mbg$ have solution over $\Z$.  Hence, as the partition $\beta'$ has last part equal to $1$, we deduce from results of \Cref{sec:TypeA} that
\[
\Mod(\wbg) = \left\{ \sum_{i=1}^9 c_i \alpha_i^\vee \ \middle| \ c_i \in \Z, \ c_3 = 0, \ c_6 = 0, \ c_9 \equiv 0 \; \operatorname{mod}\left( 2\right) \right\},
\]
which has a $\Z$-basis given by $\{ \alpha_1^\vee, \alpha_2^\vee\} \cup \{ \alpha_4^\vee, \alpha_5^\vee\} \cup \{\alpha_7^\vee, \alpha_8^\vee, 2\alpha_9^\vee\}$.  

To find the Smith normal form $\Sbg$, we start with $\Mbg$, and first use column operations on columns $7,8,9$ to obtain
\[
\Tbg' = 
\begin{pmatrix}
1 & 1 & -1 & & & & && \\
-1 & 2 & -1 & & & & &&\\
0 & 0 & 0 & 0 & 0 & 0 && \\
& & -1&1&1&-1&&& \\
 &  & &-1&2&-1& &&\\
&&&0&0&0&0&0&0\\
&&&&& \circled{$-1$} & 0 & \boxed{-1} & 0 \\
&&&& & \circled{$-1$} & 0 & 0 & \boxed{-1} \\
&&&&& \circled{$-2$} & \boxed{2} & 0 & 0 \\
\end{pmatrix}. 
\]
We then use the boxed pivots to clear the circled entries in rows $7,8,9$.  Now we can apply the same column operations as in \Cref{sec:TypeA} to the submatrix $M_{\beta'}$, to obtain
\[
\Tbg = 
\begin{pmatrix}
\boxed{1} & 0 & 0 & & & & && \\
0 & \boxed{1} & 0 & & & & &&\\
0 & 0 & 0 & 0 & 0 & 0 && \\
& & 0&\boxed{1}&0&0&& \\
 &  & &0&\boxed{1}&0& &&\\
&&&0&0&0&0&0&0\\
&&&&& 0& 0 & \boxed{-1} & 0 \\
&&&& & 0 & 0 & 0 & \boxed{-1} \\
&&&&& 0 & \boxed{2} & 0 & 0 \\
\end{pmatrix}. 
\]
Therefore $\Sbg = \diag(1^6, 2,0^2)$, and so $R^\vee/\Mod(\wbg) \cong (\Z/2\Z) \oplus \Z^2$.

\end{example}

\subsection{Proofs in type $C$}\label{sec:proofsC}

In order to prove \Cref{thm:Cn}, we will use many of our definitions and results from \Cref{sec:TypeA}.  We also define several additional auxiliary matrices and record some results for them in~\Cref{sec:matricesC}.   We then determine the matrix for $w_{\beta,\gamma}$, and hence for $M_{\beta,\gamma} = \Id - w_{\beta,\gamma}$, with respect to $\Delta^\vee$ in Section~\ref{sec:wMC}, and complete the proof of \Cref{thm:Cn} in Section~\ref{sec:propsCn}.

\subsubsection{Matrix definitions and results in type $C$}\label{sec:matricesC}  

In this section, we define several additional matrices which we will use in type $C$, and record some results for these.  

We define $1 \times 1$ matrices $V_1 = [-1]$ and $N_1 = \Id - V_1 = [2]$.  Then for all $r \geq 2$ we define $r \times r$ matrices
\begin{equation}\label{eq:Vr}
V_r = \begin{pmatrix}
-1 & 1 &  0 & \cdots & \cdots & 0 \\ 
-1 & 0 & \ddots & \ddots &  & \vdots \\ 
\vdots & \vdots & \ddots & \ddots & \ddots & \vdots \\ 
\vdots & \vdots &  & \ddots & 1 & 0 \\
-1 & 0 & \cdots & \cdots & 0 & 1 \\
-2 & 0 & \cdots & \cdots  & 0 & 1
\end{pmatrix}
\quad \mbox{and} \quad
N_r = \begin{pmatrix}
2 & -1 &  0 & \cdots & \cdots & 0 \\ 
1 & 1 & \ddots & \ddots &  & \vdots \\ 
1 & 0 & \ddots & \ddots & \ddots & \vdots \\ 
\vdots & \vdots & \ddots & 1 & -1 & 0 \\
1 & 0 & \cdots & 0 & 1 & -1 \\
2 & 0 & \cdots & 0  & 0 & 0
\end{pmatrix}.
\end{equation}
That is, the $(1,1)$- and $(r,1)$-entry of $N_r = \Id - V_r$ is $2$, the $(i,1)$- and $(i,i)$-entries are $1$ for $2 \leq i \leq r-1$, the $(i,i+1)$-entry is $-1$ for $1\leq i \leq r-1$, and all other entries of $N_r$ are $0$.

\begin{lemma}  For all $r \geq 1$, the $\Z$-linear system $N_r \x = \bfc$ has solution $\x \in \Z^r$ if and only if $c_r \equiv 0 \; \operatorname{mod}\left(2\right)$.
\end{lemma}
\begin{proof}  This is clear for $r =1$.  When $r \geq 2$, for $1 \leq i \leq r-1$, the last nonzero entry in row $i$ of $N_r$ is the $-1$ in position $(i,i+1)$, while the only nonzero entry in row $r$ of $N_r$ is the $2$ in column $1$.  The result follows. \end{proof}

Consider the coroot lattice $R^\vee$ as a free $\Z$-module with $\Z$-basis $\Delta^\vee = \{ \alpha_i^\vee \}$. In this section, when taking a subsystem of type $C_r$ for some $1 \leq r < n$, we denote the coroot lattice by $R_r^\vee$ with corresponding basis $\Delta^\vee_r = \{ \alpha_1^\vee, \dots, \alpha_r^\vee\}$, where $r$ indexes the special node.

\begin{corollary}  For all $r \geq 1$, we have
\[N_r R_r^\vee = (\Id - V_r)R_r^\vee = \left\{ \sum_{i=1}^r c_i \alpha_i^\vee \ \middle|\  c_i \in \Z, \ c_r \equiv 0 \; \operatorname{mod}\left(2\right) \right\},\]
with a $\Z$-basis given by $\{ \alpha_1^\vee, \dots, \alpha^\vee_{r-1}, 2\alpha_r^\vee \}$.
\end{corollary}

Now define $U_r$ to be the matrix obtained from $N_r$ by replacing $\cC_1(N_r)$ by $\sum_{i=1}^r i\cC_i(N_r)$, and replacing $\cC_i(N_r)$ for $2 \leq i \leq r-1$ by $\sum_{j=i}^r \cC_j(N_r)$.  Thus
\begin{equation}\label{eq:Ur}
U_r = 
\begin{pmatrix}
0 & -1 &  0 & \cdots & \cdots & 0 \\ 
0 & 0 & -1 & \ddots &  & \vdots \\ 
\vdots & \ddots & \ddots & \ddots & \ddots & \vdots \\ 
\vdots &  & \ddots & 0 & -1 & 0 \\
0 & \cdots & \cdots & 0 & 0 & -1 \\
2 & 0 & \cdots & 0  & 0 & 0
\end{pmatrix}.
\end{equation}
For $2 \leq i \leq r$, multiply $\cC_i(U_r)$ by $-1$, and then permute columns to obtain the Smith normal form.

\begin{lemma}\label{lem:UrSNF}  
The matrix $N_r$ has Smith normal form $\diag(1^{r-1},2)$.
\end{lemma}

\begin{corollary}\label{cor:UrQuotient}
$R_r^\vee /N_r R_r^\vee = R_r^\vee / (\Id - V_r)R_r^\vee \cong \Z/2\Z.$
\end{corollary}

We also define several additional matrices, which will appear as blocks in subsequent arguments, as follows.  For $k \in \Z$ and $m,n \geq 1$ we define $m \times n$ matrices
\[
L_{m,n}(k) = 
\begin{pmatrix}
k & 0 & \cdots & 0 \\
\vdots & \vdots &  & \vdots \\
k & 0 & \cdots & 0
\end{pmatrix}
\quad\mbox{and}\quad
R_{m,n}(k) = 
\begin{pmatrix}
0 & \cdots & 0 & k\\
\vdots & & \vdots &  \vdots \\
0 & \cdots& 0 & k
\end{pmatrix}.
\]
That is, $L_{m,n}(k)$ (respectively, $R_{m,n}(k)$) is all $0$s except for its first (respectively, last) column, which is all $k$s. 
Then for $k,k' \in \Z$, $m \geq 1$, and $n \geq 2$, we define
\[
N_{m,n}(k;k')   =  L_{m,n}(k) + R_{m,n}(k') = 
\begin{pmatrix}
k & 0 & \cdots & 0 & k'\\
\vdots & \vdots & &\vdots & \vdots  \\
k & 0 & \cdots& 0 & k'
\end{pmatrix}.
\]
For $k \in \Z$, $m \geq 2$, and $n \geq 1$, we define
\[
R_{m,n}(1,2)=\begin{pmatrix}
0 & \cdots & 0 & 1\\
\vdots & & \vdots &  \vdots \\
0 & \cdots& 0 & 1 \\
0 & \cdots& 0 & 2
\end{pmatrix}
\quad\mbox{and}\quad
R_{m,n}(-1,-2)=\begin{pmatrix}
0 & \cdots & 0 & -1\\
\vdots & & \vdots &  \vdots \\
0 & \cdots& 0 & -1 \\
0 & \cdots& 0 & -2
\end{pmatrix}.
\]
That is, $R_{m,n}(1,2)$ is all $0$s except for its last column, where the $(m,n)$-entry is $2$ and all other entries are $1$, while $R_{m,n}(-1,-2)=-R(1,2)$. 
Finally, for $k \in \Z$, $m \geq 2$, and $n \geq 1$, we define 
\begin{eqnarray*}
N_{m,n}(k; 1,2) & = & L_{m,n}(k) + R_{m,n}(1,2) \\
N_{m,n}(k; -1,-2) & = &L_{m,n}(k) + R_{m,n}(-1,-2).
\end{eqnarray*}
To simplify notation, we will omit the subscript $m,n$ when the size of these matrices is clear.

\subsubsection{The matrices for $w_{\beta,\gamma}$ and $M_{\beta,\gamma}$}\label{sec:wMC}

In this section we determine the matrices for $w_{\beta,\gamma}$ and hence $M_{\beta,\gamma} = \Id - w_{\beta,\gamma}$, with respect to the basis $\Delta^\vee$ for $R^\vee$ in type $C_n$ with $n \geq 2$.  
Most of this section is devoted to finding the matrix for $w_\gamma$.  The matrix for $w_\beta$ then has a relatively straightforward description, and it is not difficult to multiply these matrices together to obtain the matrix for $w_{\beta,\gamma} = w_\beta \cdot w_\gamma$.  

To construct the matrix for $w_\gamma$, we first record the following special cases.

\begin{lemma}\label{lem:gammaCoxeterCn}  Let $\sW$ be of type $C_n$, for $n \geq 2$.  
\begin{enumerate}
\item If $\gamma = (n)$, then the matrix for $w_\gamma$ with respect to $\Delta^\vee$ is $V_n$.
\item\label{gamma1s} If $\gamma = (1,\dots,1)$ with $|\gamma| = n$, then the matrix for $w_\gamma$ with respect to $\Delta^\vee$ is $-\Id_n$.
\end{enumerate}
\end{lemma}
\begin{proof}  
The proof of (1) is a straightforward generalization of the corresponding computations in Example~\ref{eg:beta0gamma4C4}.  For (2), we have $w_\gamma = t_0 t_1 \dots t_{n-1}$ in the notation of \eqref{eq:ti}.  Then $w_\gamma$ is the longest element in $\sW$ (see, for instance, Table 1 of~\cite{BenkartKangOhPark}).  Since we are in type $C_n$, we then have $w_\gamma(\alpha_i^\vee) = - \alpha_i^\vee$ for all $i \in [n]$, as claimed.
\end{proof}

We next find matrices for $w^\gamma_1$ and $w^\gamma_q$, as defined in \eqref{eq:wgamma1} and \eqref{eq:wgammak}, respectively, for $|\gamma| = n$.

\begin{lemma}\label{lem:1q}  Let $\sW$ be of type $C_n$ for $n \geq 2$, and let $\gamma = (\gamma_1, \dots,\gamma_q)$ be a weakly increasing composition of $n$. 
\begin{enumerate}
\item\label{1q1} If $2 \leq \gamma_1 < n$, then the matrix for $w^\gamma_1$ with respect to $\Delta^\vee$ is block lower-triangular of the form
\[w^\gamma_1 = 
\begin{pmatrix} 
\Id_{n - \gamma_1} & \\
R(1,2)& V_{\gamma_1}
\end{pmatrix}.\]
\item\label{1q2} If $2 \leq \gamma_q < n$, then the matrix for $w^\gamma_q$ with respect to $\Delta^\vee$ is block lower-triangular of the form \[ w^\gamma_q = \begin{pmatrix}V_{\gamma_q} & \\ L(-2) & \Id_{n-\gamma_q} \end{pmatrix}.\]
\end{enumerate}
\end{lemma}
\begin{proof}   Parts \eqref{1q1} and \eqref{1q2} in this statement are obtained using the same methods as in the computation of $w^\gamma_1$ and $w^\gamma_2$ from Example~\ref{eg:beta0gamma34C7}, respectively.
\end{proof}

We then obtain the matrix for $w_\gamma = w^\gamma_1 \dots w^\gamma_q$ when $|\gamma| =n$ and $\gamma_1 \geq 2$ as follows.

\begin{corollary}\label{cor:wgammano1parts}
Let $\sW$ be of type $C_n$ for $n \geq 2$, and let $\gamma = (\gamma_1, \dots,\gamma_q)$ be a weakly increasing composition of $n$.  If $\gamma_1 \geq 2$, then the matrix for $w_\gamma$ with respect to $\Delta^\vee$ is block lower-triangular of the form
\[ w_\gamma = \begin{pmatrix}V_{\gamma_q} & & & &\\ 
N(-2;1,2) & V_{\gamma_{q-1}}&  & &\\
N(-2;2) & N(-2;1,2)  & V_{\gamma_{q-2}} &&  \\
\vdots & \ddots & \ddots & \ddots && \\
N(-2;2) & \cdots &N(-2;2)   & N(-2;1,2) & V_{\gamma_1}
 \end{pmatrix}.\]
 \end{corollary}

\begin{proof}
This statement is obtained by generalizing the argument in Example~\ref{eg:beta0gamma234C9}, by induction on the number of parts of $\gamma$, and using the matrices given by \Cref{lem:1q}. 
\end{proof}

We now consider the case $|\gamma| = n$ and $\gamma_1 = 1$.

\begin{lemma}\label{lem:wgamma1parts} 
Let $\sW$ be of type $C_n$ for $n \geq 2$, and let $\gamma = (\gamma_1, \dots,\gamma_q)$ be a weakly increasing composition of $n$. If the first $m$ parts of $\gamma$ are equal to $1$, with $1 \leq m < n$, then the matrix for $w^\gamma_1 \cdots w^\gamma_m$ with respect to $\Delta^\vee$ is block lower-triangular of the form
\[w^\gamma_1 \cdots w^\gamma_m = 
\begin{pmatrix} 
\Id_{n - m} & \\
R(2)& -\Id_m
\end{pmatrix}.\]
\end{lemma}

\begin{proof}
If $m = 1$ we have $w^\gamma_1 = t_0 = s_n$, and the result is immediate (see Example \ref{eg:beta0gamma13C4} for the case $n = 4$).  Now if $m \geq 2$, we have  $w^\gamma_1 \cdots w^\gamma_m = t_0 \cdots t_{m-1}$.  This is the longest element in the type $C_m$ subsystem indexed by the last $m$ nodes, which fixes~$\alpha_i^\vee$ for $1 \leq i \leq n-m-1$.  Hence by \Cref{lem:gammaCoxeterCn}\eqref{gamma1s}, it now suffices to prove that the matrix for $w^\gamma_1 \cdots w^\gamma_m$ has $(n-m,n-m)$-entry equal to $1$, and $(i,n-m)$-entry equal to $2$ for $n-m+1 \leq i \leq n$.  For this, we first obtain using similar arguments to those used to prove \Cref{cor:wgammano1parts} that
\[
w^\gamma_m(\alpha^\vee_{n-m}) = (s_{n-m+1} \dots s_{n-1})(s_n s_{n-1} \dots s_{n-m+1})(\alpha^\vee_{n-m}) =  \alpha_{n-m}^\vee + \sum_{i=n-m+1}^n 2\alpha_i^\vee.
\]
Using this result and induction on $m \geq 2$, we then get via \eqref{eq:siaction} that
\begin{eqnarray*}
w^\gamma_1 \cdots w^\gamma_m(\alpha_{n-m}^\vee) & = & w^\gamma_1 \dots w^\gamma_{m-1}\left( \alpha_{n-m}^\vee + \sum_{i=n-m+1}^n 2\alpha_i^\vee \right) \\
& = &  \alpha_{n-m}^\vee + 2\left(\alpha_{n-m+1}^\vee + \sum_{i=n-m+2}^n 2\alpha_i^\vee \right) + 2\left( \sum_{i=n-m+2}^n -\alpha_i^\vee \right) 
\end{eqnarray*}
which simplifies to give the required result.
\end{proof}

Combining \Cref{lem:wgamma1parts} and \Cref{cor:wgammano1parts}, we thus obtain the following uniform description of the matrix for $w_\gamma$ when $|\gamma| = n$.

\begin{corollary}\label{cor:wgamma1parts}
Let $\sW$ be of type $C_n$ for $n \geq 2$, and let $\gamma = (\gamma_1, \dots,\gamma_q)$ be a weakly increasing composition of $n$.  Define $0 \leq m \leq q$ such that $\gamma_k = 1$ for all $1 \leq k \leq m$ and $\gamma_k \geq 2$ for all $m+1 \leq k \leq q$.
Then the matrix for $w_\gamma$ with respect to $\Delta^\vee$ is block lower-triangular of the form
\[ w_\gamma = \begin{pmatrix}V_{\gamma_q} & & & & &\\ 
N(-2;1,2) & V_{\gamma_{q-1}}&  & &&\\
N(-2;2) & N(-2;1,2)  & V_{\gamma_{q-2}} &&&  \\
\vdots & \ddots & \ddots & \ddots && &\\
N(-2;2) & \cdots &N(-2;2)   & N(-2;1,2) & V_{\gamma_{m+1}}& \\
N(-2;2)&\cdots&N(-2;2)&N(-2;2)&N(-2;2)&-\Id_m
 \end{pmatrix}.\]
 \end{corollary}

We can now put together the above results to describe the matrix for $w_\gamma$ when $1 \leq |\gamma| \leq n$ (when $|\gamma|=0$, the element $w_\gamma$ is trivial).  The proof of the next result generalizes the computation of the matrix for $w_\gamma$ from Example~\ref{eg:beta4gamma3C7}.

\begin{corollary}\label{cor:wgammaC}  
Let $\sW$ be of type $C_n$ for $n \geq 2$, and let $\gamma = (\gamma_1, \dots,\gamma_q)$ be a weakly increasing composition such that $1 \leq |\gamma| \leq n$.  Let $0 \leq m \leq q$ be such that $\gamma_k = 1$ for all $1 \leq k \leq m$ and $\gamma_k \geq 2$ for all $m+1 \leq k \leq q$. Then the matrix for $w_\gamma$ with respect to $\Delta^\vee$ is block lower-triangular of the form
\[ w_\gamma = \begin{pmatrix}
\Id_{n-|\gamma|} &&&&& 
\\R(1,2)&V_{\gamma_q} & & & & &\\ 
R(2)&N(-2;1,2) & V_{\gamma_{q-1}}&  & &&\\
\vdots&N(-2;2) & N(-2;1,2)  & V_{\gamma_{q-2}} &&&  \\
\vdots&\vdots & \ddots & \ddots & \ddots && &\\
R(2) &N(-2;2)  & \cdots &N(-2;2)   & N(-2;1,2) & V_{\gamma_{m+1}}& \\
R(2) & N(-2;2)&\cdots&N(-2;2)&N(-2;2)&N(-2;2)&-\Id_m
 \end{pmatrix}.\]
\end{corollary}

We now turn to the matrix for $w_\beta$ in type $C$. The formula is given by the next result, which generalizes the computation of the matrix for $w_\beta$ in Example~\ref{eg:beta4gamma3C7}.

\begin{lemma}\label{lem:wbetaC}  Let $\sW$ be of type $C_n$ for $n \geq 2$, and let $\beta = (\beta_1, \dots,\beta_p)$ be a partition with $|\beta| \leq n$.  
\begin{enumerate}
\item\label{beta01} If $|\beta| \in \{ 0,1\}$ then $w_\beta = \Id_{n}$.
\item\label{beta2} If $|\beta| \geq 2$, let $\beta'$ be the partition of $|\beta|+1$ given by $\beta' = (\beta_1,\dots,\beta_p,1)$, and let $w_{\beta'}$ be the matrix for the corresponding conjugacy class representative in type~$A_{|\beta|}$.  Then the matrix for $w_\beta \in \sW$ with respect to $\Delta^\vee$ is the block-diagonal matrix $w_\beta = \begin{pmatrix} w_{\beta'} & \\  & \Id_{n-|\beta|} \end{pmatrix}$.
\end{enumerate}
\end{lemma}

\begin{proof}  If $|\beta| \in \{ 0,1\}$ then by definition $w_\beta$ is trivial, so the result in \eqref{beta01} is clear.  

Assume now that $|\beta| = m \geq 2$, and recall from Section \ref{sec:repsC} that $w_\beta$ is the product, in increasing order, of certain simple reflections in the type $A_{m-1}$ subsystem indexed by the first $m-1$ nodes.   Hence $w_\beta$ fixes $\alpha_i^\vee$ for all $m+1 \leq i \leq n$, which explains the block $\Id_{n-|\beta|}$.  

Now since the last part of $\beta'$ equals $1$, the expression for $w_\beta$ in terms of the simple reflections in type $A_{m-1}$ is equal to the expression for $w_{\beta'}$ in terms of the simple reflections in type $A_{m}$.   Moreover, for all $2 \leq m \leq n$, the formula $s_{m-1}(\alpha^\vee_m) = \alpha^\vee_{m-1} + \alpha^\vee_{m}$ holds in type $C_{n}$ as well as in type $A_n$ by \eqref{eq:siaction}, since the Cartan matrix entries $c_{m-1,m} = -1$ for all $2 \leq m \leq n$ for both types. Hence for $1 \leq i \leq m$, the expression for $w_\beta(\alpha^\vee_i)$ in terms of the simple coroots $\{ \alpha_1^\vee,\dots,\alpha_{m}^\vee \}$ in type $C_n$ is equal to the expression for $w_{\beta'}(\alpha^\vee_i)$ in terms of the simple coroots $\{ \alpha_1^\vee,\dots,\alpha_{m}^\vee \}$ in type $A_n$.  The result in \eqref{beta2} follows. 
\end{proof}

The identity matrix blocks of $w_\beta$ and $w_\gamma$ mean that we can now easily determine the matrix for $\wbg = w_\beta \cdot w_\gamma$ as follows.

\begin{corollary}\label{cor:wbetagammaC}  Let $\sW$ be of type $C_n$ for $n \geq 2$.  Let $(\beta, \gamma)$ be a pair of compositions so that $\beta = (\beta_1, \dots,\beta_p)$ is weakly decreasing, $\gamma = (\gamma_1,\dots,\gamma_q)$ is weakly increasing, and $|\beta|+ |\gamma| = n$.   Let $\beta'$ be the partition of $|\beta| + 1$ given by $\beta' = (\beta_1, \dots, \beta_p,1)$.
\begin{enumerate}
\item If $|\beta| = n$, then the matrix for $\wbg = w_\beta$ with respect to $\Delta^\vee$ in type $C_n$ the same as the matrix for $w_{\beta'}$ with respect to $\Delta^\vee$ in type $A_{n}$.
\item If $|\gamma| = n$, then the matrix for $\wbg = w_\gamma$ with respect to $\Delta^\vee$ in type $C_n$ is given by \Cref{cor:wgammaC} above.
\item Otherwise, the matrix for $\wbg = w_\beta \cdot w_\gamma$ with respect to $\Delta^\vee$ in type $C_n$ is the block lower-triangular matrix obtained by replacing the $\Id_{n-|\gamma|}$ block in the top left of the matrix in \Cref{cor:wgammaC} by the matrix $w_{\beta'}$.
\end{enumerate}
\end{corollary}

Finally, we are able to determine the matrix for $\Mbg = \Id - \wbg$.  Recall $N_r$ from \eqref{eq:Vr}.

\begin{corollary}\label{cor:MbetagammaC}  Let $\sW$ be of type $C_n$ for $n \geq 2$.  Let $(\beta, \gamma)$ be a pair of compositions so that $\beta = (\beta_1, \dots,\beta_p)$ is weakly decreasing, $\gamma = (\gamma_1,\dots,\gamma_q)$ is weakly increasing, and $|\beta|+ |\gamma| = n$.  Let $\beta'$ be the partition of $|\beta| + 1$ given by $\beta' = (\beta_1, \dots, \beta_p,1)$, and let $0 \leq m \leq q$ be such that $\gamma_k = 1$ for all $1 \leq k \leq m$ and $\gamma_k \geq 2$ for all $m+1 \leq k \leq q$.
\begin{enumerate}
\item If $|\beta| = n$, then the matrix for $\Mbg = \Id - w_\beta$ with respect to $\Delta^\vee$ is equal to the matrix $M_{\beta'}$ from Corollary \ref{cor:matrixMbetaA}.
\item\label{Cgamman} If $|\gamma| = n$, then the matrix for $\Mbg = \Id - w_\gamma$ with respect to $\Delta^\vee$ is block lower-triangular of the form
\[ \Mbg = \begin{pmatrix}N_{\gamma_q} & & & & &\\ 
N(2;-1,-2) & N_{\gamma_{q-1}}&  & &&\\
N(2;-2) & N(2;-1,-2)  & N_{\gamma_{q-2}} &&&  \\
\vdots & \ddots & \ddots & \ddots && &\\
N(2;-2) & \cdots &N(2;-2)   & N(2;-1,-2) & N_{\gamma_{m+1}}& \\
N(2;-2)&\cdots&N(2;-2)&N(2;-2)&N(2;-2)&2\Id_m
 \end{pmatrix}.\]
\item\label{Cbetagamma} Otherwise, the matrix for $\Mbg = \Id - \wbg$ with respect to $\Delta^\vee$ is block lower-triangular of the form
\[ \Mbg = \begin{pmatrix}
M_{\beta'} &&&&& 
\\R(-1,-2)&N_{\gamma_q} & & & & &\\ 
R(-2)&N(2;-1,-2) & N_{\gamma_{q-1}}&  & &&\\
\vdots&N(2;-2) & N(2;-1,-2)  & N_{\gamma_{q-2}} &&&  \\
\vdots&\vdots & \ddots & \ddots & \ddots && &\\
R(-2) &N(2;-2)  & \cdots &N(2;-2)   & N(2;-1,-2) & N_{\gamma_{m+1}}& \\
R(-2) & N(2;-2)&\cdots&N(2;-2)&N(2;-2)&N(2;-2)&2\Id_m
 \end{pmatrix}.\]
\end{enumerate}
\end{corollary}

\subsubsection{Proof of \Cref{thm:Cn}}\label{sec:propsCn}

In this section we complete the proof of \Cref{thm:Cn}.  Let $\sW$ be of type $C_n$ with $n \geq 2$, and let $(\beta,\gamma)$ be a pair of compositions such that $\beta = (\beta_1,\dots,\beta_p)$ is weakly decreasing, $\gamma = (\gamma_1,\dots,\gamma_q)$ is weakly increasing, and $|\beta| + |\gamma| = n$.  Let $\Mbg = \Id - \wbg$ be the matrix given by \Cref{cor:MbetagammaC}.

We first assume that $|\beta| = n$.  Then $I_\gamma = J_\gamma = \emptyset$ and $q = 0$.  Now since $\Mbg = M_{\beta'}$ where $\beta' = (\beta_1, \dots, \beta_p,1)$ has last part equal to $1$, the result follows from \Cref{thm:An} in this case.

We next assume that $|\gamma| = n$, in which case $J_\beta = \emptyset$.  By Corollary \ref{cor:MbetagammaC}\eqref{Cgamman}, for all $i \in I_\gamma$, the nonzero entries in row $i$ of $\Mbg$ are~$\pm 2$, while for all $i \in [n] \setminus I_\gamma$, the last nonzero entry in row $i$ is the $-1$ in the $(i,i+1)$-position.  Hence, if $|\beta| = 0$, then
\[  \Mod(\wbg) = \left\{\sum_{i=1}^n c_i \alpha_i^\vee \ \middle| \ c_i \in \Z, \ c_i \equiv 0 \; \operatorname{mod}\left(2\right) \mbox{ for } i \in I_\gamma  \right\},\]
which clearly has a $\Z$-basis given by $\{ \alpha_i^\vee \mid i \in [n] \setminus I_\gamma \} \cup \{ 2\alpha_i^\vee \mid i \in I_\gamma \}$.

 For the Smith normal form when $|\gamma| = n$, we first use the pivot entries in the $2\Id_m$ block in the lower right to clear all nonzero entries in the last $m$ rows of $\Mbg$.  We then use column operations to convert $N_{\gamma_{m+1}}$ to the matrix $U_{\gamma_{m+1}}$ defined in \eqref{eq:Ur}.  Since $U_{\gamma_{m+1}}$ has pivot entries $-1$ in all but its last row, and a pivot entry $2$ in its last row, we can then use column operations to clear all nonzero entries to the left of $U_{\gamma_{m+1}}$.  We then repeat this process moving up one block at a time, to obtain the block-diagonal matrix
 \[
 U_{\beta,\gamma} = \begin{pmatrix} U_{\gamma_q} &&& \\ & \ddots && \\ && U_{\gamma_{m+1}}& \\ &&& 2\Id_m \end{pmatrix}.
 \]
Now $2\Id_m$ is in Smith normal form already, while for $m+1 \leq k \leq q$, the matrix $U_{\gamma_k}$ has Smith normal form $\diag(1^{\gamma_k - 1},2)$ by Lemma \ref{lem:UrSNF}.  We thus obtain that $\Mbg$ has Smith normal form $\Sbg = \diag(1^{n-q},2^q)$, and hence $R^\vee/\Mod(\wbg) \cong (\Z/2\Z)^q$.  This completes the proof of \Cref{thm:Cn} in the case $|\gamma| = n$.

The remaining case is when $1 \leq |\beta| \leq n-1$ and  $1 \leq |\gamma| \leq n-1$, which corresponds to case \eqref{Cbetagamma} of Corollary \ref{cor:MbetagammaC}. Applying the same argument as for $|\gamma| = n$, we initially see that $\Mbg \x = \bfc$ has solution over $\Z$ if and only if $c_i \equiv 0 \; \operatorname{mod}\left(2\right)$ for all $i \in I_\gamma$ and the $|\beta|$ equations represented by the matrix $M_{\beta'}$ in the top left of $\Mbg$ have solution over $\Z$.  Now as the partition $\beta'$ has last part equal to $1$, these $|\beta|$ equations have solution over $\Z$ if and only if $c_i = 0$ for all $i \in \{ 1,\dots,|\beta|\} \setminus J_\beta$ by part \eqref{ModA} of Theorem \ref{thm:An}.  This establishes the description of $\Mod(\wbg)$ from \Cref{thm:Cn}\eqref{ModC}, and the $\Z$-basis given in \eqref{BasisC} is immediate.

 For the Smith normal form in this final case, we first apply to the $|\gamma| \times |\gamma|$ block in the bottom right-hand corner of $\Mbg$ the same sequence of column operations as in the case $|\gamma| = n$.  This converts $\Mbg$ to the block lower-triangular matrix
 \[
U_{\beta,\gamma} = \begin{pmatrix}
M_{\beta'} &&&&&& \\
R(-1,-2)&U_{\gamma_q} & & & &&\\ 
R(-2)& & U_{\gamma_{q-1}}&  & &&\\
\vdots& &   & U_{\gamma_{q-2}} &&&  \\
\vdots& &  &  & \ddots && &\\
R(-2)& &  &  & & U_{\gamma_{m+1}}& \\
R(-2)& &  &  & & &2\Id_m
 \end{pmatrix}.
 \]
We can then use the pivot entries in the $U_{\gamma_k}$ and $2\Id_m$ to clear the nonzero entries in the corresponding matrices $R(-1,-2)$ and $R(-2)$.  We now have the block-diagonal matrix
  \[
 U'_{\beta,\gamma} = \begin{pmatrix} M_{\beta'}&&&& \\ &U_{\gamma_q} &&& \\ && \ddots && \\ &&& U_{\gamma_{m+1}}& \\ &&& &2\Id_m \end{pmatrix}.
 \] 
 Since the partition $\beta'$ has $p+1$ parts with last part equal to $1$, by \Cref{thm:An} the Smith normal form for $M_{\beta'}$ is $\diag(1^{|\beta| - (p+1)},1,0^{(p+1)-1}) = \diag(1^{|\beta| - p},0^p)$.  Now from the case $|\gamma| = n$, the $|\gamma| \times |\gamma|$ matrix in the bottom right has Smith normal form $\diag(1^{|\gamma| - q},2^q)$.  Putting this together with the equation $|\beta| + |\gamma| = n$, we obtain that $\Mbg$ has Smith normal form $\Sbg = (1^{n - p - q}, 2^q, 0^p)$.  Therefore $R^\vee/\Mod(\wbg) \cong (\Z/2\Z)^q \oplus  \Z^p$.  
 
 This completes the proof of \Cref{thm:Cn}. \qed

\section{Type $B$ mod-sets}\label{sec:TypeB}

In this section, we give a description of all mod-sets in type $B$.   Since the minimal length representatives of conjugacy classes in $\sW$ are identical in types $B_n$ and $C_n$, we start by directly introducing the general notation required for the main results stated in Section \ref{sec:repsB}.  We then give several key examples in \Cref{sec:examplesB}, designed to fully illustrate the proofs of our results on mod-sets.

Let $\sW$ be the finite Weyl group of type $B_n$, with $n \geq 2$. Throughout this section, we order the nodes of the Dynkin diagram increasing from left to right, as in both Bourbaki \cite{Bourbaki4-6} and Sage \cite{sagemath}, so that the first $n-1$ nodes form a type $A_{n-1}$ subsystem, and the special node is indexed by $n$ on the right.  We note that our labeling is the reverse of the ordering of nodes used in~\cite{GeckPfeifferBook}, and also that in Sage \cite{sagemath} the Cartan matrices in types $B_n$ and $C_n$ are reversed with respect to the Dynkin diagrams; see~\Cref{app:dynkin} for a direct comparison of these conventions.

\subsection{Conjugacy class representatives and mod-sets in type $B$}\label{sec:repsB}

Following \cite[Proposition 3.4.7]{GeckPfeifferBook} as reviewed in Section \ref{sec:repsC}, the conjugacy classes of $\sW$ of type $B_n$ are parameterized by ordered pairs of compositions~$(\beta, \gamma)$ such that $\beta$ is weakly decreasing, $\gamma$ is weakly increasing, and $|\beta| + |\gamma| = n \geq 2$. For each such pair $(\beta,\gamma)$, recall from Section \ref{sec:repsC} the standard parabolic subgroups $\sW_{\beta}$ and $\sW_\gamma$ of $\sW$, and the  element $\wbg = w_\beta \cdot w_\gamma$ with $w_\beta$ cuspidal in $\sW_\beta$ and $w_\gamma$ cuspidal in $ \sW_\gamma$ such that $\wbg$ form a complete system of minimal length representatives for the conjugacy classes of $\sW$.

We now establish the notation needed to formally state our results on mod-sets in type $B$.  Throughout this section, we let 
\[\beta= (\beta_1,\dots,\beta_p) \quad \mbox{and} \quad \gamma = (\gamma_1,\dots,\gamma_{q}) \] 
be a pair of compositions such that $\beta$ is weakly decreasing, $\gamma$ is weakly increasing, and $|\beta| + |\gamma| = n$, with corresponding conjugacy class representative $\wbg \in \sW$, constructed in \Cref{sec:repsC}. Write $|\beta| = m$ so that $|\gamma| = n-m$. Recall the subsets $J_\beta$ and $J_\gamma$ which index the support of $w_\beta$ and $w_\gamma$, as well as the sequences $(j_k^\beta)$ and $(j_k^\gamma)$ of partial sums of the compositions $\beta$ and $\gamma$, which both satisfy $j_1^\beta = j_1^\gamma = 0$. 

For the weakly increasing composition $\gamma$, we define the following $q$-element subset of $[n]$ exactly as in type $C_n$:
\[
I_\gamma = \left\{ n - j^\gamma_k \mid 1 \leq k \leq q \right\} = \left\{ n, n - \gamma_1, \dots, n - \sum_{k=1}^{q-1} \gamma_k \right\} = \left\{ n - (|\gamma| - \gamma_q), \dots, n - \gamma_1, n \right\}.
\]
In type $B_n$, unlike type $C_n$, we also have the following subset depending on the partition $\beta$:
\[ I_\beta = \{ j_k^\beta \mid 2 \leq k \leq p \} \cup \{ m\}.\] Given any subset $I \subseteq [n]$, we denote by $I- 1 = \{ i-1 \mid i \in I\}$.

To simplify notation in the statements below, we further define $\gamma_{1,0} = 0$ and $\gamma_{k,k-1} = \left(\gamma_k + \gamma_{k-1}\right) \operatorname{mod} \left(2\right)$ for $2 \leq k \leq q$. That is, $\gamma_{k,k-1} = 0$ if the parts $\gamma_k$ and $\gamma_{k-1}$ have the same parity, and otherwise $\gamma_{k,k-1}=1$.  In particular, note that all parts of $\gamma$ have the same parity if and only if $\gamma_{k,k-1} = 0$ for all $2 \leq k \leq q$. We define $\gcd(\beta_k,2) = \gcd(\beta_1,\dots,\beta_k,2)$, so that $\gcd(\beta_k,2) = 1$ if and only if $\beta$ has at least one odd part. Finally, if $\beta$ has last part $\beta_p \geq 2$, we write $\gcd(\beta_k, \beta_p - 2) = \gcd(\beta_1, \dots, \beta_p, \beta_p - 2)$, noting that this integer equals $\gcd(\beta_k)$ if $\beta_p = 2$.

With this notation established, we now state our results describing mod-sets in type~$B$.   \Cref{thm:SNFB} provides the Smith normal form of $(\Id - \wbg)$ in the type $B_n$ basis $\Delta^\vee$, from which the isomorphism type of the quotient $R^\vee/\Mod(\wbg)$ is immediate.  Note that even though the Cartan matrices in types $B_n$ and $C_n$ differ only by exchanging the $(n-1,n)$- and $(n,n-1)$-entries, the results on mod-sets are rather different; compare Theorems \ref{thm:SNFB} and \ref{thm:BasisB} in type $B$ below to their (more straightforward) type $C$ counterpart Theorem \ref{thm:Cn}.

\begin{thm}\label{thm:SNFB}
Suppose $\sW$ is of type $B_n$ with $n \geq 2$.  Let $(\beta,\gamma)$ be a pair of compositions such that $\beta = (\beta_1,\dots,\beta_p)$ is weakly decreasing, $\gamma = (\gamma_1,\dots,\gamma_q)$ is weakly increasing, and $|\beta| + |\gamma| = n$, with corresponding conjugacy class representative $\wbg \in \sW$. 
 
 Then for any $w \in [\wbg]$, the Smith normal form of $(\Id-w)$ is as follows:
 \begin{enumerate}
  \item\label{SNFB_gcd2_parity} If $\gcd(\beta_k,2) = 2$ (including $|\beta| = 0$), $|\gamma| \geq 1$, and all parts of $\gamma$ have the same parity, then 
  \[
 \Sbg = \diag(1^{n - q-p}, 2^{q},0^p).
 \]
   \item\label{SNFB_gcd2_nonparity} If $\gcd(\beta_k,2) = 2$ (including $|\beta| = 0$), $|\gamma| \geq 1$, and $\gamma$ has a change in parity, then 
  \[
 \Sbg = \diag(1^{n - q-p+1}, 2^{q-2},4,0^p).
 \]
  \item\label{SNFB_gcd1} If $\gcd(\beta_k,2) = 1$ and $|\gamma| \geq 1$, then 
  \[
 \Sbg = \diag(1^{n-q-p+1}, 2^{q-1}, 0^p).
 \]
 \item\label{SNFB_gamma0_2} If $|\gamma| = 0$, $\beta_p \geq 2$, and $\gcd(\beta_k,\beta_p - 2) \geq 2$, then 
 \[
 \Sbg = \diag(1^{n-p-1}, \gcd(\beta_k,\beta_p -2), 0^p).
 \]
  \item\label{SNFB_gamma0_1} If $|\gamma| = 0$, and either $\beta_p = 1$, or $\beta_p \geq 2$ and $\gcd(\beta_k, \beta_p - 2) = 1$, then 
 \[
 \Sbg = \diag(1^{n-p}, 0^p).
 \]
 \end{enumerate}
\end{thm}

We also give a basis for $\Mod(\wbg)$ in Theorem \ref{thm:BasisB} in the cases where either $|\beta| = 0$ or $|\gamma| = 0$. Note that $|\beta| = 0$ if and only if $\wbg = w_\gamma$ is cuspidal in $\sW$.

\begin{thm}\label{thm:BasisB}
Suppose $\sW$ is of type $B_n$ with $n \geq 2$.  Let $(\beta,\gamma)$ be a pair of compositions such that $\beta = (\beta_1,\dots,\beta_p)$ is weakly decreasing, $\gamma = (\gamma_1,\dots,\gamma_q)$ is weakly increasing, and $|\beta| + |\gamma| = n$, with corresponding conjugacy class representative $\wbg \in \sW$.  

If either $|\beta| = 0$ or $|\gamma| =0$, then the module $\Mod(\wbg)$ has a $\Z$-basis as follows: 

\begin{enumerate}

\item\label{B_beta0_gamma3} If $|\beta|=0$ and $\gamma_1 \geq 3$, then $\Mod(\wbg)$ has $\Z$-basis given by  
\[ \left\{ 2 \alpha_{n-j^\gamma_k}^\vee + \gamma_{k,k-1}\alpha_{n}^\vee \ \middle| \  2 \leq k \leq q \right\} \cup 
\left\{ 2\alpha_{n-1}^\vee,\; (\gamma_1\operatorname{mod} 2) \alpha_{n-1}^\vee + \alpha_n^\vee \right\} \]
\[ \cup \ \left\{ -\alpha_i^\vee + \alpha_{i+1}^\vee \ \middle| \ i \not \in I_\gamma \cup (I_\gamma- 1)  \right\}  \]
\[ \cup \  \left \{ -\alpha_{i}^\vee + \alpha_{i+2}^\vee \ \middle| \ i \in (I_\gamma-1),\; i \neq n-1 \right\}.\]

\item\label{B_beta0_gamma2} If $|\beta|=0$ and $\gamma_1 = 2$, then $\Mod(\wbg)$ has $\Z$-basis given by 
\[ \left\{ 2 \alpha_{n-j^\gamma_k}^\vee + \gamma_{k,k-1}\alpha_{n-1}^\vee \ \middle| \  2 \leq k \leq q \right\} \cup 
\left\{ 2\alpha_{n-1}^\vee,\; (\gamma_1\operatorname{mod} 2) \alpha_{n-1}^\vee + \alpha_n^\vee \right\} \]
\[ \cup \ \left\{ -\alpha_i^\vee + \alpha_{i+1}^\vee \ \middle| \ i \not \in I_\gamma \cup (I_\gamma- 1)  \right\}  \]
\[ \cup \  \left \{ -\alpha_{i}^\vee + \alpha_{i+2}^\vee \ \middle| \ i \in (I_\gamma-1),\; i \neq n-1 \right\}.\]

\item\label{B_beta0_gamma12}  If $|\beta| = 0$, $\gamma_1 = 1$, and $\gamma_2 \geq 2$, then $\Mod(\wbg)$ has $\Z$-basis given by 
\[ \left\{ 2 \alpha_{n-j^\gamma_k}^\vee + \gamma_{k,k-1}\alpha_{n}^\vee \ \middle| \  1 \leq k \leq q\right\} \]
\[ \cup \  \left\{ -\alpha_i^\vee + \alpha_{i+1}^\vee \ \middle| \  i \not \in I_\gamma \cup (I_\gamma - 1)   \right\}  \]
\[ \cup \  \left\{ -\alpha_{i}^\vee + \alpha_{i+2}^\vee \ \middle|\  i \in (I_\gamma-1), \; i \neq n-1 \right\}.\]

\item\label{B_beta0_gamma11}  If $|\beta| = 0$ and $\gamma_k = 1$ for all $1 \leq k \leq \ell < n$ for some $\ell \geq 2$, while $\gamma_{\ell + 1} \geq 2$, then $\Mod(\wbg)$ has $\Z$-basis given by 
\[ \left\{ 2 \alpha_{n-j^\gamma_k}^\vee + \gamma_{k,k-1}\alpha_{n}^\vee \ \middle| \  1 \leq k \leq q\right\} \cup \ \{ -\alpha_{n-\ell-1}^\vee + \alpha_n^\vee \} \]
\[ \cup \  \left\{ -\alpha_i^\vee + \alpha_{i+1}^\vee \ \middle| \  i \not \in I_\gamma \cup (I_\gamma - 1) \right\}  \]
\[ \cup \  \left\{ -\alpha_{i}^\vee + \alpha_{i+2}^\vee \ \middle|\  i \in (I_\gamma-1), \; i < n-\ell - 1 \right\}.\]

\item\label{B_beta0_gamma1s} If $|\beta| = 0$ and $\gamma = (1,\dots,1)$, then $\Mod(\wbg)$ has $\Z$-basis given by \[ \{ 2\alpha_i^\vee \mid i \in [n] \}.\]

\item\label{B_gamma0_2} If $|\gamma|=0$, $\beta_p \geq 2$, and $\gcd(\beta_k, \beta_p - 2) \geq 2$, then $\Mod(\wbg)$ has $\Z$-basis given by
\[ \left\{ \alpha_i^\vee - \alpha_{i+1}^\vee \ \middle|\ i \in J_{\beta} \backslash (I_{\beta} - 1) \right\} \ \cup \] 
\[ \left\{ \alpha_i^\vee - \alpha_{i+2}^\vee \ \middle| \ i \in (I_{\beta} - 1) \backslash \{n-1\} \right\}\ \cup \ \left\{ \gcd(\beta_k,\beta_p - 2)\alpha_{n-1}^\vee \right\}.\]

\item\label{B_gamma0_1} If $|\gamma|=0$, and either $\beta_p = 1$, or $\beta_p \geq 2$ and $\gcd(\beta_k, \beta_p - 2) = 1$,  then $\Mod(\wbg)$ has $\Z$-basis given by
\[
\{ \alpha_j^\vee \mid j \in J_\beta \}.
\]

\end{enumerate}
\end{thm}

As seen above, the two special cases considered in Theorem~\ref{thm:BasisB} already involve several subcases and delicate statements. Thus for the remaining ``mixed" case, where $|\beta| \geq 1$ and $|\gamma| \geq 1$, we have opted not to give even lengthier general statements, but instead to provide illustrative examples of how to identify convenient bases. In all cases, a set of equations explicitly describing $\Mod(\wbg)$ can easily be deduced once the (relatively sparse) basis is known.

\subsection{Examples in type $B$}\label{sec:examplesB}

In this section, we illustrate both Theorems~\ref{thm:SNFB} and~\ref{thm:BasisB} with a sequence of examples, which can then be generalized in a straightforward manner to prove these results, as well as to find a basis in all cases. In Section~\ref{sec:examplesBC}, we consider the same pairs of compositions as in Section~\ref{sec:examplesC}:
\begin{itemize}
\item $|\beta| = 0$ and $\gamma = (4)$  in Example~\ref{eg:beta0gamma4B4};
\item $|\beta| = 0$ and $\gamma = (3,4)$  in Example~\ref{eg:beta0gamma34B7};
\item $|\beta| = 0$ and $\gamma = (2,3,4)$  in Example~\ref{eg:beta0gamma234B9};
\item $|\beta| = 0$ and $\gamma = (1,3)$  in Example~\ref{eg:beta0gamma13B4}; 
\item $\beta = (1)$ and $\gamma = (3)$  in Example~\ref{eg:beta1gamma3B4}; 
\item $\beta = (4)$ and $\gamma = (3)$  in Example~\ref{eg:beta4gamma3B7}; and
\item $\beta = (3,3)$ and $\gamma = (3)$  in Example~\ref{eg:beta33gamma3B9}.  
\end{itemize}
In Section~\ref{sec:examplesBextra}, we then also consider:
\begin{itemize}
\item $|\beta| = 0$ and $\gamma = (1,\dots,1)$ in Example~\ref{eg:beta0gamma1s}; 
\item $|\beta| = 0$ and $\gamma = (1,1,1,4)$ in Example~\ref{eg:beta0gamma1114B7};  
\item $\beta = (1)$ and $\gamma = (3,4)$ in Example~\ref{eg:beta1gamma34B8};
\item $\beta = (5)$ and $\gamma = (3)$ in Example~\ref{eg:beta5gamma3B8};
\item $\beta = (4)$ and $|\gamma| = 0$ in Example~\ref{eg:beta4gamma0B4}; 
\item $\beta = (3,3)$ and $|\gamma| = 0$ in Example~\ref{eg:beta33gamma0B6}; and 
\item $\beta = (4,1)$ and $|\gamma| = 0$ in Example~\ref{eg:beta41gamma0B5}.
\end{itemize}
For each of these examples, we construct a $\Z$-basis for $\Mod(\wbg)$ and find the Smith normal form for $(\Id - \wbg)$ in type $B$.  

\subsubsection{Examples from type $C$, reconsidered in type $B$}\label{sec:examplesBC}

In this section, we illustrate our results in type $B$ by reconsidering the same examples as in Section~\ref{sec:examplesC}. 

\begin{example}\label{eg:beta0gamma4B4}  This example illustrates part~\eqref{SNFB_gcd2_parity} of Theorem~\ref{thm:SNFB} and part~\eqref{B_beta0_gamma3} of Theorem~\ref{thm:BasisB}. Let $\sW$ be of type $B_4$ and suppose $|\beta| = 0$ and $\gamma = (4)$, so that $\wbg = w_\gamma = s_4 s_3 s_2 s_1$ is Coxeter in $\sW$.  The matrices for $w_\gamma$ and $N_\gamma = \Id - w_\gamma$ with respect to the type $B_4$ basis $\Delta^\vee$ are given by
\[
w_\gamma = \begin{pmatrix}
-1 & 1 &  0 & 0\\ 
-1 & 0 & 1 & 0\\ 
-1 & 0 & 0 & \circled{$2$} \\
\circled{$-1$} & 0 & 0 & 1
\end{pmatrix}
\quad \mbox{and} \quad
N_\gamma = \begin{pmatrix}
2 & -1 &  0 & 0\\ 
1 & 1 & -1 & 0\\ 
1 & 0 & 1 & -2 \\
1 & 0 & 0 & 0
\end{pmatrix}.
\]
Note that the matrix for $w_\gamma$ differs from that of Example \ref{eg:beta0gamma4C4} in only the circled entries, due to the slight difference in the Cartan matrices.

To determine $\Mod(w_\gamma)$, we use similar column operations to those we used to obtain $B_\beta$ from $M_\beta$ in Example \ref{eg:CoxeterA3}. In particular, replace $\cC_1(N_\gamma)$ by $\sum_{i=1}^3 i \cC_i(N_\gamma)$ to give us $B'_\gamma$:
\[
B'_\gamma = \begin{pmatrix}
0 & -1 &  0 & 0\\ 
0 & 1 & -1 & 0\\ 
\circled{4} & 0 & 1 & \boxed{-2} \\
1 & 0 & 0 & 0
\end{pmatrix} 
\quad \to \quad
B_\gamma = \begin{pmatrix}
0 & -1 &  0 & 0\\ 
0 & 1 & -1 & 0\\ 
0 & 0 & 1 & -2 \\
1 & 0 & 0 & 0
\end{pmatrix}. 
\]
Since $\gcd(\beta_k, 2) = 2$ in this case, we can use the boxed $-2$ in column 4 to clear out the circled entry in column 1, giving us $B_\gamma$. The corresponding $\Z$-basis for $\Mod(\wbg)$ is then given by $\{ -\alpha_1^\vee+\alpha_2^\vee, -\alpha_2^\vee + \alpha_3^\vee, 2\alpha_3^\vee, \alpha_4^\vee \}$, matching part~\eqref{B_beta0_gamma3} of Theorem~\ref{thm:BasisB}.

To determine the Smith normal form $S_\gamma$ for $N_\gamma$, it is clear that following several row operations on $B_\gamma$, we obtain 
\[
T_\gamma = \begin{pmatrix}
0 & -1 &  0 & 0\\ 
0 & 0 & -1 & 0\\ 
0 & 0 & 0 & -2 \\
1 & 0 & 0 & 0
\end{pmatrix}.
\]
Hence $S_\gamma = \diag(1,1,1,2) $, and $R^\vee /\Mod(w_\gamma) \cong \Z/2\Z$, confirming part~\eqref{SNFB_gcd2_parity} of Theorem~\ref{thm:SNFB}.
\end{example}

\begin{example}\label{eg:beta0gamma34B7}  
This example illustrates part \eqref{SNFB_gcd2_nonparity}  of Theorem~\ref{thm:SNFB} and part~\eqref{B_beta0_gamma3} of Theorem~\ref{thm:BasisB}, in the case that $\gamma$ has more than one part. 
Let $\sW$ be of type $B_7$ and suppose $|\beta| = 0$ and $\gamma = (3,4)$, so that $\wbg = w_\gamma = w^\gamma_1 w^\gamma_2 = (s_7 s_6 s_5)(s_4 s_5 s_6 s_7 s_6 s_5 s_4 s_3 s_2 s_1)$.  The matrices for $w_\gamma$ and $N_\gamma = \Id-w_\gamma$ with respect to the type $B_7$ basis $\Delta^\vee$ are given by
\[
w_\gamma = 
\begin{pmatrix}
-1 & 1 & 0 & 0 &  &  &  \\
-1 & 0 & 1 & 0 &  &  &  \\
-1 & 0 & 0 & 1 &  &  &  \\
-2 & 0 & 0 & 1 &  &  &  \\
-2 & 0 & 0 & 1 & -1 & 1 & 0 \\
-2 & 0 & 0 & 1 & -1 & 0 & \circled{2} \\
\circled{$-1$} & 0 & 0 & \circled{1} & \circled{$-1$} & 0 & 1 \\
\end{pmatrix},
\quad
N_\gamma = 
\begin{pmatrix}
2 & -1 & 0 & 0 &  &  &  \\
1 & 1 & -1 & 0 &  &  &  \\
1 & 0 & 1 & -1 &  &  &  \\
2 & 0 & 0 & 0 &  &  &  \\
2 & 0 & 0 & -1 & 2 & -1 & 0 \\
2 & 0 & 0 & -1 & 1 & 1 & -2 \\
1 & 0 & 0 & -1 & 1 & 0 & 0 \\
\end{pmatrix}.
\]
Note that the matrix for $w_\gamma$ differs from that of Example \ref{eg:beta0gamma34C7} in only the circled entries, due to the slight difference in the Cartan matrices. 

For $\Mod(w_\gamma)$, we carry out column operations on $N_\gamma$ similar to those in the previous example within blocks. First add column 5 to column 4.  We then replace $\cC_1(N_\gamma)$ by $\cC_1(N_\gamma) + 2\cC_2(N_\gamma) + 3\cC_3(N_\gamma)$ and then replace $\cC_5(N_\gamma)$ by $\cC_5(N_\gamma) + 2\cC_6(N_\gamma)$ to give us $B_\gamma'$:
\[
B'_\gamma = 
\begin{pmatrix}
0 & -1 & 0 & 0 &  &  &  \\
0 & 1 & -1 & 0 &  &  &  \\
\circled{4} & 0 & 1 & -1 &  &  &  \\
2 & 0 & 0 & 0 &  &  &  \\
\circled{2} & 0 & 0 & 1 & 0 & -1 & 0 \\
2 & 0 & 0 & 0 & 3 & 1 & -2 \\
1 & 0 & 0 & 0 & 1 & 0 & 0 \\
\end{pmatrix}
\ \to \ 
B''_\gamma = 
\begin{pmatrix}
0 & -1 & 0 & 0 &  &  &  \\
0 & 1 & -1 & 0 &  &  &  \\
0 & 0 & 1 & -1 &  &  &  \\
2 & 0 & 0 & 0 &  &  &  \\
0 & 0 & 0 & 1 & 0 & -1 & 0 \\
\circled{8} & 0 & 0 & 0 & \circled{3} & 1 & \boxed{-2} \\
1 & 0 & 0 & 0 & 1 & 0 & 0 \\
\end{pmatrix}.
\]
Using columns 4 and 6 to clear the circled entries in column 1, we obtain $B''_\gamma$ above. Using the boxed pivot $-2$ in the last column, we can then reduce all other entries in row $n-1$ modulo $2$ to obtain $B_\gamma$
\[ 
B_\gamma = 
\begin{pmatrix}
0 & -1 & 0 & 0 &  &  &  \\
0 & 1 & -1 & 0 &  &  &  \\
0 & 0 & 1 & -1 &  &  &  \\
2 & 0 & 0 & 0 &  &  &  \\
0 & 0 & 0 & 1 & 0 & -1 & 0 \\
0 & 0 & 0 & 0 & 1 & 1 & -2 \\
1 & 0 & 0 & 0 & 1 & 0 & 0 \\
\end{pmatrix}.
\]
It may be checked that the $(n-1,1)$-entry of $B_\gamma$ equals $\gamma_2 (\operatorname{mod} 2)$, and the $(n-1,5)$-entry of $B_\gamma$ equals $\gamma_1 (\operatorname{mod} 2)$, so that these entries differ if and only if $\gamma_{2,1} = 1$. From here, the $\Z$-basis $ \{ 2\alpha_4^\vee + \alpha_7^\vee\} \cup \{ 2\alpha_{6}^\vee, \alpha_{6}^\vee + \alpha_7^\vee\} \cup \{ -\alpha_1^\vee + \alpha_2^\vee, -\alpha_2^\vee +\alpha_3^\vee, -\alpha_5^\vee+ \alpha_6^\vee \} \cup  \{- \alpha_3^\vee + \alpha_5^\vee \}$
can be read off, matching part~\eqref{B_beta0_gamma3} of Theorem~\ref{thm:BasisB}.

For the Smith normal form, continuing from $B_\gamma$, we can add rows 1 to 2, 2 to 3, 3 to 5, and 5 to 6 successively to obtain $T'_\gamma$:
\[ 
T'_\gamma = 
\begin{pmatrix}
0 & -1 & 0 & 0 &  &  &  \\
0 & 0 & -1 & 0 &  &  &  \\
0 & 0 & 0 & -1 &  &  &  \\
2 & 0 & 0 & 0 &  &  &  \\
0 & 0 & 0 & 0 & 0 & -1 & 0 \\
0 & 0 & 0 & 0 & 1 & 0 & -2 \\
1 & 0 & 0 & 0 & 1 & 0 & 0 \\
\end{pmatrix}.
\]
Extracting the $3\times 3$ submatrix containing the nonzero entries of columns 1, 5, and 7, we then have
\[ 
\begin{pmatrix}
2 & 0 & 0 \\
0 & 1 & -2 \\
1 & 1 & 0
\end{pmatrix}
\to 
\begin{pmatrix}
0 & -2 & 0 \\
0 & 1 & -2 \\
1 & 1 & 0
\end{pmatrix}
\to 
\begin{pmatrix}
0 & 0 & -4 \\
0 & 1 & -2 \\
1 & 0 & 0
\end{pmatrix}
\to 
\begin{pmatrix}
0 & 0 & -4 \\
0 & 1 & 0 \\
1 & 0 & 0
\end{pmatrix}.
\]
Inserting these columns back into $T'_\gamma$ then gives us the matrix $T_\gamma$:
\[ 
T_\gamma = 
\begin{pmatrix}
0 & -1 & 0 & 0 &  &  &  \\
0 & 0 & -1 & 0 &  &  &  \\
0 & 0 & 0 & -1 &  &  &  \\
0 & 0 & 0 & 0 & 0 & 0 & -4 \\
0 & 0 & 0 & 0 & 0 & -1 & 0 \\
0 & 0 & 0 & 0 & 1 & 0 & 0 \\
1 & 0 & 0 & 0 & 0 & 0 & 0 \\
\end{pmatrix}.
\]
Thus $S_\gamma = \diag(1^5, 4)$, and so $R^\vee/\Mod(w_\gamma) \cong (\Z/4\Z)$, confirming part \eqref{SNFB_gcd2_nonparity}  of Theorem~\ref{thm:SNFB}.
\end{example}

\begin{example}\label{eg:beta0gamma234B9}  
This example illustrates part \eqref{SNFB_gcd2_nonparity} of Theorem~\ref{thm:SNFB} and part~\eqref{B_beta0_gamma2} of Theorem~\ref{thm:BasisB}. 
Let $\sW$ be of type $B_9$ and suppose $|\beta| = 0$ and $\gamma = (2,3,4)$, so that $\wbg = w_\gamma = w^\gamma_1 w^\gamma_2 w^\gamma_3 = (s_9 s_8)(s_7 s_8 s_9 s_8s_7s_6 s_5)(s_4 \cdots s_8 s_9 s_8 \cdots s_4 s_3s_2  s_1)$.  The matrices for $w_\gamma$ and $N_\gamma = \Id - w_\gamma$ with respect to the type $B_9$ basis $\Delta^\vee$ are given by
\[
w_\gamma = 
\begin{pmatrix}
-1 & 1 & 0 & 0 & &&&&     \\
-1 & 0 & 1 & 0 &  &&&&    \\
-1 & 0 & 0 & 1 &   &&&&   \\
-2 & 0 & 0 & 1 &    &&&&  \\
-2&0&0&1 & -1 & 1 & 0 &      \\
-2&0&0&1 & -1 & 0 & 1 &      \\
-2&0&0&2 & -2 & 0 & 1 &      \\
-2&0&0&2 & -2 & 0 & 1 &  -1 & \circled{2}    \\
\circled{$-1$}&0&0&\circled{1} & \circled{$-1$} & 0 & \circled{1} & \circled{$-1$} & 1
\end{pmatrix},
\]
\[
N_\gamma = 
\begin{pmatrix}
2 & -1 & 0 & 0 & &&&&     \\
1 & 1 & -1 & 0 &  &&&&    \\
1 & 0 & 1 & -1 &   &&&&   \\
2 & 0 & 0 & 0 &    &&&&  \\
2&0&0&-1 & 2 & -1 & 0 &      \\
2&0&0&-1 & 1 & 1 & -1 &      \\
2&0&0&-2 & 2 & 0 & 0 &      \\
2&0&0&-2 & 2 & 0 & -1 &  2 & -2    \\
1&0&0&-1 & 1 & 0 & -1 & 1 & 0
\end{pmatrix}.
\]
Note that the matrix for $w_\gamma$ differs from that of Example \ref{eg:beta0gamma234C9}  in only the circled entries, due to the slight difference in the Cartan matrices.

For $\Mod(w_\gamma)$, we carry out column operations on $N_\gamma$ similar to those in the previous two examples.  Start by adding columns 5 and 8 to columns 4 and 7, respectively. Clearing entries within the first columns of the blocks on the diagonal, we obtain $B'_\gamma$:
\[
B'_\gamma = 
\begin{pmatrix}
0 & -1 & 0 & 0 & &&&&     \\
0 & 1 & -1 & 0 &  &&&&    \\
\circled{4} & 0 & 1 & -1 &   &&&&   \\
2 & 0 & 0 & 0 &    &&&&  \\
\circled{2}&0&0&1 & 0 & -1 & 0 &      \\
\circled{2}&0&0&0 & \circled{3} & 1 & -1 &      \\
2&0&0&0 & 2 & 0 & 0 &      \\
2&0&0&0 & 2 & 0 & 1 &  0 & -2    \\
1&0&0&0 & 1 & 0 & 0 & \boxed{1} & 0
\end{pmatrix}.
\]
Notice that we now have a boxed pivot entry $1$ in the last row, since $\delta_1 = 2$. We use this to clear the last row, and also clear the circled entries using columns to their right with first entry $-1$, as in the previous examples, and we get
\[
B''_\gamma = 
\begin{pmatrix}
0 & -1 & 0 & 0 & &&&&     \\
0 & 1 & -1 & 0 &  &&&&    \\
0 & 0 & 1 & -1 &   &&&&   \\
2 & 0 & 0 & 0 &    &&&&  \\
0&0&0&1 & 0 & -1 & 0 &      \\
0&0&0&0 & 0 & 1 & -1 &      \\
\circled{2}&0&0& 0 & 2 & 0 & 0 & &    \\
10&0&0&0 & 5 & 0 & 1 &  0 & -2    \\
0&0&0&0 & 0 & 0 & 0 & 1 & 0
\end{pmatrix}.
\]
We then clear the circled entry using column $5$, and carry out several final column operations to obtain $B_\gamma$:
\[
B_\gamma = 
\begin{pmatrix}
0 & -1 & 0 & 0 & &&&&     \\
0 & 1 & -1 & 0 &  &&&&    \\
0 & 0 & 1 & -1 &   &&&&   \\
2 & 0 & 0 & 0 &    &&&&  \\
0&0&0&1 & 0 & -1 & 0 &      \\
0&0&0&0 & 0 & 1 & -1 &      \\
0&0&0& 0 & 2 & 0 & 0 & &    \\
1&0&0&0 & 1 & 0 & 1 &  0 & 2    \\
0&0&0&0 & 0 & 0 & 0 & 1 & 0
\end{pmatrix},
\]
from which the $\Z$-basis $ \{ 2\alpha_4^\vee + \alpha_8^\vee, 2\alpha_7^\vee + \alpha_8^\vee\} \cup \{ 2\alpha_{8}^\vee, \alpha_{9}^\vee \} \cup \{ -\alpha_1^\vee + \alpha_2^\vee, -\alpha_2^\vee +\alpha_3^\vee, -\alpha_5^\vee+ \alpha_6^\vee \} \cup  \{- \alpha_3^\vee + \alpha_5^\vee, - \alpha_6^\vee + \alpha_8^\vee \}$
can be read off, matching part~\eqref{B_beta0_gamma2} of Theorem~\ref{thm:BasisB}.

For the Smith normal form, continuing from $B_\gamma$, we can add rows 1 to 2, 2 to 3, 3 to 5, 5 to 6, and 6 to 8 successively to obtain $T'_\gamma$:
\[
T'_\gamma = 
\begin{pmatrix}
0 & -1 & 0 & 0 & &&&&     \\
0 & 0 & -1 & 0 &  &&&&    \\
0 & 0 & 0 & -1 &   &&&&   \\
2 & 0 & 0 & 0 &    &&&&  \\
0&0&0&0 & 0 & -1 & 0 &      \\
0&0&0&0 & 0 & 0 & -1 &      \\
0&0&0& 0 & 2 & 0 & 0 & &    \\
1&0&0&0 & 1 & 0 & 0 &  0 & -2    \\
0&0&0&0 & 0 & 0 & 0 & 1 & 0
\end{pmatrix},
\]
Extracting the $3 \times 3$ submatrix containing the nonzero entries of columns 1, 5, and 8, that is, the columns of $T'_\gamma$ which have $2$s, we have
\[ 
\begin{pmatrix}
2 & 0 & 0 \\
0 & 2 & 0 \\
1 & 1 & 2
\end{pmatrix}
\to 
\begin{pmatrix}
0 & -2 & -4 \\
0 & 2 & 0 \\
1 & 1 & 2
\end{pmatrix}
\to 
\begin{pmatrix}
0 & 0 & 4 \\
0 & 2 & 0 \\
1 & 0 & 0
\end{pmatrix}
\]
Inserting these columns back into $T'_\gamma$ then gives us the matrix $T_\gamma$:
\[
T_\gamma = 
\begin{pmatrix}
0 & -1 & 0 & 0 & &&&&     \\
0 & 0 & -1 & 0 &  &&&&    \\
0 & 0 & 0 & -1 &   &&&&   \\
0 & 0 & 0 & 0 &  0 &0&0&0&4  \\
0&0&0&0 & 0 & -1 & 0 &  0 & 0    \\
0&0&0&0 & 0 & 0 & -1 &    0 & 0  \\
0&0&0& 0 & 2 & 0 & 0 &0 &  0   \\
1&0&0&0 & 0 & 0 & 0 &  0 & 0    \\
0&0&0&0 & 0 & 0 & 0 & 1 & 0
\end{pmatrix}.
\]
Thus $S_\gamma = \diag(1^6, 2, 4)$, and so $R^\vee/\Mod(w_\gamma) \cong (\Z/2\Z) \oplus (\Z/4\Z)$, confirming part \eqref{SNFB_gcd2_nonparity} of Theorem~\ref{thm:SNFB}.
\end{example}

\begin{example}\label{eg:beta0gamma13B4}  
This example illustrates part \eqref{SNFB_gcd2_parity} of Theorem~\ref{thm:SNFB} and part~\eqref{B_beta0_gamma12} of Theorem~\ref{thm:BasisB}. 
Let $\sW$ be of type $B_4$ and suppose $|\beta| = 0$ and $\gamma = (1,3)$, so that $\wbg = w_\gamma = w_1^\gamma w_2^\gamma = (s_4)(s_3 s_4 s_3 s_2 s_1)$. The matrices for $w_\gamma$ and $N_\gamma = \Id - w_\gamma$ with respect to the type $B_4$ basis $\Delta^\vee$ are given by
\[
w_\gamma = 
\begin{pmatrix} 
-1 &1&0& 0\\
-1 & 0 &1&0 \\
-2&0&1&0 \\
\circled{$-1$}&0&\circled{1}&-1
\end{pmatrix},
\quad\quad
N_\gamma = 
\begin{pmatrix} 
2 &-1&0& 0\\
1 & 1 &-1&0 \\
2&0&0& 0\\
1&0&-1&2
\end{pmatrix}.
\]
Note that the matrix for $w_\gamma$ differs from that of Example \ref{eg:beta0gamma13C4} in only the circled entries, due to the slight difference in the Cartan matrices.

Proceeding using similar column operations to the previous examples, we obtain 
\[
B_\gamma = 
\begin{pmatrix} 
0 &-1&0& 0\\
0 & 1 &-1&0 \\
2&0&0& 0\\
0&0&1&2
\end{pmatrix},
\]
from which the $\Z$-basis $ \{ 2\alpha_3^\vee, 2\alpha_4^\vee\} \cup \{ -\alpha_1^\vee + \alpha_2^\vee \} \cup \{ -\alpha_2^\vee + \alpha_4^\vee\}$ can be read off. Here, we have $I_\gamma = \{3,4\}$ so $I_{\gamma} - 1 = \{ 2,3\}$, and this matches part~\eqref{B_beta0_gamma12} of Theorem~\ref{thm:BasisB}. 

To obtain the Smith normal form, continue by doing successive row operations on $B_\gamma$ as in the previous examples, to yield 
\[
T_\gamma = 
\begin{pmatrix} 
0 &-1&0& 0\\
0 & 0 &-1&0 \\
2&0&0& 0\\
0&0&0&2
\end{pmatrix},
\]
Thus $S_\gamma = \diag(1,1,2,2)$, and so $R^\vee / \Mod(w_\gamma) \cong (\Z/2\Z)^2$, confirming part \eqref{SNFB_gcd2_parity} of Theorem~\ref{thm:SNFB}.
\end{example}

Our final examples in this section illustrate the case when the compositions $\beta$ and $\gamma$ are both nonempty.

\begin{example}\label{eg:beta1gamma3B4}   This example illustrates part \eqref{SNFB_gcd1} of Theorem~\ref{thm:SNFB}, and shows how to obtain a basis for $\Mod(\wbg)$. 
Let $\sW$ be of type $B_4$ and suppose $\beta = (1)$ and $\gamma = (3)$, so that $\wbg = w_\gamma = s_4 s_3 s_2$. The matrices for $\wbg$ and $\Mbg = \Id - \wbg$ with respect to the type $B_4$ basis $\Delta^\vee$ are given by
\[
\wbg = 
\begin{pmatrix}
1&0&0&0\\
 1 & -1 & 1 & 0 \\
 1 & -1 & 0 & \circled{2} \\
  \circled{1} & \circled{$-1$} & 0 & 1 \\
\end{pmatrix} \quad \mbox{and} \quad
\Mbg = 
\begin{pmatrix}
0&0&0&0\\
 -1 & 2 & -1 & 0 \\
 -1 & 1 & 1 & -2 \\
  -1 & 1 & 0 & 0 \\
\end{pmatrix}. 
\]
Note that the matrix for $\wbg$ differs from that of Example \ref{eg:beta1gamma3C4} in only the circled entries, due to the slight difference in the Cartan matrices.

To find $\Mod(\wbg)$, we now  add column 2 to column 1 of $\Mbg$, to obtain $\Bbg'$:
\[
\Bbg' = 
\begin{pmatrix}
0&0&0&0\\
 \boxed{1} & 2 & -1 & 0 \\
 0 & 1 & 1 & -2 \\
 0 & 1 & 0 & 0 \\
\end{pmatrix}. 
\]
We then use the boxed pivot entry $1$ in the first column of $\Bbg'$ to clear the remainder of row $2$, and then use the resulting pivot entry $1$ in column $3$ to clear the remainder of row $3$, to obtain $\Bbg$:
\[
\Bbg = 
\begin{pmatrix}
0&0&0&0\\
 \boxed{1} & 0 & 0 & 0 \\
 0 & 0 & \boxed{1} & 0 \\
 0 & \boxed{1} & 0 & 0 \\
\end{pmatrix}. 
\]
The resulting $\Z$-basis is $\{ \alpha_2^\vee, \alpha_3^\vee, \alpha_4^\vee \}$.
We also have $\Sbg = \diag(1^3, 0)$, and so $R^\vee/\Mod(\wbg) \cong \Z$, confirming part \eqref{SNFB_gcd1} of Theorem~\ref{thm:SNFB}.
\end{example}

\begin{example}\label{eg:beta4gamma3B7}   This example illustrates part \eqref{SNFB_gcd2_parity} of Theorem~\ref{thm:SNFB} in the case $|\beta| \neq 0$, and shows how to obtain a basis for $\Mod(\wbg)$. 
Let $\sW$ be of type $B_7$ and suppose $\beta = (4)$ and $\gamma = (3)$ so that $\wbg = w_\beta w_\gamma = (s_1 s_2 s_3)(s_7 s_6 s_5)$. The matrices for $\wbg$ and $\Mbg = \Id - \wbg$ with respect to the type $B_7$ basis $\Delta^\vee$ are given by
\[
\wbg = 
\begin{pmatrix}
0&0&-1&1& \\
1&0&-1&1& \\
0&1&-1&1& \\
0&0&0&1&0&0&0\\
&&& 1 & -1 & 1 & 0 \\
&& & 1 & -1 & 0 & \circled{2} \\
&&& \circled{1} & \circled{$-1$} & 0 & 1 \\
\end{pmatrix} \quad
\Mbg = 
\begin{pmatrix}
1&0&1&-1& \\
-1&1&1&-1& \\
0&-1&2&-1& \\
0&0&0&0&0&0&0\\
&&& -1 & 2 & -1 & 0 \\
&& & -1 & 1 & 1 & -2 \\
&&& -1 & 1 & 0 & 0 \\
\end{pmatrix}. 
\]
Note that the matrix for $\wbg$ differs from that of Example \ref{eg:beta4gamma3C7} in only the circled entries, due to the slight difference in the Cartan matrices.

We start by adding columns 3 and 5 to column 4, and then proceed with column operations as in previous type $A$ examples, to obtain 
\[ 
B'_{\beta,\gamma} = 
\begin{pmatrix}
1&0&0&0& \\
-1&1&0&0& \\
0&-1&4&1& \\
0&0&0&0&0&0&0\\
&&& 1 & 2 & -1 & 0 \\
&& & 0 & 1 & 1 & -2 \\
&&& 0 & 1 & 0 & 0 \\
\end{pmatrix}
\quad \to \quad
B''_{\beta,\gamma} = 
\begin{pmatrix}
1&0&0&0& \\
-1&1&0&0& \\
0&-1&0&1& \\
0&0&0&0&0&0&0\\
&& 0 & 1 & 2 & -1 & 0 \\
&&-4 & 0 & 1 & 1 & -2 \\
&&& 0 & 1 & 0 & 0 \\
\end{pmatrix}.
\]
Using columns 4 and 6, we can move the $\beta_k = 4$ down into the row with the $-2$ in the last column, giving $B''_{\beta,\gamma}$. At this point, we replace the $-4$ by $\gcd(\beta_k,2) \in \{1,2\}$.  

Note that the lower right-hand $3 \times 3$ matrix is the same as for $\gamma = (3)$ in type $B_3$, so that the basis coming from the last 3 columns is the same as if we were working in type $B_3$.  Finally, we use $\gcd(\beta_k,2)$ to clear the final column, removing $2\alpha_6^\vee$ from the basis for the lower right $3\times 3$ minor.  The resulting $\Z$-basis from the matrix $B_{\beta,\gamma}$ below
\[
B_{\beta,\gamma} = 
\begin{pmatrix}
1&0&0&0& \\
-1&1&0&0& \\
0&-1&0&1& \\
0&0&0&0&0&0&0\\
&&0& 1 & 0 & -1 & 0 \\
&&2 & 0 & 1 & 1 & 0 \\
&&& 0 & 1 & 0 & 0 \\
\end{pmatrix}.
\]
is read off as $\{ \alpha_1^\vee - \alpha_2^\vee, \alpha_2^\vee - \alpha_3^\vee, \alpha_5^\vee - \alpha_6^\vee\} \cup \{ \alpha_3^\vee + \alpha_5^\vee\} \cup \{ \gcd(\beta_k,2)\alpha_6^\vee\} \cup  \{ \alpha_6^\vee + \alpha_7^\vee\}$.

To find the Smith normal form, continue by doing successive row operations on $B_{\beta,\gamma}$ as in previous examples, to yield
\[
T_{\beta,\gamma} = 
\begin{pmatrix}
1&0&0&0& \\
0&1&0&0& \\
0&0&0&1& \\
0&0&0&0&0&0&0\\
&&0& 0 & 0 & -1 & 0 \\
&&2 & 0 & 0 & 0 & 0 \\
&&& 0 & 1 & 0 & 0 \\
\end{pmatrix}.
\]
Therefore $\Sbg = \diag(1^5, 2,0)$, and so $R^\vee/\Mod(\wbg) \cong (\Z/2\Z) \oplus \Z$, confirming part \eqref{SNFB_gcd2_parity} of Theorem~\ref{thm:SNFB}.
\end{example}

\begin{example}\label{eg:beta33gamma3B9}  
This example illustrates part \eqref{SNFB_gcd1} of Theorem~\ref{thm:SNFB}, and shows how to obtain a basis for $\Mod(\wbg)$. 
Let $\sW$ be of type $B_9$ and suppose $\beta = (3,3)$ and $\gamma = (3)$.  The matrix $\Mbg = \Id - \wbg$ with respect to the type $B_9$ basis $\Delta^\vee$ is given by
\[
\Mbg = 
\begin{pmatrix}
1 & 1 & -1 & & & & && \\
-1 & 2 & -1 & & & & &&\\
0 & 0 & 0 & 0 & 0 & 0 && \\
& & -1&1&1&-1&&& \\
 &  & &-1&2&-1& &&\\
 & & &0&0&0& 0 &0&0\\
&&&&& -1 & 2 & -1 & 0 \\
&&&& & -1 & 1 & 1 & -2 \\
&&&&& -1 & 1 & 0 & 0 \\
\end{pmatrix}. 
\]
We add columns 5 and 7 to column 6, carry out column operations on the first 5 columns as done in type $A$, and then move $\gcd(\beta_k) = 3$ down to row $n-1$, giving us $\Bbg'$:
\[
\Bbg' = 
\begin{pmatrix}
1 & 0 & 0 & & & & && \\
-1 & 0 & 1 & & & & &&\\
0 & 0 & 0 & 0 & 0 & 0 && \\
& & -1&1&0&0&&& \\
 &  & &-1&0&1& &&\\
 & & &0&0&0& 0 &0&0\\
&&&&0& 1 & 2 & -1 & 0 \\
&&&& -3 & 0 & 1 & 1 & -2 \\
&&&&& 0 & 1 & 0 & 0 \\
\end{pmatrix}. 
\]
Now since $\gcd(\beta_k,2) = 1$ we get a pivot $1$ in entry $(n-1,m-1)$. We can use this pivot to successively clear entries in the last $\gamma_1 = 3$ rows, and then clear entries in successive rows, using only column operations, and moving up the matrix, resulting in $\Bbg$:
\[
\Bbg= 
\begin{pmatrix}
\boxed{1} & 0 & 0 & & & & && \\
0 & 0 & \boxed{1} & & & & &&\\
0 & 0 & 0 & 0 & 0 & 0 && \\
& & 0&\boxed{1}&0&0&&& \\
 &  & &0&0&\boxed{1}& &&\\
 & & &0&0&0& 0 &0&0\\
&&&&0& 0 & 0 & \boxed{1} & 0 \\
&&&& \boxed{1} & 0 & 0 & 0 & 0 \\
&&&&& 0 & \boxed{1} & 0 & 0 \\
\end{pmatrix}. 
\]
Therefore $\Sbg = \diag(1^7, 0^2)$, and so $R^\vee/\Mod(\wbg) \cong \Z^2$, confirming part \eqref{SNFB_gcd1} of Theorem~\ref{thm:SNFB}.
\end{example}

\subsubsection{Additional examples in type $B$}\label{sec:examplesBextra}

We now give some further examples in type $B$, to illustrate the differences in behavior compared to type $C$. See the introduction to Section~\ref{sec:examplesB} for the list of further examples we consider in type $B$.
  
\begin{example}\label{eg:beta0gamma1s} 
This example proves a special case of part \eqref{SNFB_gcd2_parity} of Theorem \ref{thm:SNFB}, and part \eqref{B_beta0_gamma1s} of Theorem~\ref{thm:BasisB}.  
 Let $\sW$ be of type $B_{n}$ with $n \geq 2$, and let $|\beta| = 0$ and $\gamma = (1,\dots,1)$. For this case, we have that
\[\wbg = w_\gamma = s_n (s_{n-1} s_n s_{n-1}) \cdots (s_1 s_2 \cdots s_{n-1} s_n s_{n-1} \cdots s_2 s_1),\] 
which equals the longest element $w_0$ of $\sW$; see Table 1 in \cite{BenkartKangOhPark}, for example. Thus $w_0$ is central in $\sW$ and acts as $-\Id$ on $R^\vee$.
Hence, in this case
 \[\Mod(\wbg) = (\Id - w_\gamma)R^\vee = (2\Id) R^\vee = 2R^\vee,\] 
which has $\Z$-basis $\{ 2\alpha_i^\vee \mid i \in [n] \}$. The Smith normal form for $(\Id - \wbg)$ is also clearly $\Sbg = (2^{n})$.
\end{example}    
  
\begin{example}\label{eg:beta0gamma1114B7} 
This example illustrates part \eqref{SNFB_gcd2_nonparity} of Theorem~\ref{thm:SNFB} and part~\eqref{B_beta0_gamma11} of Theorem~\ref{thm:BasisB}.  Let $\sW$ be of type $B_{7}$, and let $|\beta| = 0$ and $\gamma = (1,1,1,4)$. The matrix for $N_\gamma = \Id - w_\gamma$ with respect to the type $B_7$ basis $\Delta^\vee$ is given by
\[
N_\gamma = 
\begin{pmatrix}
2 & -1&0&0&&& \\
1 & 1 & -1 & 0 &  &      \\
1 &0 & 1 & -1 &  &      \\
2 &0 & 0 & 0 & 0 & 0 & 0   \\
2 &  & &  -2 & 2 & 0 & 0   \\
2  &  & &  -2 & 0 & 2 & 0  \\
1 &  &  &-1 & 0 & 0 & 2  
\end{pmatrix}.
\]
We carry out column operations on the first $4$ columns as in type $C$, and reduce entries in the last $\ell = 3$ rows modulo $2$, to obtain $B_\gamma$:
\[
B_\gamma = 
\begin{pmatrix}
0 & -1&0&0&&& \\
0 & 1 & -1 & 0 &  &      \\
0 &0 & 1 & -1 &  &      \\
2 &0 & 0 & 0 & 0 & 0 & 0   \\
0 &  & &  0 & 2 & 0 & 0   \\
0  &  & &  0 & 0 & 2 & 0  \\
1 &  &  & 1 & 0 & 0 & 2  
\end{pmatrix}.
\]
Thus a $\Z$-basis for $\Mod(\wbg)$ is given by the union of $\{ 2\alpha_4^\vee + \alpha_7^\vee, 2\alpha_5^\vee, 2\alpha_6^\vee, 2\alpha_7^\vee\}$, $ \{ -\alpha_3^\vee + \alpha_7^\vee \}$, and $\{ -\alpha_1^\vee + \alpha_2^\vee, -\alpha_2^\vee + \alpha_3^\vee\}$. Here, we have $I_\gamma = \{ 7,6,5,4\}$, so $I_\gamma - 1 = \{ 6,5,4,3\}$, and we have $n - \ell - 1 = 3$. Hence this matches part~\eqref{B_beta0_gamma11} of Theorem~\ref{thm:BasisB}. Note that the set of vectors $\left\{ -\alpha_{i}^\vee + \alpha_{i+2}^\vee \ \middle|\  i \in (I_\gamma-1), \; i < n-\ell - 1 \right\}$ is empty in this example, since $\gamma$ has only one part which is strictly greater than $1$.

To find the Smith normal form, we continue from $B_\gamma$, now using some row operations. We successively add row 1 to row 2, row 2 to row 3, and row 3 to row 7, to obtain
\[
T_\gamma = 
\begin{pmatrix}
0 & -1&0&0&&&& \\
0 & 0 & -1 & 0 &  &  &    \\
0 &0 & 0 & -1 &  &  &    \\
2 &0 & 0 & 0 & 0 & 0 & 0   \\
0 &  & &  0 & 2 & 0 & 0   \\
0  &  & &  0 & 0 & 2 & 0  \\
1 &  &  & 0 & 0 & 0 & 2  
\end{pmatrix}.
\]
The entry 1 in the bottom left of $T_\gamma$ is due to the parity change in $\gamma$. We can now extract a suitable minor containing this entry, and using row and column column operations verify that $S_\gamma = \diag(1^4,2^2,4)$, confirming part \eqref{SNFB_gcd2_nonparity} of Theorem~\ref{thm:SNFB}.
\end{example}  
  
\begin{example}\label{eg:beta1gamma34B8}  
This example illustrates part \eqref{SNFB_gcd1} of Theorem~\ref{thm:SNFB} in the case $\beta = (1)$, and shows how to obtain a basis for $\Mod(\wbg)$, for $\gamma$ with more than one part. 
Let $\sW$ be of type $B_8$ and suppose $\beta = (1)$ and $\gamma = (3,4)$, so that $\wbg = w_\gamma = w^\gamma_1 w^\gamma_2 = (s_8 s_7 s_6)(s_5 s_6 s_7 s_8 s_7 s_6 s_5 s_4 s_3 s_2)$.  The matrix for $N_\gamma = \Id-w_\gamma$ with respect to the type $B_8$ basis $\Delta^\vee$ is given by
\[
N_\gamma = 
\begin{pmatrix}
0 & &&&&&& \\
-1 & 2 & -1 & 0 & 0 &  &  &  \\
-1 &1 & 1 & -1 & 0 &  &  &  \\
-1 &1 & 0 & 1 & -1 &  &  &  \\
-2 & 2 & 0 & 0 & 0 &  &  &  \\
 &  &  &  & -1 & 2 & -1 & 0 \\
 &  &  & & -1 & 1 & 1 & -2 \\
 &  &  & & -1 & 1 & 0 & 0 \\
\end{pmatrix}.
\]

For $\Mod(w_\gamma)$, we first add column 2 to column 1, which creates a pivot entry 1 in the first column. We then use this pivot to clear entries in the $4 \times 4$ block corresponding to $\gamma_2 = 4$, which gives $B_\gamma'$:
\[
B_\gamma' = 
\begin{pmatrix}
0 & &&&&&& \\
\boxed{1} & 0 & 0 & 0 & 0 &  &  &  \\
0 & 0& \boxed{1} & 0 & 0 &  &  &  \\
0 &0 & 0 & \boxed{1} & 0 &  &  &  \\
0 & \boxed{2} & 0 & 0 & 0 &  &  &  \\
 &  &  &  & -1 & 2 & -1 & 0 \\
 &  &  & & -1 & 1 & 1 & -2 \\
 &  &  & & -1 & 1 & 0 & 0 \\
\end{pmatrix}.
\]
We then add column 6 to column 5, and similarly clear entries in the $3 \times 3$ block corresponding to $\gamma_1 = 3$, to give $B_\gamma$:
\[
B_\gamma = 
\begin{pmatrix}
0 & &&&&&& \\
\boxed{1} & 0 & 0 & 0 & 0 &  &  &  \\
0 & 0& \boxed{1} & 0 & 0 &  &  &  \\
0 &0 & 0 & \boxed{1} & 0 &  &  &  \\
0 & \boxed{2} & 0 & 0 & 0 &  &  &  \\
 &  &  &  & \boxed{1} & 0 & 0 & 0 \\
 &  &  & & 0 & 0 & \boxed{1} & 0 \\
 &  &  & & 0 & \boxed{1} & 0 & 0 \\
\end{pmatrix}.
\]
The resulting $\Z$-basis is $\{ \alpha_2^\vee, \alpha_3^\vee, \alpha_4^\vee, 2\alpha_5^\vee \} \cup \{ \alpha_6^\vee, \alpha_7^\vee, \alpha_8^\vee\}$. We have $S_\gamma = \diag(1^6, 2, 0)$, and so $R^\vee/\Mod(w_\gamma) \cong (\Z/2\Z)$, confirming part \eqref{SNFB_gcd1} of Theorem~\ref{thm:SNFB}.
\end{example}

\begin{example}\label{eg:beta5gamma3B8}   This example illustrates part \eqref{SNFB_gcd1} of Theorem~\ref{thm:SNFB} in the case $|\beta| \geq 2$. Let $\sW$ be of type $B_8$ and suppose $\beta = (5)$ and $\gamma = (3)$  so that $\wbg = w_\beta w_\gamma = (s_1 s_2 s_3s_4)(s_8 s_7 s_6)$. The matrix for $\Mbg = \Id - \wbg$ with respect to the type $B_8$ basis $\Delta^\vee$ is given by
\[
\Mbg = 
\begin{pmatrix}
1&0&0&1&-1& \\
-1&1&0&1&-1& \\
0&-1&1&1&-1& \\
0&0&-1&2&-1&&&\\
0&0&0&0&0&0&0&0\\
&&&& -1 & 2 & -1 & 0 \\
&& && -1 & 1 & 1 & -2 \\
&&&& -1 & 1 & 0 & 0 \\
\end{pmatrix}. 
\]

Then similarly to Example~\ref{eg:beta4gamma3B7}, we carry out column operations on $\Mbg$ to obtain $\Bbg''$:
\[
B''_{\beta,\gamma} = 
\begin{pmatrix}
1&0&0&0& &&&\\
-1&1&0&0& &&&\\
0&-1&1&0& &&&\\
0 & 0 & -1 & 0 & 1 && \\
0&0&0&0&0&0&0& 0\\
&&& 0 & 1 & 2 & -1 & 0 \\
&&&-5 & 0 & 1 & 1 & -2 \\
&&&& 0 & 1 & 0 & 0 \\
\end{pmatrix}.
\]
At this point, we replace the $-5$ by $\gcd(\beta_k,2) \in \{1,2\}$. But now, since $\gcd(\beta_k,2) = 1$, we create a pivot entry $1$ in row $n-1$. We can now successively clear entries in the last $\gamma_1 = 3$ rows, and so obtain a pivot entry $1$ in row 4 of the matrix $\Bbg'$:
\[
B'_{\beta,\gamma} = 
\begin{pmatrix}
1&0&0&0& &&&\\
-1&1&0&0& &&&\\
0&-1&1&0& &&&\\
0 & 0 & -1 & 0 & \boxed{1} && \\
0&0&0&0&0&0&0& 0\\
&&& 0 & 0 & 0 & \boxed{1} & 0 \\
&&&\boxed{1} & 0 & 0 & 0 & 0 \\
&&&& 0 & \boxed{1} & 0 & 0 \\
\end{pmatrix}.
\]
We can then clear entries in the first $ 4$ rows, to get $\Bbg$:
\[
B_{\beta,\gamma} = 
\begin{pmatrix}
\boxed{1}&0&0&0& &&&\\
0&\boxed{1}&0&0& &&&\\
0&0&\boxed{1}&0& &&&\\
0 & 0 & 0 & 0 & \boxed{1} && \\
0&0&0&0&0&0&0& 0\\
&&& 0 & 0 & 0 & \boxed{1} & 0 \\
&&&\boxed{1} & 0 & 0 & 0 & 0 \\
&&&& 0 & \boxed{1} & 0 & 0 \\
\end{pmatrix}.
\]
Therefore $\Sbg = \diag(1^7, 0)$, and so $R^\vee/\Mod(\wbg) \cong \Z$, confirming part \eqref{SNFB_gcd1} of Theorem~\ref{thm:SNFB}.
\end{example}

\begin{example}\label{eg:beta4gamma0B4} This example illustrates part \eqref{SNFB_gamma0_2} of Theorem~\ref{thm:SNFB}, and part~\eqref{B_gamma0_2} of Theorem~\ref{thm:BasisB}. 

Let $\sW$ be of type $B_4$ and suppose $\beta = (4)$ and $|\gamma| = 0$, so that $\wbg = w_\beta = s_1 s_2 s_3$. The matrix for $\Mbg = \Id - \wbg$ with respect to the type $B_4$ basis $\Delta^\vee$ is given by
\[
\Mbg = 
\begin{pmatrix}
1&0&1&-2 \\
-1&1&1&-2 \\
0&-1&2&-2 \\
0&0&0&0
\end{pmatrix}. 
\]
Note that the last row of this matrix is all $0$s, while the $(n-1) \times (n-1)$ minor in the top left is the same as the matrix $M_\beta$ in type $A_{n-1}$.

We start by adding column 3 to column 4, then multiply column 4 by $-1$, to obtain $B'_{\beta,\gamma}$:
\[ 
B'_{\beta,\gamma} = 
\begin{pmatrix}
1&0&1&1 \\
-1&1&1&1 \\
0&-1&2&0 \\
0&0&0&0
\end{pmatrix}
\]
Note that the number of $1$s in column $n$ of $B'_{\beta,\gamma}$ is equal to $(\beta_p - 2) \geq 0$. (This is also true when $\beta_p = 2$.)
We then add to both columns $n-1$ and $n$ suitable linear combinations of columns 1 and 2, to obtain $B''_{\beta,\gamma}$: 
\[
B''_{\beta,\gamma} = 
\begin{pmatrix}
1&0&0&0 \\
-1&1&0&0 \\
0&-1&4&2 \\
0&0&0&0
\end{pmatrix}.
\]
We see that in $B''_{\beta,\gamma}$, the $(n-1,n-1)$-entry is $\gcd(\beta_k)$, and the $(n-1,n)$-entry is $(\beta_p - 2)$. 
We can thus replace the $(n-1,n-1)$-entry of $B''_{\beta,\gamma}$ by $\gcd(\beta_k,\beta_p-2)$ and clear its $(n-1,n)$-entry. In this example, since $\gcd(\beta_k,\beta_p - 2) = 2$, this gives us:
\[ 
B_{\beta,\gamma} = 
\begin{pmatrix}
1&0&0&0 \\
-1&1&0&0 \\
0&-1&2&0 \\
0&0&0&0
\end{pmatrix}.
\]
From here, the basis $\{\alpha_1^\vee - \alpha_2^\vee, \alpha_2^\vee - \alpha_3^\vee, 2 \alpha_{n-1}^\vee \}$ can be read off, which matches part~\eqref{B_gamma0_2} of Theorem~\ref{thm:BasisB}. Note that in this example, the set of basis vectors $\left\{ \alpha_i^\vee - \alpha_{i+2}^\vee \ \middle| \ i \in (I_{\beta} - 1) \backslash \{n-1\} \right\}$ is empty, since $\beta$ has only one part.

To find the Smith normal form, we successively add row 1 to row 2, then row 2 to row 3, to directly obtain $S_{\beta,\gamma}$:
\[ 
S_{\beta,\gamma} = 
\begin{pmatrix}
1&0&0&0 \\
0&1&0&0 \\
0&0&2&0 \\
0&0&0&0
\end{pmatrix}.
\]
In other words, $\Sbg = \diag(1^2,2,0)$, and so $R^\vee/\Mod(\wbg) \cong (\Z /2\Z) \oplus \Z$, confirming part \eqref{SNFB_gamma0_2} of Theorem~\ref{thm:SNFB}.
\end{example}

\begin{example}\label{eg:beta33gamma0B6} This example illustrates part \eqref{SNFB_gamma0_1} of Theorem~\ref{thm:SNFB} and part \eqref{B_gamma0_1} of Theorem~\ref{thm:BasisB} in the case where $\beta_p \geq 2$. (The case $\beta_p = 1$ is discussed in Example~\ref{eg:beta41gamma0B5} below.)

Let $\sW$ be of type $B_6$ and suppose $\beta = (3,3)$ and $|\gamma| = 0$ so that $\wbg = w_\beta = (s_1 s_2)(s_4s_5)$.  The matrix $\Mbg = \Id - \wbg$ with respect to the type $B_6$ basis $\Delta^\vee$ is given by
\[
\Mbg = 
\begin{pmatrix}
1 & 1 & -1 & & &   \\
-1 & 2 & -1 & & &  \\
0 & 0 & 0 & 0 & 0 & 0  \\
& & -1&1&1&-2 \\
 &  &0& -1&2&-2\\
 & & 0&0&0&0\\
\end{pmatrix}. 
\]
We add column 2 to column 3 and then proceed as in type $A$ and the previous example to obtain:
\[
\Bbg' = 
\begin{pmatrix}
1 & 0 & 0 & &  &   \\
-1 & 0 & 1 & & &  \\
0 & 0 & 0 & 0 & 0 & 0  \\
& & -1&1&0&0 \\
 &  &0& -1&3&1\\
 & & 0&0&0&0\\
\end{pmatrix}. 
\]
We see that in $\Bbg'$, the $(n-1,n-1)$-entry is $\gcd(\beta_k)$, and the $(n-1,n)$-entry is $(\beta_p - 2) \geq 0$. We can thus replace the $(n-1,n-1)$-entry by $\gcd(\beta_k,\beta_p - 2)$ and clear the $(n-1,n)$-entry. In this example, since $\gcd(\beta_k,\beta_p-2) = 1$, we get a pivot $1$ in the  $(n-1,n-1)$-entry.  We can use this resulting pivot to successively clear entries moving up the matrix, resulting in $\Bbg$:
\[
\Bbg = 
\begin{pmatrix}
1 & 0 & 0 & & &   \\
0 & 0 & 1 & & & \\
0 & 0 & 0 & 0 & 0 & 0  \\
& & 0&1&0&0 \\
 &  &0& 0&1&0\\
 & & 0&0&0&0\\
\end{pmatrix}. 
\]
From here, we read off the basis $\{ \alpha_1^\vee, \alpha_2^\vee \} \cup \{ \alpha_4^\vee, \alpha_5^\vee \}$, confirming part~\eqref{B_gamma0_1} of Theorem~\ref{thm:BasisB} in this case. 
The Smith normal form is also now clearly $\Sbg = \diag(1^4, 0^2)$ and so $R^\vee/\Mod(\wbg) \cong \Z^2$, confirming part \eqref{SNFB_gamma0_1} of Theorem~\ref{thm:SNFB} in this case.
\end{example}

\begin{example}\label{eg:beta41gamma0B5}
This example illustrates part \eqref{SNFB_gamma0_1} of Theorem~\ref{thm:SNFB} and part~\eqref{B_gamma0_1} of Theorem~\ref{thm:BasisB}, in the case that $\beta$ has last part $\beta_p=1$. 

Let $\sW$ be of type $B_5$ and suppose $\beta = (4,1)$ and $|\gamma| = 0$, so that $\wbg = w_\beta = s_1 s_2 s_3$. Note that since $\beta_p = 1$, the element $w_\beta$ has support in the type $A_{n-2}$ subsystem indexed by the first $n-2$ nodes, and so in particular $w_\beta$ fixes $\alpha_n^\vee$. Moreover, letting $\beta' = (4,1,1)$ be the partition obtained by adding last part $1$ to $\beta$, the element $w_\beta$ acts on $\alpha_{n-1}^\vee$ exactly as does the element $w_{\beta'}$ in type $A_{n-1}$. Hence the matrix for $\Mbg = \Id - \wbg$ with respect to the type $B_5$ basis $\Delta^\vee$ is given by
\[
\Mbg = 
\begin{pmatrix}
1&0&1&-1 & 0\\
-1&1&1&-1 & 0\\
0&-1&2&-1&0 \\
0&0&0&0&0 \\
0&0&0&0&0
\end{pmatrix}. 
\]
Note that if we denote by $\ell$ the number of parts of $\beta$ which are equal to 1, then the last $\ell + 1$ rows of this matrix are all $0$s. 

We then proceed as in type $A_{n-1}$ in the $(n-1) \times (n-1)$ submatrix in the top left (compare Example~\ref{eg:41A4}, which considers $\beta = (4,1)$ in type $A_4$, as well as Example~\ref{eg:431A7}, which considers $\beta = (4,3,1)$ in type $A_7$). Since $\beta_p=1$, there is a pivot 1 in column $n-1$, and we thus obtain
\[
\Bbg = 
\begin{pmatrix}
1&0&0&0 & 0\\
0&1&0&0 & 0\\
0&0&0&1&0 \\
0&0&0&0&0 \\
0&0&0&0&0
\end{pmatrix}.
\]
(If instead $\gcd(\beta_k, \beta_p-2)=1$, the pivot 1 is created in the same row of column $n-2$.)
From here, we read off the basis $\{ \alpha_1^\vee, \alpha_2^\vee, \alpha_3^\vee \}$, confirming part~\eqref{B_gamma0_1} of Theorem~\ref{thm:BasisB} in the case $\beta_p = 1$. In addition, clearly $\Sbg = \diag(1^3, 0^2)$, confirming part \eqref{SNFB_gamma0_1} of Theorem~\ref{thm:SNFB} in the case $\beta_p = 1$.
\end{example}

\section{Type $D$ mod-sets}\label{sec:TypeD}

In this section we give a description of all mod-sets in type $D$. Let $\sW$ be the finite Weyl group of type $D_n$, for $n \geq 4$.  In \Cref{sec:repsD}, we review the complete system of minimal length representatives for the conjugacy classes of $\sW$ provided in \cite[Chapter~3]{GeckPfeifferBook}. Our results are stated in Section \ref{sec:resultsD}, and we then give examples in Section \ref{sec:examplesD}, designed to fully illustrate the proofs of our results on mod-sets.

Throughout this section, we label the nodes of the Dynkin diagram by $[n]$ as in both Bourbaki \cite{Bourbaki4-6} and Sage \cite{sagemath}, so that the first $n-1$ nodes form a type $A_{n-1}$ subsystem, and the unique node of valence~$3$ is indexed by $(n-2)$. We remark that this is different from the indexing used in~\cite{GeckPfeifferBook}, where the nodes of the Dynkin diagram in type $D_n$ are labeled by $\{0,\dots,n-1\}$, and the unique node of valence $3$ is indexed by $2$. See Table \ref{table:dynkin} in~\Cref{app:dynkin} for a direct comparison of these choices of labeling.

\subsection{Conjugacy class representatives in type $D$}\label{sec:repsD}

Following \cite{GeckPfeifferBook}, we first explain how the conjugacy classes of $\sW$ of type $D_n$ are parameterized by ordered pairs of compositions~$(\beta, \delta)$ such that $\beta$ is weakly decreasing, $\delta$ is weakly increasing and has an even number of parts, and $|\beta| + |\gamma| = n \geq 4$.  In the special case $(\beta,0)$ where all parts of $\beta$ are even, there are two distinct conjugacy classes, denoted by $\beta^+$ and $\beta^-$, or  $\beta^\pm$ to include both cases. For each  pair $(\beta,\delta)$, we will define standard parabolic subgroups $\sW_{\beta}$ and $\sW_\delta$ of $\sW$, and an  element $w_{\beta, \delta} = w_{\beta} \cdot w_\delta$ with $w_{\beta}$ cuspidal in $\sW_{\beta}$ and $w_\delta$ cuspidal in $ \sW_\delta$. The set of all $w_{\beta^\pm, \delta}$ forms a complete system of minimal length representatives for the conjugacy classes of $\sW$.

If $|\beta| = 0$ or $|\beta| = 1$, we define $w_\beta$ and $\sW_\beta$ to both be trivial. If $\beta$ is a partition of $m$ with $2 \leq m < n$, then we define both $w_{\beta}$ and $\sW_{\beta}$ as in \Cref{sec:repsA}, so that the~$w_{\beta}$ are the minimal length conjugacy class representatives in the type $A_{m-1}$ subsystem of $\sW$ indexed by the first $m-1$ nodes of the Dynkin diagram.  If $|\beta| = n$ and $\beta$ has at least one odd part, then $w_\beta$ is again defined as in Section \ref{sec:repsA}. If $|\beta| = n$ and $\beta$ has all parts even, then $w_{\beta^+}$ is defined to be the minimal length conjugacy class representative $w_\beta$ in the type $A_{n-1}$ subsystem of $\sW$ indexed by the first $n-1$ nodes of the Dynkin diagram, as in Section~\ref{sec:repsA}.  If $|\beta| = n$ and $\beta$ has all parts even, the element $w_{\beta^-}$ is defined by replacing the element $s_{n-1}$ in $w_{\beta^+}$ by $s_n$.

Now let $\delta = (\delta_1, \dots, \delta_{2r})$ be any weakly increasing composition of $n-m$, with an even number of parts.  If $|\delta| = 0$, we define $w_\delta$ and $\sW_\delta$ to both be trivial.  Note that since $\delta$ must have an even number of parts, the case $|\delta| = 1$ does not occur. It thus remains to define $w_\delta$ and $\sW_\delta$  for  $2 \leq m < n$. 

Following \cite[Section 3.4]{GeckPfeifferBook} but using different labeling conventions and notation, we first define elements \[ u_0 = s_n \quad \mbox{and} \quad u_1 = s_n s_{n-1},\] and then for $2 \leq j \leq n-1$, inductively define \[u_j = s_{n-j} u_{j-1} s_{n-j}.\]  
We also define a strictly increasing subsequence of $\{0,1,\dots,n\}$ in the same way as we did for the compositions $\beta$  and $\gamma$ in Sections \ref{sec:repsA} and \ref{sec:repsC}, respectively.  
That is, let $j^\delta_1 = 0$ and for all $2 \leq k \leq 2r+1$, define
\begin{equation}\label{eq:jdeltak}
j^\delta_k = \sum_{i=1}^{k-1} \delta_i = \delta_1 + \dots + \delta_{k-1}.
\end{equation}
In particular, note that $j^\delta_2 = \delta_1$ and $j^\delta_{2r+1} = n-m$. 
Notice also that if $\delta_1 = \dots = \delta_k = 1$ for some $k \geq 1$, then $j^\delta_{k+1} = k$.
Now for all $2 \leq k \leq 2r$ such that $\delta_k \geq 2$, define the element 
\begin{equation}\label{eq:vdeltak}
v_k^\delta = u_{j^\delta_k} s_{n-(j^\delta_k+1)} s_{n-(j^\delta_k+2)}\cdots s_{n - (j^\delta_k + \delta_k - 1)},
\end{equation}
analogous to \eqref{eq:wgammak-t}. In particular, $v_k^\delta$ ends with a product of $\delta_k - 1 $ consecutive simple reflections ordered by decreasing index. Considering the initial $u_{j^\delta_k}$ as a single term, the expression for $v_k^\delta$ has $\delta_k \geq 2$ terms. Note also that if the last part $\delta_{2r} \geq 2$, then the element $v_{2r}^\delta$ has final simple reflection equal to $s_{n - |\delta| + 1}$.

We can now define the element $w_\delta$ for $2 \leq m < n$, which has several cases.  We first consider the cases where at least one part of $\delta$ equals 1. In the special case that all $2r$ parts of $\delta$ are equal to $1$ (and hence $|\delta| = n-m = 2r$), we define
\begin{equation}\label{eq:wdelta1s}
w_\delta = u_1 u_2 \dots u_{2r - 1}.
\end{equation}
Next, if $\delta_1 = \dots = \delta_{\ell} = 1$ while $\delta_{\ell+1} \geq 2$, we define
\begin{equation}\label{eq:wdelta1s2s}
w_\delta = (u_1 \cdots u_{\ell-1}) v^\delta_{\ell+1} v^\delta_{\ell+2} \dots v^\delta_{2r},
\end{equation}
where the product $(u_1\cdots u_{\ell-1})$ is trivial if $\ell = 1$. Note that $v^\delta_{\ell+1}$ begins with $u_\ell$, and hence in all cases where $\delta_1 = \dots = \delta_\ell = 1$ for some $\ell \geq 1$, the element $w_\delta$ begins with the product $u_1 \dots u_\ell$. 
We next consider the cases where $\delta$ does not contain any parts equal to 1.  Define
\begin{equation}\label{eq:wdelta3s}
w_\delta = \left(u_0 s_{n-2} s_{n-3} \dots s_{n-(\delta_1 - 1)} \right) v_2^\delta v_3^\delta \dots v_{2r}^\delta,
\end{equation}
where the product $s_{n-2} s_{n-3} \dots s_{n-(\delta_1 - 1)}$ is trivial if $\delta_1 = 2$,
so that the parenthetical expression in the formula for $w_\delta$ contains $\delta_1 - 1$ terms, including $u_0 = s_n$. 

By \cite[Prop.~3.4.11]{GeckPfeifferBook}, the element $w_\delta$ is cuspidal in the parabolic subgroup $W_\delta$ of type $D_{n-m}$, generated by the simple reflections indexed by the last $|\delta| = n-m$ nodes of the Dynkin diagram.   By \cite[Prop.~3.4.12]{GeckPfeifferBook}, the set of $w_{\beta^\pm, \delta} = w_{\beta^\pm} \cdot w_\delta$ for all distinct pairs of compositions $(\beta, \delta)$ such that $\beta$ is weakly decreasing, $\delta$ is weakly increasing with an even number of parts, and $|\beta| + |\delta| = n$ forms a complete system of minimal length representatives of the conjugacy classes of $\sW$ in type $D_n$.

\subsection{Results on mod-sets in type $D$}\label{sec:resultsD}

In this section, we state our results in type~$D$.  It turns out that there are quite a few cases in type $D$, and so we state three separate theorems.  Throughout this section, we assume $\sW$ is of type $D_n$ with $n \geq 4$, and we let \[\beta= (\beta_1,\dots,\beta_p) \quad \mbox{and} \quad \delta = (\delta_1,\dots,\delta_{2r}) \] be a pair of compositions such that $\beta $ is weakly decreasing, $\delta$ is weakly increasing with an even number of parts, and $|\beta| + |\delta| = n$, with corresponding conjugacy class representative $\wbd \in \sW$, constructed as in \Cref{sec:repsD}. In the special case $|\beta| = n$ and all parts of $\beta$ are even, recall that we have two cases $\beta^\pm$.

\Cref{thm:DnSNF} provides the Smith normal form of $(\Id - \wbd)$, from which the isomorphism type of the quotient $R^\vee/\Mod(\wbd)$ is immediate. In~\Cref{thm:DnCuspidal}, we give a basis for $\Mod(\wbd)$ in the case that $|\beta| = 0$, equivalently whenever $\wbd = w_\delta$ is cuspidal in $W$. In~\Cref{thm:Dnnoncuspidal}, we then give a basis for $\Mod(\wbd)$ when $|\delta| = 0$. As in type $B_n$, for the remaining noncuspidal cases, where $|\beta| \geq 1$ and $|\delta| \geq 2$, we provide illustrative examples of how to identify convenient bases. In all cases, a set of equations explicitly describing $\Mod(\wbd)$ can easily be deduced once the (relatively sparse) basis is known.

Throughout this section, we write $\gcd(\beta_k,2) = \gcd(\beta_1, \dots, \beta_p, 2)$.  In particular, $\gcd(\beta_k,2) = 2$ when all parts of $\beta$ are even (including the case that $|\beta| = 0$), and otherwise $\gcd(\beta_k,2) = 1$.  We now state our results on the Smith normal form for $(\Id-\wbd)$ in type $D$.

\begin{thm}\label{thm:DnSNF}
Suppose $\sW$ is of type $D_n$ with $n \geq 4$.  Let $(\beta,\delta)$ be a pair of compositions such that $\beta = (\beta_1,\dots,\beta_p)$ is weakly decreasing, $\delta = (\delta_1,\dots,\delta_{2r})$ is weakly increasing with an even number of parts, and $|\beta| + |\delta| = n$.  Let $w_{\beta^\pm, \delta} \in \sW$ be the corresponding conjugacy class representative in $\sW$.  

Then for any $w \in [w_{\beta^\pm, \delta}]$, the Smith normal form of $(\Id-w)$ is as follows:
\begin{enumerate}
 \item\label{SNFD_gcd2_parity} If $\gcd(\beta_k,2) = 2$ (including the case that $|\beta| = 0$), $|\delta| \geq 2$, and all parts of $\delta$ have the same parity, then 
  \[
 \Sbd = \diag(1^{n - 2r-p}, 2^{2r},0^p).
 \]
   \item\label{SNFD_gcd2_nonparity} If $\gcd(\beta_k,2) = 2$ (including the case that $|\beta| = 0$), $|\delta| \geq 2$, and $\delta$ has a change in parity, then 
  \[
 \Sbd = \diag(1^{n - 2r-p+1}, 2^{2r-2},4,0^p).
 \]
\item\label{D_gcd1} If $\gcd(\beta_k,2) = 1$ and $|\delta| \geq 2$, then
\[ \Sbd = \diag(1^{n-2r-p+1},2^{2r-1},0^p).\] 
\item\label{D_beta_delta0} If $|\delta| = 0$ and $\beta = (1,\dots, 1)$, so that $\wbd$ is trivial, then $\Sbd = \diag(0^n)$. If $|\delta| = 0$ and $\beta \neq (1,\dots, 1)$, then \[ S_{\beta^\pm, \delta} = \diag(1^{n-p-1}, \gcd(\beta_k,2), 0^{p}).\] 
\end{enumerate}
\end{thm}

\noindent Since Smith normal form is canonical, the delicate nature of the results in type $D$ is evident already from the many cases necessarily appearing in Theorem \ref{thm:DnSNF} above.

The next theorem gives a $\Z$-basis for $\Mod(\wbd)$ when $|\beta| = 0$, equivalently whenever $\wbd = w_\delta$ is cuspidal in $\sW$. Before stating this result, we require some additional notation.  

Given $\beta = (\beta_1,\dots,\beta_p)$ with $|\beta| = m$, recall from \eqref{eq:Jbetak} the pairwise disjoint subintervals $J^\beta_k$ of $[m-1]$, for $1 \leq k \leq p$. These subintervals are constructed such that  $J_\beta = \sqcup_{k=1}^p J^\beta_k$ indexes the simple reflections whose product equals $w_\beta$, when written in increasing order. In type $D_n$, as in type $B_n$, we also have the following subsets
\[ I_\beta = \{ j_k^\beta \mid 2 \leq k \leq p \} \cup \{ m\} \quad \text{and} \quad I_\beta - 1 = \{ i-1 \mid i \in I_\beta\}.\]

Given a weakly increasing composition $\delta = (\delta_1, \dots, \delta_{2r})$ with $|\delta| = n-m$, recall from \eqref{eq:jdeltak} that $j^\delta_1 = 0$ and $j^\delta_k = \sum_{i=1}^{k-1} \delta_i$ for $2 \leq k \leq 2r$. 
Define the following $2r$-element subset of $[n]$ associated to $\delta$:
\[ I_\delta = \left\{ n - j^\delta_k \ \middle| 1 \leq k \leq 2r \right\} =\left\{ m+ \delta_{2r}, m+\delta_{2r} + \delta_{2r-1}, \dots, n - \delta_1, n \right\}.
\] 
In addition, we denote by $I_\delta - 1 =\{ i - 1 \mid i \in I_\delta \}$. 

To simplify notation in the statements below, we further define $\delta_{1,0} = 0$ and $\delta_{k,k-1} = \left(\delta_k + \delta_{k-1}\right) \operatorname{mod} \left(2\right)$ for $2 \leq k \leq 2r$. That is, $\delta_{k,k-1} = 0$ if the parts $\delta_k$ and $\delta_{k-1}$ have the same parity, and otherwise $\delta_{k,k-1}=1$.  In particular, note that all parts of $\delta$ have the same parity if and only if $\delta_{k,k-1} = 0$ for all $2 \leq k \leq 2r$. 

With this additional notation established, we now state our results providing bases for the mod-sets in type $D$.

\begin{thm}\label{thm:DnCuspidal}
Suppose $\sW$ is of type $D_n$ with $n \geq 4$.  Let $(\beta, \delta)$ be a pair of compositions such that $|\beta| = 0$, $\delta = (\delta_1,\dots,\delta_{2r})$ is weakly increasing with an even number of parts, and $|\delta| = n$.  Let $\wbd = w_\delta$ be the corresponding conjugacy class representative in $\sW$, which is cuspidal in $\sW$. Then:

\begin{enumerate}
\item\label{D_delta1s} If $\delta = (1, \dots, 1)$, then $\Mod(\wbd) = 2R^\vee$ and has $\Z$-basis given by \[\{ 2\alpha_i^\vee \mid i \in [n]\}.\]

\item\label{D_delta1s_2} If $\delta_k = 1$ for all $1 \leq k \leq \ell < n$ for some $\ell \geq 2$, while $\delta_{\ell+1} \geq 2$, then $\Mod(\wbd)$  has $\Z$-basis given by 
\[ \left\{ 2 \alpha_{n-j^\delta_k}^\vee + \delta_{k,k-1}(\alpha_{n-1}^\vee + \alpha_n^\vee)\  \middle| \  1 \leq k \leq 2r\right\}  \]
\[ \cup \  \left\{ -\alpha_i^\vee + \alpha_{i+1}^\vee\  \middle| \  i \in [n],\; i \not \in I_\delta \cup (I_\delta - 1) \cup \{ n-\ell-1\} \right\}  \]
\[ \cup \  \left\{ -\alpha_{i}^\vee + \alpha_{i+2}^\vee\ \middle| \ i \in (I_\delta-1),\; i \leq n-(\delta_{\ell + 1} + \ell)\right\} \]
\[\cup\  \left\{ -\alpha_{n-\ell-1}^\vee + \alpha_{n-1}^\vee + \alpha_n^\vee \right\}.\]

\item\label{D_delta12} If $\delta_1 = 1$ and $\delta_2 \geq 2$, then  $\Mod(\wbd)$ has $\Z$-basis given by
\[ \left\{ 2 \alpha_{n-j^\delta_k}^\vee + \delta_{k,k-1}\alpha_{n-2}^\vee \ \middle| \  1 \leq k \leq 2r \right\} \]
\[ \cup \  \left\{ -\alpha_i^\vee + \alpha_{i+1}^\vee \ \middle|\  i \in [n],\; i \not \in I_\delta \cup (I_\delta - 1) \cup \{ n-1,n\} \right\}  \]
\[ \cup \  \left\{ -\alpha_{i}^\vee + \alpha_{i+2}^\vee \ \middle|\  i \in (I_\delta-1),\; i \leq n-(\delta_2 + 1)\right\}\]
\[ \cup\  \left\{- \alpha_{n-2}^\vee - \alpha_{n-1}^\vee + \alpha_n^\vee \right\}.\]

\item\label{D_delta2} If $\delta_1 = 2$, then $\Mod(\wbd)$ has $\Z$-basis given by 
\[ \left\{ 2 \alpha_{n-j^\delta_k}^\vee + \delta_{k,k-1}\alpha_{n}^\vee \ \middle| \  1 \leq k \leq 2r\right\} \]
\[ \cup \  \left\{ -\alpha_i^\vee + \alpha_{i+1}^\vee \ \middle| \  i \not \in I_\delta \cup (I_\delta - 1)  \ \operatorname{or} \ i=n-1 \right\}  \]
\[ \cup \  \left\{ -\alpha_{i}^\vee + \alpha_{i+2}^\vee \ \middle|\  i \in (I_\delta-1), \; i \neq n-1 \right\}.\]

\item\label{D_delta3} If $\delta_1 \geq 3$, then $\Mod(\wbd)$ has $\Z$-basis given by 
\[ \left\{ 2 \alpha_{n-j^\delta_k}^\vee + \delta_{k,k-1}\alpha_{n-1}^\vee \ \middle| \  2 \leq k \leq 2r \right\} \cup 
\left\{ 2\alpha_{n-1}^\vee,\; ((\delta_1 - 1)\operatorname{mod} 2) \alpha_{n-1}^\vee + \alpha_n^\vee \right\} \]
\[ \cup \ \left\{ -\alpha_i^\vee + \alpha_{i+1}^\vee \ \middle| \ i \not \in I_\delta \cup (I_\delta - 1)  \right\}  \]
\[ \cup \  \left \{ -\alpha_{i}^\vee + \alpha_{i+2}^\vee \ \middle| \ i \in (I_\delta-1),\; i \neq n-1 \right\}.\]
\end{enumerate}
\end{thm}

Finally, we give a $\Z$-basis for $\Mod(w_{\beta^\pm, \delta})$ in the case that $|\delta| =0$.

\begin{thm}\label{thm:Dnnoncuspidal}
Suppose $\sW$ is of type $D_n$ with $n \geq 4$.  Let $(\beta,\delta)$ be a pair of compositions such that $\beta = (\beta_1,\dots,\beta_p)$ is weakly decreasing, $|\beta| = n$, and $|\delta| = 0$. Let $w_{\beta^\pm, \delta} = w_{\beta^\pm}$ be the corresponding conjugacy class representative in $\sW$, which is not cuspidal in $\sW$. 
Then:

\begin{enumerate}

\item\label{D_beta1} If $\beta_p = 1$, then $\Mod(\wbd)$ has $\Z$-basis given by
\[ \{ \alpha_j^\vee \mid j \in J_\beta \}. \]

\item\label{D_beta2} If $\beta_p \geq 2$, then:
\begin{enumerate}
\item Both $\Mod(w_{\beta, \delta})$ and $\Mod(w_{\beta^+, \delta})$ have $\Z$-bases given by
\[ \left\{ \alpha_i^\vee - \alpha_{i+1}^\vee \ \middle|\ i \in J_{\beta} \backslash (I_{\beta} - 1) \right\} \ \cup \] 
\[ \left\{ \alpha_i^\vee - \alpha_{i+2}^\vee \ \middle| \ i \in (I_{\beta} - 1) \backslash \{n-1\} \right\}\ \cup \ \left\{ \gcd(\beta_k,2)\alpha_{n-1}^\vee \right\}.\]
\item Replace every $\alpha_{n-1}^\vee$ by $\alpha_n^\vee$ in the basis for $\Mod(w_{\beta^+, \delta})$ to obtain a basis for $\Mod(w_{\beta^-, \delta})$.
\end{enumerate}
\end{enumerate}
\end{thm}

As seen in Theorems \ref{thm:DnCuspidal} and \ref{thm:Dnnoncuspidal}, the two special cases considered here already involve many subcases and delicate statements.  Thus for the remaining ``mixed'' case, where $|\beta| \geq 1$ and $|\delta| \geq 2$, we have opted not to give even lengthier general statements, but instead to provide illustrative examples of how to identify convenient bases. 
We illustrate each of these type $D_n$ results with several examples in Section \ref{sec:examplesD}, which can then be generalized in a straightforward manner to prove each of these three theorems.

\subsection{Examples in type $D$}\label{sec:examplesD}

In this section, we present a sequence of examples which can be generalized to obtain the proofs of the theorems stated in \Cref{sec:resultsD}. The collection of examples is chosen to fully illustrate every case that arises in Theorems \ref{thm:DnSNF}, \ref{thm:DnCuspidal}, and \ref{thm:Dnnoncuspidal}.

In Section \ref{sec:cuspidal}, we restrict to the case $|\beta| = 0$, equivalently $\wbd = w_\delta$ is cuspidal. These examples illustrate the proofs of parts \eqref{SNFD_gcd2_parity} and \eqref{SNFD_gcd2_nonparity} of \Cref{thm:DnSNF} in the case $|\beta| = 0$, and the proof of \Cref{thm:DnCuspidal}. The cuspidal examples we consider are: 
\begin{itemize}
\item $|\beta| = 0$ and $\delta = (1,\dots,1)$ with $n \geq 4$ even in Example~\ref{eg:delta1111};
\item $|\beta| = 0$ and $\delta = (1,1,1,2,2,3)$ in Example~\ref{eg:delta111223};
\item $|\beta| = 0$ and $\delta = (1,3,3,4)$ in Example~\ref{eg:delta1334};
\item $|\beta| = 0$ and $\delta = (2,3,3,3)$ in Example~\ref{eg:delta2333}; and 
\item $|\beta| = 0$ and $\delta = (3,3,3,3)$ in Example~\ref{eg:delta3333}.
\end{itemize}

In \Cref{sec:noncuspidal2}, we consider examples with $\beta_p = 1$ and $|\delta| \geq 2$, to illustrate the proof of part \eqref{D_gcd1} of \Cref{thm:DnSNF} in this case. The examples with $\beta_p = 1$ we consider are: 
\begin{itemize}
\item $\beta = (1)$ and $\delta = (1,\dots,1)$ with $n \geq 5$ odd in Example~\ref{eg:beta1delta1111};
\item $\beta = (1)$ and $\delta = (1,1,1,2)$ in Example~\ref{eg:beta1delta1112};
\item $\beta = (1)$ and $\delta = (1,3)$ in Example~\ref{eg:beta1delta13};
\item $\beta = (1)$ and $\delta = (2,2)$ in Example~\ref{eg:beta1delta22}; 
\item $\beta = (1)$ and $\delta = (3,3)$ in Example~\ref{eg:beta1delta33}; and
\item $\beta = (3,1)$ and $\delta = (1,2)$ in Example~\ref{eg:beta31delta12}.
\end{itemize}

In \Cref{sec:noncuspidal3}, we consider examples with $\beta_p \geq 2$ and $|\delta| \geq 2$. These examples illustrate parts \eqref{SNFD_gcd2_parity}, \eqref{SNFD_gcd2_nonparity}, and \eqref{D_gcd1} of \Cref{thm:DnSNF} in this case. The examples we consider in this section are: 
\begin{itemize}
\item $\beta = (4)$ and $\delta = (1,1,1,2)$ in Example~\ref{eg:beta4delta1112};
\item $\beta = (4)$ and $\delta = (1,2)$ in Example~\ref{eg:beta4delta12};
\item $\beta = (4)$ and $\delta = (2,2)$ in Example~\ref{eg:beta4delta22}; 
\item $\beta = (4)$ and $\delta = (3,3)$ in Example~\ref{eg:beta4delta33}; and 
\item $\beta = (3)$ and $\delta = (3,3)$ in Example~\ref{eg:beta3delta33}.
\end{itemize}

Finally, in \Cref{sec:noncuspidal1},  we give examples in the case where $|\delta| = 0$, illustrating the proof of part \eqref{D_beta_delta0} of \Cref{thm:DnSNF}, and the proof of \Cref{thm:Dnnoncuspidal}. The examples with $|\delta|=0$ we consider are: 
\begin{itemize}
\item $\beta = (3,1)$ and $|\delta| = 0$ in Example~\ref{eg:beta31delta0};
\item $\beta^\pm = (2,2)$ and $|\delta| = 0$ in Example~\ref{eg:beta22delta0}; and
\item $\beta^\pm = (4)$ and $|\delta| = 0$ in Example~\ref{eg:beta4delta0}.
\end{itemize}

For each of the examples listed above, we construct a $\Z$-basis for $\Mod(\wbd)$ and find the Smith normal form for $ (\Id - \wbd)$. From this information, it is then possible to give an explicit description of the $\Z$-module $\Mod(\wbd) = (\Id - \wbd)R^\vee$, and to easily obtain the isomorphism class of the quotient $R^\vee / \Mod(\wbd)$.

\subsubsection{Cuspidal examples}\label{sec:cuspidal}

We first consider the situation where $|\beta| = 0$, equivalently $\wbd = w_\delta$ is cuspidal in $\sW$. These examples illustrate the proofs of parts \eqref{SNFD_gcd2_parity} and \eqref{SNFD_gcd2_nonparity} of \Cref{thm:DnSNF} in the case $|\beta| = 0$, and the proof of \Cref{thm:DnCuspidal}; see the introduction to Section \ref{sec:examplesD} for the list of examples.


\begin{example}\label{eg:delta1111}\label{eg:delta1111}
This example proves part \eqref{SNFD_gcd2_parity} of Theorem \ref{thm:DnSNF} in the special case $|\beta| = 0$ and $\delta = (1,\dots,1)$, and part \eqref{D_delta1s} of Theorem~\ref{thm:DnCuspidal}.
 Let $\sW$ be of type $D_{n}$ with $n=2r \geq 4$ even. Let $|\beta| = 0$ and $\delta = (1,\dots,1)$. For this case, recall from \eqref{eq:wdelta1s} that
\[\wbd = w_\delta = u_1 u_2 \dots u_{n-1},\] 
which equals the longest element $w_0$ of $\sW$; see Table 1 in \cite{BenkartKangOhPark}, for example. In type $D_n$ with $n$ even, note that $w_0$ is central in $\sW$ and acts as $-\Id$ on $R^\vee$.
Hence, in this case
 \[\Mod(\wbd) = (\Id - w_\delta)R^\vee = (2\Id) R^\vee = 2R^\vee,\] 
which has a $\Z$-basis $\{ 2\alpha_i^\vee \mid i \in [n] \}$. The Smith normal form for $(\Id - \wbd)$ is also clearly $\Sbd = (2^{n})$, confirming part \eqref{SNFD_gcd2_parity} of Theorem \ref{thm:DnSNF} and part \eqref{D_delta1s} of Theorem \ref{thm:DnCuspidal} in this case.
\end{example}


\begin{example}\label{eg:delta111223} 
This example illustrates  part \eqref{SNFD_gcd2_nonparity} of \Cref{thm:DnSNF} in the case $|\beta| = 0$, and part \eqref{D_delta1s_2} of \Cref{thm:DnCuspidal}. Let $\sW$ be of type $D_{10}$, and let $|\beta| = 0$ and $\delta = (1,1,1,2,2,3)$. By \eqref{eq:wdelta1s2s} with $\ell = 3$, we have 
\[ \wbd = w_\delta = u_1 u_2 v^\delta_4 v^\delta_5 v^\delta_6 = u_1 u_2 (u_3 s_6)(u_5 s_4)(u_7 s_2 s_1).\]  
Applying \eqref{eq:siaction}, the matrix for $\Id - w_\delta$, with respect to the basis $\Delta^\vee$, is thus given by
\[
P_\delta = 
\begin{pmatrix}
2 & -1 & 0 &&&&&&& \\
1 & 1 & -1  &&&&&&& \\
2 &0 &0 &&&&&&& \\
2 &  & -1 & 2 & -1 &&&&&\\
2 &  & -2 & 2 & 0 &&&&&\\
2 &  &-2 & 2 & -1 & 2 & -1 &&&\\
2 &  &-2 & 2 & -2 & 2 & 0 &&&\\
2 &  &-2 & 2 & -2 & 2 & -2 & 2 & &\\
1 &  &-1 & 1 & -1 & 1 & -1 &  & 2 & \\
1 &  &-1 & 1 & -1 & 1 & -1 &  &  & 2
\end{pmatrix},
\]
where we have omitted most $0$ entries for readability. Notice that $P_\delta$ is block lower-triangular with diagonal blocks of sizes $3,2,2,1,1,1$ from left to right, obtained from the parts of $\delta$ in reverse order.

More specifically, the diagonal blocks of size~$1$ have entry equal to $2$, with all $0$s underneath. The diagonal blocks of size $\geq 2$ have first column $2,1, \dots, 1, 2$, diagonal entries $2, 1, \dots, 1, 0$, above-diagonal entries $-1$, and all other entries $0$. Underneath the diagonal blocks of size $\geq 2$, the last two rows have~$1$s in the leftmost column and $-1$s in the rightmost column, and then 2s fill the remaining entries of the leftmost column. The rightmost column below the $\delta_{\ell+1}$ block is filled with $-2$s.  For the blocks of size $\delta_k$ with $k \geq \ell+2$, place $-1$s in the first $(\delta_{\ell+1}-1)$ entries of the rightmost column, then put $-2$s in the remaining entries of the rightmost column. (In this example, $\delta_{\ell+1} -1 = \delta_4-1 = 2 - 1 = 1$, so we have one $-1$ at the top of the rightmost column below the leftmost blocks of sizes 3 and 2, but not in the $\delta_4$ block itself.) All other entries equal $0$.  One can check via \eqref{eq:siaction} that this description in fact characterizes $P_\delta$ for all $\delta$ with $|\delta| = n$ and $\delta_1 = \cdots = \delta_\ell = 1$ for some $2 \leq \ell < n$ such that $\delta_{\ell+1} \geq 2$; that is, the case considered in part \eqref{D_delta1s_2} of \Cref{thm:DnCuspidal}.

We now carry out a series of column operations on $P_\delta$ to obtain a matrix from which a basis for $\Mod(w_\delta)$ can be read off; see the matrix $Q_\delta$ in \eqref{eq:Qdelta111223} below.
We start by carrying out column operations on columns~1 and 3; that is, the first and last columns of the leftmost block of $P_\delta$. We replace $\cC_1(P_\delta)$ by $\cC_1(P_\delta) + 2\cC_2(P_\delta)  + 3\cC_3(P_\delta)$, and then replace $\cC_3(P_\delta)$ by $\cC_3(P_\delta) + \cC_4(P_\delta)$, to obtain
\[
P_\delta' = 
\begin{pmatrix}
0 & -1 & 0&&&&&&& \\
0 & 1 & -1 &&&&&&& \\
2 & 0 & 0 &&&&&&& \\
\circled{$-1$} &  &1 & \circled{2} & -1 &&&&&\\
-4 &  & & 2 & 0 &&&&&\\
-4& &  & 2 & -1 & 2 & -1 &&&\\
-4 &  & & 2 & \circled{$-2$} & 2 & 0 &&&\\
-4 &  & & 2 & \circled{$-2$} & 2 & -2 & 2 & &\\
-2 &  &  &1 & \circled{$-1$} & 1 & -1 &  & 2 & \\
-2 &  &  &1 & \circled{$-1$} & 1 & -1 &  &  & 2
\end{pmatrix}.
\]

We now clear the circled entry in the first column of $P_\delta'$ using its second block, by replacing $\cC_1(P_\delta')$ by $\cC_1(P_\delta') - \cC_5(P_\delta')$. We then clear the circled entries in columns 4 and~5 of $P'_\delta$ by carrying out column operations on these columns (i.e. the first and last columns of the second block) similar to those we did on columns~1 and~3 in the previous step. Namely, we replace $\cC_4(P_\delta')$ by $\cC_4(P_\delta') + 2\cC_5(P_\delta')$, and then replace $\cC_5(P_\delta')$ by $\cC_5(P_\delta') + \cC_6(P_\delta')$.  This yields
\[
P_\delta'' = 
\begin{pmatrix}
0 & -1 & 0&&&&&&& \\
0 & 1 & -1 &&&&&&& \\
2 & 0 & 0 &&&&&&& \\
0 &  &1 & 0 & -1 &&&&&\\
-4 &  & & 2 & 0 &&&&&\\
\circled{$-3$}& &  &0  & 1 & \circled{2} & -1 &&&\\
-2 & & & -2 &  & 2 & 0 &&&\\
-2 &  & & -2 &  & 2 & -2 & 2 & &\\
-1 &  &  &-1 &  & 1 & -1 & 0 & 2 & \\
-1 &  &  &-1 &  & 1 & -1 & 0 & 0 & 2
\end{pmatrix}.
\]
We now clear the circled entry in the first column of $P_\delta''$ using its third block, by replacing $\cC_1(P_\delta'')$ by $\cC_1(P_\delta'') - 3\cC_7(P_\delta'')$. In the third block, which is the last block of size $\geq 2$, we clear the circled entry of $P''_\delta$ by replacing $\cC_6(P_\delta')$ by $\cC_6(P_\delta') + 2\cC_7(P_\delta')$, to obtain
\[
P_\delta''' = 
\begin{pmatrix}
0 & -1 &0 &&&&&&& \\
0 & 1 & -1 &&&&&&& \\
2 & 0 & 0 &&&&&&& \\
0 &  &1 & 0 & -1 &&&&&\\
\circled{$-4$} &  & & \mathbf{2} & 0 &&&&&\\
0& &  & 0 & 1 & 0 & -1 &&&\\
\circled{$-2$} &  & & \circled{$-2$} &  & \mathbf{2} & 0 &&&\\
4 &  & & -2 &  & -2 & -2 & \boxed{2} & &\\
2 &  &  &-1 &  & -1 & -1 &  & \boxed{2} & \\
2 &  &  &-1 &  & -1 & -1 &  &  & \boxed{2}
\end{pmatrix}.
\]

From here on, we will freely work mod 2 in the $(i,j)$ entries for $8 \leq i \leq 10$, since there are boxed pivot entries $2$ in the $(j,j)$-entry of $P'''_\delta$ for all $8 \leq j \leq 10$. In particular, this allows us to clear all entries in row 8 other than the $(8,8)$-entry, since these are all even. We can also clear any even entries which occur in rows 9 and 10 throughout the remainder of this argument.

Our next goal is to clear the circled entries of $P_\delta'''$. For this, we use the $2$s in the bottom left corners of the second and third blocks, shown in bold above, using the following column operations. First add column 6 to column~4 and column~1, and then add twice the resulting column 4 to column 1, to obtain:
\begin{equation}\label{eq:Qdelta111223}
Q_\delta = 
\begin{pmatrix}
0 & -1 &0 &&&&&&& \\
0 & 1 & -1 &&&&&&& \\
2 & 0 & 0 &&&&&&& \\
 &  &1 & 0 & -1 &&&&&\\
 &  & & 2 & 0 &&&&&\\
& &  &  & 1 & 0 & -1 &&&\\
 &  & &  &  & 2 & 0 &&&\\
 &  & &  &  &  &  & 2 & &\\
1 &  &  & &  & 1 & 1 &  & 2 & \\
1 &  &  & &  & 1 & 1 &  &  & 2
\end{pmatrix}.
\end{equation}

Since we have only used column operations so far, a $\Z$-basis for $\Mod(w_\delta)$ can be read off from the columns of $Q_\delta$, as follows. The columns with $2$s in them yield
\[
\{ 2\alpha_3^\vee+(\alpha_9^\vee + \alpha_{10}^\vee), \quad 2\alpha_5^\vee, \quad 2\alpha_7^\vee + (\alpha_9^\vee + \alpha_{10}^\vee), \quad 2\alpha_8^\vee, \quad 2\alpha_9^\vee,\quad 2\alpha_{10}^\vee\}  \]
and then the remaining columns give the basis elements \[  \{ -\alpha_1^\vee + \alpha_2^\vee \}\  \cup\  \{ -\alpha_2^\vee+ \alpha^\vee_{4}, \ -\alpha^\vee_4 + \alpha_6^\vee \} \ \cup\  \{-\alpha_6^\vee + \alpha_9^\vee + \alpha_{10}^\vee \}.
\]
Now recall that $\delta = (1,1,1,2,2,3)$, so the sequence $(j^\delta_k)=(0,1,2,3,5,7)$ and the sequence $(\delta_{k,k-1})=(0,0,0,1,0,1)$. Hence, the first set in this basis equals
\[ 
\{ 2\alpha_{n-j^\delta_k}^\vee + \delta_{k,k-1}(\alpha_{n-1}^\vee + \alpha_{n}^\vee) \mid 1 \leq k \leq 2r\} \]
with $n = 10$ and $2r = 6$. We also have $I_\delta = \{ 3, 5, 7, 8, 9, 10 \}$ hence $I_\delta - 1 = \{ 2, 4, 6, 7, 8, 9 \}$, and the first $\ell = 3$ parts of $\delta$ equal $1$ while $\delta_{\ell + 1} \geq 2$, so 
\[
\{ -\alpha_1^\vee + \alpha_2^\vee \} = \{ -\alpha_i^\vee + \alpha_{i+1}^\vee \mid i \in [n],\; i \not \in I_\delta \cup (I_\delta - 1) \cup \{ n - \ell - 1\} \} \]
and 
\[
\{ -\alpha_2^\vee+ \alpha^\vee_{4}, -\alpha^\vee_4 + \alpha_6^\vee \} = \{ -\alpha_i^\vee + \alpha_{i+2}^\vee \mid i \in I_\delta - 1, i \leq n-\ell -2\},
\]
while 
$\{-\alpha_6^\vee + \alpha_9^\vee + \alpha_{10}^\vee \} = \{-\alpha_{n - \ell - 1}^\vee + \alpha_{n-1}^\vee + \alpha_{n}^\vee \}$, confirming part \eqref{D_delta1s_2} of \Cref{thm:DnCuspidal}.

To find the Smith normal form for $P_\delta$, we continue from $Q_\delta$, now using some row operations. Working from left to right, we use the $-1$s above the diagonal to successively clear all of the $1$s underneath them. This gives 
\[
Q_\delta' = 
\begin{pmatrix}
0 & -1 &0 &&&&&&& \\
0 & 0 & -1 &&&&&&& \\
2 & 0 & 0 &&&&&&& \\
&  & & 0 & -1 &&&&&\\
 &  & & 2 & 0 &&&&&\\
& &  &  &  & 0 & -1 &&&\\
 &  & &  &  & 2 & 0 &&&\\
 &  & &  &  &  &  & 2 & &\\
1 &  &  & &  & 1 &  &  & 2 & \\
1 &  &  & &  & 1 &  &  &  & 2
\end{pmatrix}.
\]

From here we extract the $6 \times 6$ minor of $Q'_\delta$ consisting of all rows and columns containing 2s. We then carry out a series of row and then column operations, as follows. First subtract row 6 from row 5, then subtract twice row 6 from row 1, and then add row 3 to row 1 to obtain the second matrix. Then use the boxed pivots to clear out the remaining entries in the same row, via column operations:
\[
\begin{pmatrix}
2 & 0 & 0 & 0 & 0 & 0\\
0 &2 &  0 & 0 & 0 & 0 \\
0 &0 &  2 & 0 & 0 & 0 \\
0 &0 &  0 & 2 & 0 & 0 \\
1 & 0 & 1 & 0 & 2 & 0\\
1 & 0 & 1 & 0 & 0 & 2\\
 \end{pmatrix} \to 
\begin{pmatrix}
0 & 0 & 0 & 0 & 0 & -4\\
0 &2 &  0 & 0 & 0 & 0 \\
0 &0 &  2 & 0 & 0 & 0 \\
0 &0 &  0 & 2 & 0 & 0 \\
0 & 0 & 0 & 0 & \boxed{2} & -2\\
\boxed{1} & 0 & 1 & 0 & 0 & 2\\
 \end{pmatrix} \to 
\begin{pmatrix}
0 & 0 & 0 & 0 & 0 & -4\\
0 &2 &  0 & 0 & 0 & 0 \\
0 &0 &  2 & 0 & 0 & 0 \\
0 &0 &  0 & 2 & 0 & 0 \\
0 & 0 & 0 & 0 & 2 & 0\\
1 & 0 & 0 & 0 & 0 & 0\\
 \end{pmatrix}.
\]
After scaling and rearrangement, this implies that $P_\delta$ has Smith normal form $\diag(1^5,2^4,4)$, confirming part \eqref{SNFD_gcd2_nonparity} of Theorem \ref{thm:DnSNF} in this case.
\end{example}


\begin{example}\label{eg:delta1334}
This example illustrates part \eqref{SNFD_gcd2_nonparity} of \Cref{thm:DnSNF} in the case $|\beta| = 0$, and part \eqref{D_delta12} of \Cref{thm:DnCuspidal}. Let $\sW$ be of type $D_{11}$ and let $|\beta| = 0$ and $\delta = (1,3,3,4)$. By \eqref{eq:wdelta1s2s} with $\ell=1$,  we have
 \[ \wbd = w_\delta = v^\delta_2 v^\delta_3 v^\delta_4 = (u_1 s_9 s_8)(u_4 s_6 s_5)(u_7 s_3 s_2 s_1).\]  
Applying \eqref{eq:siaction}, the matrix for $\Id - w_\delta$ with respect to the basis $\Delta^\vee$ is given by
\[
P_\delta = 
\begin{pmatrix}
2 & -1 &0 & 0 &&&&&&& \\
1 & 1 & -1& 0 &&&&&&& \\
1 & 0  & 1 & -1&&&&&&& \\
2 &  0 &0 & 0 &&&&&&& \\
2 &&& -1 & 2 & -1& 0 &&&&\\
2 &&& -1 & 1 & 1 & -1 &&&&\\
2 && & -2 & 2 & 0 & 0 &&&&\\
2 & &  & -2 & 2 &  & -1 & 2 & -1 &0&0\\
2 & &  & -2 & 2 &  & -1 & 1 &  1&-1&-1\\
1 &&  &-1 & 1 & & -1 & 1 & 0 & 1 &-1 \\
1 & & & -1 & 1 &  & -1 & 1 &0& -1 & 1  
\end{pmatrix}.
\]

Notice that $P_\delta$ is block lower-triangular with diagonal blocks of sizes $4,3,$ and $3+1$ going from left to right.
 That is, the sizes of the diagonal blocks of $P_\delta$ are the parts of $\delta$ in reverse order, except for the first two parts of $\delta$, which correspond to a single block of size $\delta_2 + \delta_1 = \delta_2 + 1$ on the right.
The diagonal blocks and entries underneath them corresponding to the parts $\delta_k$ for $k \geq 3$ are the same as described in Example~\ref{eg:delta111223} above. The rightmost diagonal block of size $\delta_2+1$ has first column $2,1,\dots,1$, diagonal entries $2,1,\dots, 1$, above-diagonal entries all $-1$, an additional $-1$ in  entries $(n,n-1)$ and $(n-2,n)$ of $P_\delta$, and all other entries $0$.  One can check via \eqref{eq:siaction} that this description in fact characterizes $P_\delta$ for all $|\delta| = n$ with $\delta_1 = 1$ and $\delta_2 \geq 2$; that is, the case considered in part \eqref{D_delta12} of \Cref{thm:DnCuspidal}.

We now carry out a series of column operations on $P_\delta$ to obtain a matrix from which a basis for $\Mod(w_\delta)$ can be read off; see the matrix $Q_\delta$ in \eqref{eq:Qdelta1334} below. We start with column operations on columns 1 and 4 similar to those done on columns 1 and 3 in the previous example. 
More specifically, we replace $\cC_1(P_\delta)$ by $\cC_1(P_\delta) + 2\cC_2(P_\delta) + 3\cC_3(P_\delta) + 4\cC_4(P_\delta)$, and then replace $\cC_4(P_\delta)$ by $\cC_4(P_\delta) + \cC_5(P_\delta)$, to obtain
\[
P_\delta' = 
\begin{pmatrix}
0 & -1 &0 & 0 &&&&&&& \\
0 & 1 & -1& 0 &&&&&&& \\
0 & 0  & 1 & -1&&&&&&& \\
2 & 0  &0 & 0 &&&&&&& \\
\circled{$-2$} &&& 1 & \circled{2} & -1& 0 &&&&\\
\circled{$-2$} && &  & \circled{1} & 1 & -1 &&&&\\
-6 && &  & 2 & 0 & 0 &&&&\\
-6 & &  &  & 2 &  & -1 & 2 & -1 &0&0\\
-6 & &  &  & 2 &  & \circled{$-1$} & 1 &  1&-1&-1\\
-3 &&  & & 1 & & \circled{$-1$}  & 1 & 0 & 1 &-1 \\
-3 & & &  & 1 &  & \circled{$-1$}  & 1 &0& -1 & 1  
\end{pmatrix}.
\]
We now clear the circled entries in column 1 using the second block, specifically by replacing $\cC_1(P_\delta')$ by $\cC_1(P_\delta') - 2(\cC_6(P_\delta')+2\cC_7(P_\delta'))$. Similar to how we began with the leftmost block, we now also clear the circled entries in columns 5 and 7 by  replacing $\cC_5(P_\delta')$ by $\cC_5(P_\delta') + 2\cC_6(P_\delta')+3\cC_7(P_\delta')$, and then replacing $\cC_7(P_\delta')$ by $\cC_7(P_\delta') + \cC_8(P_\delta')$.

In the last $4$ columns, corresponding to the parts $\delta_2 + \delta_1 = 3 + 1$ of $\delta$, the procedure is a little different. We replace $\cC_8(P_\delta')$ by $\cC_8(P_\delta') + 2\cC_9(P_\delta') + 3[\cC_{10}(P_\delta') + \cC_{11}(P_\delta') ]$, and then add column $10$ to column $11$. Altogether, this yields
\[
P_\delta'' = 
\begin{pmatrix}
0 & -1 &0 & 0 &&&&&&& \\
0 & 1 & -1& 0 &&&&&&& \\
0 & 0  & 1 & -1&&&&&&& \\
2 & 0  &0 & 0 &&&&&&& \\
0 &&& 1 & 0 & -1& 0 &&&&\\
0 && &  & 0 & 1 & -1 &&&&\\
-6 && &  & 2 & 0 & 0 &&&&\\
\circled{$-2$} & &  &  & \circled{$-1$} &  & 1 & 0 & -1 &0&0\\
-2 & &  &  & -1 &  &  & -3 &  1&-1&\boxed{-2}\\
\circled{$1$} &&  & & \circled{$-2$} & &  & 1 & 0 & 1 &0 \\
\circled{$1$} & & &  & \circled{$-2$} &  &  & 1 &0& -1 & 0  
\end{pmatrix}.
\]

From here on, we will freely work mod 2 in row 9, using the boxed pivot entry to clear all entries as much as possible. We now clear the circled entries in columns 1 and 5 of $P_\delta''$ by adding multiples of columns 8 and 9. We then add column 10 to column 8, to obtain
\[
P_\delta''' = 
\begin{pmatrix}
0 & -1 &0 & 0 &&&&&&& \\
0 & 1 & -1& 0 &&&&&&& \\
0 & 0  & 1 & -1&&&&&&& \\
2 & 0  &0 & 0 &&&&&&& \\
 &&& 1 & 0 & -1& 0 &&&&\\
 && &  & 0 & 1 & -1 &&&&\\
\circled{$-6$} && &  & \boxed{2} & 0 & 0 &&&&\\
 & &  &  &  &  & 1 & 0 & -1 &0&0\\
1 & &  &  &  &  & & 0 &  1&1&\circled{$-2$}\\
 0&&  & &  & &  & 2 & 0 & 1 &0 \\
 0& & &  &  &  &  &  && -1 & 0  
\end{pmatrix}.
\]

We then clear the circled entry in column 1 of $P_\delta'''$ using the boxed entry in column 5. In the rightmost block, we first use column 10 to clear the circled entry in column 11, then use column 8 to clear the resulting $(10,11)$-entry mod 2, and lastly multiply columns 10 and 11 by $-1$, to obtain
\begin{equation}\label{eq:Qdelta1334}
Q_\delta = 
\begin{pmatrix}
0 & -1 &0 & 0 &&&&&&& \\
0 & 1 & -1& 0 &&&&&&& \\
0 & 0  & 1 & -1&&&&&&& \\
2 & 0  &0 & 0 &&&&&&& \\
 &&& 1 & 0 & -1& 0 &&&&\\ 
 && &  & 0 & 1 & -1 &&&&\\
 && &  & 2 & 0 & 0 &&&&\\
 & &  &  &  &  & 1 & 0 & -1 &0&\\
1 & &  &  &  &  &  & 0 &  1&-1& \\
0 &&  & &  &&  & 2 & 0 & -1 & \\
0 & & &  &  &  &  &  && 1 & 2  
\end{pmatrix}.
\end{equation}

Since we have only used column operations, the following basis for $\Mod(w_\delta)$ can be read off from the columns of $Q_\delta$:
\[
 \{ 2\alpha_4^\vee + \alpha_9^\vee, \quad 2\alpha_7^\vee, \quad 2\alpha_{10}^\vee, \quad 2\alpha_{11}^\vee\} \] 
together with \[ \{ -\alpha_i^\vee + \alpha_{i+1}^\vee \mid i = 1,2,5,8\} \cup \{ -\alpha_i^\vee+ \alpha_{i+2} \mid i = 3,6\} \cup \{-\alpha_9^\vee - \alpha_{10}^\vee + \alpha_{11}^\vee \},
\]
confirming part \eqref{D_delta12} of Theorem \ref{thm:DnCuspidal}.

To find the Smith normal form for $P_\delta$, we continue from $Q_\delta$, now using some row operations. Working from left to right, we use the $-1$s in columns 2, 3, 4, 6, 7, and 9 to successively clear the $1$s underneath them, similar to how we obtained $Q'_\delta$ in Example \ref{eg:delta111223}.
From there we extract the $5 \times 5$ minor of $Q_\delta$ consisting of all columns containing 2s, in addition to column $n-1 = 10$, together with the rows containing the nonzero entries of these columns, resulting in the lefthand matrix below. We then carry out a series of row and then column operations, as follows:
\[
\begin{pmatrix}
2 & 0 & 0 & 0 & 0 \\
0 &2 &  0 & 0 & 0  \\
1 &0 &  0 & -1 & 0  \\
0 &0 &  2 & -1 & 0  \\
0 & 0 & 0 & 1 & 2 \\
 \end{pmatrix} \to 
\begin{pmatrix}
2 & 0 & 0 & 0 & 0 \\
0 &2 &  0 & 0 & 0  \\
\boxed{1} &0 &  0 & 0 & 2  \\
0 &0 &  2 & 0 & 0  \\
0 & 0 & 0 & 1 & 0 \\
 \end{pmatrix} \to 
\begin{pmatrix}
0 & 0 & 0 & 0 & -4 \\
0 &2 &  0 & 0 & 0  \\
1 &0 &  0 & 0 & 0  \\
0 &0 &  2 & 0 & 0  \\
0 & 0 & 0 & 1 & 0 \\
 \end{pmatrix}.
\]
On this minor, we first add row 5 to rows 3 and 4, and then use columns 3 and 4 to clear the entries in column 5, to obtain the second matrix. Then use the boxed entry in column 1 to clear the 2 in column 5. Finally, use the same boxed entry in row 3 to clear the 2 in row 1, to obtain the final matrix.
After scaling and rearrangement, this implies that $P_\delta$ has Smith normal form $\diag(1^8,2^2,4)$, confirming part \eqref{SNFD_gcd2_nonparity} of Theorem \ref{thm:DnSNF} in this case.
\end{example}


\begin{example}\label{eg:delta2333}  
This example illustrates part \eqref{SNFD_gcd2_nonparity} of \Cref{thm:DnSNF} in the case $|\beta| =0$, and part \eqref{D_delta2} of \Cref{thm:DnCuspidal}. Let $\sW$ be of type $D_{11}$ and let $|\beta| = 0$ and $\delta = (2,3,3,3)$. By \eqref{eq:wdelta3s}, we have
 \[ \wbd = w_\delta = u_0 v^\delta_2 v^\delta_3 v^\delta_4 = u_0(u_2 s_8 s_7)(u_5 s_5 s_4)(u_{8} s_2 s_1).\]  
 Applying \eqref{eq:siaction}, the matrix for $\Id - w_\delta$ with respect to the basis $\Delta^\vee$ is given by
\[
P_\delta = 
\begin{pmatrix}
2 & -1  & 0 & &&&&&&& \\
1 & 1& -1 & &&&&&&& \\
2 & 0  & 0 & &&&&&&& \\
2 &   & -1 &2 &-1&0&&&&& \\
2 & & -1 & 1 & 1& -1 &&&&&\\
2 &&  -2 & 2 & 0 & 0 &&&&&\\
2 & & -2 & 2 &  & -1 &2&-1&0&&\\
2 & &  -2 & 2 &  & -1 & 1 & 1 &-1&&\\
2 & &  -2 & 2 &  & -2 & 2 & 0 &0&&\\
1 & &-1 & 1 & & -1 & 1 &  & 0 &1 & -1\\
1 & & -1 & 1 &  & -1 & 1 && -1 & 1  &1
\end{pmatrix}.
\]

Notice that $P_\delta$ is block lower-triangular with diagonal blocks of sizes $3,3,3,2$ going from left to right, obtained from the parts of $\delta$ in reverse order.  For $k \geq 3$, both the diagonal blocks and the entries underneath them corresponding to the parts $\delta_k$ are the same as described in Example~\ref{eg:delta111223} above. The diagonal block corresponding to~$\delta_2$ is the same as in Example~\ref{eg:delta111223} above, and the diagonal block for $\delta_1 = 2$ is the $2 \times 2$ matrix in the lower right depicted in $P_\delta$. Finally, in the $2$ rows underneath the block for $\delta_2$, the leftmost column has entries $1,1$, the rightmost column has entries $0,-1$, and all other entries equal~$0$. 
One can check via \eqref{eq:siaction} that this description in fact characterizes $P_\delta$ for all $|\delta| = n$ with $\delta_1 =2$; that is, the case considered in part \eqref{D_delta2} of \Cref{thm:DnCuspidal}.

We now carry out a series of column operations on $P_\delta$ to obtain a matrix from which a basis for $\Mod(w_\delta)$ can be read off; see the matrix $Q_\delta$ in \eqref{eq:Qdelta2333} below. We start with column operations on  columns 1, 3, 4, 6, and 7 modeled upon those in the two previous examples, clearing each of these columns as much as possible, to obtain
\[
P_\delta' = 
\begin{pmatrix}
0 & -1  & 0 & &&&&&&& \\
0 & 1& -1 & &&&&&&& \\
2 & 0  & 0 & &&&&&&& \\
 &   & 1 &0 &-1&0&&&&& \\
 & &  & 0 & 1& -1 &&&&&\\
 &&   & 2 & 0 & 0 &&&&&\\
 & &  &  &  & 1 &0&-1&0&&\\
 & &   &  &  &  & 0 & 1 &-1&&\\
 & &   &  &  &  & 2 & 0 &0&&\\
 & & &  & &  &  &  & 0 &1 & -1\\
\circled{$-20$} &  &  & \circled{$-4$} &  &  & \circled{$-3$} && -1 & 1  &1
\end{pmatrix}.
\]

The process of clearing the bottom row of $P'_\delta$ is somewhat unique to the case of $\delta_2 = 2$. We first add column 10 to column 9, and then add column 11 to column 10, resulting in the boxed entry below, which is used to clear the circled entries in the bottom row of $P'_\delta$ mod 2, giving us
\begin{equation}\label{eq:Qdelta2333}
Q_\delta = 
\begin{pmatrix}
0 & -1  & 0 & &&&&&&& \\
0 & 1& -1 & &&&&&&& \\
2 & 0  & 0 & &&&&&&& \\
 &   & 1 &0 &-1&0&&&&& \\
 & &  & 0 & 1& -1 &&&&&\\
 &&   & 2 & 0 & 0 &&&&&\\
 & &  &  &  & 1 &0&-1&0&&\\
 & &   &  &  &  & 0 & 1 &-1&&\\
 & &   &  &  &  & 2 & 0 &0&&\\
 & & &  & &  &  &  & 1 & 0& -1\\
 &  &  &  &  &  & 1 &  &  & \boxed{2}  &1
\end{pmatrix}.
\end{equation}

Since we have only used column operations, the following basis for $\Mod(w_\delta)$ can be read off from the columns of $Q_\delta$. The columns with 2s yield
\[
 \{ 2\alpha_3^\vee, \quad 2\alpha_6^\vee, \quad 2\alpha_{9}^\vee + \alpha_{11}^\vee, \quad 2\alpha_{11}^\vee\}, \]
 and the remaining columns give us
  \[ \{ -\alpha_i^\vee + \alpha_{i+1}^\vee \mid i = 1,4,7,10\} \cup \{ -\alpha_i^\vee+ \alpha_{i+2} \mid i = 2,5,8\}.
\]
In particular, since $\delta = (2,3,3,3)$, then we obtain the sequences $(j_k^\delta) = (0,2,5,8)$ and $(\delta_{k,k-1}) = (0,1,0,0)$. Hence, the first set in this basis equals 
\[ \left\{ 2 \alpha_{n-j^\delta_k}^\vee + \delta_{k,k-1}\alpha_{n}^\vee \ \middle| \  1 \leq k \leq 2r\right\} \]
with $n=11$ and $2r = 4$. We also have $I_\delta = \{3,6,9,11\}$ and hence $I_\delta - 1 = \{2,5,8,10\}$, and so the second pair of sets in the basis obtained from $Q_\delta$ coincides with
\[ \left\{ -\alpha_i^\vee + \alpha_{i+1}^\vee \ \middle| \  i \not \in I_\delta \cup (I_\delta - 1)  \ \text{or} \ i=n-1 \right\} \ \cup \  \left\{ -\alpha_{i}^\vee + \alpha_{i+2}^\vee \ \middle|\  i \in (I_\delta-1), \; i \neq n-1 \right\},\]
confirming part \eqref{D_delta2} of Theorem \ref{thm:DnCuspidal}.

To find the Smith normal form for $P_\delta$, we continue from $Q_\delta$. Working from left to right, we use the $-1$s above the diagonal to successively clear the $1$s underneath them. This yields 
\[
Q_\delta' = 
\begin{pmatrix}
0 & -1  & 0 & &&&&&&& \\
0 & 0& -1 & &&&&&&& \\
2 & 0  & 0 & &&&&&&& \\
 &   &  &0 &-1&0&&&&& \\
 & &  & 0 & 0& -1 &&&&&\\
 &&   & 2 & 0 & 0 &&&&&\\
 & &  &  &  &  &0&-1&0&&\\
 & &   &  &  &  & 0 & 0 &-1&&\\
 & &   &  &  &  & 2 & 0 &0&&\\
 & & &  & &  &  &  &  & 0& -1\\
 &  &  &  &  &  & 1 &  &  & 2  &0
\end{pmatrix}.
\]
From here, noting that $\begin{pmatrix} 2 & 0 \\ 1 & 2 \end{pmatrix} \to \begin{pmatrix} 0 & -4 \\ 1 & 0 \end{pmatrix}$ via one row and one column operation, it is easy to see that $P_\delta$ has Smith normal form $\diag(1^8,2^2,4)$, confirming part \eqref{SNFD_gcd2_nonparity} of Theorem \ref{thm:DnSNF} in this case.
\end{example}


\begin{example}\label{eg:delta3333} 
This example illustrates part \eqref{SNFD_gcd2_parity} of \Cref{thm:DnSNF} in the case $|\beta| = 0$, and part \eqref{D_delta3} of \Cref{thm:DnCuspidal}.  Let $\sW$ be of type $D_{12}$ and let $|\beta| = 0$ and $\delta = (3,3,3,3)$. By \eqref{eq:wdelta3s}, we have
 \[ \wbd = w_\delta = (u_0 s_{10})v^\delta_2 v^\delta_3 v^\delta_4 = (u_0 s_{10})(u_3 s_8 s_7)(u_6 s_5 s_4)(u_9 s_2 s_1).\]
Applying \eqref{eq:siaction}, the matrix for $\Id - w_\delta$ with respect to the basis $\Delta^\vee$ is given by
\[
P_\delta = 
\begin{pmatrix}
2 & -1  & 0 & &&&&&&&& \\
1 & 1& -1 & &&&&&&&& \\
2 & 0  & 0 & &&&&&&&& \\
2 &   & -1 &2 &-1&0&&&&&& \\
2 & & -1 & 1 & 1& -1 &&&&&&\\
2 &&  -2 & 2 & 0 & 0 &&&&&&\\
2 & & -2 & 2 &  & -1 &2&-1&0&&&\\
2 & &  -2 & 2 &  & -1 & 1 & 1 &-1&&&\\
2 & &  -2 & 2 &  & -2 & 2 & 0 &0&&&\\
2 & &  -2 & 2 &  & -2 & 2 &  &-1&2&-1&-1\\
1 & &-1 & 1 & & -1 & 1 &  & 0 & 0 &1 & -1\\
1 & & -1 & 1 &  & -1 & 1 && -1 & 1& 0  &0
\end{pmatrix}.
\]

As in most previous examples, the sizes of the diagonal blocks of $P_\delta$ are the parts of $\delta$ in reverse order. The diagonal blocks and entries underneath them corresponding to the parts $\delta_k$ for $k \geq 3$ are the same as described in Example~\ref{eg:delta111223} above. The diagonal block corresponding to~$\delta_2$ is also the same as in Example~\ref{eg:delta111223}. Underneath the $\delta_2$ block, the leftmost column has entries $2,\dots,2,1,1$ like other blocks to its left; however, the rightmost column has entries $-1,\dots,-1,0,-1$, and all other entries equal~$0$. The diagonal block for $\delta_1$ is the $\delta_1 \times \delta_1$ matrix which has first column $(2,1^{\delta_1 - 3},0,1)$, diagonal entries $2,1, \dots, 1, 0$, above-diagonal entries $-1$, and an additional $-1$ entry in the $(n-2,n)$ position.  One can check via \eqref{eq:siaction} that this description characterizes $P_\delta$ for all $|\delta| = n$ with $\delta_1 \geq 3$; that is, the case considered in part \eqref{D_delta3} of \Cref{thm:DnCuspidal}.

We now carry out a series of column operations on $P_\delta$ to obtain a matrix $Q_\delta$ from which a basis for $\Mod(w_\delta)$ can be read off; see \eqref{eq:Qdelta3333} below.  We start with column operations on columns 1, 3, 4, 6, 7, and 9 modeled upon those in previous examples, clearing each of these columns as much as possible, to obtain
\[
P_\delta' = 
\begin{pmatrix}
0 & -1  & 0 & &&&&&&&& \\
0 & 1& -1 & &&&&&&&& \\
2 & 0  & 0 & &&&&&&&& \\
 &   & 1 &0 &-1&0&&&&&& \\
&&& 0 & 1& -1 &&&&&&\\
&&& 2 & 0 & 0 &&&&&&\\
&&&  &  & 1 &0&-1&0&&&\\
&&& &&& 0 & 1 &-1&&&\\
&&& &&& 2 & 0 &0&&&\\
&&& &&&  &  &1&2&-1&-1\\
-32&&& -8&&& &&& 0 &1 & -1\\
-6&&&-6&&&-2 &&& 1& 0  &0
\end{pmatrix}.
\]
After the initial steps of replacing $\cC_1(P_\delta)$ by $\cC_1(P_\delta) + 2\cC_2(P_\delta) + 3\cC_3(P_\delta)$ and $\cC_3(P_\delta)$ by $\cC_3(P_\delta) + \cC_4(P_\delta)$ and so on, we point out that all subsequent column operations clearing columns 1, 4, and 7 below the diagonal result in even numbers in these columns, due to the absence of any change of parity in $\delta$. The result is that we are able to completely clear columns 1, 4, and 7 in this example; compare $P'_\delta$ in Example \ref{eg:delta2333}, where the odd number in column 7 results from the change in parity between the first two parts of $\delta = (2,3,3,3)$.

The process of clearing the bottom $\delta_1-1$ rows in columns 1, 4, and 7 now involves several column operations unique to the case of $\delta_1 \geq 3$. In the last block, we change its first column to $0,\dots, 0,\delta_1-1,1$ by replacing its first column, which has index $n-\delta_1 + 1$, by the sum $\cC_{n-\delta_1 + 1}(P_\delta) + 2 \cC_{n-\delta_1 + 2}(P_\delta) + 3 \cC_{n-\delta_1 + 3}(P_\delta) + \cdots + (\delta_1 - 1) \cC_{n-1}(P_\delta)$. We then subtract column $n-1$ from column $n$. The resulting boxed entries in the $\delta_1$ block can be used to clear the remaining entries mod 2 in row $n-1$, and completely in row $n$, to obtain
\begin{equation}\label{eq:Qdelta3333}
Q_\delta = 
\begin{pmatrix}
0 & -1  & 0 & &&&&&&&& \\
0 & 1& -1 & &&&&&&&& \\
2 & 0  & 0 & &&&&&&&& \\
 &   & 1 &0 &-1&0&&&&&& \\
&&& 0 & 1& -1 &&&&&&\\
&&& 2 & 0 & 0 &&&&&&\\
&&&  &  & 1 &0&-1&0&&&\\
&&& &&& 0 & 1 &-1&&&\\
&&& &&& 2 & 0 &0&&&\\
&&& &&&  &  &1&0&-1&0\\
&&& &&& &&& 0 &1 & \boxed{-2}\\
&&&&&& &&& \boxed{1}& 0  &0
\end{pmatrix}.
\end{equation}

Since we have only used column operations, the following basis for $\Mod(w_\delta)$ can be read off from the columns of $Q_\delta$.  The columns with 2s, together with column $n-\delta_1+1$, yields
\[
 \{ 2\alpha_3^\vee, \quad 2\alpha_6^\vee, \quad 2\alpha_{9}^\vee\} \ \cup \ \{ 2\alpha_{11}^\vee, \quad \alpha_{12}^\vee\}, \] 
 and the remaining columns give us 
 \[ \{ -\alpha_i^\vee + \alpha_{i+1}^\vee \mid i = 1,4,7,10\} \cup \{ -\alpha_i^\vee+ \alpha_{i+2} \mid i = 2,5,8\}.
\]
Since $\delta= (3,3,3,3)$ has no parity changes, then $(\delta_{k,k-1}) = (0,0,0,0)$, and $(j_k^\delta) = (0,3,6,9)$. Therefore, the first sets agree with 
\[ \left\{ 2 \alpha_{n-j^\delta_k}^\vee + \delta_{k,k-1}\alpha_{n-1}^\vee \ \middle| \  2 \leq k \leq 2r \right\} \cup 
\left\{ 2\alpha_{n-1}^\vee,\; ((\delta_1 - 1)\operatorname{mod} 2) \alpha_{n-1}^\vee + \alpha_n^\vee \right\}\]
for $n=12$ and $2r=4$. In this example, $I_\delta = \{3,6,9,12\}$ and so $I_\delta - 1 = \{2,5,8,11\}$. The second pair of sets in the basis obtained from $Q_\delta$ thus coincides with
\[ \left\{ -\alpha_i^\vee + \alpha_{i+1}^\vee \ \middle| \ i \not \in I_\delta \cup (I_\delta - 1)  \right\} \cup   \left \{ -\alpha_{i}^\vee + \alpha_{i+2}^\vee \ \middle| \ i \in (I_\delta-1),\; i \neq n-1 \right\},\]
confirming part \eqref{D_delta3} of Theorem \ref{thm:DnCuspidal}.

To find the Smith normal form for $P_\delta$ from $Q_\delta$, we work left to right as in previous examples, using row operations with the $-1$s above the diagonal to successively clear the 1s below, giving us
\[
Q_\delta' = 
\begin{pmatrix}
0 & -1  & 0 & &&&&&&&& \\
0 & 0& -1 & &&&&&&&& \\
2 & 0  & 0 & &&&&&&&& \\
 &   &  &0 &-1&0&&&&&& \\
&&& 0 & 0& -1 &&&&&&\\
&&& 2 & 0 & 0 &&&&&&\\
&&&  &  &  &0&-1&0&&&\\
&&& &&& 0 & 0 &-1&&&\\
&&& &&& 2 & 0 &0&&&\\
&&& &&&  &  &&0&-1&0\\
&&& &&& &&& 0 &0 & -2\\
&&&&&& &&& 1& 0  &0
\end{pmatrix}.
\]
Therefore, $P_\delta$ has Smith normal form $\diag(1^6,2^4)$, confirming part \eqref{SNFD_gcd2_parity} of Theorem \ref{thm:DnSNF} in this case.
\end{example}


\subsubsection{Noncuspidal examples with $\beta_p = 1$}\label{sec:noncuspidal2}
We now consider the situation where $\beta_p = 1$.  This includes the special case $\beta = (1)$. Recall from \Cref{sec:repsD} that when $|\beta| = 1$, we have $w_\beta$ trivial and $J_\beta = \emptyset$ (as in the case $|\beta| = 0$). 
The following examples illustrate the proof of part~\eqref{D_gcd1} of \Cref{thm:DnSNF}, and how to obtain a basis in this case; see the introduction to Section \ref{sec:examplesD} for the list of examples. 


\begin{example}\label{eg:beta1delta1111}
Let $\sW$ be of type $D_n$ with $n \geq 5$ odd, and let $\beta = (1)$ and $\delta = (1,\dots,1)$ with $n-1$ parts. We prove this special case of part \eqref{D_gcd1} of \Cref{thm:DnSNF}, and find a basis. Recall from \eqref{eq:wdelta1s} that 
\[
\wbd = w_\delta = u_1 u_2 \dots u_{n-2},
\]
which equals the longest element in the type $D_{n-1}$ subsystem indexed by the last $n-1$ nodes; see Table 1 in \cite{BenkartKangOhPark}, for example. Therefore, $w_\delta$ acts as $(-\Id)$ on the submodule of $R^\vee$ spanned by $\{\alpha_i^\vee \mid 2 \leq i \leq n\}$. 

From \eqref{eq:siaction}, the matrix for $\Id - w_\delta$ with respect to $\Delta^\vee$ is therefore
\[
\Pd  = 
\begin{pmatrix}
0 &&&&& \\
-2 & 2 &  & & &\\
\vdots & & \ddots &&& \\
-2 &  & & 2 & & \\
-1 &  & & &2 &  \\
-1 &  &  & & & 2
\end{pmatrix}.
\]
That is, first row of this matrix is all $0$s, the $(n-1) \times (n-1)$ minor in the bottom right is~$2\Id$, and the first column is $0,-2,\dots,-2,-1,-1$. 

We now add columns $2$ through $n$ to column $1$ to obtain $P_\delta'$, and then use two further column operations to clear the circled entry and obtain $\Qd$ below:
\[
\Pd'  = 
\begin{pmatrix}
0 &&&&& \\
0 & 2 &  & & &\\
\vdots & & \ddots &&& \\
0 &  & & 2 & & \\
1 &  & & &2 &  \\
1 &  &  & & & \circled{2}
\end{pmatrix}
\quad \to \quad
\Qd = 
\begin{pmatrix}
0 &&&&& \\
0 & 2 &  & & &\\
\vdots & & \ddots &&& \\
0 &  & & 2 & & \\
1 &  & & &2 &  \\
1 &  &  & & & 0
\end{pmatrix}.
\]
A basis for $\Mod(\wbd)$ can then be read off from $\Qd$; namely $\{ 2\alpha_2^\vee, \dots, 2 \alpha_{n-1}^\vee \} \cup \{ \alpha_{n-1}^\vee + \alpha_n\}$.

Now on the matrix $\Qd$, we subtract row $n$ from row $n-1$ to obtain
\[
\Qd' = 
\begin{pmatrix}
0 &&&&& \\
0 & 2 &  & & &\\
\vdots & & \ddots &&& \\
0 &  & & 2 & & \\
0 &  & & &2 &  \\
1 &  &  & & & 0
\end{pmatrix}.
\]
The Smith normal form $(1,2^{n-2},0)$ can be seen immediately from $\Qd'$, 
which establishes this case of part \eqref{D_gcd1} of \Cref{thm:DnSNF}.
\end{example}


\begin{example}\label{eg:beta1delta1112} 
Let $\sW$ be of type $D_6$ and let $\beta = (1)$ and $\delta = (1,1,1,2)$.  This example illustrates part \eqref{D_gcd1} of \Cref{thm:DnSNF} and how to find a basis, in the case that $\beta = (1)$, $\delta_1 = \delta_2 = 1$, and $\delta \neq (1,\dots,1)$. By \eqref{eq:wdelta1s2s} with $\ell = 3$, 
\[
\wbd = w_\delta = u_1 u_2 v_4^\delta = u_1 u_2 (u_3 s_2).
\]
Applying \eqref{eq:siaction}, the matrix for $\Id - w_\delta$ with respect to $\Delta^\vee$ is 
\[
\Pd  = 
\begin{pmatrix}
0 &0&0&&& \\
-1 & 2 & -1 & & & \\
-2 & 2 & 0 &  & & \\
-2 & 2 & -2 & 2 && \\
 -1& 1 & -1 &  & 2&\\
-1 & 1 & -1 &  & & 2
\end{pmatrix}.
\]
Notice that the first row of this matrix is all $0$s, while the $(n-1) \times (n-1)$ minor in the bottom right is the same as the matrix for the cuspidal element $w_\delta$ in the subsystem of type $D_{n-1}$ indexed by the last $n-1$ nodes. 
The first column of $\Pd$ is the negative of its second column, except in the $(2,1)$-entry, which equals $-1$. The same holds whenever $\beta = (1)$. 

We now add column $2$ to column $1$, so that the $(2,1)$-entry becomes the boxed pivot entry~$1$. We then carry out the same column operations on the $(n-1)\times(n-1)$ minor in the bottom right as we did in the cuspidal case (see Example~\ref{eg:delta111223}), to obtain
\[
\Pd'  = 
\begin{pmatrix}
0 &&&&& \\
\boxed{1} & 0 & \circled{$-1$} & & & \\
 & 2 & 0 &  & & \\
 & 0 &0  & 2 && \\
& \circled{$1$} & \boxed{1} &  & 2&\\
 & \circled{$1$} & \boxed{1} &  & & 2
\end{pmatrix}.
\]
Now, working from left to right, we can apply column operations which use the boxed entries to clear the circled entries of $P'_\delta$ to  obtain the matrix $P_\delta''$ 
\[
P_\delta'' = 
\begin{pmatrix}
0 &&&&& \\
1 & 0 & 0 & & & \\
 & 2 & 0 &  & & \\
 & 0 & 0 & 2 && \\
&  & 1 &  & 2&\\
 &  & 1 &  & & 2
\end{pmatrix}
\quad \to \quad
\Qd = 
\begin{pmatrix}
0 &&&&& \\
1 & 0 & 0 & & & \\
 & 2 & 0 &  & & \\
 & 0 & 0 & 2 && \\
&  & 1 &  & 2&\\
 &  & 1 &  & & 0
\end{pmatrix}
\]
Finally, use columns 3 and 5 to clear column $n=6$, which yields the matrix $\Qd$, from which we can read off a basis for $\Mod(\wbd)$, namely  $\{ \alpha_2^\vee , 2\alpha_3^\vee \} \cup \{  2\alpha_4^\vee, 2\alpha_5^\vee \} \cup \{ \alpha_{5}^\vee + \alpha_6\}$. 

From here, observe that $\begin{pmatrix} 1 & 2 \\ 1 & 0 \end{pmatrix} \to \begin{pmatrix} 0 & 2 \\ 1 & 0 \end{pmatrix}$ via one row operation. Therefore, $P_\delta$ clearly has Smith normal form $\diag(1^2,2^3,0)$, confirming part \eqref{D_gcd1} of Theorem \ref{thm:DnSNF}.
\end{example}


\begin{example}\label{eg:beta1delta13} 
Let $\sW$ be of type $D_5$ and let $\beta = (1)$ and $\delta = (1,3)$.  This example illustrates part \eqref{D_gcd1} of \Cref{thm:DnSNF}, and how to find a basis, in the case $\beta = (1)$, $\delta_1 = 1$, and $\delta_2 \geq 2$. By \eqref{eq:wdelta1s2s} with $\ell = 1$, we have 
\[
\wbd = w_\delta = v_2^\delta = u_1 s_3 s_2.
\]
Applying \eqref{eq:siaction}, the matrix for  $\Id - w_\delta$ with respect to $\Delta^\vee$ is
\[
\Pd  = 
\begin{pmatrix}
0 &&&& \\
-1 & 2 & -1 & 0 & 0\\
-1 & 1 & 1 & -1 & -1 \\
-1 & 1 & 0 & 1 & -1\\
-1 & 1 & 0 & -1 & 1
\end{pmatrix}.
\] 
We now add column $2$ to column $1$, so that the $(2,1)$-entry becomes a pivot entry $1$. We then carry out the same column operations on the $(n-1)\times(n-1)$ minor in the bottom right as we did in the cuspidal case (see Example~\ref{eg:delta1334}), to obtain
\[
\Pd'  = 
\begin{pmatrix}
0 &&&& \\
\boxed{1} & 0 & \circled{$-1$} & 0 & \\
 & 0 &\boxed{1} & \circled{$-1$} &  \\
 & 2 & 0 & -1 & \\
 &  &  & 1 & 2
\end{pmatrix}.
\]
Compare the bottom right $4\times 4$ minor of $\Qd$ in \eqref{eq:Qdelta1334}. Working from left to right, we can apply column operations which use the boxed entries to clear the circled entries of $P_\delta'$, yielding
\[
\Pd''  = 
\begin{pmatrix}
0 &&&& \\
1 & 0 & 0 & 0 & \\
 & 0 & 1 & 0 &  \\
 & 2 & 0 & -1 & \\
 &  &  & 1 & \boxed{2}
\end{pmatrix}.
\]
We now add twice column 4 to column 2, and then used the boxed pivot entry $2$ in column $n$ of $\Pd''$ to clear the resulting $(n,2)$-entry. This gives us
\[
\Qd  = 
\begin{pmatrix}
0 &&&& \\
\boxed{1} & 0 & 0 & 0 & \\
 & 0 & \boxed{1} & 0 &  \\
 & 0 & 0 & -1 & \\
 &  &  & \circled{1} & \boxed{2}
\end{pmatrix}.
\]
A basis for $\Mod(\wbd)$ given by $\{\alpha_2^\vee, \alpha_3^\vee \} \cup \{ -\alpha_4^\vee + \alpha_5^\vee\} \cup \{2\alpha_5^\vee \}$  can now be read off $\Qd$. 

To obtain the Smith normal form of $\Pd$, we add row 4 to row 5, which clears the circled entry of $\Qd$ and gives us
\[
\Qd'  = 
\begin{pmatrix}
0 &&&& \\
\boxed{1} & 0 & 0 & 0 & \\
 & 0 &\boxed{1} & 0 &  \\
 & 0 & 0 & \boxed{-1} & \\
 &  &  &  & \boxed{2}
\end{pmatrix}.
\]
Therefore, $P_\delta$ has Smith normal form $\diag(1^3,2^1,0)$, confirming part \eqref{D_gcd1} of Theorem \ref{thm:DnSNF}.
\end{example}


\begin{example}\label{eg:beta1delta22}
Let $\sW$ be of type $D_5$ and let $\beta = (1)$ and $\delta = (2,2)$. This example illustrates part \eqref{D_gcd1} of \Cref{thm:DnSNF} and shows how to obtain a basis, in the case that $\beta = (1)$ and $\delta$ has first part $\delta_1 = 2$. By \eqref{eq:wdelta3s}, we have 
\[
\wbd = w_\delta = u_0 v_2^\delta = u_0 u_2 s_2,
\]
and so by \eqref{eq:siaction} the matrix for $\Id - w_\delta$ with respect to $\Delta^\vee$ is
\[
\Pd  = 
\begin{pmatrix}
0 &&&& \\
-1 & 2 & -1 &  & \\
-2 & 2 & 0 &  &  \\
-1 & 1 & 0 & 1 & -1\\
-1 & 1 & -1 & 1 & 1
\end{pmatrix}.
\]
We now add column $2$ to column $1$, so that the $(2,1)$-entry becomes a pivot entry $1$. We then carry out the same column operations on the $(n-1)\times(n-1)$ minor in the bottom right as we did in the cuspidal case (see Example~\ref{eg:delta2333}), to obtain
\[
\Pd'  = 
\begin{pmatrix}
0 &&&& \\
\boxed{1} & 0 & \circled{$-1$} &  & \\
 & 2 & 0 &  &  \\
 &  & \boxed{1} & 0 & \circled{$-1$}\\
 &  & & \circled{2} & \boxed{1}
\end{pmatrix}
\quad \to \quad
\Qd  = 
\begin{pmatrix}
0 &&&& \\
1 & 0 & 0 &  & \\
 &2 & 0 &  &  \\
 &  & 1 & 0 & 0\\
 &  & & 0 & 1
\end{pmatrix}.
\]
Now, working from left to right, we can apply further column operations which use the boxed entries of $P_\delta'$ to clear the circled entries, yielding $\Qd$. A basis for $\Mod(\wbd)$ given by $\{\alpha_2^\vee, 2\alpha_3^\vee \} \cup \{ \alpha_4^\vee, \alpha_5^\vee\}$  can now be read off from $\Qd$. 

Moreover, we can easily see from $\Qd$ that $P_\delta$ has Smith normal form $\diag(1^3, 2^1, 0)$, confirming part \eqref{D_gcd1} of Theorem \ref{thm:DnSNF}.
\end{example}


\begin{example}\label{eg:beta1delta33}
Let $\sW$ be of type $D_7$ and let $\beta = (1)$ and $\delta = (3,3)$. This example illustrates part \eqref{D_gcd1} of \Cref{thm:DnSNF} and how to find a basis, in the case that $\beta = (1)$ and $\delta$ has first part $\delta_1\geq 3$. By \eqref{eq:wdelta3s}, we have 
\[
\wbd = w_\delta = (u_0 s_5)v^\delta_2 = (u_0 s_5)( u_3 s_3 s_2).
\]
By \eqref{eq:siaction}, the matrix for $\Id - w_\delta$ with respect to $\Delta^\vee$ is
\[
\Pd  = 
\begin{pmatrix}
0 &&&&&& \\
-1 & 2 & -1 & 0 & &&\\
-1& 1 & 1 & -1 &&& \\
-2 & 2 & 0 & 0 &&& \\
-2 & 2 & & -1 &2 &-1& -1\\
-1 & 1 && 0 &0& 1 & -1\\
-1 & 1 && -1 & 1 & 0 & 0  
\end{pmatrix}.
\]
We now add column $2$ to column $1$, so that the $(2,1)$-entry becomes a pivot entry $1$. We then carry out the same column operations on the $(n-1)\times(n-1)$ minor in the bottom right as we did in the cuspidal case (see Example~\ref{eg:delta3333}), to obtain $P'_\delta$, which agrees with the lower right $6\times 6$ minor in \eqref{eq:Qdelta3333}.
\[
\Pd'  = 
\begin{pmatrix}
0 &&&&&& \\
\boxed{1} & 0 & \circled{$-1$} & 0 & &&\\
& 0 & \boxed{1} & \circled{$-1$} &&& \\
 & 2 & 0 & 0 &&& \\
 &  & & \boxed{1} &0 &\circled{$-1$} & 0\\
 &  &&  &0& \boxed{1} & \circled{$-2$}\\
 &  &&  & 1 & 0 & 0  
\end{pmatrix}
\quad \to \quad
\Qd  = 
\begin{pmatrix}
0 &&&&&& \\
1 & 0 & 0 & 0 & &&\\
& 0 & 1 & 0 &&& \\
 & 2 & 0 & 0 &&& \\
 &  & & 1 &0 &0 & 0\\
 &  &&  &0& 1 & 0\\
 &  &&  & 1 & 0 & 0  
\end{pmatrix}
\]
Now, working from left to right, we can apply column operations which use the boxed entries to clear the circled entries of $P'_\delta$, yielding $\Qd$ above. A basis for $\Mod(\wbd)$ given by $\{\alpha_2^\vee, \alpha_3^\vee, 2\alpha_4^\vee \} \cup \{ \alpha_5^\vee, \alpha_6^\vee, \alpha_7^\vee\}$  can now be read off from $\Qd$. 

Moreover, we can easily see from $\Qd$ that $P_\delta$ has Smith normal form $\diag(1^5, 2^1, 0)$, confirming part \eqref{D_gcd1} of Theorem \ref{thm:DnSNF}.
\end{example}


\begin{example}\label{eg:beta31delta12} 
This example illustrates part \eqref{D_gcd1} of \Cref{thm:DnSNF} in the case that $\beta_p = 1$ and $\beta$ has more than one part, and shows how to obtain a basis.

Let $\sW$ be of type $D_{7}$, and let $\beta = (3,1)$ and $\delta = (1,2)$. By Proposition \ref{prop:wbetaA} and Equation \eqref{eq:wdelta1s2s}, we have
\[
\wbd  = (s_1 s_2)(u_1 s_5).
\]
Let $|\beta| = m \leq n-2$ (so that $m = 4$ in this example). Then since $\beta_p = 1$, the element $w_\beta$ has support in the type $A_{m-2}$ subsystem indexed by the first $m-2$ nodes, so in particular, $w_\beta$ fixes $\alpha_i^\vee$ for $m \leq i \leq n$. Moreover, letting $\beta' = (3,1,1)$ be the partition obtained by adding last part $1$ to $\beta$, the element $w_\beta$ acts on $\alpha_{m-1}^\vee$ exactly as does the element $w_{\beta'}$ in type $A_{m-1}$. Thus applying \eqref{eq:siaction}, the matrix for $\Id - \wbd$ with respect to $\Delta^\vee$ is given by $\Pbd$ below:
\[
\Pbd  = 
\begin{pmatrix}
1 & 1 & -1 & 0 & &&\\
-1 & 2 & -1 & 0 &&& \\
0 & 0 & 0 & 0 &&& \\
0 & 0 & 0 & 0 & 0 & 0 & 0\\
 &  &  & -1 & 2 & -1 & -1 \\
  &  &  & -1 & 1 & 1 & -1 \\
   &  &  & -1 & 1 & -1 & 1 
\end{pmatrix}.
\]
Notice that the $(m-1) \times (m-1)$ minor in the top left (here $m-1 = 3$) is the same as the matrix $M_{\beta}$ from case \eqref{Abeta1} of \Cref{cor:matrixMbetaA} in type $A_{m-1}$, since the last part of $\beta$ equals $1$. Also, the $(n-m+1) \times (n-m+1)$ minor in the bottom right (here $n-m+1 = 4$) is analogous to the matrix $P_\delta$ for the case $\beta = (1)$ considered in Example~\ref{eg:beta1delta13}. We can thus carry out column operations on columns $1$ through $m-1$ as in type $A$, and on columns $m$ through $n$ as in the case $\beta = (1)$, to obtain 
\[
\Qbd  = 
\begin{pmatrix}
1 & 0 & 0 & 0 & &&\\
0 & 0 & 1 & 0 &&& \\
0 & 0 & 0 & 0 &&& \\
0 & 0 & 0 & 0 & 0 & 0 & 0\\
 &  &  & 1 & 0 & 0 & 0 \\
  &  &  & 0 & 0 & 2 & 0 \\
   &  &  & 0 & 1 & 0 & 0 
\end{pmatrix}.
\]

Note that, since the last part of $\beta$ equals $1$, we have $\gcd(\beta_k,2)  = 1$ in this case. Hence in columns $1$ through $m-1$, we obtain diagonal entries $(1^{m-p},0^{p-1})$. In columns $m$ through $n$, by the case $\beta = (1)$ we obtain diagonal entries $(1^{n-m+1-2r}, 2^{2r-1},0)$. Putting this together gives us $\Sbd = \diag(1^{n-2r-p+1},2^{2r-1},0^p)$, which confirms part~\eqref{D_gcd1} of Theorem \ref{thm:DnSNF}. The basis read off from $\Qbd$ yields $\{ \alpha_j^\vee \mid j \in J_\beta \}$ together with the basis obtained from the last $n-m+1$ columns as in the case $\beta = (1)$.
\end{example}


\subsubsection{Noncuspidal examples with $\beta_p \geq 2$}\label{sec:noncuspidal3}

We continue by considering some examples in which $\beta_p \geq 2$, to illustrate parts \eqref{SNFD_gcd2_parity}, \eqref{SNFD_gcd2_nonparity}, and~\eqref{D_gcd1} of \Cref{thm:DnSNF} in this case. See the start of Section~\ref{sec:examplesD} for the full list of examples.


\begin{example}\label{eg:beta4delta1112} 

This example illustrates part \eqref{SNFD_gcd2_nonparity} of \Cref{thm:DnSNF} in the case $\beta_p \geq 2$ and $\delta_1 = \delta_2 = 1$. Let $\sW$ be of type $D_9$ and let $\beta = (4)$ and $\delta = (1,1,1,2)$. By \eqref{eq:wdelta1s2s} with $\ell = 3$, 
\[
\wbd = w_\beta w_\delta = (s_1 s_2 s_3)u_1 u_2 v_4^\delta = (s_1 s_2 s_3)u_1 u_2 (u_3 s_5),
\]
and the matrix for $\Id - \wbd$ with respect to $\Delta^\vee$ is given by
\[
\Pbd = 
\begin{pmatrix}
1&0&1&-1&&&&& \\
-1&1&1&-1&&&& \\
0&-1&2&-1&&&& \\
0&0&0&0&0&0&0&0&0 \\
&&& -1 & 2 & -1&&&  \\
&&& -2 & 2 & 0&&&  \\
&&& -2 & 2 & -2&2&&  \\
&&& -1 & 1 & -1&&2&  \\
&&& -1 & 1 & -1&&&2
\end{pmatrix}. 
\]
Let $m = |\beta|$ and define $\beta' = (4,1)$; that is, $\beta'$ is the partition of $m+1$ obtained by adjoining a last part $1$ to $\beta$. Then the $m \times m$ minor in the top left (here $m = 4$) is the matrix $M_{\beta'}$ from case \eqref{Abeta1} of \Cref{cor:matrixMbetaA}, since the last part of $\beta'$ equals $1$. Also, the $(n-m+1) \times (n-m+1)$ minor in the bottom right (here $n-m+1 = 6$) is the matrix $P_\delta$ for the case $\beta = (1)$ and $\delta = (1,1,1,2)$ considered in Example~\ref{eg:beta1delta1112}.

To find a basis for $\Mod(\wbd)$, we first add column $m-1=3$ and column $m+1=5$ to column $m=4$, then carry out column operations in the first $m-1$ columns as in type $A$ to obtain $(m-1,m-1)$-entry equal to $\gcd(\beta_k) = 4$. We also carry out column operations on the last $(n-m)$ columns (here $n-m = 5$) as done in the case $|\beta| = 0$ to obtain the matrix $Q_\delta$ (compare Example~\ref{eg:delta111223}). This gives us $\Pbd'$:
\[
\Pbd' = 
\begin{pmatrix}
1&0&0&0&&&&& \\
-1&1&0&0&&&& \\
0&-1&4&1&&&& \\
0&0&0&0&0&0&0&0&0 \\
&&& 1 & 0 & -1&&&  \\
&&& 0 & 2 & 0&&&  \\
&&& 0 & 0 & 0&2&&  \\
&&& 0 & 1 & 1&&2&  \\
&&& 0 & 1 & 1&&&2
\end{pmatrix}. 
\]

We now subtract $\gcd(\beta_k) = 4$ times column $m$ from column $m-1$. Then, using a suitable linear combination of the columns to the right of column $m$  which do not have $2$s, we move $(-\gcd(\beta_k)) = -4$ down to the last two entries in column $m-1$, as seen in $\Pbd''$:
\[
\Pbd'' = 
\begin{pmatrix}
1&0&0&0&&&&& \\
-1&1&0&0&&&& \\
0&-1&0&1&&&& \\
0&0&0&0&0&0&0&0&0 \\
&&& 1 & 0 & -1&&&  \\
&&& 0 & 2 & 0&&&  \\
&&& 0 & 0 & 0&2&&  \\
&&-4& 0 & 1 & 1&&\boxed{2}&  \\
&&-4& 0 & 1 & 1&&&\boxed{2}
\end{pmatrix}. 
\]
Then, using the boxed entries $2$ in the last two columns of $\Pbd''$, we can clear column $m-1$ and replace both of these boxed $2$s by $\gcd(\beta_k,2) \in \{1,2\}$. In this example, since $\gcd(\beta_k,2) = 2$, we obtain
\[
\Qbd = 
\begin{pmatrix}
1&0&0&0&&&&& \\
-1&1&0&0&&&& \\
0&-1&0&1&&&& \\
0&0&0&0&0&0&0&0&0 \\
&&& 1 & 0 & -1&&&  \\
&&& 0 & 2 & 0&&&  \\
&&& 0 & 0 & 0&2&&  \\
&&& 0 & 1 & 1&&2&  \\
&&& 0 & 1 & 1&&&2
\end{pmatrix}. 
\]
Notice that the $(n-m) \times (n-m)$ minor in the bottom right (here $n-m = 5$) does not change at this step in the case $\gcd(\beta_k,2) = 2$. Since we have only used column operations so far, and the nonzero columns of $\Qbd$ are linearly independent over $\Z$, a basis for $\Mod(\wbd)$ can be read off from the columns of $\Qbd$.

To find the Smith normal form, we will carry out some row operations on $\Qbd$. First, in the $(m-1) \times (m-1)$ minor in the top left (here $m -1 = 3$), we use successive row operations to clear the $-1$s under $1$s. We then subtract row $m-1 = 3$ from row $m+1=5$,  to obtain
\[
\Qbd' = 
\begin{pmatrix}
\boxed{1}&0&0&0&&&&& \\
0&\boxed{1}&0&0&&&& \\
0&0&0&\boxed{1}&&&& \\
0&0&0&0&0&0&0&0&0 \\
&&& 0 & 0 & -1&&&  \\
&&& 0 & 2 & 0&&&  \\
&&& 0 & 0 & 0&2&&  \\
&&& 0 & 1 & 1&&2&  \\
&&& 0 & 1 & 1&&&2
\end{pmatrix}. 
\]
After some rearrangement in columns $1$ through $m = 4$, we obtain diagonal entries $(1^{m-p},0^{p}) = (1^3,0)$, since $\beta' = (4,1)$ has last part equal to $1$. We can then continue to find the Smith normal form for the $(n-m) \times (n-m)$ minor in the bottom right (here $n-m = 5$) as in the case $|\beta| = 0$ (compare  Example~\ref{eg:delta111223}). This will have diagonal entries $(1^{(n-m)-2r+1},2^{2r-2},4) = (1^2,2^2,4)$, since $\delta=(1,1,1,2)$ has a change of parity in this example. Putting this together, we obtain $\Sbd = (1^{n-2r-p+1}, 2^{2r-2},4,0^p) = (1^5,2^2,4,0)$, confirming part~\eqref{SNFD_gcd2_nonparity} of Theorem~\ref{thm:DnSNF}.
\end{example}


\begin{example}\label{eg:beta4delta12} 

This example illustrates part \eqref{SNFD_gcd2_nonparity} of \Cref{thm:DnSNF} in the case $\beta_p \geq 2$, $\delta_1 = 1$, and $\delta_2 \geq 2$. Let $\sW$ be of type $D_7$ and let $\beta = (4)$ and $\delta = (1,2)$. By \eqref{eq:wdelta1s2s}, we have 
\[
\wbd = w_\beta w_\delta = (s_1 s_2 s_3)(u_1s_5),
\]
and the matrix for $\Id - \wbd$ with respect to $\Delta^\vee$ is given by
\[
\Pbd = 
\begin{pmatrix}
1&0&1&-1 &&\\
-1&1&1&-1 &&\\
 0 & -1& 2 &-1&& \\
0&0&0&0&0 &0&0\\
&& &-1 & 2 & -1 &-1 \\
&& &-1 & 1 & 1 &-1\\
&& &-1 & 1 & -1&1
\end{pmatrix}. 
\]

Let $|\beta| = m$. As in the previous example, we add columns $m-1=3$ and $m+1=5$ to column $m=4$, carry out column operations on the first $m-1$ columns as in type $A$, and subtract $\gcd(\beta_k) = 4$ times column $m$ from column $m-1$. This has the effect of moving $(-\gcd(\beta_k))=-4$ down into row $m+1$.

 The next step is particular to the case $\delta_1 = 1$ and $\delta_2 \geq 2$. We add to column $n-\delta_2 = 5$ twice the sum of columns $n-1$ and $n$, and then add column $n-1$ to column $n$ (this step is analogous to that used to obtain the matrix $P_\delta''$ in Example~\ref{eg:delta1334}). This creates a pivot $-2$ in the $(n-2,n)$-entry, as well as clearing much of column $n-\delta_2=5$, as shown in $\Pbd'$:
\[
\Pbd' = 
\begin{pmatrix}
1&0&0& &&\\
-1&1&0& &&\\
 0 & -1& 0 &1&& \\
0&0&0&0&0 &0&0\\
&&-4 &1 & 0 & -1 &\boxed{-2} \\
&& & & 1 & 1 &0\\
&& & & 1 & -1&0
\end{pmatrix}. 
\] 
In this particular example, we have $m+1 = n-2 = 5$, and so we can immediately use this boxed $-2$ to clear column $m-1=3$ and replace the $(n-2,n)$-entry by $\gcd(\beta_k,2) \in \{1,2\}$. (In general for $\delta_1 = 1$ and $\delta_2 \geq 2$, we can use a suitable linear combination of the columns to the right of column $m = 4$  which have nonzero entries $1$ and $-1$ to move $(-\gcd(\beta_k)) = -4$ down to row $n-2$. We can then use the boxed $-2$ to clear column $m-1$ and replace the $(n-2,n)$-entry by $\gcd(\beta_k,2) \in \{1,2\}$.) In this example, since $\gcd(\beta_k,2) = 2$, we obtain
\[
\Pbd'' = 
\begin{pmatrix}
1&0&0& &&\\
-1&1&0& &&\\
 0 & -1& 0 &1&& \\
0&0&0&0&0 &0&0\\
&& &1 & 0 & -1 &2 \\
&& & & 1 & 1 &0\\
&& & & 1 & -1&0
\end{pmatrix}. 
\] 
Notice that the $(n-m) \times (n-m)$ minor in the bottom right (here $n-m = 3$) does not change at this step in the case $\gcd(\beta_k,2) = 2$.

In order to write down a basis in the case $\delta_1 = 1$ and $\delta_2 \geq 2$, we will take the vectors corresponding to the nonzero columns of index $1$ through $m = 4$ of $\Pbd''$  without further changes, but carry out additional column operations on the $(n-m) \times (n-m)$ minor in the bottom right (here $n-m = 3$). If $|\delta| \geq 4$, we carry out column operations on this minor to obtain a basis as in the case $|\beta| = 0$ (compare Example~\ref{eg:delta1334}). In the special case $\delta = (1,2)$ considered in this example, a basis can be read off from the matrix $\Qbd$ below, which is obtained from $\Pbd''$ by column operations on the last $n-m = 3$ columns:
\[
\Qbd = 
\begin{pmatrix}
1&0&0& &&\\
-1&1&0& &&\\
 0 & -1& 0 &1&& \\
0&0&0&0&0 &0&0\\
&& &1 & 1 & 1 &2 \\
&& & & 2 & 1 &0\\
&& & & 0 & -1&0
\end{pmatrix}. 
\] 

To find the Smith normal form, as in the previous example, we carry out row operations on the first $m - 1 = 3$ rows of $\Qbd$, and then subtract row $m-1 = 3$ from row $m + 1 = 5$, to obtain $\Qbd'$ below:
\[
\Qbd' = 
\begin{pmatrix}
1&0&0& &&\\
0&1&0& &&\\
 0 & 0& 0 &1&& \\
0&0&0&0&0 &0&0\\
&& &0 & 1 & 1 &2 \\
&& & & 2 & 1 &0\\
&& & & 0 & -1&0
\end{pmatrix} 
\quad \to \quad
\Qbd'' = 
\begin{pmatrix}
1&0&0& &&\\
0&1&0& &&\\
 0 & 0& 0 &1&& \\
0&0&0&0&0 &0&0\\
&& &0 & 1 & 0 &2 \\
&& & & 2 & 0 &0\\
&& & & 0 & 1&0
\end{pmatrix}.
\] 
Then Smith normal form for the entire matrix $\Pbd$ can be found by combining the Smith normal form for the $m \times m$ minor in the top left (here $m = 4$), namely $(1^{m-p},0^p) = (1^3,0)$, with the Smith normal form for the $(n-m) \times (n-m)$ minor in the bottom right (here $n-m = 3$). (If $|\delta| \geq 4$, then the Smith normal form from the case $|\beta| = 0$ can be used for this $(n-m) \times (n-m)$ minor; compare Example~\ref{eg:delta1334}.) 

In the special case $\delta = (1,2)$ considered in this example, the diagonal entries in the Smith normal form are $(1,1,4)$, by using row operations to obtain $\Qbd''$, and then noting similarly to Example \ref{eg:delta2333} that $\begin{pmatrix} 1 & 2 \\ 2 & 0 \end{pmatrix} \to \begin{pmatrix} 1 & 0 \\ 0 & -4 \end{pmatrix}$ via one row and one column operation. The resulting Smith normal form for the entire matrix is thus $(1^5,4,0)$, confirming part~\eqref{SNFD_gcd2_nonparity} of \Cref{thm:DnSNF}.
\end{example}


\begin{example}\label{eg:beta4delta22}
This example illustrates part \eqref{SNFD_gcd2_parity} of \Cref{thm:DnSNF} in the case $\beta_p \geq 2$ and $\delta_1 = 2$.
Let $\sW$ be of type $D_8$ and let $\beta = (4)$ and $\delta = (2,2)$. By \eqref{eq:wdelta3s}, we have 
\[
\wbd = w_\beta w_\delta = (s_1 s_2 s_3)(u_0u_2s_5).
\] 
Here, the matrix for $\Id - w_{\beta,\delta}$ with respect to $\Delta^\vee$ is given by
\[
\Pbd  = 
\begin{pmatrix}
1&0&1&-1 &&&\\
-1&1&1&-1 &&&\\
 0 & -1& 2 &-1&&& \\
0&0&0&0&0 &0&&\\
&&&-1 & 2 & -1 &  & \\
&&&-2 & 2 & 0 &  &  \\
&&&-1 & 1 & 0 & 1 & -1\\
&&&-1 & 1 & -1 & 1 & 1
\end{pmatrix}.
\]

Let $|\beta| = m$. As in the previous two examples, we add columns $m-1 = 3$ and $m+1 = 5$ to column $m = 4$, and then carry out column operations on columns $1$ through $m-1=3$ as in type $A$. Now, on columns $m+1 = 5$ through $n$, we carry out the same column operations as in the case $|\beta| = 0$ (compare Example~\ref{eg:delta2333}), to obtain:
\[
\Pbd'  = 
\begin{pmatrix}
1&0&0&0 &&&\\
-1&1&0&0 &&&\\
 0 & -1& 0 &1&&& \\
0&0&0&0&0 &0&&\\
&&-4&1 & 0 & -1 &  & \\
&&&0 & 2 & 0 &  &  \\
&&&0 & 0 & 1 & 0 & -1\\
&&&0 & 0 & 0 & \boxed{2} & 1
\end{pmatrix}.
\]

In this case, since $\delta_1 = 2$, we obtain the boxed pivot $2$ in the $(n,n-1)$-entry. We can now use further column operations, as in the previous examples, to move  $(-\gcd(\beta_k)) = -4$ down to row $n$, and then clear the $(n,m-1)$-entry and replace the boxed pivot by $\gcd(\beta_k,2) \in \{1,2\}$. In this example, since $\gcd(\beta_k,2) = 2$, we obtain:
\[
\Qbd  = 
\begin{pmatrix}
1&0&0&0 &&&\\
-1&1&0&0 &&&\\
 0 & -1& 0 &1&&& \\
0&0&0&0&0 &0&&\\
&&&1 & 0 & -1 &  & \\
&&&0 & 2 & 0 &  &  \\
&&&0 & 0 & 1 & 0 & -1\\
&&&0 & 0 & 0 & 2 & 1
\end{pmatrix}.
\]
Notice that the $(n-m) \times (n-m)$ minor in the bottom right (here $n-m = 4$) does not change at this step in the case $\gcd(\beta_k,2) = 2$.  Since we have only used column operations so far, and the nonzero columns of $\Qbd$ are linearly independent over $\Z$, a basis for $\Mod(\wbd)$ can be read off from the columns of $\Qbd$.

From $\Qbd$, the Smith normal form is obtained similarly to the previous two examples. In this example, since all parts of $\delta = (2,2)$ have the same parity, the $(n-m) \times (n-m)$ minor in the bottom right (here $n-m = 4$) has Smith normal form $\diag(1^2,2^2)$. Hence the Smith normal form of the entire matrix $\Pbd$ is given by $\Sbd = \diag(1^5,2^2,0)$, which confirms part~\eqref{SNFD_gcd2_parity} of Theorem~\ref{thm:DnSNF}.
\end{example}


\begin{example}\label{eg:beta4delta33}
This example illustrates part \eqref{SNFD_gcd2_parity} of \Cref{thm:DnSNF} in the case $\beta_p \geq 2$ and $\delta_1 \geq 3$. 
Let $\sW$ be of type $D_{10}$ and let $\beta = (4)$ and $\delta = (3,3)$. By \eqref{eq:wdelta3s}, we have 
\[
\wbd = w_\beta w_\delta = (s_1 s_2 s_3)(u_0s_8\cdot u_3s_6s_5).
\] 
 Here, the matrix for $\Id - w_\delta$ with respect to $\Delta^\vee$ is given by
\[
\Pbd  = 
\begin{pmatrix}
1&0&1&-1 &&&&&\\
-1&1&1&-1 &&&&&\\
 0 & -1& 2 &-1&&&&& \\
0&0&0&0&0 &0&0&&&\\
&&&-1 & 2 & -1 & 0 &&& \\
&&&-1 & 1 & 1 & -1 &&& \\
&&&-2 & 2 & 0 & 0 &&&  \\
&&&-2 & 2 &  & -1 & 2 &-1& -1 \\
&&&-1 & 1 &  & 0 & 0 & 1 & -1\\
&&&-1 & 1 & &  -1 & 1 & 0 & 0
\end{pmatrix}.
\]

Let $|\beta| = m$. As in the previous three examples, we add columns $m-1 = 3$ and $m+1 = 5$ to column $m = 4$, then carry out column operations on columns $1$ through $m-1=3$ as in type $A$. On columns $m+1=5$ through $n$, we then carry out the same column operations as in the case $|\beta| = 0$ (compare Example~\ref{eg:delta3333}), to obtain:
\[
\Pbd'  = 
\begin{pmatrix}
1&0&0&0 &&&&&\\
-1&1&0&0 &&&&&\\
 0 & -1& 0 &1&&&&& \\
0&0&0&0&0 &0&0&&&\\
&&-4&1 & 0 & -1 & 0 &&& \\
&&& & 0 & 1 & -1 &&& \\
&&& & 2 & 0 & 0 &&&  \\
&&& &  &  & 1 & 0 &-1& 0 \\
&&& &  &  &  & 0 & 1 & \boxed{-2}\\
&&& &  & &   & 1 & 0 & 0
\end{pmatrix}.
\]
In this case, since $\delta_1 \geq 3$, we obtain the boxed pivot $-2$ in the $(n-1,n)$-entry. 
We can then use further column operations to move $(-\gcd(\beta_k)) = -4$ down to row $n-1$, and then clear the $(n-1,m-1)$-entry and replace the boxed pivot by $\gcd(\beta_k,2) \in \{1,2\}$. In this example, since $\gcd(\beta_k,2) = 2$, we obtain:
\[
\Qbd = 
\begin{pmatrix}
1&0&0&0 &&&&&\\
-1&1&0&0 &&&&&\\
 0 & -1& 0 &1&&&&& \\
0&0&0&0&0 &0&0&&&\\
&&&1 & 0 & -1 & 0 &&& \\
&&& & 0 & 1 & -1 &&& \\
&&& & 2 & 0 & 0 &&&  \\
&&& &  &  & 1 & 0 &-1& 0 \\
&&& &  &  &  & 0 & 1 & 2\\
&&& &  & &   & 1 & 0 & 0
\end{pmatrix}.
\] 
Notice that the $(n-m) \times (n-m)$ minor in the bottom right (here $n-m = 6$) does not change at this step in the case $\gcd(\beta_k,2) = 2$. Since we have only used column operations so far, and the nonzero columns of $\Qbd$ are linearly independent over $\Z$, a basis for $\Mod(\wbd)$ can be read off from the columns of $\Qbd$.

From $\Qbd$, the Smith normal form is obtained similarly to the previous three examples. We note that in this example, since all parts of $\delta = (3,3)$ have the same parity, the $(n-m) \times (n-m)$ minor in the bottom right has Smith normal form $\diag(1^4,2^2)$. Hence the Smith normal form of the entire matrix $\Pbd$ is given by $\Sbd = \diag(1^7,2^2,0)$, which confirms part~\eqref{SNFD_gcd2_parity}  of Theorem~\ref{thm:DnSNF}.
\end{example}


\begin{example}\label{eg:beta3delta33}
This example illustrates part \eqref{D_gcd1} of \Cref{thm:DnSNF}, which considers the case $\gcd(\beta_k,2) = 1$ and $|\delta| \geq 2$, when $\beta_p \geq 2$ and $\delta_1 \geq 3$. The argument is similar for $\delta_1 \in \{1,2\}$. 

Let $\sW$ be of type $D_{9}$ and let $\beta = (3)$ and $\delta = (3,3)$. By \eqref{eq:wdelta3s}, we have 
\[
\wbd = w_\beta w_\delta = (s_1 s_2)(u_0s_7\cdot u_3s_5s_4).
\] 
Here, the matrix for $\Id - w_\delta$ with respect to $\Delta^\vee$ is given by
\[
\Pbd  = 
\begin{pmatrix}
1&1&-1 &&&&&\\
  -1& 2 &-1&&&&& \\
0&0&0&0 &0&0&&&\\
&&-1 & 2 & -1 & 0 &&& \\
&&-1 & 1 & 1 & -1 &&& \\
&&-2 & 2 & 0 & 0 &&&  \\
&&-2 & 2 &  & -1 & 2 &-1& -1 \\
&&-1 & 1 &  & 0 & 0 & 1 & -1\\
&&-1 & 1 & &  -1 & 1 & 0 & 0
\end{pmatrix}.
\]

We proceed similarly to the previous example, with $|\beta| = m = 3$ here. We move $(-\gcd(\beta_k)) = -3$ down to row $n-1$, and then clear the $(n-1,m-1)$-entry and replace the $(n-1,n)$-entry by $\gcd(\beta_k,2) \in \{1,2\}$. Now, since $\gcd(\beta_k,2) = 1$ in this example, we obtain a pivot entry $1$ in column $n$:
\[
\Pbd'  = 
\begin{pmatrix}
1&0&0 &&&&&\\
  -1& 0 &1&&&&& \\
0&0&0&0 &0&0&&&\\
&&1 & 0 & -1 & 0 &&& \\
&& & 0 & 1 & -1 &&& \\
&& & 2 & 0 & 0 &&&  \\
&& &  &  & 1 & 0 &-1& 0 \\
&& &  &  &  & 0 & 1 & \boxed{1}\\
&& &  & &   & 1 & 0 & 0
\end{pmatrix}.
\] 
We can now use column operations and this boxed pivot entry $1$ to clear many entries, moving up the matrix, to obtain:
\[
\Qbd  = 
\begin{pmatrix}
1&0&0 &&&&&\\
 0& 0 &1&&&&& \\
0&0&0&0 &0&0&&&\\
&& & 0 & -1 & 0 &&& \\
&& & 0 & 0 & -1 &&& \\
&& & 2 & 0 & 0 &&&  \\
&& &  &  &  & 0 &-1& 0 \\
&& &  &  &  & 0 & 0 & 1\\
&& &  & &   & 1 & 0 & 0
\end{pmatrix},
\]
from which a basis for $\Mod(\wbd)$ can be easily read off.
In addition, the Smith normal form is clearly given by $\Sbd = \diag(1^{n-2r-p+1},2^{2r-1},0^p) = \diag(1^7,2,0)$, confirming  part~\eqref{D_gcd1}  of Theorem~\ref{thm:DnSNF}. 
\end{example}


\subsubsection{Noncuspidal examples with $|\delta| = 0$}\label{sec:noncuspidal1}

We conclude by considering the situation where $|\delta| = 0$ and $|\beta| = n$. If $\beta$ has at least one odd part, then $\wbd = w_\beta$ is cuspidal  in the type $A_{n-1}$ subsystem indexed by the first $n-1$ nodes. If all parts of $\beta$ are even, then  $\wbd = w_{\beta^+}$ is also cuspidal  in the type $A_{n-1}$ subsystem indexed by the first $n-1$ nodes, while $\wbd = w_{\beta^-}$ is obtained from $w_{\beta^+}$ by replacing $s_{n-1}$ by $s_n$, hence $w_{\beta^-}$ is cuspidal in the type $A_{n-1}$ subsystem indexed by $\{s_1,\dots,s_{n-2},s_n\}$.

The following examples illustrate the proofs of part \eqref{D_beta_delta0} of \Cref{thm:DnSNF} and parts \eqref{D_beta1} and \eqref{D_beta2} of \Cref{thm:Dnnoncuspidal}; see the introduction to Section \ref{sec:examplesD} for the list of examples we consider.


\begin{example}\label{eg:beta31delta0} 
This example illustrates part \eqref{D_beta_delta0} of \Cref{thm:DnSNF}, and part \eqref{D_beta1} of \Cref{thm:Dnnoncuspidal}. Let $\sW$ be of type $D_{4}$, and let $\beta = (3,1)$ and $|\delta| = 0$. By Proposition \ref{prop:wbetaA}, we have
\[
\wbd = w_\beta = s_1 s_2.
\]
Applying \eqref{eq:siaction}, the matrix for $\Id - w_\beta$ with respect to $\Delta^\vee$ is given by $\Pb$ below:
\[
\Pb  = 
\begin{pmatrix}
1 & 1 & -1 & -1 \\
-1 & 2 & -1 & -1 \\
0 & 0 & 0 & 0 \\
0 & 0 & 0 & 0
\end{pmatrix}
\quad \to \quad  
Q_\beta  = 
\begin{pmatrix}
1 & 0 & 0 & 0 \\
0 & 0 & 1 & 0 \\
0 & 0 & 0 & 0 \\
0 & 0 & 0 & 0
\end{pmatrix}.
\]
Notice that the last row of $\Pb$ is all $0$s, while the $(n-1) \times (n-1)$ minor in the top left (for $n = 4$) is the same as the matrix $M_{\beta}$ from case \eqref{Abeta1} of \Cref{cor:matrixMbetaA} in type $A_{n-1}$, since the last part of $\beta$ equals $1$. Also, the last two columns of $\Pb$ are identical. 

Since the last two columns of $\Pb$ are identical, we subtract column $(n-1)$ from column~$n$, and then work as in type $A_{n-1}$ in the top left minor,  to obtain $Q_\beta$ using only column operations. The resulting basis for $\Mod(\wbd) = \Mod(w_\beta)$ is given then by part \eqref{D_beta1} of \Cref{thm:Dnnoncuspidal}, using the first case of part \eqref{BasisA} of Theorem \ref{thm:An}.

The Smith normal form for $(\Id - \wbd) = (\Id - w_\beta)$ will also be the same as that corresponding to $\beta$ in type $A_{n-1}$, with one extra $0$. By part \eqref{SNFA} of Theorem \ref{thm:An}, and using the fact that $\gcd(\beta_k,2)=1$, we thus have $\Sbd = (1,1,0,0)$, confirming part \eqref{D_beta_delta0} of Theorem \ref{thm:DnSNF}.
\end{example}


\begin{example}\label{eg:beta22delta0} 
This example illustrates part \eqref{D_beta_delta0} of \Cref{thm:DnSNF}, and part \eqref{D_beta2} of \Cref{thm:Dnnoncuspidal}. Let $\sW$ be of type $D_{4}$, and let $\beta = (2,2)$ and $|\delta| = 0$. Note that all parts of $\beta$ are even, so there are two cases $\beta^\pm$ to consider. By Proposition \ref{prop:wbetaA}, we have
\[
w_{\beta^+, \delta} = w_{\beta^+} = s_1 s_3.
\]
Applying \eqref{eq:siaction}, the matrix for $\Id - w_{\beta^+}$ with respect to $\Delta^\vee$ is given by $P_{\beta^+}$ below:
\[
P_{\beta^+}  = 
\begin{pmatrix}
2 & -1 & 0 & 0 \\
0 & 0 & 0 & 0 \\
0 & -1 & 2 & 0 \\
0 & 0 & 0 & 0
\end{pmatrix}
\quad \to \quad  
Q_{\beta^+}  = 
\begin{pmatrix}
0 & 1 & 0 & 0 \\
0 & 0 & 0 & 0 \\
0 & -1 & 2 & 0 \\
0 & 0 & 0 & 0
\end{pmatrix}.
\]
Notice that the last column and last row of $P_{\beta^+}$ are all $0$s, while the $(n-1) \times (n-1)$ minor in the top left (for $n = 4$) is same as the matrix $M_{\beta}$ from case \eqref{Abeta2} of \Cref{cor:matrixMbetaA} in type $A_{n-1}$.   

Since the last row and last column of $P_{\beta^+}$ are all $0$s, we can work in the $(n-1) \times (n-1)$ minor in the top left exactly as in type $A_{n-1}$, to obtain $Q_{\beta^+}$ using only column operations. More precisely, following the same column operations as used to obtain the matrix $B_\beta'$ in Example \ref{eg:64A9}, we first add column 1 to column 2. Since here each block $M_{\beta_k}$ is already in the form $B_{\beta_k}$ from Example \ref{eg:CoxeterA3}, we proceed to move the part $\beta_1 = 2$ to the bottom nonzero row, by subtracting twice column 2 from column 1. The pivot $\beta_2 = 2$ in the third column can then clear the column 1, giving $Q_{\beta^+}$ above.

The resulting basis for $\Mod(w_{\beta^+, \delta}) = \Mod(w_{\beta^+})$ is therefore 
\[ \{ \alpha_1^\vee - \alpha_3^\vee, 2\alpha_3^\vee \}. \]
Since $J_\beta = \{1,3\}$ and $I_\beta = \{2,4\}$, then $J_\beta \backslash (I_\beta - 1) = \emptyset$, and so this basis matches the second case of part \eqref{BasisA} of Theorem \ref{thm:An} with $n=3$, equivalently part \eqref{D_beta2} of \Cref{thm:Dnnoncuspidal} with $n=4$.

The Smith normal form for $(\Id -w_{\beta^+, \delta}) = (\Id - w_{\beta^+})$ will then be the same as that corresponding to $\beta$ in type $A_{n-1}$, with one extra $0$. By part \eqref{SNFA} of Theorem \ref{thm:An}, we thus have $S_{\beta^+, \delta} = (1,2,0,0)$, confirming part \eqref{D_beta_delta0} of Theorem \ref{thm:DnSNF}. 

We now consider the case of $\beta^-$.  We switch the $s_3$ for $s_4$ in $w_{\beta^+}$ and so obtain
\[
w_{\beta^-, \delta} = w_{\beta^-} = s_1 s_4.
\]
The matrix for $\Id - w_{\beta^-}$ with respect to $\Delta^\vee$ can then be obtained from $P_{\beta^+}$ by exchanging both rows and columns $n-1$ and $n$:
\[
P_{\beta^-}  = 
\begin{pmatrix}
2 & -1 & 0 & 0 \\
0 & 0 & 0 & 0 \\
0 & 0 & 0 & 0 \\
0 & -1 & 0 & 2
\end{pmatrix}
\quad \to \quad  
Q_{\beta^-}  = 
\begin{pmatrix}
0 & 1 & 0 & 0 \\
0 & 0 & 0 & 0 \\
0 & 0 & 0 & 0 \\
0 & -1 & 0 & 2
\end{pmatrix}.
\]
Since $P_{\beta^-}$ is equivalent to $P_{\beta^+}$ up to row and column operations, then clearly $S_{\beta^-,\delta} = S_{\beta^+,\delta} = (1,2,0,0)$, confirming part \eqref{D_beta_delta0} of Theorem \ref{thm:DnSNF}.  

Similarly, using the same column operations as performed on $P_{\beta^+}$, we obtain $Q_{\beta^-}$, which differs only from $Q_{\beta^+}$ by exchanging rows $n-1$ and $n$, confirming the description for $\Mod(w_{\beta^-, \delta}) = \Mod(w_{\beta^-})$  provided by part \eqref{D_beta2} of Theorem \ref{thm:Dnnoncuspidal}.
\end{example}


\begin{example}\label{eg:beta4delta0} 
 This example illustrates part \eqref{D_beta_delta0} of \Cref{thm:DnSNF}, and part \eqref{D_beta2} of \Cref{thm:Dnnoncuspidal}. Let $\sW$ be of type $D_{4}$ and let $\beta = (4)$ and $|\delta| = 0$. Note that all parts of $\beta$ are even, so there are two cases $\beta^\pm$ to consider. By Proposition \ref{prop:wbetaA}, we have
 \[
w_{\beta^+,\delta} = w_{\beta^+} = s_1 s_2 s_3,
\]
and applying \eqref{eq:siaction} says that the matrix for $\Id - w_{\beta^+}$ with respect to $\Delta^\vee$ is
\[
P_{\beta^+} = 
\begin{pmatrix}
1 & 0 & 1 & -1 \\
-1 & 1 & 1 & -1 \\
0 & -1 & 2 & 0 \\
0 & 0 & 0 & 0
\end{pmatrix}.
\]
Notice that the last row of this matrix is all $0$s, while the $(n-1) \times (n-1)$ minor in the top left (for $n = 4$) is the same as the matrix $M_{\beta}$ from case \eqref{Abeta2} of \Cref{cor:matrixMbetaA} in type $A_{n-1}$. Also, the last column of $P_{\beta^+}$ is the negative of its second-last column, except for the $(n-1,n)$-entry, which equals $0$. 

Upon adding column $n-1$ to column $n$ we obtain $P_{\beta^+}'$ below:
\[
P_{\beta^+}'  = 
\begin{pmatrix}
1 & 0 & 1 & 0 \\
-1 & 1 & 1 & 0 \\
0 & -1 & 2 & 2 \\
0 & 0 & 0 & 0
\end{pmatrix}
\  \to \
P_{\beta^+}''  = 
\begin{pmatrix}
1 & 0 & 0 & 0 \\
-1 & 1 & 0 & 0 \\
0 & -1 & 4 & \boxed{2} \\
0 & 0 & 0 & 0
\end{pmatrix}
\ \to \ 
Q_{\beta^+}  = 
\begin{pmatrix}
1 & 0 & 0 & 0 \\
-1 & 1 & 0 & 0 \\
0 & -1 & 0 & 2 \\
0 & 0 & 0 & 0
\end{pmatrix}.
\]
We can now work on the first $(n-1)$ rows and columns of $P_{\beta^+}'$ as in type $A_{n-1}$, without changing column $n$, to obtain $P_{\beta^+}''$. The boxed 2 in row 3 can be used to clear column 3 entirely in this case, since here $\gcd(\beta_k,2) = 2$.  The resulting basis for $\Mod(w_{\beta^+, \delta}) = \Mod(w_{\beta^+})$ is given by part \eqref{D_beta2} of \Cref{thm:Dnnoncuspidal}, noting that $J_\beta = \{1,2,3\}$ and $I_\beta = \{4\}$.

Finally, the Smith normal form for $(\Id - w_{\beta^+,\delta})$ is the same as in type $A_{n-1}$, except that $\gcd(\beta_k,2)$ replaces $\gcd(\beta_k)$ due to the pivot 2 in column 4, and we have one extra 0. By part \eqref{SNFA} of Theorem \ref{thm:An}, we thus have $S_{\beta^+,\delta} = (1,1,2,0)$, confirming part \eqref{D_beta_delta0} of Theorem \ref{thm:DnSNF}.

The argument for $w_{\beta^-} = s_1s_2s_4$ is the same as in the previous example.
\end{example}




\section{Exceptional types}\label{sec:exceptional}

In this section, we record an implicit description of all mod-sets in the exceptional types, for the sake of completeness.  We provide a complete system of minimal length representatives for all conjugacy classes of $\sW$, where $\sW$ is the finite Weyl group of type $E_6, E_7, E_8, F_4,$ or $G_2$.  For each representative $w \in \sW$, we present the Smith normal form $S_w$ for $(\Id - w)$ with respect to $\Delta^\vee$, which is both canonical and fully characterizes the isomorphism type of the quotient $R^\vee/\Mod(w)$.  In the exceptional types, the Smith normal form $S_w$ was calculated using the \verb!smith_form()! 
command in Sage \cite{sagemath}. 

For all exceptional types, we index the nodes of the Dynkin diagram using the common labeling from Bourbaki \cite{Bourbaki4-6}, Geck and Pfeiffer \cite{GeckPfeifferBook}, and Sage \cite{sagemath}; see Table \ref{table:dynkin} in~\Cref{app:dynkin}. For brevity, we denote the product $s_is_j$ by $s_{ij}$ in our tables.

\subsection{Type $E_6$}

To obtain a system of minimal length representatives for the 25 conjugacy classes in $\sW$ of type $E_6$, we use Table B.4 in \cite{GeckPfeifferBook} to identify the 5 cuspidal classes.  We then iterate over all proper parabolic subgroups of $\sW$ and choose representatives for the cuspidal conjugacy classes in each parabolic, using the classification in \cite{GeckPfeifferBook} for types $A$ and $D$. See Table \ref{table:SNFE6} for a summary of our results in type $E_6$.

\begin{table}[htp]
\begin{center}
\begin{tabular}{|l|c|c|} 
\hline
Representative $w\in \wconj$ & Type of $\Pc(w)$ & $S_w = \diag(\cdots)$ \\ 
\hline \hline
1&1&(0,0,0,0,0,0) \\
$s_1$& $A_1$ &(1,0,0,0,0,0) \\
$s_{12}$& $A_1\times A_1$ &(1,1,0,0,0,0) \\
$s_{13}$& $A_2$ &(1,1,0,0,0,0) \\
$s_{125}$& $A_1 \times A_1 \times A_1$ &(1,1,1,0,0,0) \\
$s_{132}$& $A_2 \times A_1$ &(1,1,1,0,0,0) \\
$s_{134}$& $A_3$ &(1,1,1,0,0,0) \\
$s_{3456}$& $A_4$ &(1,1,1,1,0,0) \\
$s_{3256}$& $A_1 \times A_1 \times A_2$ &(1,1,1,1,0,0) \\
$s_{1346}$& $A_3 \times A_1$ &(1,1,1,1,0,0) \\
$s_{1356}$& $A_2 \times A_2$ &(1,1,1,3,0,0) \\
$s_{13456}$& $A_5$ &(1,1,1,1,3,0) \\
$s_{12456}$& $A_4 \times A_1$ &(1,1,1,1,1,0) \\
$s_{13256}$& $A_2 \times A_2 \times A_1$ &(1,1,1,1,3,0) \\
$s_{2345}$& $D_4$ &(1,1,1,1,0,0) \\
$s_{242345}$& $D_4$ &(1,1,1,1,0,0) \\
$s_{234234542345}$& $D_4$ &(1,1,2,2,0,0) \\
$s_{23456}$& $D_5$ &(1,1,1,1,1,0) \\
$s_{2423456}$& $D_5$ &(1,1,1,1,1,0) \\
$s_{2342345423456}$& $D_5$ &(1,1,1,2,2,0) \\
$s_{123456}$& $E_6$ &(1,1,1,1,1,3) \\
$s_{12342546}$& $E_6$ &(1,1,1,1,1,3) \\
$s_{123142345465}$& $E_6$ &(1,1,1,1,1,3) \\
$s_{12342345423456}$& $E_6$ &(1,1,1,1,2,6) \\
$s_{123142314542314565423456}$& $E_6$ &(1,1,1,3,3,3) \\
\hline
\end{tabular}
\end{center}
\caption{Smith normal form $S_w$ for $(\Id-w)$ with $w \in \sW$ of type $E_6$.}
\label{table:SNFE6}
\end{table}%

\subsection{Type $E_7$}

To obtain a system of minimal length representatives for the 60 conjugacy classes in $\sW$ of type $E_7$, we use Table B.5 in \cite{GeckPfeifferBook} to identify the 12 cuspidal classes.  We then iterate over all proper parabolic subgroups of $\sW$ and choose representatives for the cuspidal conjugacy classes in each parabolic, using the classification in \cite{GeckPfeifferBook} for types $A, D$, and $E_6$. See Tables \ref{table:SNFE7-noncusp} and \ref{table:SNFE7-cusp} for a summary of our results in type $E_7$.

\begin{table}[htp]
\begin{center}
\begin{tabular}{|l|c|c|} 
\hline
Representative $w\in \wconj$ & Type of $\Pc(w)$ & $S_w = \diag(\cdots)$ \\ 
\hline \hline
1&1&(0,0,0,0,0,0,0) \\
$s_1$& $A_1$ &(1,0,0,0,0,0,0) \\
$s_{12}$& $A_1 \times A_1$ &(1,1,0,0,0,0,0) \\
$s_{13}$& $A_2$ &(1,1,0,0,0,0,0) \\
$s_{132}$& $A_2 \times A_1$ &(1,1,1,0,0,0,0) \\
$s_{134}$& $A_3$ &(1,1,1,0,0,0,0) \\
$s_{125}$& $A_1 \times A_1 \times A_1$ &(1,1,1,0,0,0,0) \\
$s_{257}$& $A_1 \times A_1 \times A_1$ &(1,1,2,0,0,0,0) \\
$s_{4567}$& $A_4$ &(1,1,1,1,0,0,0) \\
$s_{3567}$& $A_3 \times A_1$ &(1,1,1,1,0,0,0) \\
$s_{2567}$& $A_3 \times A_1$ &(1,1,1,2,0,0,0) \\
$s_{3467}$& $A_2 \times A_2$ &(1,1,1,1,0,0,0) \\
$s_{2367}$& $A_2 \times A_1 \times A_1$ &(1,1,1,1,0,0,0) \\
$s_{2357}$& $A_1 \times A_1 \times A_1 \times A_1$ &(1,1,1,2,0,0,0) \\
$s_{34567}$& $A_5$ &(1,1,1,1,1,0,0) \\
$s_{24567}$& $A_5$ &(1,1,1,1,2,0,0) \\
$s_{14567}$& $A_4 \times A_1$ &(1,1,1,1,1,0,0) \\
$s_{24367}$& $A_3 \times A_2$ &(1,1,1,1,1,0,0) \\
$s_{12467}$& $A_2 \times A_2 \times A_1$ &(1,1,1,1,1,0,0) \\
$s_{23567}$& $A_3 \times A_1 \times A_1$ &(1,1,1,1,2,0,0) \\
$s_{13257}$& $A_2 \times A_1 \times A_1 \times A_1$ &(1,1,1,1,2,0,0) \\
$s_{134567}$& $A_6$ &(1,1,1,1,1,1,0) \\
$s_{124567}$& $A_5 \times A_1$ &(1,1,1,1,1,2,0) \\
$s_{134267}$& $A_4 \times A_2$ &(1,1,1,1,1,1,0) \\
$s_{132567}$& $A_3 \times A_2 \times A_1$ &(1,1,1,1,1,2,0) \\
$s_{2345}$& $D_4$ &(1,1,1,1,0,0,0) \\
$s_{242345}$& $D_4$ &(1,1,1,1,0,0,0) \\
$s_{234234542345}$& $D_4$ &(1,1,2,2,0,0,0) \\
$s_{23457}$& $D_4 \times A_1$ &(1,1,1,1,2,0,0) \\
$s_{2423457}$& $D_4 \times A_1$ &(1,1,1,1,2,0,0) \\
$s_{2342345423457}$& $D_4 \times A_1$ &(1,1,2,2,2,0,0) \\
$s_{23456}$& $D_5$ &(1,1,1,1,1,0,0) \\
$s_{2423456}$& $D_5$ &(1,1,1,1,1,0,0) \\
$s_{2342345423456}$& $D_5$ &(1,1,1,2,2,0,0) \\
$s_{254317}$& $D_5 \times A_1$ &(1,1,1,1,1,2,0) \\
$s_{24254317}$& $D_5 \times A_1$ &(1,1,1,1,1,2,0) \\
$s_{25425434254317}$& $D_5 \times A_1$ &(1,1,1,2,2,2,0) \\
$s_{234567}$& $D_6$ &(1,1,1,1,1,2,0) \\
$s_{24234567}$& $D_6$ &(1,1,1,1,1,2,0) \\
$s_{2454234567}$& $D_6$ &(1,1,1,1,1,2,0) \\
$s_{23423454234567}$& $D_6$ &(1,1,1,2,2,2,0) \\
$s_{2342345654234567}$& $D_6$ &(1,1,1,1,2,4,0) \\
$s_{234234542345654234567654234567}$& $D_6$ &(1,2,2,2,2,2,0) \\
$s_{123456}$& $E_6$ &(1,1,1,1,1,1,0) \\
$s_{12342546}$& $E_6$ &(1,1,1,1,1,1,0) \\
$s_{123142345465}$& $E_6$ &(1,1,1,1,1,1,0) \\
$s_{12342345423456}$& $E_6$ &(1,1,1,1,2,2,0) \\
$s_{123142314542314565423456}$& $E_6$ &(1,1,1,1,3,3,0) \\
\hline
\end{tabular}
\end{center}
\caption{Smith normal form $S_w$ for $(\Id-w)$ with $w \in \sW$ non-cuspidal in type $E_7$.}
\label{table:SNFE7-noncusp}
\end{table}%

\begin{table}[htp]
\begin{center}
\begin{tabular}{|l|c|c|} 
\hline
Representative $w\in \wconj$ & Type of $\Pc(w)$ & $S_w = \diag(\cdots)$ \\ 
\hline \hline
$s_{1234567}$ & $E_7$ &(1,1,1,1,1,1,2) \\
$s_{123425467}$ & $E_7$ &(1,1,1,1,1,1,2) \\
$s_{12342546576}$ & $E_7$ &(1,1,1,1,1,1,2) \\
$s_{1234254234567}$ & $E_7$ &(1,1,1,1,1,1,2) \\
$s_{123423454234567}$ & $E_7$ &(1,1,1,1,2,2,2) \\
$s_{12342345423456576}$ & $E_7$ &(1,1,1,1,1,2,4) \\
$s_{123142314354234654765}$ & $E_7$ &(1,1,1,1,1,1,2) \\
$s_{12314231435423143546576}$ & $E_7$ &(1,1,1,1,2,2,2) \\
$s_{1231423145423145654234567}$ & $E_7$ &(1,1,1,1,1,3,6) \\
$s_{1234234542345654234567654234567}$ & $E_7$ &(1,1,2,2,2,2,2) \\
$s_{123142345423145654234567654234567}$ & $E_7$ &(1,1,1,1,2,4,4) \\
$w_0 = s_{765432456713456245341324567134562453413245624534132453413241321}$ & $E_7$ &(2,2,2,2,2,2,2) \\
\hline
\end{tabular}
\end{center}
\caption{Smith normal form $S_w$ for $(\Id-w)$ with $w \in \sW$ cuspidal in type $E_7$.}
\label{table:SNFE7-cusp}
\end{table}%

\subsection{Type $E_8$}

To obtain a system of minimal length representatives for the 112 conjugacy classes in $\sW$ of type $E_8$, we use Table B.6 in \cite{GeckPfeifferBook} to identify the 30 cuspidal classes.  We then iterate over all proper parabolic subgroups of $\sW$ and choose representatives for the cuspidal conjugacy classes in each parabolic,  using the classification in \cite{GeckPfeifferBook} for types $A, D, E_6$, and $E_7$. See Tables \ref{table:SNFE8-noncusp1}, \ref{table:SNFE8-noncusp2}, and \ref{table:SNFE8-cusp} for a summary of our results in type $E_8$.

\begin{table}[htp]
\begin{center}
\begin{tabular}{|l|c|c|} 
\hline
Representative $w\in \wconj$ & Type of $\Pc(w)$ & $S_w = \diag(\cdots)$ \\ 
\hline \hline
1&1&(0,0,0,0,0,0,0,0) \\
$s_1$& $A_1$ &(1,0,0,0,0,0,0,0) \\
$s_{13}$& $A_2$ &(1,1,0,0,0,0,0,0) \\
$s_{14}$& $A_1 \times A_1$ &(1,1,0,0,0,0,0,0) \\
$s_{134}$& $A_3$ &(1,1,1,0,0,0,0,0) \\
$s_{135}$& $A_2 \times A_1$ &(1,1,1,0,0,0,0,0) \\
$s_{146}$& $A_1 \times A_1 \times A_1$ &(1,1,1,0,0,0,0,0) \\
$s_{1345}$& $A_4$ &(1,1,1,1,0,0,0,0) \\
$s_{1346}$& $A_3 \times A_1$ &(1,1,1,1,0,0,0,0) \\
$s_{1356}$& $A_2 \times A_2$ &(1,1,1,1,0,0,0,0) \\
$s_{1357}$& $A_2 \times A_1 \times A_1$ &(1,1,1,1,0,0,0,0) \\
$s_{1468}$& $A_1 \times A_1 \times A_1 \times A_1$ &(1,1,1,1,0,0,0,0) \\
$s_{13456}$& $A_5$ &(1,1,1,1,1,0,0,0) \\
$s_{13457}$& $A_4 \times A_1$ &(1,1,1,1,1,0,0,0) \\
$s_{13467}$& $A_3 \times A_2$ &(1,1,1,1,1,0,0,0) \\
$s_{13468}$& $A_3 \times A_1 \times A_1$ &(1,1,1,1,1,0,0,0) \\
$s_{13568}$& $A_2 \times A_2 \times A_1$ &(1,1,1,1,1,0,0,0) \\
$s_{13257}$& $A_2 \times A_1 \times A_1 \times A_1$ &(1,1,1,1,1,0,0,0) \\
$s_{134567}$& $A_6$ &(1,1,1,1,1,1,0,0) \\
$s_{134568}$& $A_5 \times A_1$ &(1,1,1,1,1,1,0,0) \\
$s_{134578}$& $A_4 \times A_2$ &(1,1,1,1,1,1,0,0) \\
$s_{567812}$& $A_4 \times A_1 \times A_1$ &(1,1,1,1,1,1,0,0) \\
$s_{134678}$& $A_3 \times A_3$ &(1,1,1,1,1,1,0,0) \\
$s_{678124}$& $A_3 \times A_2 \times A_1$ &(1,1,1,1,1,1,0,0) \\
$s_{135628}$& $A_2 \times A_2 \times A_1 \times A_1$ &(1,1,1,1,1,1,0,0) \\
$s_{1345678}$& $A_7$ &(1,1,1,1,1,1,1,0) \\
$s_{1245678}$& $A_6 \times A_1$ &(1,1,1,1,1,1,1,0) \\
$s_{1342678}$& $A_4 \times A_3$ &(1,1,1,1,1,1,1,0) \\
$s_{1325678}$& $A_4 \times A_2 \times A_1$ &(1,1,1,1,1,1,1,0) \\
$s_{2345}$& $D_4$ &(1,1,1,1,0,0,0,0) \\
$s_{242345}$& $D_4$ &(1,1,1,1,0,0,0,0) \\
$s_{234234542345}$& $D_4$ &(1,1,2,2,0,0,0,0) \\
$s_{23457}$& $D_4 \times A_1$ &(1,1,1,1,1,0,0,0) \\
$s_{2423457}$& $D_4 \times A_1$ &(1,1,1,1,1,0,0,0) \\
$s_{2342345423457}$& $D_4 \times A_1$ &(1,1,1,2,2,0,0,0) \\
$s_{234578}$& $D_4 \times A_2$ &(1,1,1,1,1,1,0,0) \\
$s_{24234578}$& $D_4 \times A_2$ &(1,1,1,1,1,1,0,0) \\
$s_{23423454234578}$& $D_4 \times A_2$ &(1,1,1,1,2,2,0,0) \\
$s_{23456}$& $D_5$ &(1,1,1,1,1,0,0,0) \\
$s_{2423456}$& $D_5$ &(1,1,1,1,1,0,0,0) \\
$s_{2342345423456}$& $D_5$ &(1,1,1,2,2,0,0,0) \\
\hline
\end{tabular}
\end{center}
\caption{Smith normal form $S_w$ for $(\Id-w)$ with $w \in \sW$ non-cuspidal in type $E_8$ (continued in Table \ref{table:SNFE8-noncusp2}).}
\label{table:SNFE8-noncusp1}
\end{table}%

\begin{table}[htp]
\begin{center}
\begin{tabular}{|l|c|c|} 
\hline
Representative $w\in \wconj$ & Type of $\Pc(w)$ & $S_w = \diag(\cdots)$ \\ 
\hline \hline
$s_{234568}$& $D_5 \times A_1$ &(1,1,1,1,1,1,0,0) \\
$s_{24234568}$& $D_5 \times A_1$ &(1,1,1,1,1,1,0,0) \\
$s_{23423454234568}$& $D_5 \times A_1$ &(1,1,1,1,2,2,0,0) \\
$s_{2543178}$& $D_5 \times A_2$ &(1,1,1,1,1,1,1,0) \\
$s_{242543178}$& $D_5 \times A_2$ &(1,1,1,1,1,1,1,0) \\
$s_{254254342543178}$& $D_5 \times A_2$ &(1,1,1,1,1,2,2,0) \\
$s_{234567}$& $D_6$ &(1,1,1,1,1,1,0,0) \\
$s_{24234567}$& $D_6$ &(1,1,1,1,1,1,0,0) \\
$s_{2454234567}$& $D_6$ &(1,1,1,1,1,1,0,0) \\
$s_{23423454234567}$& $D_6$ &(1,1,1,1,2,2,0,0) \\
$s_{2342345654234567}$& $D_6$ &(1,1,1,1,2,2,0,0) \\
$s_{234234542345654234567654234567}$& $D_6$ &(1,1,2,2,2,2,0,0) \\
$s_{2345678}$& $D_7$ &(1,1,1,1,1,1,1,0) \\
$s_{242345678}$& $D_7$ &(1,1,1,1,1,1,1,0) \\
$s_{24542345678}$& $D_7$ &(1,1,1,1,1,1,1,0) \\
$s_{234234542345678}$& $D_7$ &(1,1,1,1,1,2,2,0) \\
$s_{23423456542345678}$& $D_7$ &(1,1,1,1,1,2,2,0) \\
$s_{234542345676542345678}$& $D_7$ &(1,1,1,1,1,2,2,0) \\
$s_{2342345423456542345676542345678}$& $D_7$ &(1,1,1,2,2,2,2,0) \\
$s_{123456}$& $E_6$ &(1,1,1,1,1,1,0,0) \\
$s_{12342546}$& $E_6$ &(1,1,1,1,1,1,0,0) \\
$s_{123142345465}$& $E_6$ &(1,1,1,1,1,1,0,0) \\
$s_{12342345423456}$& $E_6$ &(1,1,1,1,2,2,0,0) \\
$s_{123142314542314565423456}$& $E_6$ &(1,1,1,1,3,3,0,0) \\
$s_{1234568}$& $E_6 \times A_1$ &(1,1,1,1,1,1,1,0) \\
$s_{123425468}$& $E_6 \times A_1$ &(1,1,1,1,1,1,1,0) \\
$s_{1231423454658}$& $E_6 \times A_1$ &(1,1,1,1,1,1,1,0) \\
$s_{123423454234568}$& $E_6 \times A_1$ &(1,1,1,1,1,2,2,0) \\
$s_{1231423145423145654234568}$& $E_6 \times A_1$ &(1,1,1,1,1,3,3,0) \\
$s_{1234567}$ & $E_7$ &(1,1,1,1,1,1,1,0) \\
$s_{123425467}$ & $E_7$ &(1,1,1,1,1,1,1,0) \\
$s_{12342546576}$ & $E_7$ &(1,1,1,1,1,1,1,0) \\
$s_{1234254234567}$ & $E_7$ &(1,1,1,1,1,1,1,0) \\
$s_{123423454234567}$ & $E_7$ &(1,1,1,1,1,2,2,0) \\
$s_{12342345423456576}$ & $E_7$ &(1,1,1,1,1,2,2,0) \\
$s_{123142314354234654765}$ & $E_7$ &(1,1,1,1,1,1,1,0) \\
$s_{12314231435423143546576}$ & $E_7$ &(1,1,1,1,1,2,2,0) \\
$s_{1231423145423145654234567}$ & $E_7$ &(1,1,1,1,1,3,3,0) \\
$s_{1234234542345654234567654234567}$ & $E_7$ &(1,1,1,2,2,2,2,0) \\
$s_{123142345423145654234567654234567}$ & $E_7$ &(1,1,1,1,1,4,4,0) \\
$s_{765432456713456245341324567134562453413245624534132453413241321}$ & $E_7$ &(1,2,2,2,2,2,2,0) \\
\hline
\end{tabular}
\end{center}
\caption{Smith normal form $S_w$ for $(\Id-w)$ with $w \in \sW$ non-cuspidal in type $E_8$ (continued from Table \ref{table:SNFE8-noncusp1}).}
\label{table:SNFE8-noncusp2}
\end{table}%

\begin{table}[htp]
\begin{center}
\begin{tabular}{|l|c|c|} 
\hline
Representative $w\in \wconj$ & Type of $\Pc(w)$ & $S_w = \diag(\cdots)$ \\ 
\hline \hline
$s_{12345678}$& $E_8$ &(1,1,1,1,1,1,1,1) \\
$s_{1234254678}$& $E_8$ &(1,1,1,1,1,1,1,1) \\
$s_{123425465478}$& $E_8$ &(1,1,1,1,1,1,1,1) \\
$s_{12342542345678}$& $E_8$ &(1,1,1,1,1,1,1,1) \\
$s_{1231423454657658}$& $E_8$ &(1,1,1,1,1,1,1,1) \\
$s_{1234234542345678}$& $E_8$ &(1,1,1,1,1,1,2,2) \\
$s_{123423454234565768}$& $E_8$ &(1,1,1,1,1,1,2,2) \\
$s_{12314234542365476548}$& $E_8$ &(1,1,1,1,1,1,1,1) \\
$s_{1234234542345654765876}$& $E_8$ &(1,1,1,1,1,1,2,2) \\
$s_{1231423454316542345678}$& $E_8$ &(1,1,1,1,1,1,1,1) \\
$s_{123142314542345654765876}$& $E_8$ &(1,1,1,1,1,1,1,1) \\
$s_{123142314354316542345678}$& $E_8$ &(1,1,1,1,1,1,2,2) \\
$s_{12314231454231456542345678}$& $E_8$ &(1,1,1,1,1,1,3,3) \\
$s_{12342354234654276542345678}$& $E_8$ &(1,1,1,1,1,1,2,2) \\
$s_{1231423145423145654234567687}$& $E_8$ &(1,1,1,1,1,1,3,3) \\
$s_{123142314354231465423476548765}$& $E_8$ &(1,1,1,1,1,1,2,2) \\
$s_{12342345423456542345676542345678}$& $E_8$ &(1,1,1,1,2,2,2,2) \\
$s_{1231423454231456542345676542345678}$& $E_8$ &(1,1,1,1,1,1,4,4) \\
$s_{1231423143542314565423145676542345678765}$& $E_8$ &(1,1,1,1,1,1,1,1) \\
$s_{123142314354231435426542314567654234567876}$& $E_8$ &(1,1,1,1,1,1,2,2) \\
$s_{12314231435423456542314567654231435465768765}$& $E_8$ &(1,1,1,1,2,2,2,2) \\
$s_{12314231454231436542314354265431765423456787}$& $E_8$ &(1,1,1,1,1,1,3,3) \\
$s_{1231423145423145654234567654234567876542345678}$& $E_8$ &(1,1,1,1,1,1,6,6) \\
$s_{1231423154231654317654231435426543176542345678}$& $E_8$ &(1,1,1,1,1,1,4,4) \\
$s_{123142314542314565423145676542314567876542345678}$& $E_8$ &(1,1,1,1,1,1,5,5) \\
$s_{123142314354231435426542345765423143548765423143542654765876}$& $E_8$ &(1,1,1,1,2,2,2,2) \\
$s_{1231423143542314354265423143542654317654231435426543176542345678}$& $E_8$ &(1,1,2,2,2,2,2,2) \\
$s_{123142314354231435426542314354265431765423143542654317876542345678}$& $E_8$ &(1,1,1,1,2,2,4,4) \\
$s_{1231423143542314356542314354267654231435426543178765423}$& $E_8$ &(1,1,1,1,3,3,3,3) \\
$\quad \quad \quad \cdot s_{435426543176542345678765}$&  & \\
$w_0 = s_{876543245671345624534132456787654324567134562453413245678765}$& $E_8$ &(2,2,2,2,2,2,2,2) \\
$\quad \quad \quad \cdot s_{432456713456245341324567134562453413245624534132453413241321}$&  & \\
\hline
\end{tabular}
\end{center}
\caption{Smith normal form $S_w$ for $(\Id-w)$ with $w \in \sW$ cuspidal in type $E_8$.}
\label{table:SNFE8-cusp}
\end{table}%

\subsection{Type $F_4$}

To obtain a system of minimal length representatives for the 25 conjugacy classes in $\sW$ of type $F_4$, we use Table B.3 in \cite{GeckPfeifferBook} to identify the 9 cuspidal classes.  We then iterate over all proper parabolic subgroups of $\sW$ and choose representatives for the cuspidal conjugacy classes in each parabolic, using the classification in \cite{GeckPfeifferBook} for types $A, B$, and $C$. See Table \ref{table:SNFF4} for a summary of our results in type $F_4$.

\begin{table}[htp]
\begin{center}
\begin{tabular}{|l|c|c|} 
\hline
Representative $w\in \wconj$ & Type of $\Pc(w)$ & $S_w = \diag(\cdots)$ \\ 
\hline \hline
1&1&(0,0,0,0) \\
$s_1$& $A_1$ &(1,0,0,0) \\
$s_3$& $A_1$ &(1,0,0,0) \\
$s_{12}$& $A_2$ &(1,1,0,0) \\
$s_{34}$& $A_2$ &(1,1,0,0) \\
$s_{13}$& $A_1 \times A_1$ &(1,1,0,0) \\
$s_{134}$& $A_1 \times A_2$ &(1,1,1,0) \\
$s_{124}$& $A_2 \times A_1$ &(1,1,1,0) \\
$s_{23}$& $B_2$ &(1,1,0,0) \\
$s_{3232}$& $B_2$ &(1,2,0,0) \\
$s_{123}$& $B_3$ &(1,1,1,0) \\
$s_{32321}$& $B_3$ &(1,1,2,0) \\
$s_{323212321}$& $B_3$ &(1,2,2,0) \\
$s_{234}$& $C_3$ &(1,1,1,0) \\
$s_{23234}$& $C_3$ &(1,1,2,0) \\
$s_{232343234}$& $C_3$ &(1,2,2,0) \\
$s_{1234}$& $F_4$ &(1,1,1,1) \\
$s_{123234}$& $F_4$ &(1,1,1,2) \\
$s_{12132343}$& $F_4$ &(1,1,1,1) \\
$s_{1232343234}$& $F_4$ &(1,1,2,2) \\
$s_{1213213234}$& $F_4$ &(1,1,2,2) \\
$s_{121321343234}$& $F_4$ &(1,1,2,2) \\
$s_{12132132343234}$& $F_4$ &(1,1,2,4) \\
$s_{1213213432132343}$& $F_4$ &(1,1,3,3) \\
$w_0 = s_{323212321432132343213234}$& $F_4$ &(2,2,2,2) \\
\hline
\end{tabular}
\end{center}
\caption{Smith normal form $S_w$ for $(\Id-w)$ with $w \in \sW$ of type $F_4$.}
\label{table:SNFF4}
\end{table}%

\subsection{Type $G_2$}

To obtain a system of minimal length representatives for the 6 conjugacy classes in $\sW$ of type $G_2$, we calculate the conjugacy classes directly by hand, due to the small size of the group. See Table \ref{table:SNFG2} for a summary of our results in type $G_2$.

\begin{table}[h]
\begin{center}
\begin{tabular}{|l|c|c|} 
\hline
Representative $w\in \wconj$ & Type of $\Pc(w)$ & $S_w = \diag(\cdots)$ \\ 
\hline \hline
1&1&(0,0) \\
$s_1$& $A_1$ &(1,0) \\
$s_2$& $A_1$ &(1,0) \\
$s_{12}$& $G_2$ &(1,1) \\
$s_{1212}$& $G_2$ &(1,3) \\
$w_0 = s_{121212}$& $G_2$ &(2,2) \\
\hline
\end{tabular}
\end{center}
\caption{Smith normal form $S_w$ for $(\Id-w)$ with $w \in \sW$ of type $G_2$.}
\label{table:SNFG2}
\end{table}%


\appendix

\section{Crystallographic groups and affine Coxeter groups}\label{app:cryst}

In this appendix we briefly discuss the relationship between split $n$-dimensional crystallographic groups $\aH = T_\aH \rtimes \sH$ and affine Coxeter groups. For any such $\aH$, the \emph{holohedry group} of the associated lattice $L_\aH$ is its setwise stabilizer in $\On$. Thus as $L_\aH$ is $\sH$-invariant, $\sH$ is a subgroup of the holohedry group of $L_\aH$.

In dimension $2$, it can be verified from the classification of wallpaper groups that every (split) wallpaper group is contained in an affine Coxeter group. More precisely, up to affine equivalence every $2$-dimensional lattice $L_\aH$ is either a square lattice or a hexagonal lattice. These lattices have holohedry groups the finite Weyl groups of types $C_2$ and $G_2$, respectively. Hence every $2$-dimensional split crystallographic group $\aH$ is contained in an affine Coxeter group of type $\tilde{C}_2$ or $\tilde{G}_2$.

Now let $\aH = T_\aH \rtimes \sH$ be a split $3$-dimensional crystallographic group. Then the lattice $L_\aH$ is, up to affine equivalence, a special case of one of the four ``maximal lattices" given by~\cite[Figure~2]{Grimmer}. From the information in~\cite{Grimmer}, it can be checked that the holohedry group of each of these four ``maximal lattices" is a finite Weyl group of type either $C_3$ or $A_1 \times G_2$, and that each such lattice is, up to affine equivalence, the coroot lattice of the same type as the corresponding Weyl group. It follows that in dimension $3$, any split crystallographic $\aH$ is contained in an affine Coxeter group of type either $\tilde{C}_3$ or $\tilde{A}_1 \times \tilde{G}_2$. 

In dimension $4$, however, there exist split crystallographic groups $\aH = T_\aH \rtimes \sH$ such that the reflections in the holohedry group of $L_\aH$ generate a proper subgroup of $\sH$ (see, for example,~\cite{Horvath}). Since any finite Weyl group preserving $L_\aH$ must be generated by reflections which preserve $L_\aH$, it follows that $\aH$ is not contained in any affine Coxeter group. Then for all dimensions $n \ge 5$, one can take direct products to obtain examples which are also not contained in affine Coxeter groups.


\section{Dynkin diagram labeling conventions}\label{app:dynkin}

We gather in Table~\ref{table:dynkin} several useful conventions for labeling the nodes of the Dynkin diagrams for finite Weyl groups. In the second column of this table, we give the labelings used in Plates I--IX of Bourbaki~\cite{Bourbaki4-6}, which are the same as the labelings in Sage \cite{sagemath}. These are the conventions we follow throughout this work. Note, however, that in Sage \cite{sagemath}, the Cartan matrices in types $B_n$ and $C_n$ are swapped relative to the Dynkin diagrams, so that to carry out calculations in type $B_n$ as labeled below, one should call type $C_n$, and vice versa.

In the third column, for the classical types we give the labelings used by Geck and Pfeiffer in Chapter 3 of~\cite{GeckPfeifferBook}, since their characterization of the conjugacy classes in these types uses their conventions in Chapter 3. Note that the labeling used in Chapter 3 of~\cite{GeckPfeifferBook} is different to that given in Table 1.2 of~\cite{GeckPfeifferBook} in types $B_n, C_n$, and $D_n$, however, due to the embedding of the Weyl group of type $D_n$ as a normal subgroup of index 2 in that of type~$B_n$. For the exceptional types, we give the common labelings used in both Table 1.2 and Appendix B of~\cite{GeckPfeifferBook}, which agree with those in Bourbaki \cite{Bourbaki4-6} and Sage \cite{sagemath}.

\vskip 10pt

\begin{table}[ht]
\begin{center}
\begin{tabular}{|l|c|c|c|} 
\hline
Type & Bourbaki~\cite{Bourbaki4-6} \& Sage \cite{sagemath} & Geck and Pfeiffer~\cite{GeckPfeifferBook}  \\ 
\hline \hline
$A_n$, $n \geq 1$ && \\
&
\begin{tikzpicture}[scale=0.7,double distance=3pt,thick]
\begin{scope}[scale=1.2] 
\fill (0,0) circle (3pt); 
\fill (1,0) circle (3pt); 
\fill (3,0) circle (3pt); 
\fill (4,0) circle (3pt); 
\draw (0,0)--(1,0)--(1.5,0);
\draw [dashed] (1.5,0)--(2.5,0);
\draw (2.5,0)--(3,0)--(4,0);
\draw (0,-0.3) node {$s_1$};	
\draw (1,-0.3) node {$s_2$};	
\draw (3,-0.3) node {$s_{n-1}$};	
\draw (4,-0.3) node {$s_n$};		
\end{scope}
\end{tikzpicture}
 & 
\begin{tikzpicture}[scale=0.7,double distance=3pt,thick]
\begin{scope}[scale=1.2] 
\fill (0,0) circle (3pt); 
\fill (1,0) circle (3pt); 
\fill (3,0) circle (3pt); 
\fill (4,0) circle (3pt); 
\draw (0,0)--(1,0)--(1.5,0);
\draw [dashed] (1.5,0)--(2.5,0);
\draw (2.5,0)--(3,0)--(4,0);
\draw (0,-0.3) node {$s_1$};	
\draw (1,-0.3) node {$s_{2}$};	
\draw (3,-0.3) node {$s_{n-1}$};	
\draw (4,-0.3) node {$s_n$};		
\end{scope}
\end{tikzpicture}
 \\
\hline
$B_n$, $n \geq 2$ & &  \\
&
\begin{tikzpicture}[scale=0.7,double distance=3pt,thick]
\begin{scope}[scale=1.2] 
\fill (0,0) circle (3pt); 
\fill (1,0) circle (3pt); 
\fill (3,0) circle (3pt); 
\fill (4,0) circle (3pt); 
\fill (5,0) circle (3pt); 
\draw (0,0)--(1,0)--(1.5,0);
\draw [dashed] (1.5,0)--(2.5,0);
\draw (2.5,0)--(3,0)--(4,0);	
\draw (4,0.05)--(5,0.05);
\draw (4,-0.05)--(5,-0.05);
\draw (4.4,0.15)--(4.55,0);
\draw (4.55,0)--(4.4,-0.15);
\draw (0,-0.3) node {$s_1$};	
\draw (1,-0.3) node {$s_2$};	
\draw (3,-0.3) node {$s_{n-2}$};	
\draw (4,-0.3) node {$s_{n-1}$};	
\draw (5,-0.3) node {$s_n$};
\end{scope}
\end{tikzpicture}
&
\begin{tikzpicture}[scale=0.7,double distance=3pt,thick]
\begin{scope}[scale=1.2] 
\fill (0,0) circle (3pt); 
\fill (1,0) circle (3pt); 
\fill (3,0) circle (3pt); 
\fill (4,0) circle (3pt); 
\fill (5,0) circle (3pt); 
\draw (0,0)--(1,0)--(1.5,0);
\draw [dashed] (1.5,0)--(2.5,0);
\draw (2.5,0)--(3,0)--(4,0);	
\draw (4,0.05)--(5,0.05);
\draw (4,-0.05)--(5,-0.05);
\draw (4.4,0.15)--(4.55,0);
\draw (4.55,0)--(4.4,-0.15);
\draw (0,-0.3) node {$s_{n-1}$};	
\draw (1,-0.3) node {$s_{n-2}$};	
\draw (3,-0.3) node {$s_2$};	
\draw (4,-0.3) node {$s_1$};	
\draw (5,-0.3) node {$s_0$};
\end{scope}
\end{tikzpicture}
\\
\hline
$C_n$, $n \geq 2$ & &  \\
&
\begin{tikzpicture}[scale=0.7,double distance=3pt,thick]
\begin{scope}[scale=1.2] 
\fill (0,0) circle (3pt); 
\fill (1,0) circle (3pt); 
\fill (3,0) circle (3pt); 
\fill (4,0) circle (3pt); 
\fill (5,0) circle (3pt); 
\draw (0,0)--(1,0)--(1.5,0);
\draw [dashed] (1.5,0)--(2.5,0);
\draw (2.5,0)--(3,0)--(4,0);	
\draw (4,0.05)--(5,0.05);
\draw (4,-0.05)--(5,-0.05);
\draw (4.6,0.15)--(4.45,0);
\draw (4.45,0)--(4.6,-0.15);
\draw (0,-0.3) node {$s_1$};	
\draw (1,-0.3) node {$s_2$};	
\draw (3,-0.3) node {$s_{n-2}$};	
\draw (4,-0.3) node {$s_{n-1}$};	
\draw (5,-0.3) node {$s_n$};
\end{scope}
\end{tikzpicture}
&
\begin{tikzpicture}[scale=0.7,double distance=3pt,thick]
\begin{scope}[scale=1.2] 
\fill (0,0) circle (3pt); 
\fill (1,0) circle (3pt); 
\fill (3,0) circle (3pt); 
\fill (4,0) circle (3pt); 
\fill (5,0) circle (3pt); 
\draw (0,0)--(1,0)--(1.5,0);
\draw [dashed] (1.5,0)--(2.5,0);
\draw (2.5,0)--(3,0)--(4,0);	
\draw (4,0.05)--(5,0.05);
\draw (4,-0.05)--(5,-0.05);
\draw (4.6,0.15)--(4.45,0);
\draw (4.45,0)--(4.6,-0.15);
\draw (0,-0.3) node {$s_{n-1}$};	
\draw (1,-0.3) node {$s_{n-2}$};	
\draw (3,-0.3) node {$s_2$};	
\draw (4,-0.3) node {$s_1$};	
\draw (5,-0.3) node {$s_0$};
\end{scope}
\end{tikzpicture}
\\
\hline
$D_n$, $n \geq 4$ & &  \\
&
\begin{tikzpicture}[scale=0.7,double distance=3pt,thick]
\begin{scope}[scale=1.2] 
\fill (0,0) circle (3pt); 
\fill (1,0) circle (3pt); 
\fill (3,0) circle (3pt); 
\fill (4,0) circle (3pt); 
\fill (5,0.5) circle (3pt); 
\fill (5,-0.5) circle (3pt); 
\draw (0,0)--(1,0)--(1.5,0);
\draw [dashed] (1.5,0)--(2.5,0);
\draw (2.5,0)--(3,0)--(4,0); 
\draw (4,0)--(5,0.5);
\draw (4,0)--(5,-0.5);
\draw (0,-0.3) node {$s_1$};	
\draw (1,-0.3) node {$s_2$};	
\draw (3,-0.3) node {$s_{n-3}$};	
\draw (4,-0.3) node {$s_{n-2}$};	
\draw (5.2,0.8) node {$s_{n-1}$};
\draw (5,-0.8) node {$s_{n}$};
\end{scope}
\end{tikzpicture}
&
\begin{tikzpicture}[scale=0.7,double distance=3pt,thick]
\begin{scope}[scale=1.2] 
\fill (0,0) circle (3pt); 
\fill (1,0) circle (3pt); 
\fill (3,0) circle (3pt); 
\fill (4,0) circle (3pt); 
\fill (5,0.5) circle (3pt); 
\fill (5,-0.5) circle (3pt); 
\draw (0,0)--(1,0)--(1.5,0);
\draw [dashed] (1.5,0)--(2.5,0);
\draw (2.5,0)--(3,0)--(4,0); 
\draw (4,0)--(5,0.5);
\draw (4,0)--(5,-0.5);
\draw (0,-0.3) node {$s_{n-1}$};	
\draw (1,-0.3) node {$s_{n-2}$};	
\draw (3,-0.3) node {$s_3$};	
\draw (4,-0.3) node {$s_2$};	
\draw (5,0.8) node {$s_0$};
\draw (5,-0.8) node {$s_1$};
\end{scope}
\end{tikzpicture}
 \\
\hline
$E_6$ & &  \\
&
\begin{tikzpicture}[scale=0.7,double distance=3pt,thick]
\begin{scope}[scale=1.2] 
\fill (0,0) circle (3pt); 
\fill (1,0) circle (3pt); 
\fill (2,0) circle (3pt); 
\fill (3,0) circle (3pt); 
\fill (4,0) circle (3pt); 
\fill (2,-1) circle (3pt); 
\draw (0,0)--(1,0)--(2,0)--(3,0)--(4,0);
\draw (2,0)--(2,-1);
\draw (0,-0.3) node {$s_1$};	
\draw (1,-0.3) node {$s_3$};
\draw (2.3,-0.3) node {$s_4$};
\draw (3,-0.3) node {$s_5$};
\draw (4,-0.3) node {$s_6$};
\draw (2.4,-1) node {$s_2$};	
\end{scope}
\end{tikzpicture}
&
\begin{tikzpicture}[scale=0.7,double distance=3pt,thick]
\begin{scope}[scale=1.2] 
\fill (0,0) circle (3pt); 
\fill (1,0) circle (3pt); 
\fill (2,0) circle (3pt); 
\fill (3,0) circle (3pt); 
\fill (4,0) circle (3pt); 
\fill (2,-1) circle (3pt); 
\draw (0,0)--(1,0)--(2,0)--(3,0)--(4,0);
\draw (2,0)--(2,-1);
\draw (0,-0.3) node {$s_1$};	
\draw (1,-0.3) node {$s_3$};
\draw (2.3,-0.3) node {$s_4$};
\draw (3,-0.3) node {$s_5$};
\draw (4,-0.3) node {$s_6$};
\draw (2.4,-1) node {$s_2$};	
\end{scope}
\end{tikzpicture}
\\
\hline
$E_7$ & &  \\
&
\begin{tikzpicture}[scale=0.7,double distance=3pt,thick]
\begin{scope}[scale=1.2] 
\fill (0,0) circle (3pt); 
\fill (1,0) circle (3pt); 
\fill (2,0) circle (3pt); 
\fill (3,0) circle (3pt); 
\fill (4,0) circle (3pt); 
\fill (5,0) circle (3pt); 
\fill (2,-1) circle (3pt); 
\draw (0,0)--(1,0)--(2,0)--(3,0)--(4,0)--(5,0);
\draw (2,0)--(2,-1);
\draw (0,-0.3) node {$s_1$};	
\draw (1,-0.3) node {$s_3$};
\draw (2.3,-0.3) node {$s_4$};
\draw (3,-0.3) node {$s_5$};
\draw (4,-0.3) node {$s_6$};
\draw (5,-0.3) node {$s_7$};
\draw (2.4,-1) node {$s_2$};	
\end{scope}
\end{tikzpicture}
&
\begin{tikzpicture}[scale=0.7,double distance=3pt,thick]
\begin{scope}[scale=1.2] 
\fill (0,0) circle (3pt); 
\fill (1,0) circle (3pt); 
\fill (2,0) circle (3pt); 
\fill (3,0) circle (3pt); 
\fill (4,0) circle (3pt); 
\fill (5,0) circle (3pt); 
\fill (2,-1) circle (3pt); 
\draw (0,0)--(1,0)--(2,0)--(3,0)--(4,0)--(5,0);
\draw (2,0)--(2,-1);
\draw (0,-0.3) node {$s_1$};	
\draw (1,-0.3) node {$s_3$};
\draw (2.3,-0.3) node {$s_4$};
\draw (3,-0.3) node {$s_5$};
\draw (4,-0.3) node {$s_6$};
\draw (5,-0.3) node {$s_7$};
\draw (2.4,-1) node {$s_2$};	
\end{scope}
\end{tikzpicture}
 \\
\hline
$E_8$ & &  \\
&
\begin{tikzpicture}[scale=0.7,double distance=3pt,thick]
\begin{scope}[scale=1] 
\fill (0,0) circle (3pt); 
\fill (1,0) circle (3pt); 
\fill (2,0) circle (3pt); 
\fill (3,0) circle (3pt); 
\fill (4,0) circle (3pt); 
\fill (5,0) circle (3pt); 
\fill (6,0) circle (3pt); 
\fill (2,-1) circle (3pt); 
\draw (0,0)--(1,0)--(2,0)--(3,0)--(4,0)--(5,0)--(6,0);
\draw (2,0)--(2,-1);
\draw (0,-0.3) node {$s_1$};	
\draw (1,-0.3) node {$s_3$};
\draw (2.3,-0.3) node {$s_4$};
\draw (3,-0.3) node {$s_5$};
\draw (4,-0.3) node {$s_6$};
\draw (5,-0.3) node {$s_7$};
\draw (6,-0.3) node {$s_8$};
\draw (2.4,-1) node {$s_2$};	
\end{scope}
\end{tikzpicture}
&
\begin{tikzpicture}[scale=0.7,double distance=3pt,thick]
\begin{scope}[scale=1] 
\fill (0,0) circle (3pt); 
\fill (1,0) circle (3pt); 
\fill (2,0) circle (3pt); 
\fill (3,0) circle (3pt); 
\fill (4,0) circle (3pt); 
\fill (5,0) circle (3pt); 
\fill (6,0) circle (3pt); 
\fill (2,-1) circle (3pt); 
\draw (0,0)--(1,0)--(2,0)--(3,0)--(4,0)--(5,0)--(6,0);
\draw (2,0)--(2,-1);
\draw (0,-0.3) node {$s_1$};	
\draw (1,-0.3) node {$s_3$};
\draw (2.3,-0.3) node {$s_4$};
\draw (3,-0.3) node {$s_5$};
\draw (4,-0.3) node {$s_6$};
\draw (5,-0.3) node {$s_7$};
\draw (6,-0.3) node {$s_8$};
\draw (2.4,-1) node {$s_2$};	
\end{scope}
\end{tikzpicture}
 \\
\hline
$F_4$ & &  \\
&
\begin{tikzpicture}[scale=0.7,double distance=3pt,thick]
\begin{scope}[scale=1.2] 
\fill (0,0) circle (3pt); 
\fill (1,0) circle (3pt); 
\fill (2,0) circle (3pt); 
\fill (3,0) circle (3pt); 
\draw (0,0)--(1,0);
\draw (2,0)--(3,0);
\draw (1,0.05)--(2,0.05);
\draw (1,-0.05)--(2,-0.05);
\draw (1.4,0.15)--(1.55,0);
\draw (1.55,0)--(1.4,-0.15);
\draw (0,-0.3) node {$s_1$};	
\draw (1,-0.3) node {$s_2$};
\draw (2,-0.3) node {$s_3$};
\draw (3,-0.3) node {$s_4$};
\end{scope}
\end{tikzpicture}
&
\begin{tikzpicture}[scale=0.7,double distance=3pt,thick]
\begin{scope}[scale=1.2] 
\fill (0,0) circle (3pt); 
\fill (1,0) circle (3pt); 
\fill (2,0) circle (3pt); 
\fill (3,0) circle (3pt); 
\draw (0,0)--(1,0);
\draw (2,0)--(3,0);
\draw (1,0.05)--(2,0.05);
\draw (1,-0.05)--(2,-0.05);
\draw (1.4,0.15)--(1.55,0);
\draw (1.55,0)--(1.4,-0.15);
\draw (0,-0.3) node {$s_1$};	
\draw (1,-0.3) node {$s_2$};
\draw (2,-0.3) node {$s_3$};
\draw (3,-0.3) node {$s_4$};
\end{scope}
\end{tikzpicture}
 \\
\hline
$G_2$ & &  \\
&
\begin{tikzpicture}[scale=0.7,double distance=3pt,thick]
\begin{scope}[scale=1.2] 
\fill (0,0) circle (3pt); 
\fill (1,0) circle (3pt); 
\draw (0,0.07)--(1,0.07);
\draw (0,0)--(1,0);
\draw (0,-0.07)--(1,-0.07);
\draw (0.6,0.15)--(0.45,0);
\draw (0.45,0)--(0.6,-0.15);
\draw (0,-0.3) node {$s_1$};	
\draw (1,-0.3) node {$s_2$};
\end{scope}
\end{tikzpicture}
&
\begin{tikzpicture}[scale=0.7,double distance=3pt,thick]
\begin{scope}[scale=1.2] 
\fill (0,0) circle (3pt); 
\fill (1,0) circle (3pt); 
\draw (0,0.07)--(1,0.07);
\draw (0,0)--(1,0);
\draw (0,-0.07)--(1,-0.07);
\draw (0.4,0.15)--(0.55,0);
\draw (0.55,0)--(0.4,-0.15);
\draw (0,-0.3) node {$s_1$};	
\draw (1,-0.3) node {$s_2$};
\end{scope}
\end{tikzpicture}
 \\
\hline
\end{tabular}
\end{center}
\caption{Dynkin diagram labeling conventions}
\label{table:dynkin}
\end{table}%

\clearpage 


\renewcommand{\refname}{Bibliography}
\bibliography{bibliographyConj}

\begin{thebibliography}{BKOP14}

\bibitem[AHN21]{AHN}
Jeffrey Adams, Xuhua He, and Sian Nie.
\newblock Partial orders on conjugacy classes in the {W}eyl group and on
  unipotent conjugacy classes.
\newblock {\em Adv. Math.}, 383:Paper No. 107688, 28, 2021.

\bibitem[All13]{AllcockCentralizers}
Daniel Allcock.
\newblock Reflection centralizers in {C}oxeter groups.
\newblock {\em Transform. Groups}, 18(3):599--613, 2013.

\bibitem[BB05]{BjoernerBrenti}
A.~Bj{\"o}rner and F.~Brenti.
\newblock {\em Combinatorics of {C}oxeter groups}, volume 231 of {\em Graduate
  Texts in Mathematics}.
\newblock Springer, New York, 2005.

\bibitem[BKOP14]{BenkartKangOhPark}
Georgia Benkart, Seok-Jin Kang, Se-jin Oh, and Euiyong Park.
\newblock Construction of irreducible representations over
  {K}hovanov-{L}auda-{R}ouquier algebras of finite classical type.
\newblock {\em Int. Math. Res. Not. IMRN}, (5):1312--1366, 2014.

\bibitem[Blo89]{Blokhina}
A.~P. Blokhina.
\newblock On the centralizer of a {C}oxeter element.
\newblock {\em Vestnik Moskov. Univ. Ser. I Mat. Mekh.}, (3):21--25, 103, 1989.

\bibitem[BM15]{BradyMcCammond}
Noel Brady and Jon McCammond.
\newblock Factoring {E}uclidean isometries.
\newblock {\em Internat. J. Algebra Comput.}, 25(1-2):325--347, 2015.

\bibitem[Bou02]{Bourbaki4-6}
N.~Bourbaki.
\newblock {\em {L}ie groups and {L}ie algebras. {C}hapters 4--6}.
\newblock Elements of Mathematics (Berlin). Springer Verlag, Berlin, 2002.
\newblock Translated from the 1968 French original by Andrew Pressley.

\bibitem[Bri96]{Brink}
Brigitte Brink.
\newblock On centralizers of reflections in {C}oxeter groups.
\newblock {\em Bull. London Math. Soc.}, 28(5):465--470, 1996.

\bibitem[Car72]{Carter72}
R.~W. Carter.
\newblock Conjugacy classes in the {W}eyl group.
\newblock {\em Compositio Math.}, 25:1--59, 1972.

\bibitem[CH17]{CiuHe}
Dan Ciubotaru and Xuhua He.
\newblock Cocenters and representations of affine {H}ecke algebras.
\newblock {\em J. Eur. Math. Soc. (JEMS)}, 19(10):3143--3177, 2017.

\bibitem[DE23]{DermenjianEvetts}
Aram Dermenjian and Alex Evetts.
\newblock Conjugacy class growth in virtually abelian groups, 2023.
\newblock \href {http://arxiv.org/abs/2309.06144} {\path{arXiv:2309.06144}}.

\bibitem[GKP00]{GKP}
Meinolf Geck, Sungsoon Kim, and G\"{o}tz Pfeiffer.
\newblock Minimal length elements in twisted conjugacy classes of finite
  {C}oxeter groups.
\newblock {\em J. Algebra}, 229(2):570--600, 2000.

\bibitem[GP93]{GeckPfeiffer}
M.~Geck and G.~Pfeiffer.
\newblock On the irreducible characters of {H}ecke algebras.
\newblock {\em Adv. Math.}, 102(1):79--94, 1993.

\bibitem[GP00]{GeckPfeifferBook}
Meinolf Geck and G\"{o}tz Pfeiffer.
\newblock {\em Characters of finite {C}oxeter groups and {I}wahori-{H}ecke
  algebras}, volume~21 of {\em London Mathematical Society Monographs. New
  Series}.
\newblock The Clarendon Press, Oxford University Press, New York, 2000.

\bibitem[Gri15]{Grimmer}
Hans Grimmer.
\newblock Partial order among the 14 bravais types of lattices: basics and
  applications.
\newblock {\em Acta Cryst.}, A71:143--149, 2015.

\bibitem[HN12]{HeNie12}
X.~He and S.~Nie.
\newblock Minimal length elements of finite {C}oxeter groups.
\newblock {\em Duke Math. J.}, 161(15):2945--2967, 2012.

\bibitem[HN14]{HeNie14}
X.~He and S.~Nie.
\newblock Minimal length elements of extended affine {W}eyl groups.
\newblock {\em Compos. Math.}, 150(11):1903--1927, 2014.

\bibitem[HN15]{HeNie15}
Xuhua He and Sian Nie.
\newblock {$P$}-alcoves, parabolic subalgebras and cocenters of affine {H}ecke
  algebras.
\newblock {\em Selecta Math. (N.S.)}, 21(3):995--1019, 2015.

\bibitem[Hor06]{Horvath}
Eszter Horv\'{a}th.
\newblock On four-dimensional crystallographic groups.
\newblock {\em Acta Cryst.}, 4/2:391--404, 2006.

\bibitem[How80]{Howlett}
Robert~B. Howlett.
\newblock Normalizers of parabolic subgroups of reflection groups.
\newblock {\em J. London Math. Soc. (2)}, 21(1):62--80, 1980.

\bibitem[Hum72]{Humphreys}
J.~E. Humphreys.
\newblock {\em Introduction to {L}ie algebras and representation theory}.
\newblock Springer-Verlag, New York, 1972.
\newblock Graduate Texts in Mathematics, Vol. 9.

\bibitem[HW22]{HollenbachWegener}
Ruwen Hollenbach and Patrick Wegener.
\newblock The centralizer of a {C}oxeter element.
\newblock {\em Bull. Lond. Math. Soc.}, 54(2):682--693, 2022.

\bibitem[KPR11]{KonvalinkaPfeifferRoever}
Matja\v{z} Konvalinka, G\"{o}tz Pfeiffer, and Claas~E. R\"{o}ver.
\newblock A note on element centralizers in finite {C}oxeter groups.
\newblock {\em J. Group Theory}, 14(5):727--745, 2011.

\bibitem[Kra09]{Krammer}
Daan Krammer.
\newblock The conjugacy problem for {C}oxeter groups.
\newblock {\em Groups Geom. Dyn.}, 3(1):71--171, 2009.

\bibitem[LMPS19]{LMPS}
Joel~Brewster Lewis, Jon McCammond, T.~Kyle Petersen, and Petra Schwer.
\newblock Computing reflection length in an affine {C}oxeter group.
\newblock {\em Trans. Amer. Math. Soc.}, 371(6):4097--4127, 2019.

\bibitem[LS12]{LiebSeit}
Martin~W. Liebeck and Gary~M. Seitz.
\newblock {\em Unipotent and nilpotent classes in simple algebraic groups and
  {L}ie algebras}, volume 180 of {\em Mathematical Surveys and Monographs}.
\newblock American Mathematical Society, Providence, RI, 2012.

\bibitem[Lus11]{LuszRT11}
G.~Lusztig.
\newblock From conjugacy classes in the {W}eyl group to unipotent classes.
\newblock {\em Represent. Theory}, 15:494--530, 2011.

\bibitem[Mar21]{Marquis21}
Timoth\'ee Marquis.
\newblock Cyclically reduced elements in {C}oxeter groups.
\newblock {\em Ann. Sci. \'Ec. Norm. Sup\'er. (4)}, 54(2):483--502, 2021.

\bibitem[Mar23]{Marquis23}
Timothée Marquis.
\newblock Structure of conjugacy classes in coxeter groups, 2023.
\newblock \href {http://arxiv.org/abs/2012.11015} {\path{arXiv:2012.11015}}.

\bibitem[MST23]{MST2}
Elizabeth Mili\'cevi\'c, Petra Schwer, and Anne Thomas.
\newblock Affine {D}eligne-{L}usztig varieties and folded galleries governed by
  chimneys.
\newblock {\em Ann. Inst. Fourier (Grenoble)}, 73(6):2469--2541, 2023.

\bibitem[MST24]{MST4}
Elizabeth Mili{\'c}evi{\'c}, Petra Schwer, and Anne Thomas.
\newblock The geometry of conjugation in {E}uclidean isometry groups, 2024.
\newblock \href {http://arxiv.org/abs/2407.08078} {\path{arXiv:2407.08078}}.

\bibitem[Ric82]{Richardson}
R.~W. Richardson.
\newblock Conjugacy classes of involutions in {C}oxeter groups.
\newblock {\em Bull. Austral. Math. Soc.}, 26(1):1--15, 1982.

\bibitem[{Sag}24]{sagemath}
{Sage Developers}.
\newblock {\em {S}ageMath, the {S}age {M}athematics {S}oftware {S}ystem
  ({V}ersion 10.3)}, 2024.
\newblock {\tt https://www.sagemath.org}.

\bibitem[Spa82]{Spalt}
Nicolas Spaltenstein.
\newblock {\em Classes unipotentes et sous-groupes de {B}orel}, volume 946 of
  {\em Lecture Notes in Mathematics}.
\newblock Springer-Verlag, Berlin, 1982.

\bibitem[Spe37]{Specht37}
Wilhelm Specht.
\newblock Darstellungstheorie der {H}yperoktaedergruppe.
\newblock {\em Math. Z.}, 42(1):629--640, 1937.

\bibitem[Spr74]{Springer}
T.~A. Springer.
\newblock Regular elements of finite reflection groups.
\newblock {\em Invent. Math.}, 25:159--198, 1974.

\bibitem[SRS24]{SantosRegoSchwer}
Yuri Santos~Rego and Petra Schwer.
\newblock The galaxy of {C}oxeter groups.
\newblock {\em J. Algebra}, 656:406--445, 2024.

\bibitem[SS70]{SprStein}
T.~A. Springer and R.~Steinberg.
\newblock Conjugacy classes.
\newblock In {\em Seminar on {A}lgebraic {G}roups and {R}elated {F}inite
  {G}roups ({T}he {I}nstitute for {A}dvanced {S}tudy, {P}rinceton, {N}.{J}.,
  1968/69)}, Lecture Notes in Mathematics, Vol. 131, pages 167--266. Springer,
  Berlin, 1970.

\bibitem[You30]{Young30}
Alfred Young.
\newblock On {Q}uantitative {S}ubstitutional {A}nalysis ({F}ifth {P}aper).
\newblock {\em Proc. London Math. Soc. (2)}, 31(4):273--288, 1930.

\end{thebibliography}
\bibliographystyle{alphaurl}

\typeout{get arXiv to do 4 passes: Label(s) may have changed. Rerun} 

\end{document}